\documentclass[a4paper,reqno]{amsart}

\pdfoutput=1

\usepackage{nicefrac,amsmath,amssymb,amsthm,url,verbatim,enumitem,stmaryrd,mathtools,microtype,bbm}
\usepackage{thmtools}
\usepackage[utf8]{inputenc}
\usepackage[T1]{fontenc}
\usepackage{libertine}
\usepackage[libertine,cmintegrals,cmbraces]{newtxmath}
\usepackage[dvipsnames,svgnames,x11names]{xcolor}
\definecolor{darkred}{RGB}{139,0,0}
\definecolor{darkblue}{RGB}{0,0,139}
\definecolor{darkgreen}{RGB}{0,100,0}
\usepackage[mathscr]{euscript}
\usepackage[pagebackref]{hyperref}
\hypersetup{
  colorlinks   = true, 
  urlcolor     = darkred, 
  linkcolor    = darkred, 
  citecolor   = darkgreen}
\usepackage{tikz}
\usepackage{tikz-cd}
\usetikzlibrary{decorations.markings}
\usepackage{multicol}
\usepackage[capitalise,noabbrev]{cleveref}

\calclayout

\AtBeginDocument{%
   \def\MR#1{}
}

\setenumerate[1]{label=(\roman*),}
\setitemize[1]{label=\raisebox{0.25ex}{\tiny$\bullet$}}

\newcommand\smallsquare{{\mathbin{\text{\raise0.17ex\hbox{\scalebox{.7}{$\blacksquare$}}}}}}

\makeatletter
\newcommand{\mylabel}[2]{#2\def\@currentlabel{#2}\label{#1}}
\makeatother

\newcommand{\circled}[1]{\raisebox{.5pt}{\textcircled{\raisebox{-.9pt} {#1}}}}

\newcommand{\BO}{\ensuremath{\mathrm{BO}}}
\newcommand{\BSO}{\ensuremath{\mathrm{BSO}}}

\newcommand{\Fr}{\ensuremath{\mathrm{Fr}}}
\newcommand{\Diff}{{\ensuremath{\mathrm{Diff}}}}

\newcommand{\Homeo}{\ensuremath{\mathrm{Homeo}}}

\newcommand{\BHomeo}{\ensuremath{\mathrm{BHomeo}}}

\newcommand{\BhAut}{\ensuremath{\mathrm{BhAut}}}

\newcommand{\Map}{\ensuremath{\mathrm{Map}}}

\newcommand{\Emb}{\ensuremath{\mathrm{Emb}}}
\newcommand{\Imm}{\ensuremath{\mathrm{Imm}}}

\newcommand{\BAut}{\ensuremath{\mathrm{BAut}}}

\newcommand{\BTop}{\ensuremath{\mathrm{BTop}}}

\newcommand{\BSpin}{\ensuremath{\mathrm{BSpin}}}

\newcommand{\E}{\ensuremath{\mathrm{E}}}

\newcommand{\BC}{\ensuremath{\mathrm{BC}}}

\newcommand{\interior}{\ensuremath{\mathrm{int}}}

\newcommand{\inc}{\ensuremath{\mathrm{inc}}}
\newcommand{\ev}{\ensuremath{\mathrm{ev}}}
\newcommand{\res}{\ensuremath{\mathrm{res}}}

\newcommand{\inj}{\ensuremath{\mathrm{inj}}}
\newcommand{\op}{\mathrm{op}}
\newcommand{\red}{\mathrm{red}}

\newcommand{\geo}{\mathrm{geo}}

\newcommand{\nsur}{\mathrm{nsur} }
\newcommand{\ninj}{\mathrm{ninj}}

\DeclareMathAlphabet{\mathpzc}{OT1}{pzc}{m}{it}
\newcommand{\catsingle}[1]{\ensuremath{\mathscr{#1}}}
\newcommand{\cat}[1]{\ensuremath{\mathsf{#1}}}
\newcommand{\icat}[1]{\ensuremath{\mathscr{#1}}}

\newcommand{\oH}{\ensuremath{\mathrm{H}}}

\newcommand{\oO}{\ensuremath{\mathrm{O}}}
\newcommand{\oB}{\ensuremath{\mathrm{B}}}

\newcommand{\oP}{\ensuremath{\mathrm{P}}}
\newcommand{\ot}{t}
\newcommand{\oo}{o}

\newcommand{\bfC}{\ensuremath{\mathbf{C}}}

\newcommand{\bfN}{\ensuremath{\mathbf{N}}}

\newcommand{\bfR}{\ensuremath{\mathbf{R}}}
\newcommand{\bfZ}{\ensuremath{\mathbf{Z}}}

\newcommand{\cA}{\ensuremath{\catsingle{A}}}
\newcommand{\cB}{\ensuremath{\catsingle{B}}}
\newcommand{\cC}{\ensuremath{\catsingle{C}}}
\newcommand{\cD}{\ensuremath{\catsingle{D}}}
\newcommand{\cE}{\ensuremath{\catsingle{E}}}
\newcommand{\cF}{\ensuremath{\catsingle{F}}}
\newcommand{\cG}{\ensuremath{\catsingle{G}}}

\newcommand{\cM}{\ensuremath{\catsingle{M}}}

\newcommand{\cO}{\ensuremath{\catsingle{O}}}
\newcommand{\cP}{\ensuremath{\catsingle{P}}}

\newcommand{\cR}{\ensuremath{\catsingle{R}}}
\newcommand{\cS}{\ensuremath{\catsingle{S}}}

\newcommand{\cU}{\ensuremath{\catsingle{U}}}

\newcommand{\CM}{\ensuremath{\catsingle{CM}}}

\newcommand{\PSh}{\ensuremath{\mathrm{PSh}}}

\newcommand{\colim}{\ensuremath{\mathrm{colim}}}

\newcommand{\holim}{\ensuremath{\mathrm{holim}}}

\newcommand{\laxlim}{\ensuremath{\mathrm{laxlim}}}

\newcommand{\ra}{\rightarrow}
\newcommand{\la}{\leftarrow}
\newcommand{\lra}{\longrightarrow}

\newcommand{\xra}[1]{\xrightarrow{#1}}
\newcommand{\xsra}[1]{\overset{#1}{\rightarrow}}
\newcommand{\xsla}[1]{\overset{#1}{\leftarrow}}
\newcommand{\xlra}[1]{\overset{#1}{\longrightarrow}}
\newcommand{\xlratwo}[2]{\smashoperator[l]{\mathop{\longrightarrow}_{#2}^{#1}}}

\newcommand{\xlhra}[1]{\overset{#1}{\longhookrightarrow}}
\newcommand{\xlla}[1]{\overset{#1}{\longleftarrow}}
\newcommand{\longhookrightarrow}{\lhook\joinrel\longrightarrow}

\newcommand{\Aut}{\mathrm{Aut}}

\newcommand{\End}{\mathrm{End}}

\newcommand{\Sym}{\mathrm{Sym}}

\renewcommand{\Top}{\mathrm{Top}}
\newcommand{\ext}{\mathrm{ext}}

\newcommand{\im}{\mathrm{im}}

\newcommand{\id}{\mathrm{id}}

\newcommand{\pr}{\mathrm{pr}}

\newcommand{\fib}{\mathrm{fib}}

\newcommand{\fr}{\mathrm{fr}}

\newcommand{\lax}{\mathrm{lax}}
\newcommand{\cgr}{\mathrm{cgr}}
\newcommand{\grtp}{\mathrm{grtp}}
\newcommand{\oplax}{\mathrm{oplax}}
\newcommand{\holink}{\mathrm{holink}}

\newcommand{\FM}{\mathrm{FM}}

\newcommand{\Barc}{\mathrm{Bar}}

\newcommand{\BG}{\mathrm{B}G}

\newcommand{\cOmega}{\smash{\overline{\Omega}}}

\newcommand{\Disc}{\ensuremath{\cat{Disc}}}
\newcommand{\DiscInf}{\ensuremath{\icat{D}\mathrm{isc}}}
\newcommand{\Man}{\ensuremath{\cat{Man}}}

\newcommand{\ManInf}{\ensuremath{\icat{M}\mathrm{an}}}
\newcommand{\Fin}{\ensuremath{\mathrm{Fin}}}
\newcommand{\Fun}{\ensuremath{\mathrm{Fun}}}

\newcommand{\ALG}{\ensuremath{\mathrm{ALG}}}
\newcommand{\Alg}{\ensuremath{\mathrm{Alg}}}
\newcommand{\Mon}{\ensuremath{\mathrm{Mon}}}
\newcommand{\CMon}{\ensuremath{\mathrm{CMon}}}

\newcommand{\BordInf}{\ensuremath{\icat{B}\mathrm{ord}}}

\newcommand{\ncBordInf}{\ensuremath{\mathrm{nc}\icat{B}\mathrm{ord}}}

\newcommand{\BM}{\ensuremath{\icat{BM}}}
\newcommand{\RM}{\ensuremath{\icat{RM}}}
\newcommand{\LM}{\ensuremath{\icat{LM}}}

\newcommand{\Kan}{\ensuremath{\cat{Kan}}}
\newcommand{\cTop}{\ensuremath{\cat{Top}}}

\newcommand{\BMod}{\ensuremath{\mathrm{BMod}}}
\newcommand{\Cocart}{\ensuremath{\mathrm{Cocart}}}
\newcommand{\Cart}{\ensuremath{\mathrm{Cart}}}
\newcommand{\Assoc}{\ensuremath{\icat{A}\mathrm{ssoc}}}
\newcommand{\Ass}{\ensuremath{\mathrm{Ass}}}
\newcommand{\Cat}{\ensuremath{\icat{C}\mathrm{at}}}
\newcommand{\twoCat}{\ensuremath{\mathbb{C}\mathrm{at}}}

\newcommand{\CCat}{\ensuremath{\mathrm{Cat}}}

\newcommand{\Cosp}{\ensuremath{\mathrm{Cosp}}}

\newcommand{\ModInf}{\ensuremath{\mathpzc{r}\icat{M}\mathrm{od}}}
\newcommand{\RMod}{\ensuremath{\mathrm{rMod}}}
\newcommand{\RRMod}{\ensuremath{\mathrm{RMod}}}
\newcommand{\LLMod}{\ensuremath{\mathrm{LMod}}}
\newcommand{\oMod}{\ensuremath{\mathrm{Mod}}}

\newcommand{\Opd}{\ensuremath{{\icat{O}\mathrm{p}}}}
\newcommand{\Com}{\ensuremath{{\icat{C}\mathrm{om}}}}

\newcommand{\st}{\ensuremath{\mathrm{st}}}

\newcommand{\Mul}{\ensuremath{\mathrm{Mul}}}

\newcommand{\col}{\ensuremath{\mathrm{col}}}
\newcommand{\const}{\ensuremath{\mathrm{const}}}
\newcommand{\Env}{\ensuremath{\mathrm{Env}}}
\newcommand{\Fam}{\ensuremath{\mathrm{Fam}}}
\newcommand{\un}{\ensuremath{\mathrm{un}}}

\newcommand{\gn}{\ensuremath{\mathrm{gn}}}
\newcommand{\rf}{\ensuremath{\mathrm{rf}}}

\newcommand{\gc}{\ensuremath{\mathrm{gc}}}
\newcommand{\Tow}{\ensuremath{\mathrm{Tow}}}
\newcommand{\Assem}{\ensuremath{\mathrm{Assem}}}
\newcommand{\seg}{\ensuremath{\mathrm{seg}}}
\newcommand{\com}{\ensuremath{\mathrm{c}}}

\newcommand{\cl}{\ensuremath{\mathrm{cl}}}

\newcommand{\rec}{\ensuremath{\mathrm{rec}}}
\newcommand{\dec}{\ensuremath{\mathrm{dec}}}
\newcommand{\wc}{\ensuremath{\mathrm{wc}}}

\newcommand{\free}{\ensuremath{\mathrm{free}}}

\newcommand{\frep}{\ensuremath{\mathrm{frep}}}
\newcommand{\cyl}{\ensuremath{\mathrm{cyl}}}

\newtheorem{bigthm}{Theorem}

\newtheorem{thm}{Theorem}[section]

\newtheorem{lem}[thm]{Lemma}
\newtheorem{lemdfn}[thm]{Lemma and Definition}
\newtheorem{prop}[thm]{Proposition}
\newtheorem{cor}[thm]{Corollary}

\theoremstyle{definition}
\newtheorem{dfn}[thm]{Definition}

\newtheorem*{nconvention}{Convention}
\theoremstyle{remark}
\newtheorem{ex}[thm]{Example}

\newtheorem{rem}[thm]{Remark}
\newtheorem*{nrem}{Remark}

\newcommand{\ul}[1]{\underline{#1}}

\begin{document}

\title{Infinity-operadic foundations for embedding calculus}
\author{Manuel Krannich}
\address{Department of Mathematics, Karlsruhe Institute of Technology, 76131 Karlsruhe, Germany}
\email{krannich@kit.edu}
\author{Alexander Kupers}
\address{Department of Computer and Mathematical Sciences, University of Toronto Scarborough, 1265 Military Trail, Toronto, ON M1C 1A4, Canada}
\email{a.kupers@utoronto.ca}

\begin{abstract}
Motivated by applications to spaces of embeddings and automorphisms of manifolds, we consider a tower of $\infty$-categories of truncated right-modules over a unital $\infty$-operad $\cO$. We study monoidality and naturality properties of this tower, identify its layers, describe the difference between the towers as $\cO$ varies, and generalise these results to the level of Morita $(\infty,2)$-categories. Applied to the $\BO(d)$-framed $E_d$-operad, this extends Goodwillie--Weiss' embedding calculus and its layer identification to the level of bordism categories. Applied to other variants of the $E_d$-operad, it yields new versions of embedding calculus, such as one for topological embeddings---based on $\BTop(d)$---or one similar to Boavida de Brito--Weiss' configuration categories---based on $\BAut(E_d)$. In addition, we prove a delooping result in the context of embedding calculus, establish a convergence result for topological embedding calculus, improve upon the smooth convergence result of Goodwillie, Klein, and Weiss, and deduce an Alexander trick for homology $4$-spheres. 
\end{abstract}

\maketitle
\thispagestyle{empty}

\vspace{-0.27cm}

\enlargethispage{0.36cm}

The study of smooth embeddings between smooth manifolds $M$ and $N$, and families thereof, plays a central role in geometric topology. From a homotopy-theoretic point of view, it corresponds to the study of the homotopy type of the space $\Emb^{\oo}(M,N)$ of smooth embeddings, equipped with the smooth topology. More generally, one often considers spaces $\Emb^{\oo}_{\partial_0}(M,N)$ of embeddings that extend a given embedding $e_{\partial_0}\colon \partial_0 M\hookrightarrow \partial N$ on a fixed submanifold $\partial_0M\subset \partial M$ of codimension $0$. There is an approach due to Goodwillie and Weiss \cite{WeissEmbeddings,GoodwillieWeiss} to analyse the homotopy types of such spaces of embeddings through their restriction maps to spaces of embeddings of submanifolds of $M$ diffeomorphic to a finite disjoint union of open discs. This goes under the name of \emph{embedding calculus} and takes the form of a tower of homotopy-theoretic approximations to $\Emb^{\oo}_{\partial_0}(M,N)$
\vspace{-0.35cm}
\begin{equation}\label{equ:embcalc-tower-intro}
	\begin{tikzcd}[row sep=0.2cm]
	&\vdots\dar\\
	&T_2\Emb^{\oo}_{\partial_0}(M,N)\dar\\
	&T_1\Emb^{\oo}_{\partial_0}(M,N)\dar\\
	\Emb^{\oo}_{\partial_0}(M,N)\rar\arrow[ur, bend left=20]\arrow[uur, bend left=20]\arrow[uuur, bend left=20,shift left=1]&T_0\Emb^{\oo}_{\partial_0}(M,N)&[-1cm]\simeq&[-1cm] \ast.\\
	\end{tikzcd}\vspace{-0.35cm}
\end{equation}
whose \emph{layers}---the fibres of the consecutive maps $T_k\Emb^{\oo}_{\partial_0}(M,N)\ra T_{k-1}\Emb^{\oo}_{\partial_0}(M,N)$---can be described explicitly in terms of the frame bundles and configuration spaces of $M$ and $N$; see \cite{WeissEmbeddings}. Goodwillie, Klein, and Weiss \cite{GoodwillieWeiss,GoodwillieKlein} proved that this tower often \emph{converges}, that is, the map \vspace{-0.05cm}
\begin{equation}\label{equ:emb-calc-intro}
	\Emb^{\oo}_{\partial_0}(M,N)\lra T_\infty\Emb^{\oo}_{\partial_0}(M,N)\coloneqq \holim_kT_k\Emb^{\oo}_{\partial_0}(M,N)\vspace{-0.05cm}
\end{equation}
to the limit of the tower is a weak homotopy equivalence in many cases. The combination of these results has led to a number of applications to spaces of embeddings, and more recently also to spaces of diffeomorphisms; see \cite{KupersFiniteness,KRWEvenDiscs,WeissDalian,GRWContractible,BKK} for examples. 

\medskip

\noindent Among the various constructions of the tower \eqref{equ:embcalc-tower-intro} \cite{WeissEmbeddings,GKWHaefliger,BoavidaWeiss,TurchinContextFree,AroneTurchin}, the one closest to this work is that by Boavida de Brito--Weiss. Let us briefly outline it, for simplicity in the case where $M$ and $N$ have no boundary and are of \smash{equal} dimension $d$. They consider the functor\vspace{-0.05cm}
\begin{equation}\label{equ:e-functor-single}
	\begin{aligned}
			E\colon \ManInf^{\oo}_d &\lra \PSh(\DiscInf^{\oo}_d) \\
			M &\longmapsto E_M \coloneqq \Emb^{\oo}(-,M)
	\end{aligned}\vspace{-0.05cm}
\end{equation}
from the $\infty$-category $\ManInf^{\oo}_d$ of smooth $d$-manifolds without boundary and spaces of embeddings between them, to the $\infty$-category of space-valued presheaves on the full subcategory $\smash{\DiscInf^{\oo}_d\subset\ManInf^{\oo}_d}$ of manifolds diffeomorphic to $S\times\bfR^d$ for finite sets $S$. On mapping spaces, this functor induces  \eqref{equ:emb-calc-intro}, and the tower \eqref{equ:embcalc-tower-intro} is obtained by combining it with the tower of $\infty$-categories induced by restriction \vspace{-0.05cm}
\begin{equation}\label{equ:intro-tower-discs}
	\PSh(\DiscInf_d^{\oo})=\PSh(\DiscInf_{d,\le \infty}^{\oo})\lra\cdots\lra \PSh(\DiscInf_{d,\le 2}^{\oo})\lra \PSh(\DiscInf_{d,\le 1}^{\oo}) \vspace{-0.05cm}
\end{equation}
where ${\DiscInf_{d,\le k}^{\oo}\subset \DiscInf_d^{\oo}}$ is the full subcategory on manifolds diffeomorphic to $S\times\bfR^d$ for $|S|\le k$.

\medskip

\noindent When studying spaces of embeddings, it is often useful to compare them for different choices of source and target manifold---e.g.\,by pre- or postcomposition or by gluing along the fixed part of the boundary---so it would be desirable to extend the tower \eqref{equ:embcalc-tower-intro} to coherently incorporate these comparison maps. As part of \cite{KKDisc}, we did this for the limit of the tower \eqref{equ:emb-calc-intro} in the equidimensional case. This took the form of a functor of symmetric monoidal $(\infty,2)$-categories
\vspace{-0.05cm}
\begin{equation}\label{equ:e-functor-tower}
	E\colon \ncBordInf^{\oo}(d)^{(\infty,2)}\lra \icat{M}\mathrm{od}^{\oo}(d)^{(\infty,2)}.\vspace{-0.05cm}
\end{equation}
whose source \emph{$\ncBordInf^{\oo}(d)^{(\infty,2)}$} is a bordism $(\infty,2)$-category with (potentially noncompact) smooth $(d-1)$-manifolds as objects, (potentially noncompact) smooth bordisms as  $1$-morphisms, and smooth embeddings between bordisms that are fixed on the boundary as $2$-morphisms. The target $\smash{\icat{M}\mathrm{od}^{\oo}(d)^{(\infty,2)}}$ is a $(\infty,2)$-Morita category of the $\infty$-category $\smash{\PSh(\DiscInf^{\oo}_d)}$ equipped with the Day convolution monoidal structure induced by disjoint union in $\DiscInf^{\oo}_d$: it has associative algebras in $\PSh(\DiscInf^{\oo}_d)$ as objects, bimodules as $1$-morphisms, and  maps of bimodules as $2$-morphisms. We refer to \cite{KKDisc} for a detailed description of the functor \eqref{equ:e-functor-tower}; for now it suffices to know, firstly, that it recovers \eqref{equ:e-functor-single} on the $\infty$-category of $1$-morphisms from the empty manifold to itself, and secondly---after replacing $M$ and $N$ up to isotopy equivalence with the nullbordisms $M'\coloneqq \interior(M)\cup\interior(\partial_0M)$ and $N'\coloneqq\interior(N)\cup\interior(e(\partial_0M))$ of $\interior(\partial_0 M)$---that it recovers the limit of the embedding calculus tower \eqref{equ:emb-calc-intro} for $\Emb^o_{\partial_0}(M,N)$ on the space of $2$-morphisms from $M'$ to $N'$ and thus equips the map \eqref{equ:emb-calc-intro} with coherent composition and gluing maps.
 
\medskip

\noindent However, with an eye towards future applications of embedding calculus, the construction of the functor \eqref{equ:e-functor-tower} from \cite{KKDisc} left several desiderata. For instance, one ought to
\begin{enumerate}[leftmargin=0.7cm]
\item\label{equ:desidarata-1} extend the functor \eqref{equ:e-functor-tower} to incorporate, firstly, the tower \eqref{equ:embcalc-tower-intro} and not just its limit \eqref{equ:emb-calc-intro}, secondly, the description of the layers of the tower, and thirdly, embedding calculus in positive codimension,
\item develop embedding calculus for (locally flat) topological embeddings between topological manifolds, and establish analogues in this setting, firstly, of the functor \eqref{equ:e-functor-tower}, secondly, of extensions as in \ref{equ:desidarata-1}, and thirdly, of Goodwillie, Klein, and Weiss' convergence result, 
\item\label{equ:desidarata-3} establish an analogue of parametrised smoothing theory (see \cite[Essay V]{KirbySiebenmann}) in the context of embedding calculus that describes the difference between the smooth and topological variants.
\end{enumerate}
In this work, we in particular establish \ref{equ:desidarata-1}-\ref{equ:desidarata-3}, by placing embedding calculus in a new framework based on the theory of $\infty$-operads that allows us to treat previous and new variants of embedding calculus---e.g.\,smooth and topological versions, or a version reminiscent of Boavida de Brito--Weiss' \emph{configuration categories} \cite{BoavidaWeissConfiguration}---as instances of the same construction, describe the layers in this generality, prove a smoothing theory result that describes the difference between the variants, and establish a convergence result for topological embedding calculus as well as improve upon the smooth one. We now outline these results in more detail, starting with the $\infty$-operadic setup.

\subsection*{A calculus for right-modules over an $\infty$-operad} Informally, an $\infty$-operad $\cO$ consists of a set of colours, spaces of multi-operations $\Mul_\cO((c_s)_{s\in S};c)$ from a collection of colours $(c_s)_{s\in S}$ indexed by a finite set $S$ to a colour $c$, and composition maps between them that satisfy the axioms of an ordinary coloured operad up to higher coherent homotopy. Restricting to collections consisting of a single colour, an $\infty$-operad $\cO$ has an underlying $\infty$-category $\cO^{\col}$ of colours. The latter is a full subcategory of a symmetric monoidal $\infty$-category $\Env(\cO)$ associated to $\cO$, called the \emph{monoidal envelope} (or \emph{associated \emph{PROP}}), with objects given by collections $(c_s)_{s\in S}$ as above, morphism spaces by $\smash{\Map_{\Env(\cO)}((c_s)_{s\in S},(d_t)_{t\in T})\simeq \sqcup_{\varphi\colon S\ra T}\sqcap_{t\in T}\Mul_{\cO}((c_s)_{s\in \varphi^{-1}(t)};d_t)}$, and monoidal structure given by disjoint union. The \emph{$\infty$-category of right-modules}\footnote{The terminology stems from the case of ordinary $1$-coloured operads $\cO$: viewing $\cO$ as a monoid in symmetric sequences with composition product, the notion of a right-module over this monoid specialises to a presheaf on the PROP of $\cO$.} 
over $\cO$ is the symmetric monoidal $\infty$-category
\[\RMod(\cO)\coloneqq \PSh(\Env(\cO))\] 
of presheaves on $\Env(\cO)$ with values in the $\infty$-category $\cS$ of spaces; the monoidal structure is inherited from $\Env(\cO)$ by Day convolution. This $\infty$-category fits into a tower of $\infty$-categories: writing $\Env(\cO)_{\le k}\subset \Env(\cO)$ for the full subcategory on the collections $(c_s)_{s\in S}$ with $|S|\le k$ and defining the \emph{$\infty$-category of $k$-truncated right-modules} as $\RMod_k(\cO)\coloneqq \PSh(\Env(\cO)_{\le k})$, the filtration $\Env(\cO)_{\le 1}\subset \Env(\cO)_{\le 2}\subset \cdots\subset \Env(\cO)$ induces by restriction a tower of $\infty$-categories \vspace{-0.05cm}
\begin{equation}\label{equ:tower-intro}
	\RMod(\cO)=\RMod_\infty(\cO)\lra \cdots\lra \RMod_2(\cO)\lra \RMod_1(\cO)\vspace{-0.05cm}
\end{equation}
that can be seen to specialise to \eqref{equ:intro-tower-discs} for a suitable choice of $\cO$; see below. This first main part of this work consists of a detailed analysis of this tower. Most of our results assume that the $\infty$-operad $\cO$ is \emph{unital}, i.e.\,its spaces of $0$-ary multi-operations $\Mul_{\cO}(\varnothing;c)$ are contractible for all colours $c$, and some of our results simplify when passing to the full subcategory $\RMod^{\un}_k(\cO)\subset \RMod_k(\cO)$ of right-modules that are \emph{unital}, i.e.\,their value at $\varnothing\in\Env(\cO)$ is contractible. Specialised to unital right-modules, our main results regarding this tower can be summarised as follows (see \cref{sec:operadic-framework}):

\begin{bigthm} \label{bigthm:calc}For any unital $\infty$-operad $\cO$, the tower of $\infty$-categories
	\begin{equation}\label{equ:tower-intro-un}\RMod^\un(\cO) = \RMod^\un_{\infty}(\cO) \lra \cdots\lra \RMod^\un_2(\cO)\lra\RMod^\un_1(\cO)\end{equation}
induced by \eqref{equ:tower-intro} satisfies the following properties:
\begin{enumerate}[leftmargin=3cm]
	\item[(Convergence)] The induced functor $\RMod^\un(\cO)\ra \lim_k\RMod^\un_k(\cO)$ is an equivalence.
	\item[(Monoidality)]The Day convolution monoidal structure on $\RMod(\cO)$ induces a lift of \eqref{equ:tower-intro-un} to a tower of symmetric monoidal $\infty$-categories.
	\item[(Naturality)] The tower \eqref{equ:tower-intro-un} and its symmetric monoidal lift can be made functorial in maps of unital $\infty$-operads via left Kan extension.
	\item[(First layer)] There is a natural equivalence of symmetric monoidal $\infty$-categories \vspace{-0.01cm}\begin{equation}\label{equ:first-layer-thmA-intro}\hspace{3cm}\RMod^\un_1(\cO) \simeq \PSh(\cO^\col)\vspace{-0.01cm}\end{equation} where $ \PSh(\cO^\col)$ is equipped with the cocartesian symmetric monoidal structure.
	\item[(Higher layers)] Writing $\cO^\col\wr \Sigma_k$ for the wreath product of the $\infty$-category $\cO^\col$ with the symmetric group on $k$ letters, there is a commutative diagram of $\infty$-categories for $1< k<\infty$ 
	\[\hspace{3cm}\begin{tikzcd}[column sep=1cm]
		\RMod^\un_k(\cO) \rar{\Lambda} \dar & \PSh(\cO^\col\wr \Sigma_k)^{[2]} \arrow[d,"(0\le 2)^*",pos=0.4] \rar{\colim} & \cS^{[2]} \arrow[d,"(0\le 2)^*",pos=0.4] \\[-3pt]
		\RMod^\un_{k-1}(\cO) \rar{\Omega} & \PSh(\cO^\col\wr \Sigma_k)^{[1]}\rar{\colim}& \cS^{[1]}\end{tikzcd}\]
		whose left square is cartesian. Moreover, if $\cO^\col$ is an $\infty$-groupoid, then the right square is cartesian as well, and hence so is the outer one. Moreover, this diagram can be made functorial in operadic right-fibrations via left Kan extension.
	\item[(Smoothing theory)]For an operadic right-fibration $\varphi\colon \cO\ra \cU$ and $1\le k\le \infty$, the square 
	\[\hspace{3cm}\begin{tikzcd}[]
	\RMod^\un_k(\cO)\rar{\varphi_!}\dar&\RMod^\un_k(\cU)\dar\\[-3pt]
	\PSh(\cO^\col)\rar{\varphi_!}&\PSh(\cU^\col)
	\end{tikzcd}
	\]	
	induced by the naturality of \eqref{equ:tower-intro-un} is a pullback of symmetric monoidal $\infty$-categories.
	\item[(Morita categories)] The tower \eqref{equ:tower-intro-un} induces a tower of symmetric monoidal Morita $(\infty,2)$-categories
		\[\hspace{2.7cm}\ModInf^\un(\cO)^{(\infty,2)} = \ModInf^\un_{\infty}(\cO)^{(\infty,2)} \ra \cdots\ra\ModInf^\un_2(\cO)^{(\infty,2)}\ra\ModInf^\un_1(\cO)^{(\infty,2)}\]
		 defined in terms of algebras and bimodules in $\RMod^\un_k(\cO)$. This tower recovers \eqref{equ:tower-intro-un} on the $\infty$-category of $1$-morphisms between the monoidal units. There are analogues of all of the above properties for this tower, for example an identification
		\[\hspace{3cm}\ModInf^\un_1(\cO)^{(\infty,2)}\simeq \Cosp(\PSh(\cO^\col))^{(\infty,2)}\]
		of the first layer as the symmetric monoidal $(\infty,2)$-category of cospans in $\PSh(\cO^\col)$, or natural pullbacks of symmetric monoidal $(\infty,2)$-categories 
		\[\hspace{3cm}\begin{tikzcd}[]
	\ModInf^\un_k(\cO)^{(\infty,2)}\rar{\varphi_!}\dar&\ModInf^\un_k(\cU)^{(\infty,2)}\dar\\[-3pt]
	\Cosp(\PSh(\cO^\col))^{(\infty,2)}\rar{\varphi_!}&\Cosp(\PSh(\cU^\col))^{(\infty,2)}
	\end{tikzcd}
	\]	
	for $1\le k\le \infty$ and any operadic right-fibration $\varphi\colon \cO\ra\cU$.
\end{enumerate}
\end{bigthm}

\begin{nrem}Some comments on the statement of \cref{bigthm:calc}:
\begin{enumerate}[leftmargin=*]
\item \emph{Operadic right-fibrations} $\cO\ra \cP$ are maps of unital $\infty$-operads for which the induced functor $\Env(\cO)\ra \Env(\cP)$ is a right-fibration on underlying $\infty$-categories, i.e.\,equivalent to the unstraightening of a functor $\Env(\cP)^{\op}\ra \cS$; see Sections \ref{sec:cocart-oprightfib} and \ref{sec:cocartesian}.
\item The functors $\Lambda$ and $\Omega$ in the part on the higher layers are constructed in terms of the restriction $\RMod_{k}(\cO)\ra \RMod_{k-1}(\cO)$, its two adjoints given by left and right Kan extension along $\Env(\cO)_{\le k-1}\subset \Env(\cO)_{\le k}$, and the restriction $\RMod_k(\cO)\ra \PSh(\cO^\col\wr \Sigma_k)$ along the inclusion of $\cO^\col\wr \Sigma_k$ in $\Env(\cO)_{\le k}$ as the subcategory on those collections $(c_s)_{s\in S}$ with $|S|=k$ and maps between them that are bijections on the indexing sets; see \cref{sec:layers}.
\item The proof that the left square in the part on the higher layers is cartesian is an application of a general decomposition result for $\infty$-categories $\PSh(\cC)$ of space-valued presheaves on an $\infty$-category $\cC$ in terms of the $\infty$-categories of presheaves on certain full subcategories $\cC_0\subset \cC$ and their complements $\cC\backslash\cC_0$; see \cref{sec:recollement}.
\item If $\cO^\col$ is a $\infty$-groupoid, then by (un)straightening the $\infty$-category $\PSh(\cO^\col)$ is equivalent to the overcategory $\cS_{/\cO^\col}$. We will implicitly make this identification in what follows.
\end{enumerate}
\end{nrem}

\subsection*{Variants of embedding calculus}\label{sec:intro-variants-emb} \cref{bigthm:calc} and embedding calculus are closely related. To explain this, let us recall the $\infty$-operad $E_d$ of little $d$-discs, given as the underlying $\infty$-operad of the ordinary $1$-coloured operad in topological spaces whose $k$-ary operations are embeddings $\sqcup_k\interior(D^d)\hookrightarrow \interior(D^d)$ of open unit $d$-discs whose restriction to each disc is a composition of a scaling and a translation. Given a \emph{tangential structure} for $E_d$, that is, a map of $\infty$-groupoids $\theta\colon B\ra \BAut(E_d)$ for some $\infty$-groupoid $B$, one can form the \emph{$\theta$-framed $E_d$-operad} as a colimit
\[\smash{E_d^{\theta}\coloneqq \colim\big(B\xra{\theta}\BAut(E_d)\hookrightarrow \Opd\big)}\]
 in the $\infty$-category of operads $\Opd$, using that $\Opd$ contains $\BAut(E_d)$ as the core of the full subcategory of $\infty$-operads equivalent to $E_d$. As we will show, the colours and multi-operations of $\smash{E_d^{\theta}}$ are given as 
\[\textstyle{(E_d^\theta)^{\col}\simeq B\quad\text{and}\quad\Mul_{E_d^\theta}\big((b_s)_{s\in S}; b\big)\simeq \Mul_{E_d}\big((\ast)_{s\in S};\ast\big)\times \bigsqcap_{s\in S}\Map_{B}\big(b_s,b\big)},\]
and any map of $\infty$-operads $E_d^\theta\ra E_d^{\theta'}$ induced by a \emph{change of tangential structure} (a map of $\infty$-groupoids over $\BAut(E_d)$) is an operadic right-fibration; see \cref{sec:gc-operads}. Specialising \cref{bigthm:calc} to these $\infty$-operads, each choice of tangential structure $\theta$ gives rise to a variant of embedding calculus together with a layer identification, an extension to bordism categories, and a description of the effect of a change of tangential structure. We now discuss the three most important cases:

\begin{enumerate}[leftmargin=*]
\item[\mylabel{equ:intro-smooth-tan}{(\oo)}]a tangential structure $\oo\colon \BO(d)\ra\BAut(E_d)$ defined on the classifying space for real vector bundles of rank $d$, giving rise to classical embedding calculus for smooth embeddings,
\item[\mylabel{equ:intro-top-tan}{(\ot)}] a tangential structure $\ot\colon\BTop(d)\ra\BAut(E_d)$ defined on the classifying space for fibre bundles with fibre $\bfR^d$, giving rise to a variant of embedding calculus for topological embeddings,
\item[\mylabel{equ:intro-part-tan}{(p)}] the universal tangential structure $p\coloneqq \id\colon \BAut(E_d)\ra\BAut(E_d)$ giving rise to a variant of embedding calculus which is part of a new \emph{particle embedding calculus}.\end{enumerate}

\subsubsection*{\ref{equ:intro-smooth-tan} Smooth embedding calculus} Conjugation with the standard action of the orthogonal group $\oO(d)$ on the open unit $d$-disc yields an $\oO(d)$-action on $E_d$ which induces a tangential structure $\oo\colon \BO(d)\ra\BAut(E_d)$. Its associated $\oo$-framed $E_d$-operad $E_d^\oo$ can be shown to satisfy (see \cref{sec:ed-variants})\footnote{The $\infty$-operad $\smash{E_d^{\oo}}$ is commonly known as the \emph{framed $\smash{E_d}$-operad}---the variant of $E_d$ where one allows rotations and reflections in addition to scalings and translations---but we prefer to avoid this terminology.} \[\smash{(E_d^{\oo})^\col\simeq \BO(d)}\quad\text{and}\quad\smash{\Env(E_d^\oo)\simeq \DiscInf_d^{\oo}},\]
so its $\infty$-category of right-modules $
\smash{\RMod(E_d^{\oo})}$ agrees with the $\infty$-category $\PSh(\DiscInf^{\oo}_d)$ considered above and the tower \eqref{equ:tower-intro} recovers the previous tower \eqref{equ:intro-tower-discs}. Moreover, since $\Emb(\varnothing,M)\simeq\ast$ for all manifolds $M\in\ManInf_d$, the functor \eqref{equ:e-functor-single} lands in the full subcategory $\smash{\RMod^{\un}(E_d^{\oo})\subset \RMod(E_d^{\oo})}$ of unital right-modules which features in the tower from \cref{bigthm:calc}. Similarly, on the level of Morita categories, the target $\smash{\icat{M}\mathrm{od}^{\oo}(d)^{(\infty,2)}}$ of the functor \eqref{equ:e-functor-tower} agrees with $\smash{\ModInf(E_d^{\oo})^{(\infty,2)}}$ and the functor lands in $\smash{\ModInf^{\un}(E_d^{\oo})^{(\infty,2)}}$, so the tower of symmetric monoidal $(\infty,2)$-categories \begin{equation}\label{equ:intro-smooth-tower}
	\cdots\lra \ModInf^\un_2(E_d^{\oo})^{(\infty,2)}\lra\ModInf^\un_1(E_d^{\oo})^{(\infty,2)}\simeq\Cosp(\cS_{/\BO(d)})^{(\infty,2)}
\end{equation}
from \cref{bigthm:calc} for $\cO=E_d^{\oo}$ becomes via \eqref{equ:e-functor-tower} a tower under the bordism category $\ncBordInf^{\oo}(d)^{(\infty,2)}$. We will see that on spaces of $2$-morphisms, (a) this tower recovers the classical embedding calculus tower \eqref{equ:embcalc-tower-intro} for the space of smooth embeddings between $d$-manifolds, and (b) the abstract layer identification from \cref{bigthm:calc} recovers the explicit layer identification for \eqref{equ:embcalc-tower-intro} in terms of the frame bundles and configuration spaces of $M$ and $N$; see \cref{sec:emb-calc}. The tower \eqref{equ:intro-smooth-tower} may thus be viewed as an extension of smooth embedding calculus to the level of bordism categories.

\subsubsection*{\ref{equ:intro-top-tan} Topological embedding calculus} The tangential structure $\oo\colon \BO(d)\ra\BAut(E_d)$ from \ref{equ:intro-smooth-tan} can be shown to factor through the forgetful map $\BO(d)\ra\BTop(d)$ via a tangential structure $\ot\colon \BTop(d)\ra\BAut(E_d)$. As for \ref{equ:intro-smooth-tan}, the associated $\ot$-framed $E_d$-operad $\smash{E_d^{\ot}}$ satisfies (see \cref{sec:ed-variants})\vspace{-0.02cm}
\[\smash{(E_d^{\ot})^{\col}\simeq \BTop(d)\quad\text{and}\quad \Env(E_d^{\ot})\simeq \DiscInf^{\ot}_d}\vspace{-0.02cm}\] where $\smash{\DiscInf^{\ot}_d}$ is the symmetric monoidal $\infty$-category of \emph{topological} manifolds homeomorphic to $S\times\bfR^d$ for finite sets $S$, \emph{topological} embeddings between them, and disjoint union as monoidal structure. Combining the resulting equivalence  $\PSh(\DiscInf^{\ot})\simeq \RMod(E_d^{\ot})$ with \cref{bigthm:calc} for $\smash{\cO=E_d^{\ot}}$, the functor \eqref{equ:e-functor-single} and its extension \eqref{equ:e-functor-tower} to a bordism category as well as everything discussed in \ref{equ:intro-smooth-tan} admit analogues for topological manifolds and topological embeddings between them; basically one replaces all smooth manifolds by topological ones, $E_d^{\oo}$ by $E_d^{\ot}$, and $\BO(d)$ by $\BTop(d)$; see \cref{sec:emb-calc}. There are comparison maps between the smooth and topological variants induced by forgetting smooth structures and the change of tangential structures $\BO(d)\ra \BTop(d)$, and the smoothing theory part of \cref{bigthm:calc} applied to $E^{\oo}_{d}\ra E^{\ot}_{d}$ describes the differences: for instance, it implies that the tower \eqref{equ:intro-smooth-tower} is pulled back from the corresponding tower for $E_d^{\ot}$ along the functor $\smash{\Cosp(\cS_{/\BO(d)})^{(\infty,2)}\ra \Cosp(\cS_{/\BTop(d)})^{(\infty,2)}}$ of cospan $(\infty,2)$-categories induced by the map $\BO(d)\ra \BTop(d)$. In particular, on spaces of $2$-morphisms, this yields an analogue for topological embeddings between $d$-manifolds of the smooth embedding calculus tower \eqref{equ:embcalc-tower-intro}, such that the latter is pulled back from the former along a map between the first stages of the towers.

\subsubsection*{\ref{equ:intro-part-tan} Particle embedding calculus} Any tangential structure for $E_d$ is obtained by a change of tangential structure from the identity $\smash{p\coloneqq\id_{\BAut(E_d)}\colon \BAut(E_d)\ra \BAut(E_d)}$. The associated $p$-framed operad $\smash{E_d^p}$ does not seem to have received any attention in the literature yet, but plays a distinguished role in our setup since the smoothing theory part of \cref{bigthm:calc} implies that the towers for $\smash{\RMod^{\un}(E_d^{\oo})}$ and $\smash{\RMod^{\un}(E_d^{\ot})}$ from \ref{equ:intro-smooth-tan} and \ref{equ:intro-top-tan} as well as their extensions to Morita categories are pulled back from the corresponding towers for $\smash{E_d^{p}}$ along the maps between the first stages induced by the changes of tangential structures $\BO(d)\ra \BTop(d)\ra \BAut(E_d)$. In particular, the classical embedding calculus tower \eqref{equ:embcalc-tower-intro} is pulled back from the tower for $\smash{\ModInf^{\un}(E_d^{p})}$ on spaces of $2$-morphisms along the map between the first stages. In fact, there is a ``deeper'' tower it is pulled back from: the $\infty$-operad $E_d^p$ fits into a sequence of unital $\infty$-operads $\smash{E_d^p\ra \cdots \ra E_{d,\le 3}^p\ra E_{d,\le 2}^p\ra E_{d,\le 1}^p}$ where $\smash{E_{d,\le k}^p}$ is constructed in the same way as $\smash{E_d^{p}}$ but replacing the role of $E_d$ by a certain $k$-truncated variant $\smash{E_{d,{\le k}}}$, which roughly speaking only remembers the multi-operations with $\le k$ inputs (see \cref{sec:tangential-structures-operads}), and hence $\BAut(E_d)$ with $\BAut(E_{d,\leq k})$. Applying \cref{bigthm:calc} to this tower of $\infty$-operads yields a $(\bfN_{>0}\cup\{\infty\},\le)^{\times 2}$-indexed diagram $\smash{\RMod^\un_{\bullet}(E^p_{d,\le \smallsquare})}$ of $\infty$-categories which recovers the previous tower by setting $\smallsquare=\infty$. Setting $\bullet=\smallsquare$ yields a diagonal tower $\smash{\RMod^\un_{\bullet}(E^p_{d,\le \bullet})}$ and setting $\bullet=1$ a tower $\cS_{/\BAut(E_{d,\le \bullet})}$ induced by truncating the $E_d$-operad. These fit into a pullback 
\begin{equation}\label{equ:intro-particle-calculus-pb}
 \begin{tikzcd}
\RMod^
 \un_{\bullet}(E^p_{d})\dar\rar&\RMod^
 \un_{\bullet}(E^p_{d,\le \bullet})\dar\\
\cS_{/\BAut(E_d)}\rar& \cS_{/\BAut(E_{d,\le \bullet})}
 \end{tikzcd}
\end{equation}
of towers. The same discussion applies to the respective Morita categories. We refer to the tower $\smash{\RMod_{\bullet}(E^p_{d,\le \bullet})}$ and its extension to Morita categories $\smash{\ModInf_{\bullet}(E^p_{d,\le \bullet})}$ as \emph{particle embedding calculus}.

\begin{nrem}Given smooth $d$-manifolds $M$ and $N$ with $d\neq 4$, a codimension $0$ submanifold $\partial_0M\subset M$ and an embedding $e_{\partial_0}\colon \partial_0M\hookrightarrow \partial N$, classical smoothing theory implies that the map induced by taking derivatives $\smash{\Emb^\oo_{\partial_0}(M,N)\ra \Map^{/\BO(d)}_{\partial_0}(M,N)}$ to the space of maps over $\BO(d)$ and under $\partial_0M$ between the classifiers of the tangent bundles, is pulled back from the analogous map  $\smash{\Emb^\ot_{\partial_0}(M,N)\ra \Map^{/\BTop(d)}_{\partial_0}(M,N)}$ taking topological derivatives; see \cref{sec:smoothing-embcalc}. By the convergence results described below, the maps $\smash{\Emb^c_{\partial_0}(M,N)\ra T_\infty\Emb^c_{\partial_0}(M,N)}$ for $c\in\{\oo,\ot\}$ are often equivalences, so the discussion in \ref{equ:intro-part-tan} shows that in this case the topological derivative map is itself pulled back, namely from the map $\smash{T_\infty\Emb^p_{\partial_0}(M,N)\ra \Map^{/\BAut(E_d)}_{\partial_0}(M,N)}$ featuring in particle embedding calculus. This may come as a surprise, since the latter is, informally speaking, constructed purely in terms of configuration space data and no longer involves the homotopy type of $\BO(d)$ or $\BTop(d)$. In the smooth case, such pullbacks were known from Boavida de Brito--Weiss' theory of configuration categories \cite{BoavidaWeissConfiguration}, but the advantage of our approach is that smooth, topological, and particle embedding calculus are all instances of the same construction and come with compatible layer identifications as well as many naturality properties due to their extensions to bordism categories. In upcoming work \cite{KKPontryagin}, we will make use of this for applications to the homotopy type of $\Top(d)$ and to that of spaces of homeomorphisms of compact manifolds. 
\end{nrem}

\begin{nrem}[Configuration categories]We expect particle embedding calculus to be closely related to the theory of configuration categories \cite{BoavidaWeissConfiguration}. We offer two pieces of evidence: firstly, as already hinted at in the previous remark, one may note that the pasting of the pullback \eqref{equ:intro-particle-calculus-pb} with the square that expresses $\smash{\RMod_\bullet^\un(E_d^{\oo})}$ as pulled back from $\smash{\RMod_\bullet^\un(E_d^{p})}$ is on mapping spaces reminiscent of the pullbacks of Theorem 5.1 loc.cit.. Secondly, the final  equivalence in \eqref{equ:delooping-special-cases} below is analogous to the Alexander trick for configuration categories; see Theorem 10.1 loc.cit.. \end{nrem}

\subsection*{Positive codimension}We focus on embedding calculus in codimension $0$ in this work, because it is the case that is relevant to the study of automorphism spaces of manifolds, and since questions about positive codimension embeddings and their embedding calculi can often be reduced to the codimension $0$ case (see e.g.\,the proof of \cref{thm:conv-smooth-general}). Nevertheless, in \cref{sec:positive-codim} we briefly explain how smooth embedding calculus in positive codimension can be incorporated into the setup of \cref{bigthm:calc}, by considering an $\infty$-operad whose monoidal envelope consists of finite disjoint unions of Euclidean spaces of \emph{potentially different} dimensions and smooth embeddings between them. Replacing it with the similar $\infty$-operad involving locally flat topological embeddings yields an embedding calculus for topological locally flat embeddings of positive codimension.

\subsection*{Delooping embedding calculus} Initiated by Dwyer and Hess \cite{DwyerHess}, there has been considerable interest in delooping results in the context of embedding calculus; see \cite{TurchinMultiplicative,BoavidaWeissConfiguration,BataninDeLeger,WeissTruncated,DucoulombierTurchinWillwacher,DucoulombierTurchin}. Our set-up turns out to be well-suited for obtaining new (and recovering former) results in this direction as well, especially when combined with Lurie's theory of centralisers \cite[5.3]{LurieHA}. We illustrate this by carrying out one example of such a delooping result in \cref{sec:delooping}, namely the analogue of Morlet's equivalence $\Diff_\partial(D^d) \simeq \Omega^{d+1} \fib(\BO(d)\ra\BTop(d))$ for the diffeomorphism group of the closed disc of dimension  $d \neq 4$ \cite[Theorem V.3.4]{KirbySiebenmann} in the context of the variants of embedding calculus for  tangential structures $\theta$. For the three tangential structures \ref{equ:intro-smooth-tan}, \ref{equ:intro-top-tan}, and \ref{equ:intro-part-tan} discussed above, it specialises for any $d\ge0$ and $1\le k\le\infty$ to equivalences 
\begin{equation}\label{equ:delooping-special-cases}
\def\arraystretch{1.3}
\begin{array}{c@{\hskip 0.1cm} l@{\hskip 0.1cm} l@{\hskip 0cm} c@{\hskip 0.1cm} l@{\hskip 0.1cm} l@{\hskip 0cm}} 
T_k\Emb^\oo_\partial(D^d,D^d)&\simeq&  \Omega^{d+1}\fib\big(&\BO(d)&\xlra{\oo}&\BAut(E_{d,\le k})\big)\\
 T_k\Emb^\ot_\partial(D^d,D^d)&\simeq& \Omega^{d+1}\fib\big(&\BTop(d)&\xlra{\ot}&\BAut(E_{d,\le k})\big)\\
   T_k\Emb^p_\partial(D^d,D^d)&\simeq& \Omega^{d+1}\fib\big(&\BAut(E_{d,\le k})&\xlra{\id}&\BAut(E_{d,\le k})\big)\simeq\ast\\
\end{array}
\end{equation}
closely related to some of the previous delooping results mentioned above; see \cref{rem:relation-to-existing-delooping}.

\subsection*{Convergence results for embedding calculus}\label{sec:intro-convergence}
Having set up a topological variant of embedding calculus, the missing ingredient for it to be well-applicable in the study of spaces of topological embeddings is an analogue of Goodwillie, Klein, and Weiss' convergence result. We establish such an analogue, assuming that the target manifold is smoothable and of dimension $d\ge5$. The proof is by reduction to the smooth case using the smoothing theory part of \cref{bigthm:calc} and classical smoothing theory. For this deduction, it is convenient to first establish an improvement of the existing smooth convergence result. We now state a special case of these results, for which we fix a quadruple $(M,\partial_0M,N,e_{\partial_0})$ consisting of topological $d$-manifolds $M$ and $N$ where $M$ is compact, a compact codimension $0$ submanifold $\partial_0M\subset M$, and an embedding $e_{\partial_0}\colon \partial_0M\hookrightarrow \partial N$. We call such a quadruple \emph{smooth}, if all manifolds and embeddings involved are smooth. In this setting, the condition in the known convergence result for \eqref{equ:emb-calc-intro} to be a weak homotopy equivalence is that $M$ can be obtained from a closed collar from $\partial_0M$ by attaching handles of index at most $\dim(N)-3$, or equivalently from a closed collar on $\partial_1M\coloneqq\partial M\backslash \interior(\partial_0 M)$ by attaching handles of index at least $3$. In the improved result, this assumption is replaced by the weaker condition that the inclusion $\partial_1M\subset M$ is an \emph{equivalence on tangential $2$-types}, which means that the inclusion $\partial_1M\subset M$ is an equivalence on fundamental groupoids and that for all components the spherical Stiefel--Whitney class $w_2\colon \pi_2(M)\ra\bfZ/2$ is nontrivial if and only if $w_2\colon \pi_2(\partial_1M)\ra\bfZ/2$ is nontrivial. Both the improved smooth convergence result and its topological counterpart involve this condition:

\begin{bigthm}\label{bigthm:convergence}Fix a quadruple $(M,\partial_0M,N,e_{\partial_0})$ as above. Assume that $d\ge5$ and that the inclusion $\partial_1M\subset M$ is an equivalence on tangential $2$-types.
\begin{enumerate}[leftmargin=*]
\item If the quadruple is smooth, then the smooth embedding calculus approximation \[\smash{\Emb^{\oo}_{\partial_0}(M,N)\lra T_\infty\Emb^{\oo}_{\partial_0}(M,N)}\] is a weak homotopy equivalence.
\item If the target manifold $N$ is smoothable, then the topological embedding calculus approximation \[\smash{\Emb^{\ot}_{\partial_0}(M,N)\lra T_\infty\Emb^{\ot}_{\partial_0}(M,N)}\] is a weak homotopy equivalence. Moreover, for $d=5$ the smoothability assumption on $N$ can be weakened to assuming that the topological tangent bundle $N\ra\BTop(5)$ lifts to $\BO(5)$.
\end{enumerate}
\end{bigthm}

\begin{nrem}Both parts of \cref{bigthm:convergence} have generalisations to positive codimension, but the statements are more involved (see \cref{sec:convergence}). We only mention two consequences of them at this point:
\begin{enumerate}[leftmargin=*]
\item In codimension $\ge3$, Goodwillie, Klein, and Weiss' result shows that smooth embedding calculus always converges. Our results imply that the same holds for topological embedding calculus whenever the target is smoothable and of dimension $\ge5$; see \cref{rem:simpler-cond-top}.
\item As an example of our convergence result in codimension $2$, we show that embedding calculus converges for the spaces of smooth or locally flat topological embeddings $\bfC P^{2n}\hookrightarrow \bfC P^{2n+1}$ when restricted to the component of the standard inclusion $\bfC P^{2n}\subset \bfC P^{2n+1}$; see \cref{ex:convergence-cp2n}. 
\end{enumerate}
\end{nrem}

\subsection*{The Alexander trick for homology $4$-spheres}As a sample application of \cref{bigthm:convergence}, we extend work of Galatius--Randal-Williams \cite{GRWContractible} on homeomorphisms of contractible manifolds and embeddings of one-sided h-cobordisms from dimensions $d\ge6$ to $d=5$. In particular, we show that the topological group $\smash{\Homeo_{\partial}(\Delta)}$ of homeomorphisms of a compact contractible topological $5$-manifold $\Delta$ that fix the boundary (a homology $4$-sphere) is weakly contractible (see \cref{sec:grw-alexander-extension}).

\subsection*{Truncated $\infty$-operads and the classification of groupoid-coloured $\infty$-operads} The proof of \cref{bigthm:calc} and the analysis of the resulting variants of embedding calculus involves various general results on $\infty$-operads. Some of them might be of independent interest, such as the following:
\begin{enumerate}[leftmargin=*]
\item Assigning to a tangential structure $\smash{\theta\colon B\ra \BAut(E_d)}$ for $E_d$ its associated $\infty$-operad $\smash{E_d^{\theta}}$ extends to an $\infty$-operad-valued functor $\smash{\cS_{/\BAut(E_d)}\ra \Opd}$ out of the $\infty$-category of spaces over $\BAut(E_d)$. This  turns out to be a replete subcategory inclusion, so in particular we have \[\smash{\Aut_{\Opd}(E_d^{\theta})\simeq \Aut_{\cS_{/\BAut(E_d)}}(\theta)}.\] The analogous statement also holds for $E_d$ replaced by any \emph{reduced $\infty$-operad} (a unital $\infty$-operad $\cO$ with $\cO^{\col}\simeq \ast$). This follows from a general classification result for \emph{groupoid-coloured operads} (unital $\infty$-operads $\cO$ for which $\cO^{\col}$ is an $\infty$-groupoid): writing $\Opd^{\red}\subset \Opd^{\gc}\subset \Opd$ for the full subcategories of reduced and groupoid-coloured $\infty$-operads respectively, assigning to a functor $F\colon B\ra\Opd^{\red}$ out of an $\infty$-groupoid $B$ its colimit in $\Opd$ induces an equivalence 
\[\smash{\textstyle{\int_{\cS}\Fun(-,\Opd^{\red})\simeq \Opd^{\gc}}}\]
between $\Opd^{\gc}$ and the unstraightening of $\Fun(-,\Opd^{\red})\colon \cS\ra\Cat$. We deduce this equivalence from Lurie's \emph{assembly and disintegration} for $\infty$-operads \cite[2.3]{LurieHA} (see \cref{sec:gc-operads}).
\item To construct the tower $\smash{E_d^p\ra \cdots \ra E_{d,\le 2}^p\ra E_{d,\le 1}^p}$ from above, we consider an $\infty$-category of $\smash{\Opd^{\le k}}$ of \emph{$k$-truncated $\infty$-operads} and show that a certain truncation functor $\smash{\tau_k\colon \Opd\ra \Opd^{\le k}}$ admits a fully faithful right adjoint on the full subcategories of unital $\infty$-operads (see \cref{sec:truncation-operads} and \cref{sec:truncation}). This was recently independently proved as part of \cite{DubeyLiu}.
\end{enumerate}

\subsection*{Acknowledgements} We would like to thank Elden Elmanto, Rune Haugseng, Fabian Hebestreit, Gijs Heuts, Nick Rozenblyum, Jan Steinebrunner for helpful comments and discussions. {MK} acknowledges funding from the European Union through an ERC grant (MaFC, 101221003). {AK} acknowledges the support of the Natural Sciences and Engineering Research Council of Canada (NSERC) [funding reference number 512156 and 512250]. {AK} was supported by an Alfred P.~Sloan Research Fellowship. 

\tableofcontents

\section{Categorical preparations}\label{sec:categorical-preliminaries}
We begin by summarising several $\infty$-categorical concepts used in the body of this work, alongside with proofs of various statements we could not locate in the literature. The topics are:

\begin{minipage}[c]{1\textwidth}
\begin{multicols}{2}
\begin{enumerate}[leftmargin=0.1cm]
	\item[\ref{sec:monoids}] Category and monoid objects
	\item[\ref{sec:towers}] Towers
	\item[\ref{sec:presheaves-yoneda}] Categories of presheaves
	\item[\ref{sec:operads}] Operads
	\item[\ref{sec:lax}] (Op)lax monoidality
	\item[\ref{sec:bimod}] Categories of bimodules
	\item[\ref{sec:morita}] Morita categories
	\item[\ref{sec:algebras-modules-intro}]Algebras over operads and their modules
		\item[\ref{sec:mate-calculus}] (Op)lax natural transformations and lax limits
	\item[\ref{sec:rfib}] Miscellaneous on right-fibrations
\end{enumerate}
\end{multicols}
\end{minipage}

\begin{nconvention}
We work in the setting of $\infty$-categories throughout this work, but we drop the $\infty$-prefix: a ``category'' is an $\infty$-category, a ``space'' or ``groupoid'' is an $\infty$-groupoid, an ``operad'' is an $\infty$-operad and so on. Whenever we refer to a concept in the ordinary sense, we explicitly say so. Regarding the model of $\infty$-categories, we in principle follow Lurie \cite{LurieHTT, LurieHA} and use quasi-categories, but most of our statements and arguments are model-independent and in line with this, we differ from loc.cit.\,in that we use the equivalence-invariant variants of the $\infty$-categorical concepts that appear, as opposed to their point-set quasi-categorical incarnations. For instance, a \emph{cocartesian fibration} or an \emph{operad} is defined as recalled in \cite[Definitions 2.1 and 2.9]{KKDisc} as opposed to as defined in \cite[2.4.2.1]{LurieHTT} and \cite[2.1.1.10]{LurieHA}. We ignore size issues throughout and leave adding the appropriate adjectives (``small'', ``large'', etc.) to the reader, c.f.~\cite[1.2.15]{LurieHTT}. Unless said otherwise, we adopt the notation from \cite[Section 2]{KKDisc}.
\end{nconvention}

\subsection{Category and monoid objects}\label{sec:monoids}A \emph{category object} in a category $\cC$ with finite limits is a simplicial object $X\in \Fun(\Delta^{\op},\cC)$ such that the Segal maps $X_{[p]}\ra X_{[1]}\times_{X_{[0]}} \cdots \times_{X_{[0]}} X_{[1]}$ are equivalences for $p\ge1$. These span a full subcategory $\CCat(\cC)\subset \Fun(\Delta^{\op},\cC)$. Requiring in addition that $X_{[0]}$ be the terminal object yields the full subcategory $\Mon(\cC)\subset \CCat(\cC)$ of \emph{monoid objects}. Replacing the role of $\Delta^{\op}$ in the definition of $\Mon(\cC)$ with the category $\Fin_\ast$ of finite pointed sets and pointed maps between them gives the full subcategory $\CMon(\cC)\subset \Fun(\Fin_\ast,\cC)$ of \emph{commutative monoid objects}. The latter is related to $\Mon(\cC)$ by the functor $\CMon(\cC)\ra \Mon(\cC)$ induced by precomposition with the functor $\Delta^{\op}\ra\Fin_\ast$ as recalled e.g.\,in \cite[Section 2.4]{KKDisc}. Writing $\Cat$ for the category of categories \cite[3.0.0.1]{LurieHTT}, relevant examples for us are:
\begin{enumerate}
	\item $\Mon(\Cat)$, the category of \emph{monoidal categories},
	\item $\CMon(\Cat)$, the category of \emph{symmetric monoidal categories},
	\item \label{enum:double-cat} $\CCat(\Cat)$,  the category of \emph{double categories}, and
	\item $\CMon(\CCat(\Cat))$, the category of \emph{symmetric monoidal double categories}.
\end{enumerate}
Regarding \ref{enum:double-cat}, we adopt the following terminology (see \cite[Section 2.5]{KKDisc} for more details): given a double category $\cM\in \CCat(\Cat)\subset\Fun(\Delta^{\op},\Cat)$, its \emph{category of objects} and \emph{category of morphisms} are the values $\cM_{[0]}$ and $\cM_{[1]}$ at $[0], [1]\in \Delta$, respectively,  its \emph{mapping category} $\cM_{A,B}$ between objects $A,B\in\cM_{[0]}$ is the fibre over $(A,B)$ of the \emph{source-target} functor $(s,t)\colon \cM_{[1]}\ra \cM_{[0]}\times \cM_{[0]}$ given as the value of the two maps $0,1\colon [0]\ra [1]$, and its \emph{composition functor} $\cM_{A,B}\times \cM_{B,C}\ra \cM_{A,C}$ for objects $A,B,C\in\cM_{[0]}$ is the functor induced by the face map $(0\le 2)^*\colon \cM_{[2]}\ra \cM_{[1]}$ and the equivalence $\cM_{[2]}\simeq \cM_{[1]}\times_{\cM_{[0]}}\cM_{[1]}$ resulting from the Segal condition.

\subsubsection{Relative mapping spaces in double categories} \label{sec:relative-mapping-spaces} Given objects $M,N\in \cM_{[1]}$ in the category of morphisms in a double category $\cM$ and maps $\varsigma\colon s(M)\ra s(N)$ and $\tau\colon t(M)\ra t(N)$ between their sources and targets, the \emph{mapping space between $M$ and $N$ relative to $\varsigma$ and $\tau$} is the fibre 
\[\Map_{\cM}(M,N;\varsigma,\tau)\coloneqq \fib_{(\varsigma,\tau)}\Big(\Map_{\cM_{[1]}}\big(M,N\big)\ra \Map_{\cM_{[0]}}\big(s(M),s(N)\big)\times\Map_{\cM_{[0]}}\big(t(M),t(N)\big)\Big).\]
If $\varsigma$ and $\tau$ are equivalences, these relative mapping spaces recover the usual mapping spaces between $M$ and $N$ in the mapping category of $\cM$ between $s(M)\simeq s(M)$ and $t(N)\simeq t(N)$. 

\begin{rem}\label{rem:st-cartesian}For many double categories $\cM$, the source-target functor $(s,t)\colon \cM_{[1]}\ra \cM_{[0]}\times \cM_{[0]}$ is a cartesian fibration (see e.g.\,\cref{rem:morita-st-cartesian} for a class of examples). In this case, any relative mapping space $\Map_{\cM}(M,N;\varsigma,\tau)$ is equivalent to an actual mapping space in the mapping category $\cM_{s(M),t(M)}$ since postcomposition with a cartesian lift $(\varsigma,\tau)^*N\ra N$ of $(\varsigma,\tau)\colon (s(M),t(M))\ra (s(N),t(N))$ induces an equivalence $\smash{\Map_{\cM_{s(M),t(M)}}(M,(\varsigma,\tau)^*N)\simeq \Map_{\cM}(M,N;\varsigma,\tau)}$.
\end{rem}

\subsubsection{Double and $(\infty,2)$-categories}\label{sec:infty-two} Double categories are closely related to $(\infty,2)$-categories in that the category $\CCat(\Cat)$ of double categories contains the category $\Cat_{(\infty,2)}$ of $(\infty,2)$-categories as a full subcategory. Moreover, any $\cM\in \CCat(\Cat)$ has an underlying $(\infty,2)$-category $\cM^{(\infty,2)}$, giving rise to a functor $(-)^{(\infty,2)}\colon\CCat(\Cat)\ra \Cat_{(\infty,2)}$. The latter has the following properties:

\begin{enumerate}[leftmargin=*]
\item\label{item:oo-2-products} It preserves finite products, and thus also symmetric monoidal structures, i.e.\,induces a functor of the form $\CMon(\CCat(\Cat))\ra \CMon(\Cat_{(\infty,2)})$,
\item\label{item:oo-2-pullbacks} More generally, it preserves pullback squares with bottom right corner in $\Cat_{(\infty,2)}\subset \CCat(\Cat)$. Since the forgetful functor $\CMon(\cC)\ra \cC$ preserves and detects limits, this also implies that $\CMon(\CCat(\Cat))\ra \CMon(\Cat_{(\infty,2)})$ preserves pullbacks squares whose bottom right corner maps via the forgetful functor $\CMon(\CCat(\Cat))\ra \CCat(\Cat)$ to $\Cat_{(\infty,2)}\subset \CCat(\Cat)$.
\item\label{item:oo-2-objects} The objects of $\cM^{(\infty,2)}$ can be identified with those of $\cM$ and the mapping category between two objects in $\cM$ with the mapping category between the corresponding objects in $\cM^{(\infty,2)}$.
\end{enumerate}

One way to define $\Cat_{(\infty,2)}\subset \CCat(\Cat)$, to implement the functor $(-)^{(\infty,2)}$, and to justify properties \ref{item:oo-2-products}-\ref{item:oo-2-objects} is as follows: recall from \cite[Theorem 4.11]{JoyalTierney} that the \emph{nerve} $N\colon \Cat\ra \PSh(\Delta)=\Fun(\Delta^\op,\cS)$ (the Yoneda embedding followed by restriction along $\Delta\subset\Cat$) is fully faithful and thus yields an equivalence $\Cat\simeq \mathrm{CSS}(\cS)$ to its essential image $\mathrm{CSS}(\cS)\subset \PSh(\Delta)$, also called the category of \emph{complete Segal spaces}. Applying $\CCat(-)$ this gives an equivalence 
$\CCat(\Cat)\simeq \CCat(\mathrm{CSS}(\cS))$. The category $\CCat(\mathrm{CSS}(\cS))$ contains a full subcategory $\mathrm{CSS}(\mathrm{CSS}(\cS))\subset \CCat(\mathrm{CSS}(\cS))$ of complete Segal objects in $\mathrm{CSS}(\cS))$; this is one of the standard models for $(\infty,2)$-categories (see e.g.\,\cite[Section 7]{HaugsengSpans}). Via the equivalence $\CCat(\Cat)\simeq \CCat(\mathrm{CSS}(\cS))$, this yields a full subcategory $\Cat_{(\infty,2)}\subset \CCat(\Cat)$. The functor $(-)^{(\infty,2)}\colon\CCat(\Cat)\ra \Cat_{(\infty,2)}$ is given by ``Segalfication followed by completion'', as explained for example in \cite[Remark 3.15]{HaugsengMorita}. This remark in loc.cit.\,also explains why this functor satisfies \ref{item:oo-2-products}. Property \ref{item:oo-2-pullbacks} holds by a minor extension of this argument: arguing as in loc.cit.\,(and adopting the notation), it suffices to show that completion $L_2\colon\mathrm{Seg}_2(\cS)\ra \mathrm{CSS}_2(\cS)$ preserves pullbacks over $2$-fold Segal spaces that are already complete. The latter holds by the argument in the proof of \cite[Lemma 7.10]{HaugsengSpans}. 
Property \ref{item:oo-2-objects} is explained in \cite[Section 2.5.6]{KKDisc}.

\begin{rem}In fact, a double category $\cM\in\CCat(\Cat)$ has \emph{two} underlying $(\infty,2)$-categories: the one given by the functor $(-)^{(\infty,2)}$ as described above, and the one given by the composition
\[\smash{
\CCat(\Cat)\xlra{\CCat(N)}\CCat(\CCat(\cS))\overset{\tau}{\underset{\simeq}\lra}\CCat(\CCat(\cS))\lra \Cat_{(\infty,2)}}
\]
where the final arrow is given by ``Segalfication followed by completion'' (see e.g.\,\cite[Remark 3.15]{HaugsengMorita}) and the middle arrow is the restriction of the self-equivalence of $\Fun(\Delta^\op\times\Delta^\op,\cS)$ given by permuting the $\Delta^\op$. However, we will only make use of the first way of extracting an $(\infty,2)$-category.
\end{rem}

\subsection{Towers}\label{sec:towers}
For a category $\cC$, we write $\smash{\Tow(\cC) \coloneqq \Fun((\bfN_{>0}\cup\{\infty\})^{\op},\cC)}$ for the category of \emph{towers} in $\cC$, where $(\bfN_{>0}\cup\{\infty\})^{\op}$ is the opposite of the poset  $(\bfN_{>0}\cup\{\infty\},\le)$. A tower $X\in\Tow(\cC)$ thus consists of an object $x\in \cC$ (the value at $\infty$) and a sequence of morphisms $\smash{\ldots\ra x_3 \ra x_2\ra x_1}$ in the undercategory $\smash{\cC_{x/}}$ (the values at $1,2,3,\ldots \in\bfN_{>0}$). If this expresses $x$ as the limit of the $x_i$s (i.e.\,if $X$ is a limit-cone), then the tower $X\in\Tow(\cC)$ is said to \emph{converge}.

\subsection{Categories of presheaves}\label{sec:presheaves-yoneda} Recall that for a category $\cC$, we write $\PSh(\cC)=\Fun(\cC^\op,\cS)$ for the category of presheaves on $\cC$ with values in the category $\cS$ of spaces. It is the target of the Yoneda embedding $y_{\cC}\colon \cC\ra \PSh(\cC)$ which is fully faithful by the Yoneda lemma \cite[5.1.3.1]{LurieHTT}. Given a functor $\varphi\colon \cC\ra \cD$, we often consider the functors 
\[\varphi^*\colon \PSh(\cD)\ra\PSh(\cC),\quad \varphi_!\colon \PSh(\cC)\ra\PSh(\cD), \quad \varphi_*\colon \PSh(\cC)\ra\PSh(\cD)\]
where $\varphi^*$ is induced by restriction and $\varphi_!$ respectively $\varphi_*$ is its left respectively right adjoint which is given by taking left respectively right Kan extension. There is one more related functor we will make frequent use of: writing  $\Fun^\colim(\PSh(\cC),\cD)\subset \Fun(\PSh(\cC),\cD)$ for the full subcategory of colimit-preserving functors into a category $\cD$ with colimits, restriction along $y_{\cC}$ induces an equivalence $\smash{\Fun^\colim(\PSh(\cC),\cD)\simeq\Fun(\cC,\cD)}$; see 5.1.5.6 loc.cit.. Given $\varphi\in\Fun(\cC,\cD)$ we write 
\begin{equation}\label{equ:colim-pres-ex}
	\smash{\lvert-\rvert_\varphi\colon \PSh(\cC)\ra \cD}
\end{equation} 
for its preimage under this equivalence---the unique colimit preserving extension of $\varphi$ along $y_\cC$. The functor $\varphi_!(-)$ is a special case of \eqref{equ:colim-pres-ex}: we have $\varphi_!(-)\simeq \lvert-\rvert_{(y_{\cD}\circ \varphi)}$. The functor $\lvert-\rvert_\varphi$ can also be expressed as a coend: we have $\lvert-\rvert_\varphi\simeq (-)\otimes_{\cC}\varphi$.

\subsection{Operads}\label{sec:operads}\label{sec:operads-intro}
We use Lurie's model for operads $\cO$: functors $\cO^{\otimes}\ra \Fin_\ast$ to the category $\Fin_\ast$ of pointed finite sets and pointed maps between these, satisfying certain properties \cite[2.1.1.10]{LurieHA}, among them the existence of cocartesian lifts over \emph{inert maps} (i.e.\,maps $\smash{\varphi\in\Map_{\Fin_\ast}(S\sqcup\ast,T\sqcup\ast)}$ with $\lvert\varphi^{-1}(t)\rvert=1$ for all $t\in T$) and the property that the functor $\smash{\cO^{\otimes}_{S\sqcup *}\ra \sqcap_{s\in S}\cO^{\otimes}_{\{s\}\sqcup \ast}}$ from the fibre over $S\sqcup\ast$ to the product of the fibres over the $\{s\}\sqcup\ast$'s induced by the cocartesian lifts of the Segal maps is an equivalence for all $S\in\Fin$. A \emph{morphism of operads} is a functor over $\Fin_\ast$ that preserves cocartesian lifts over inert maps. This defines a subcategory $\Opd$ of ${\Cat}_{/\Fin_\ast}$. The fibre $\smash{\cO^\col\coloneqq \cO^{\otimes}_{\langle 1\rangle}\in\Cat}$ over $\langle 1\rangle=\{1,\ast\}\in \Fin_\ast$ of an operad $\cO^{\otimes}\ra \Fin_\ast$ is its \emph{category of colours}. The objects in $\cO^{\otimes}$ can be identified with tuples $(c_s)_{s\in S}$ of objects in $\cO^\col$ indexed by finite sets $S\in\Fin$ (see 2.1.1.5 loc.cit.). For such a tuple $(c_s)_{s\in S}$ and a further object $d$ in $\cO^\col$, the subspace $\Mul_{\cO}((c_s)_{s\in S};d)\subset \Map_{\cO^\otimes}((c_s)_{s\in S}, d)$ of those components that map to the unique map $\varphi\colon S\sqcup\ast\ra \langle 1\rangle$ with $\varphi^{-1}(\ast)=\ast$ is called the \emph{space of multi-operations} from $(c_s)_{s\in S}$ to $d$ (c.f.\,2.1.1.16 and 2.1.1.1 loc.cit.). These spaces of multi-operations are related by composition maps that satisfy the axioms of a coloured operad in the classical $1$-categorical sense up to higher coherent homotopy (see 2.1.1.17 loc.cit.). The category of operads $\Opd$ has a terminal object, the \emph{commutative operad} $\Com$, given by $\id_{\Fin_\ast}$. From the axioms of an operad $\cO=(\cO^{\otimes}\ra \Fin_\ast)$, it follows that mapping spaces in $\cO^\otimes$ can be expressed in terms of spaces of multi-operations via natural equivalences
\begin{equation}\label{equ:maps-total-space-operads}
	\textstyle{\Map_{\cO^{\otimes}}((c_s)_{s\in S},(d_t)_{t\in T})\simeq\bigsqcup_{\varphi\in\Map_{\Fin_\ast}(S\sqcup \ast, T\sqcup \ast)}\bigsqcap_{t\in T}}\Mul_{\cO}((c_s)_{s\in\varphi^{-1}(t)};d_t).
\end{equation}
The latter in particular implies the following:

\begin{lem}\label{lem:operad-equivalence-criterion}A map of operads $\varphi\colon \cO\ra\cU$ is an equivalence if and only if it induces an equivalence on categories of colours and all spaces of multi-operations.
\end{lem}

\begin{proof}Since both forgetful functors $\smash{\Opd\ra\Cat_{/\Fin_\ast}\ra\Cat}$ are conservative, it suffices to show that $\cO^{\otimes}\ra \cU^{\otimes}$ is an equivalence in $\Cat$. Using $\smash{\cU^{\otimes}_{S\sqcup *}\ra \sqcap_{s\in S}\cU^{\otimes}_{\{s\}\sqcup \ast}}$ and that $\varphi^\col$ is an equivalence, it follows that $\cO^{\otimes}\ra \cU^{\otimes}$ is essentially surjective. That it is fully faithful follows from \eqref{equ:maps-total-space-operads} and the assumption that $\varphi$ is an equivalence on spaces of multi-operations. 
\end{proof}

\noindent The following will be useful at several occasions. For a proof, see e.g.\,\cite[Lemma 1.13]{AyalaFrancisTanaka}.

\begin{lem}\label{lem:operad-limits}The forgetful functor $\Opd\ra \Cat_{/\Fin_*}$ preserves and detects limits.
\end{lem}

\subsubsection{Relation to symmetric monoidal categories}\label{sec:operads-as-symmon}
Unstraightening induces a (non-full; see \cref{sec:lax}) subcategory inclusion $\CMon(\Cat)\hookrightarrow \Opd$, so a symmetric monoidal category can be viewed as a special case of an operad. The essential image of this inclusion consists of those operads that are also cocartesian fibrations, i.e.\,have cocartesian lifts over all morphisms in $\Fin_\ast$ instead just over inert morphisms. When viewing a symmetric monoidal category $\cC$ as an operad, the category of colours $\cC^\col$ corresponds to the underlying category of $\cC$ and the spaces of multi-operations $\Mul_{\cO}((c_s)_{s\in S};d)$ to the mapping spaces $\Map_\cC(\otimes_{s\in S}c_s,d)$ where $\otimes$ is the monoidal product in $\cC$.

\subsubsection{Monoidal envelopes}\label{sec:env}
The inclusion $\CMon(\Cat)\hookrightarrow \Opd$ from \cref{sec:operads-as-symmon} has a left adjoint  \[\Env(-)\colon \Opd\lra \CMon(\Cat),\] 
the \emph{symmetric monoidal envelope}, see 2.2.4 loc.cit. The underlying category of $\Env(\cO)$ for an operad $\cO=(\cO^\otimes\ra\Fin_\ast)$ is the wide subcategory of $\cO^\otimes$ containing the morphisms that map to an \emph{active} map in $\Fin_\ast$, i.e.\,a map in the image of the subcategory inclusion $\Fin\hookrightarrow \Fin_\ast$ that adds a disjoint basepoint (see 2.2.4.3 loc.cit.). In particular, objects of $\Env(\cO)$ are finite collections $(c_s)_{s\in S}$ of colours, the monoidal structure is given by disjoint union, the mapping spaces are (using \eqref{equ:maps-total-space-operads}) given by 
\begin{equation}\label{equ:mapping-space-env}
		\textstyle{\Map_{\Env(\cO)}\big((c_s)_{s\in S}, (d_t)_{t\in T}\big)\simeq \bigsqcup_{\varphi\in\Map_{\Fin}(S,T)}\bigsqcap_{t\in T}\Mul_{\cO}\big((c_s)_{s\in \varphi^{-1}(t)}; d_t\big)},
\end{equation} 
and the composition is induced by the operad composition. For example, the symmetric monoidal envelope of the commutative operad $\Com$ is the symmetric monoidal category of finite sets with disjoint union as monoidal structure, $\Env(\Com)\simeq \Fin$. Since $\Com$ is the terminal operad we have $\Opd\simeq \Opd_{/\Com}$, so the functor $\Env(-)$ lifts to a functor with target $\CMon(\Cat)_{/\Fin}$. In particular, there is a preferred symmetric monoidal functor $\pi_\cO\colon \Env(\cO)\lra \Fin$ for any operad $\cO$. On underlying categories, this functor is the restriction of $\cO^\otimes\ra\Fin_\ast$ to the preimage of $\Fin\subset \Fin_*$, so it sends an object $(c_i)_{i\in S}\in \Env(\cO)$ to the indexing set $S$.

\subsubsection{Unital operads} \label{sec:unital} An operad $\cO$ is \emph{unital} if the space $\Mul_{\cO}(\varnothing;c)$ of 0-ary operations with output colour $c$ is contractible for all colours $c\in \cO^\col$ (see 2.3.1.1 loc.cit.). We write $\smash{\Opd^\un\subset \Opd}$ 
for the full subcategory of unital operads. The inclusion $\smash{\Opd^\un\subset \Opd}$ admits both adjoints, so it preserves all (co)limits (see 2.3.1 loc.cit.). There are several other characterisations of unital operads; for instance, \eqref{equ:mapping-space-env} implies that an operad $\cO$ is unital if and only if the monoidal unit $\varnothing\in \Env(\cO)$ in the symmetric monoidal envelope is initial. Another characterisation of unital operads is:

\begin{lem}\label{lem:unital-cartesian-injections} $\cO$ is unital if and only if $\pi_{\cO}\colon \Env(\cO)\ra \Fin$ has cartesian lifts for injections.
\end{lem}

\begin{proof}We first prove the ``only if''-direction. Given an object $(c_{s})_{s\in S}\in \Env(\cO)$ and an injection $\iota\colon S'\rightarrow S$, we have to find a cartesian lift of $\iota$ along $\pi_{\cO}$ with target $(c_s)_{s\in S}$. The contractibility of $\Mul_{\cO}(\varnothing;c_s)$ together with the equivalence \eqref{equ:mapping-space-env} yields an equivalence 
\[\smash{\textstyle{\fib_\iota\big(\Map_{\Env(\cO)}((c_{\iota(s')})_{s'\in S'},(c_s)_{s\in S})\ra \Map_{\Fin}(S',S)\big)\simeq \bigsqcap_{s'\in S'}\Mul_{\cO}(c_{\iota(s')};c_{\iota(s')})}},\] so $(\id_{\iota(s')})_{s'\in S'}$ yields a lift $\smash{\bar{\iota}\in \Map_{\Env(\cO)}((c_{\iota(s')})_{s'\in S'},(c_s)_{s\in S})}$. We claim that $\smash{\bar{\iota}}$ is cartesian, i.e.\,that for all objects $(d_t)_{t\in T}\in\Env(\cO)$ the following square is cartesian:
\begin{equation}\label{eqn:unital-cartesian}
\begin{tikzcd}[row sep=0.3cm]
\Map_{\Env(\cO)}\big((d_t)_{t\in T},(c_{\iota(s')})_{s'\in S'}\big)\dar\rar{\bar{\iota}\circ(-)}&\Map_{\Env(\cO)}\big((d_t)_{t\in T},(c_{s})_{s\in S}\big)\dar\\
\Map_{\Fin}\big( T, S'\big)\rar{\iota\circ(-)}&\Map_{\Fin}\big( T, S\big).
\end{tikzcd}
\end{equation}
Arguing as in the construction of the lift $\bar{\iota}$ using \eqref{equ:mapping-space-env} and unitality of $\cO$, it follows that the maps on vertical fibres are equivalences, so the square is indeed cartesian and the  ``only if''-direction of the claim follows. For the ``if''-direction, we consider the pullback \eqref{eqn:unital-cartesian} for $T=S'=\varnothing$, $S$ a singleton, the injection $\iota \colon \varnothing \to S$ and a cartesian lift $\bar{\iota}$ with target $c$. Taking vertical fibres over $\ast\simeq \Map_{\Fin}(\varnothing;\varnothing)$ shows that $\Mul_{\cO}(\varnothing;\varnothing)\simeq \Mul_{\cO}(\varnothing;c)$ which implies the claim since $\Mul_{\cO}(\varnothing;\varnothing)\simeq\ast$ holds by the Segal condition in the definition of an operad.
\end{proof}

\begin{cor}\label{cor:unital-nonsurj}For a unital operad $\cO$ and an object $(c_s)_{s\in S}\in\Env(\cO)$, the assignment that sends a subset $S'\subseteq S$ to $(c_s)_{s\in S'}\in \Env(\cO)$ extends to a cubical diagram (here $2^S$ denotes the powerset of $S$) \[(c_s)_{s\in\bullet}\colon (2^S,\subseteq)\lra \Env(\cO)\] which is natural in $(c_s)_{s\in S}$ and $\cO$. Moreover, for all subsets $S'\subseteq S$ the functor
\[
\Env(\cO)_{/(c_s)_{s\in S'}}\lra\Env(\cO)_{/(c_s)_{s\in S}}
\]
induced by postcomposition with the map in the diagram $(c_s)_{s\in\bullet}$ is an equivalence onto the full subcategory spanned by those morphisms whose underlying map of finite sets has image in $S'\subseteq S$.
\end{cor}

\begin{proof}By the uniqueness of cartesian lifts, sending an inclusion $S'\subseteq S''$ to the cartesian lift $(c_{s'})_{s\in S'}\ra (c_{s''})_{s\in S''}$ constructed in the proof of \cref{lem:unital-cartesian-injections} defines a cubical diagram as claimed. In more detail, by \cref{lem:unital-cartesian-injections}, the pullback $\pi_\cO\colon \Env(\cO)^{\inj}\ra \Fin^{\inj}$ of $\pi_\cO$ along the inclusion of the wide subcategory of $\Fin$ consisting of injections, is a cartesian fibration, so $(\pi_\cO)_*\colon \Fun((2^S,\subseteq),\Env(\cO)^{\inj})\ra \Fun((2^S,\subseteq),\Fin^\inj)$ is a cartesian fibration as well \cite[3.1.2.1]{LurieHTT}. We may thus choose a cartesian lift of the natural transformation $\inc\ra \const_{S}$ in $\Fun((2^S,\subseteq),\Fin^\inj)$ with target $\const_{(c_s)_{s\in S}}\in \Fun((2^S,\subseteq),\Env(\cO)^\inj)$. The source of this lift is the desired cubical diagram. Moreover, from the fact that the lift $(c_{s'})_{s\in S'}\ra (c_{s''})_{s\in S''}$ of $S'\subseteq S''$ is cartesian, the second claim in the statement follows by observing that the functor $\Fin_{/S'} \to \Fin_{/S}$ induced by the inclusion $S'\subseteq S$ is an equivalence onto the full subcategory of those morphisms with image in $S'\subseteq S$. 
\end{proof}

We call the maps in the cube $(c_s)_{s\in\bullet}$ of \cref{cor:unital-nonsurj} \emph{inclusions} and write $(c_s)_{s\in S'}\subseteq (c_{s})_{s\in S''}$.

\subsubsection{Truncation of unital operads} \label{sec:truncation-operads} In the ordinary $1$-categorical setting, there is a notion of a \emph{$k$-truncated operad} for $1\le k\le \infty$ given by from modifying the definition of an operad to only include multi-operations with at most $k$ inputs. For $k=\infty$ this recovers (non-truncated) operads and for $k=1$ it recovers categories if one restricts to unital operads. Note that by discarding some of the multi-operations, any $k$-truncated coloured classical operad in this sense induces a $j$-truncated one for any $1\leq j\leq k \leq \infty$. As shall we explain in \cref{sec:truncation}, this can be also implemented $\infty$-categorically: restricting to unital operads, we construct a converging tower of categories
\begin{equation}\label{eqn:truncation-tower}
	\Opd^{\le \bullet,\un}=\Big(\Opd^\un=\Opd^{\le \infty,\un} \ra \cdots \ra  \Opd^{\le 2,\un}\ra \Opd^{\le 1,\un}\simeq\Cat\Big)\in\Tow(\Cat)
\end{equation}
which factors the functor $(-)^\col\colon \Opd^{\un}\ra \Cat$, and we show that all functors in this tower (called \emph{truncation functors}) are localisations, i.e.\,admit fully faithful right adjoints. Writing $\tau_k\colon \Opd^\un\ra \Opd^{\le k,\un}$ for the $k$th truncation functor, $\tau_{k*}$ for its right adjoint, and $\cO_{\le k}\coloneqq \tau_{k*}\tau_k(\cO)\in \Opd^{\un}$, the counits of the adjunctions in \eqref{eqn:truncation-tower} yield a converging tower of unital operads \begin{equation}\label{equ:truncation-counit-tower}
\cO_{\le \bullet}=\Big(\cO\simeq\cO_{\le \infty}\ra \cdots\ra \cO_{\le 2}\ra\cO_{\le 1}\Big)\in\Tow(\Opd^{\un})
,
\end{equation}
for any unital operad $\cO\in\Opd^{\un}$. By restricting the tower \eqref{eqn:truncation-tower} to the cores of the full subcategories on operads equivalent to a fixed operad $\cO\in\Opd^{\un} $ and its truncations we also obtain a tower \begin{equation}\label{equ:truncation-automorphism-tower}
\BAut(\cO_{\le \bullet})=\Big(\BAut(\cO)\simeq\BAut(\cO_{\le \infty})\ra  \cdots\ra \BAut(\cO_{\le 2})\ra\BAut(\cO_{\le 1})\Big)\in\Tow(\cS).
\end{equation}
In the mentioned appendix we will also show that the counit $\cO\ra  \cO_{\le k}$ induces an equivalence on categories of colours and that it is on spaces of multi-operations given by
\begin{equation}\label{equ:counit-multioperations}
	\smash{\Mul_{\cO}((c_s)_{s \in S};c) \lra \lim_{S' \subseteq S,\,\lvert S'\rvert\leq k} \Mul_{\cO}((c_s)_{s \in S'};c)};
\end{equation}
here the limit is induced by the cubical diagram $(c_s)_{s\in \bullet}$ from \cref{cor:unital-nonsurj}. We say that a unital operad $\cO\in\Opd^{\un}$  is \emph{$k$-truncated} if the counit $\cO\ra  \cO_{\le k}$ is an equivalence, which is by \cref{lem:operad-equivalence-criterion} equivalent to the maps \eqref{equ:counit-multioperations} on multi-operations being equivalences. 

\medskip

\noindent From the formula \eqref{equ:counit-multioperations}, one can also deduce the above claimed convergence of \eqref{equ:truncation-counit-tower}:

\begin{lem}\label{lem:counits-converge}The tower \eqref{equ:truncation-counit-tower} converges.
\end{lem}

\begin{proof}Since the forgetful functors $\smash{\Opd^{\un}\ra \Opd\ra\Cat_{/\Fin_\ast}}$ detect limits (the first by the discussion in \cref{sec:unital} and the second by \cref{lem:operad-limits}), it suffices to show that $\cO_{\le \bullet}$ converges in $\Cat_{/\Fin_\ast}$. As $\Cat_{/\Fin_\ast}$ admits all limits and taking mapping spaces preserves limits, this follows from \cref{lem:operad-equivalence-criterion} together with the fact that \eqref{equ:truncation-counit-tower} is constant on categories of colours and is eventually constant on spaces of multi-operations between any fixed collection of colours as a result of \eqref{equ:counit-multioperations}.
\end{proof}

\begin{rem}\label{rem:restricted-env-depends-on-truncation}The formula \eqref{equ:counit-multioperations} combined with the description of $\Env(\cO)$ in \cref{sec:env} also shows that the functor $\Env(\cO)\ra \Env(\cO_{\le k})$ is an equivalence when restricting source and target to the subcategories of objects that map to sets of cardinality $\le k$ under the functor to $\Fin$. 
\end{rem}

\subsubsection{Cocartesian operads and operadic right-fibrations}\label{sec:cocart-oprightfib}
The functor $(-)^\col\colon \Opd^\un\ra \Cat$ that sends a unital operad to its category of colours admits a fully faithful right adjoint $(-)^\sqcup\colon \Cat\ra  \Opd^\un$ which sends a category $\cC$ to its \emph{cocartesian operad} $\cC^\sqcup$ which satisfies (see \cite[2.4.3]{LurieHA}) \begin{equation}\label{equ:operations-cocartesian-operad}
(\cC^{\sqcup})^\col\simeq \cC\quad\text{and}\quad 
	\Mul_{\cC^\sqcup}((c_i)_{i\in S};c)\simeq \sqcap_{i\in S}\Map_{\cC}(c_i;c).
\end{equation} 
The operadic composition in $\cC^\sqcup$ is induced by composition in $\cC$. By uniqueness of adjoints, $(-)^\sqcup$ agrees with the right adjoint to the truncation functor $\Opd^\un\ra \Opd^{\le1,\un}\simeq\Cat$ from \cref{sec:truncation-operads}. 

\medskip

\noindent An operad $\cO$ is called \emph{cocartesian} if it is equivalent to $\cC^{\sqcup}$ for some category $\cC$. This is equivalent to $\cO$ being unital and $1$-truncated in the sense of \cref{sec:truncation-operads}. It is also equivalent to the counit $\cO\ra (\cO^{\col})^\sqcup\eqcolon \cO^{\col,\sqcup}$ being an equivalence, which in turn is equivalent to the induced functor on symmetric monoidal envelopes $\Env(\cO)\ra \Env(\cO^{\col,\sqcup})$ being an equivalence. Note that the latter is always an equivalence if one restricts source and target to the preimages of the wide subcategory $\Fin_{\inj}\subset \Fin$ spanned by injective maps (this follows directly from \eqref{equ:mapping-space-env} and \eqref{equ:operations-cocartesian-operad}).

\medskip

\noindent The cocartesian operad $\cC^{\sqcup}$ is a symmetric monoidal category if and only if $\cC$ admits finite coproducts (see 2.4.3.12 loc.cit.). In this case the induced symmetric monoidal structure on $\cC$ is the unique symmetric monoidal structure on $\cC$ that is given by taking coproducts on the level of homotopy categories (see 2.4.3.19 loc.cit.); it is called the \emph{cocartesian symmetric monoidal structure}. A class of examples of cocartesian symmetric monoidal categories are monoidal envelopes $\Env(\cC^\sqcup)$ of any cocartesian operad; this can be seen from the description of the homotopy category of the monoidal envelope from \cref{sec:env}.

\medskip

\noindent The notion of a cocartesian operad allows us to characterise a class of maps of operads that will play an important role in the body of this work, \emph{operadic right-fibrations}:
\begin{lemdfn}\label{lem:operadic-right-fib}\label{dfn:oprightfib}Given a map of unital operads $\varphi\colon \cO\ra \cP$, the following are equivalent:
\begin{enumerate}
\item\label{enum:orf-i} The functor $\varphi^\col\colon \cO^\col\ra \cP^\col$ is a right-fibration and the square
\[\begin{tikzcd}
\cO\rar{\varphi}\dar&\cP\dar\\
\cO^{\col,\sqcup}\rar{\varphi^{\col,\sqcup}}&\cP^{\col,\sqcup}
\end{tikzcd}
\]
induced by the unit of $(-)^\col\dashv(-)^\sqcup$ is a pullback of operads.
\item\label{enum:orf-ii} The underlying functor of categories of $\Env(\varphi)\colon \Env(\cO)\ra\Env(\cP)$ is a right-fibration.
\end{enumerate}
The map $\varphi\colon \cO\ra\cP$ is called operadic right-fibration if it satisfies one (and hence both) of these properties. The wide subcategory on these maps is denoted by $\smash{\Opd^{\rf}\subset \Opd^\un}$.
\end{lemdfn}

In order to prove Lemma \ref{lem:operadic-right-fib}, we first establish a characterisation of right-fibrations. In the proof of it, and in the subsequent proof of Lemma \ref{lem:operadic-right-fib}, we will repeatedly use that right-fibrations and cartesian fibrations are both stable under pullback, and that a functor is a right-fibration if and only if it is a cartesian fibration and all fibres are groupoids (see 2.4.2.4 loc.cit.).
 
\begin{lem}\label{lem:rfib-cocart}A functor $\psi\colon \cC\ra\cD$ is a right-fibration if and only if the underlying functor of $\Env(\psi^\sqcup)\colon \Env(\cC^\sqcup)\ra\Env(\cD^\sqcup)$ is a right-fibration.
\end{lem}

\begin{proof}
Using \eqref{equ:mapping-space-env} and \eqref{equ:operations-cocartesian-operad} one checks that $\fib_{(d_s)_{s\in S}}(\Env(\psi^\sqcup))\simeq \sqcap_{s\in S}\fib_{d_s}(\psi)$, so since a product of categories is a groupoid if and only if its factors are, it suffices to show that $\psi$ is a cartesian fibration if and only if $\Env(\psi^\sqcup)$ is. The backward direction follows by noting that $\psi$ is pulled back from $\Env(\psi^\sqcup)$ along the inclusion $\cD \subset \Env(\cD^\sqcup)$. To prove the forward direction, we need to find an $\Env(\psi^\sqcup)$-cartesian lift of any morphism $(c_s)_{s\in S}\ra (\psi(d_t))_{t\in T}$ in $\Env(\cD^\sqcup)$. The latter consists of a map $\alpha\colon S\ra T$ of finite sets and morphisms $\alpha_s\colon c_s\ra \psi(d_{\alpha(s)})$ in $\cD$ for $s \in S$. Choosing $\psi$-cartesian lifts $\widetilde{\alpha_s}\colon \widetilde{c_s}\ra d_{\alpha(s)}$ of $\alpha_s$ for $s\in S$ gives a morphism $(\widetilde{c_s})_{s\in S}\ra (d_t)_{t\in T}$ in $\Env(\cC^\sqcup)$ which one checks using \eqref{equ:mapping-space-env} and \eqref{equ:operations-cocartesian-operad} to be $\Env(\psi^\sqcup)$-cartesian.
\end{proof}

\begin{proof}[Proof of Lemma \ref{lem:operadic-right-fib}] Assume $\varphi$ satisfies \ref{enum:orf-i}. Since $\varphi^\col$ is a right-fibration, the same holds for $\Env((\varphi^\col)^\sqcup)$ by \cref{lem:rfib-cocart}. The functor $\Env(-)\colon  \Opd\ra \Cat$ preserves and detects pullbacks (use \cref{lem:operad-equivalence-criterion} and \eqref{equ:mapping-space-env} to see this), so it in particular preserves the pullback in \ref{enum:orf-i}. $\Env(\varphi)$ is thus pulled back from the right-fibration $\Env((\varphi^\col)^\sqcup)$ and is therefore a right-fibration itself. Conversely, if $\Env(\varphi)$ satisfies \ref{enum:orf-ii}, then $\varphi^\col$ is a right-fibration since it is pulled back from $\Env(\varphi)$ along the inclusion $\cP^\col\subset \Env(\cP)$. Since $\Env(-)$ detects pullbacks, it thus remains to show that $\Env(-)$ applied to the square in \ref{enum:orf-i} is a pullback. By assumption and \cref{lem:rfib-cocart}, both rows in the square are right-fibrations. From \eqref{equ:mapping-space-env} and \eqref{equ:operations-cocartesian-operad}, we see that the map $\fib_{(c_s)_{s\in S}}(\Env(\varphi))\ra \fib_{(c_s)_{s\in S}}(\Env(\varphi^{\col,\sqcup}))$ on fibres is an equivalence, so the claim follows from \cref{lem:test-fib-pullback}.
\end{proof}

\subsubsection{Wreath products} \label{sec:wreath-product}Cocartesian operads can be used to define \emph{wreath products} in a very general setting (see \cite[2.4.4]{LurieHA}). We will only need a special case of this: given a category $\cC$ and $k\ge0$, we consider the preimage of the groupoid $\Fin_k^{\simeq}\subset \Fin$ of sets of cardinality $k$ and bijections between them under the functor $\pi_{\cC^\sqcup}\colon \Env(\cC^\sqcup) \ra\Fin$
\[\cC\wr\Sigma_k\coloneqq \Env(\cC^{\sqcup})\times_{\Fin}\Fin_k^{\simeq}.\]
From \eqref{equ:mapping-space-env} and \eqref{equ:operations-cocartesian-operad}, we see that objects in this category are ordered $k$-tuples $(c_i)_{1\le i\le k}$ of objects in $\cC$, and the morphisms $(c_i)_{1\le i\le k}\ra (d_i)_{1\le i\le k}$ are given by a permutation $\sigma\in\Sigma_k$ and a $k$-tuple $(c_i\ra d_{\sigma(i)})_{1\le i\le k}$ of morphisms in $\cC$. Note that as a result of the discussion in \cref{sec:cocart-oprightfib} the counit $\cO\ra \cO^{\col,\sqcup}$ induced for any unital operad $\cO$ an equivalence \begin{equation}\label{equ:operad-wreath-subcat}\cO^\col\wr\Sigma_k\simeq \Env(\cO)\times_{\Fin}\Fin_k^{\simeq}.\end{equation}

\subsection{(Op)lax monoidality}\label{sec:lax} By unstraightening, the category of symmetric monoidal categories $\CMon(\Cat)\subset \Fun(\Fin_\ast,\Cat)$ can be viewed as a full subcategory of the category $\Cocart(\Fin_\ast)$ of cocartesian fibrations over $\Fin_\ast$ and thus as a subcategory of $\smash{{\Cat}_{/\Fin_\ast}}$, or as a full subcategory of the category $\smash{\Cart(\Fin^{\op}_\ast)}$ of cartesian fibrations and thus as a subcategory of $\smash{{\Cat}_{/\Fin^{\op}_\ast}}$. 

\subsubsection{Lax monoidality}
Viewing $\CMon(\Cat)$ as a subcategory of $\smash{{\Cat}_{/\Fin_\ast}}$, the condition on morphisms can be weakened to only having to preserve cocartesian lifts over inert maps (i.e.\,being a morphism of operads) instead of all cocartesian lifts, to arrive at the notion of a \emph{lax symmetric monoidal functor}. We write $\CMon^\lax(\Cat)$ for the category of symmetric monoidal categories with lax symmetric monoidal functors between them. By construction, $\CMon^\lax(\Cat)$ is a full subcategory of $\Opd$ which is itself a (non-full) subcategory of ${\Cat}_{/\Fin_\ast}$. 

\medskip

\noindent When viewing a symmetric monoidal category $\cC$  as a cocartesian fibration of $\Fin_\ast$, the monoidal product of two objects $c,c'\in\cC$ can be recovered as the target of the unique cocartesian lift $(c,c')\ra c\otimes c'$ in the source of the associated cocartesian fibration $\cC^\otimes\ra\Fin_\ast$ that has source $(c,c')$ and whose underlying map $\{1,2,\ast\} \ra \{1,\ast\} $ in $\Fin_\ast$ sends both $1$ and $2$ to $1$. Taking the image of this morphism under a lax monoidal functor $\varphi\colon \cC\ra \cD$ gives a morphism $(\varphi(c),\varphi(c'))\ra \varphi(c\otimes c')$ which need not longer be cocartesian. Taking its cocartesian factorisation gives a morphism $\varphi(c)\otimes\varphi(c')\ra \varphi(c\otimes c')$. These morphisms can be used to characterise those lax symmetric monoidal functors that are (strong) symmetric monoidal (i.e.\,preserve \emph{all} cocartesian morphisms) namely as those $\varphi$ for which the morphisms $\varphi(c)\otimes\varphi(c')\ra \varphi(c\otimes c')$ are equivalences.

\subsubsection{Oplax monoidality}
Similarly, viewing $\CMon(\Cat)$ as a subcategory of $\smash{{\Cat}_{/\Fin^{\op}_\ast}}$, the condition on morphisms can be weakened to only having to preserve cartesian lifts over inert maps; this defines \emph{oplax symmetric monoidal functors}. The resulting category is denoted by $\CMon^\oplax(\Cat)$. Similarly as before, an oplax symmetric monoidal functor $\varphi\colon \cC\ra\cD$ induces morphisms $\varphi(c\otimes c')\ra \varphi(c)\otimes\varphi(c')$ such that $\varphi$ is symmetric monoidal if and only if they are equivalences.

\medskip

\noindent The opposite $\varphi^\op\colon \cC^\op\ra\cD^\op$ of a lax symmetric monoidal functor is oplax monoidal with respect to the opposite symmetric monoidal structures \cite[2.4.2.7]{LurieHA}. This extends to an equivalence
\begin{equation}\label{equ:op-equivalence}
	(-)^{\op}\colon \CMon^\lax(\Cat)\xlra{\simeq }\CMon^\oplax(\Cat).
\end{equation}

\subsubsection{Adjunctions and (op)lax monoidality}
Given a lax symmetric monoidal functor $R\colon \cC\ra \cD$ whose underlying functor in $\Cat$ is the right adjoint to a functor $L\colon \cD\ra\cC$, one obtains for $d,d'\in \cD$ morphisms $L(d\otimes d')\ra L(d)\otimes L(d')$ by using (co)unit of the adjunction and the morphisms $R(c)\otimes R(c')\ra R(c\otimes c')$ from the lax structure of $R$ to form the composition
\begin{equation}\label{equ:oplax-from-lax}
	L(d\otimes d'\big)\ra L( RL(d)\otimes RL(d'))\ra LR(L(d)\otimes L(d'))\ra L(d)\otimes L(d').
\end{equation}
Conversely, starting with an oplax monoidal structure on $L\colon \cD\ra\cC$, one obtain morphisms $R(c)\otimes R(c')\ra R(c\otimes c')$ for $c,c'\in\cC$ as the compositions
\begin{equation}\label{equ:lax-from-oplax}
	R(c)\otimes R(c')\ra RL(R(c)\otimes R(c'))\ra R(LR(c)\otimes LR(c'))\ra R(c\otimes c')
\end{equation}
This suggests that the left adjoint of a lax symmetric monoidal functor can be enhanced to an oplax monoidal functor and vice versa. This is indeed the case: by \cite[Corollary C]{HHLN} or \cite[Theorem 1.1]{Torii}, the construction of \eqref{equ:oplax-from-lax} and \eqref{equ:lax-from-oplax} extends an equivalence of categories
\begin{equation}\label{equ:mates}
	\CMon^{\oplax,L}(\Cat)^\op \simeq \CMon^{\lax,R}(\Cat)
\end{equation}
where $\CMon^{\oplax,L}(\Cat)$ and $\CMon^{\lax,R}(\Cat)$ are the wide subcategories of $\CMon^{\oplax}(\Cat)$ and $\CMon^{\lax}(\Cat)$, respectively, on those (op)lax symmetric monoidal functors whose underlying functors in $\Cat$ are left and right adjoints, respectively. 

\subsection{Categories of bimodules}\label{sec:bimod}Recall (e.g.\,from the summary in \cite[Section 2.8]{KKDisc}) the functors $\Ass(-), \BMod(-)\colon\Mon(\Cat)\ra \Cat$ that send a monoidal category $\cC$ to the categories $\Ass(\cC)$ and $\BMod(\cC)$ of \emph{associative algebras} and \emph{bimodules} in $\cC$. These categories are related by a forgetful functor $\BMod(\cC)\ra \Ass(\cC)\times\cC\times\Ass(\cC)$ whose pullback along the inclusion $\{A\}\times\cC\times \{B\}\subset \Ass(\cC)\times\cC\times\Ass(\cC)$ for associative algebras $A$ and $B$ in $\cC$ is the category $\BMod_{A,B}(\cC)$ of \emph{$(A,B)$-bimodules in $\cC$} which comes with a forgetful functor $U_{A,B}\colon \BMod_{A,B}(\cC)\ra \cC$. As mentioned in  \cite[Remark 2.18]{KKDisc} there are at least two models for these categories and the functors between them that will be relevant for us: one by Haugseng via \emph{nonsymmetric generalised operads} \cite[Section 4]{HaugsengMorita} and one by Lurie via \emph{symmetric operads} \cite[4.1, 4.3]{LurieHA}. As explained in the cited remark, they are known to be equivalent.

\medskip

\noindent There is a canonical associative algebra $1\in\Ass(\cC)$ whose underlying object in $\cC$ is the monoidal unit. It has the property that $U_{11} \colon \BMod_{1,1}(\cC) \to \cC$ is an equivalence (see \cite[Corollary 4.50]{HaugsengMorita}). 

\subsubsection{Left- and right-modules}The algebra $1\in\Ass(\cC)$ allows one to define the category $\LLMod(\cC)$ and $\RRMod(\cC)$ of \emph{left-} or \emph{right-modules} in $\cC$ as the fibre at $1$ of the composition $\BMod(\cC)\ra \Ass(\cC)\times\Ass(\cC)\ra\Ass(\cC)$ whose final functor is the projection onto the right or left factor. The fibres at $A\in\Ass(\cC)$ of the maps $\LLMod(\cC)\ra \Ass(\cC)$ and $\RRMod(\cC)\ra \Ass(\cC)$ to the other factors are the categories $\LLMod_A(\cC)$ and $\RRMod_A(\cC)$ of \emph{left-} or \emph{right-modules}s over $A$. By definition we have $\LLMod_A(\cC)= \BMod_{A,1}(\cC)$ and $\RRMod_A(\cC)= \BMod_{1,A}(\cC)$. There is also a direct construction of these categories not in terms of bimodules (see \cite[4.2.1.13, 4.2.1.36, 4.3.2.8]{LurieHA}). 

\subsubsection{Free bimodules}\label{sec:free-bimodules}
The forgetful functor $U_{AB}\colon \BMod_{A,B}(\cC)\ra \cC$ is conservative and admits a left adjoint $F_{AB}\colon \cC\ra \BMod_{A,B}(\cC)$ whose essential image is the subcategory of \emph{free $(A,B)$-bimodules} \cite[4.3.3.14]{LurieHA}. The component of the unit $\id\ra U_{AB}F_{AB}$ at an object $M\in\cC$ is given by the map $M\ra A\otimes M\otimes B$ induced by tensoring with the unit maps $1\ra A$ and $1\ra B$ (see 4.3.3.14 loc.cit.) The induced monad $(U_{AB}F_{AB})\colon \cC\ra\cC$ exhibits the category $\BMod_{A,B}(\cC)$ as monadic over $\cC$ (see 4.7.3.9 loc.cit.). This implies that any bimodule $M\in \BMod_{A,B}(\cC)$ is the geometric realisation of the monadic bar construction $\Barc_{U_{AB}F_{AB}}(U_{AB}F_{AB},U_{AB}(M))_\bullet\in \Fun(\Delta^{\op},\BMod_{A,B}(\cC))$ whose $p$-simplices are $(F_{AB}U_{AB})^{p+1} (M)$ (see 4.7.2.7 loc.cit.), so after applying $U_{AB}$ given by $A^{\otimes p+1}\otimes M\otimes B^{\otimes p+1}$. In particular, this shows that every bimodule $M\in  \BMod_{A,B}(\cC)$ is a sifted colimit of free bimodules.
\medskip

\noindent A lax monoidal functor $\varphi\colon \cC\ra\cD$ induces a functor $\varphi_{AB}\colon \BMod_{A,B}(\cC)\ra \BMod_{\varphi(A),\varphi(B)}(\cD)$ which makes the left-hand square in the following diagram of categories commute
\begin{equation}\label{equ:general-bimodule-diagram}\begin{tikzcd}[column sep=2cm]
\BMod_{A,B}(\cC)\rar{U_{AB}}\dar{\varphi_{AB}}&\cC\dar{\varphi}\rar{F_{AB}}&\BMod_{A,B}(\cC)\dar{\varphi_{AB}}\\
\BMod_{\varphi(A),\varphi(B)}(\cD)\rar{U_{\varphi(A),\varphi(B)}}&\cD\rar{F_{\varphi(A),\varphi(B)}}&\BMod_{\varphi(A),\varphi(B)}(\cD).
\end{tikzcd}\end{equation}
The right-hand square does not commute in general, but its two compositions are related by a Beck-Chevalley transformation $F_{\varphi(A),\varphi(B)}\varphi\ra \varphi_{AB}F_{AB}$
that results from taking horizontal left adjoints in the commutative left-hand square.

\begin{lem}\label{lem:monoidal-functors-preserve-free-bimodules}For monoidal $\varphi$, the right square commutes: $F_{\varphi(A),\varphi(B)}\varphi\ra \varphi_{AB}F_{AB}$ is an equivalence. 
\end{lem}

\begin{proof}By conservativity of $U_{\varphi(A),\varphi(B)}$, we may check this after applying $U_{\varphi(A),\varphi(B)}$ in which case the Beck-Chevalley transformation becomes the transformation $\varphi(A)\otimes\varphi(-)\otimes\varphi(B)\ra \varphi(A\otimes(-)\otimes B)$ induced by the lax monoidality of $\varphi$. This is an equivalence if $\varphi$ is (strong) monoidal.\end{proof}

\subsubsection{Presentable monoidal categories}\label{sec:presentable-monoidal}
The category $\BMod(\cC)$ of bimodules in a monoidal category $\cC$ is particularly well-behaved if $\cC$ is a \emph{presentable monoidal category}, which means that the underlying category of $\cC$ is presentable in the sense of \cite[5.5.0.1]{LurieHTT} and the monoidal product $\otimes\colon \cC\times \cC\ra \cC$ preserves small colimits in both variables \cite[3.4.4.1]{LurieHA}. A \emph{presentable symmetric monoidal category} is a symmetric monoidal category which is presentable as a monoidal category. As an example, for a small category $\cC$, the category $\PSh(\cC)=\Fun(\cC^\op,\cS)$ of space-valued presheaves is presentable \cite[5.5.1.1]{LurieHTT} and any symmetric monoidal structure on $\cC$ turns $\PSh(\cC)$ into a  presentable symmetric monoidal category via Day convolution \cite[4.8.1.12]{LurieHA}.

\medskip

\noindent One way in which $\BMod(\cC)$ is better behaved if $\cC$ is presentable monoidal is the following:
\begin{lem}\label{lem:left-adjoint-bimodules}If a monoidal functor  $\varphi\colon \cC\ra\cD$ between presentable monoidal categories admits both adjoints, then so does $\varphi_{AB}$. In this case, if the left adjoint of $\varphi$ is fully faithful then so is that of $\varphi_{AB}$.
\end{lem}

\begin{proof}
Both source and target of $\varphi_{AB}$ are presentable by \cite[4.3.3.10 (1)]{LurieHA} and the forgetful functors $U_{AB}$ and $U_{\varphi(A)\varphi(B)}$ preserve and reflect limits as well as colimits by 4.3.3.3 and 4.3.3.9 loc.cit., so it follows that $\varphi_{AB}$ preserves limits and colimits. This implies the first part of the claim by the adjoint functor theorem \cite[5.5.2.9]{LurieHTT}. The second part is equivalent to showing that the unit $\smash{\eta_{AB}\colon \id\ra \varphi_{AB}\varphi_{AB}^L}$ of the adjunction $\varphi_{AB}^L\dashv\varphi_{AB}$ is an equivalence if the unit $\smash{\eta\colon \id\ra \varphi\varphi^L}$  of $\varphi^L\dashv\varphi$ is one. Both $\varphi_{AB}$ and $\varphi_{AB}^L$ preserve colimits since they are left adjoints, so using that any bimodule is a colimit of free bimodules by the discussion above and that $U_{\varphi(A)\varphi(B)}$ is conservative, $\eta_{AB}$ is an equivalence if $U_{\varphi(A),\varphi(B)}(\eta_{AB})F_{\varphi(A),\varphi(B)}$ is one. Using \cref{lem:monoidal-functors-preserve-free-bimodules} and commutativity of the square obtained from the left square in \eqref{equ:general-bimodule-diagram} by taking left adjoints, one sees that this transformation is equivalent to $U_{\varphi(A),\varphi(B)}F_{\varphi(A),\varphi(B)}(\eta)$, which is an equivalence if $\eta$ is one. \end{proof}

\subsubsection{Compatibility with geometric realisations and naturality of bimodules}\label{sec:cgr-subcat}There is another assumption on the monoidal category $\cC$ that is relevant in the context of bimodules and weaker than being presentably monoidal: $\cC$ is called \emph{compatible with geometric realisations} if the underlying category of $\cC$ has geometric realisations (colimits indexed by $\Delta^{\op}$) and the monoidal product $\otimes\colon \cC\times \cC\ra \cC$ preserves these geometric realisations separately in both variables. A monoidal functor is called \emph{compatible with geometric realisations} if it preserves geometric realisations. A symmetric monoidal category (or functor) is called compatible with geometric realisations if the underlying monoidal category (or functor) is. This defines (non-full) subcategories $\Mon(\Cat)^{\cgr}$ and $\CMon(\Cat)^{\cgr}$ of $\Mon(\Cat)$ and $\CMon(\Cat)$ respectively.

\medskip

\noindent One benefit of being compatible with geometric realisations is that in this case there is a relative tensor product $(-)\otimes_B(-)\colon \BMod_{A,B}(\cC)\times \BMod_{B,C}(\cC)\ra \BMod_{A,C}(\cC)$ for associative algebras $A,B,C\in\Ass(\cC)$ which is preserved by monoidal functors compatible with geometric realisations \cite[4.4.2.11, 4.4.2.8]{LurieHA}. This can in particular be used to make the category of bimodules $\BMod_{A,B}(\cC)$ natural in triples $(\cC,A,B)$, i.e.\,to extend it to a functor 
\begin{equation}\label{equ:bimodule-naturality-functor}
	\textstyle{\BMod_{-,-}(-) \colon \int_{\Mon(\Cat)^\cgr} \Ass(-)^{\times 2} \lra \Cat}.
\end{equation}
on the cocartesian unstraightening of the restriction of $\Ass(-)^{\times 2} \colon \Mon(\Cat)\ra \Cat$ to $\Mon(\Cat)^\cgr$:

\begin{lem}\label{lem:bmod-naturality} The assignment $(\cC,A,B)\mapsto \BMod_{A,B}(\cC)$ lifts to a functor of the form \eqref{equ:bimodule-naturality-functor}. On morphisms it sends $(\varphi\colon \cC\ra \cD, (\varphi(A),\varphi(B))\ra (A',B'))$ to the composition
\[ {\BMod_{A,B}(\cC)\xlra{\varphi_{AB}}\BMod_{\varphi(A),\varphi(B)}(\cD)\xra{A'\otimes_{\varphi(A)} (-)\otimes_{\varphi(B)}B'}\BMod_{\varphi(A),\varphi(B)}(\cD)}.\]
\end{lem}

\begin{proof}The natural transformation $\eta\colon \BMod(-)\ra \Ass(-)^{\times 2}$ induces by unstraightening a map \begin{equation}\label{equ:unstraightened-algebra-fibration}
		\smash{\textstyle{\int_{\Mon(\Cat)^{\cgr}}\BMod(-)\lra \int_{\Mon(\Cat)^{\grtp}}\Ass(-)^{\times 2}}}
\end{equation} 
of cocartesian fibrations over $\Mon(\Cat)^{\cgr}$. By construction fibres of  \eqref{equ:unstraightened-algebra-fibration} are $\BMod_{A,B}(\cC)$, so first part of the claim follows by showing that \eqref{equ:unstraightened-algebra-fibration} is itself a cocartesian fibration which we do by checking Conditions (1)-(4) in the dual version of \cite[Proposition 5.45]{HaugsengMorita}. Conditions (1) and (2) are satisfied since \eqref{equ:unstraightened-algebra-fibration} is a map of cocartesian fibrations over $\Mon(\Cat)^{\cgr}$ by construction. Condition (3) says that the functor $\BMod(\cC)\ra \Ass(\cC)^{\times 2}$ is a cocartesian fibration for any $\cC\in \Mon(\Cat)^{\cgr}$ which holds by Proposition 4.53 loc.cit. (note, firstly, that, as a result of Lemma 4.19 loc.cit., being compatible with geometric realisations implies having ``good relative tensor products'' in the sense of Definition 4.18 loc.cit., and secondly, that Haugseng's relative tensor product is equivalent to that of Lurie; see \cite[Remark 2.18]{KKDisc}). Finally, Condition (4) says that for all $\varphi\colon \cC\ra\cD$ in $\Mon(\Cat)^{\cgr}$ the induced functor $\varphi\colon \BMod(\cC)\ra \BMod(\cD)$ sends $\eta_\cC$-cocartesian morphisms to $\eta_\cD$-cocartesian morphisms. This follows from the description of $\eta_\cC$-cocartesian morphisms in \cite[Proposition 4.53]{HaugsengMorita} since $\varphi$ preserves relative tensor products. This finishes the proof of the first part of the claim. The second part follows from the construction of cocartesian lifts in the proof of Proposition 5.45 loc.cit., together with the description of $\eta_\cC$-cocartesian morphisms in Proposition 4.53 loc.cit..
\end{proof}

\subsection{Morita categories}\label{sec:morita}\label{sec:morita-functor} The relative tensor products $(-)\otimes_B(-)\colon \BMod_{A,B}(\cC)\times \BMod_{B,C}(\cC)\ra \BMod_{A,C}(\cC)$ for varying choices of  $A,B,C\in \Ass(\cC)$ in a monoidal category compatible with geometric realisations satisfy various coherences. This can be packaged conveniently into a double category $\ALG(\cC)\in\CCat(\Cat)$, called the \emph{Morita category} of $\cC$, whose category of objects and morphisms are $\Ass(\cC)$ and $\BMod(\cC)$ respectively and the source-target functor is the forgetful functor $\BMod(\cC)\ra\Ass(\cC)\times \Ass(\cC)$ so that the mapping categories are $\BMod_{A,B}(\cC)$ for $A,B\in\cC$. The composition functors are given by taking relative tensor products. 

\medskip

\noindent There are several constructions of the Morita category $\ALG(\cC)$. As in \cite{KKDisc} we will use the one by Haugseng \cite{HaugsengMorita} (denoted $\mathfrak{ALG}_1(\cC)$ in loc.cit.) which is equivalent to a construction by Lurie \cite[Corollary 5.14]{HaugsengComparison} that has some advantages (see \cite[Remark 2.18]{KKDisc} for more explanation). As recalled in \cite[2.9.4]{KKDisc}, Haugseng's model extends the assignment $\cC\mapsto \ALG(\cC)$ to a functor $\ALG(-)\colon \Mon(\Cat)^{\cgr}\lra \CCat(\Cat)$ which preserves products and thus induces a functor
\begin{equation}\label{equ:alg-functor-com}
	\ALG(-)\colon \CMon(\Cat)^{\cgr}\simeq \CMon(\Mon(\Cat)^{\cgr})\lra \CMon(\CCat(\Cat)).
\end{equation}
whose leftmost equivalence is the restriction of the instance of the equivalence $\CMon(\Mon(\cC))\simeq\CMon(\cC)$ for categories $\cC$ with finite products in the case $\cC=\Cat$ (see e.g.\,\cite[Remark 2.4]{KKDisc}). 

\medskip

\noindent The functor \eqref{equ:alg-functor-com} behaves well with limits:

\begin{lem}\label{lem:limits-of-morita}The functor \eqref{equ:alg-functor-com} sends left-cone diagrams in $\CMon(\Cat)^{\cgr}$ that are limits in $\CMon(\Cat)$ to limits in  $\CMon(\CCat(\Cat))$.   
\end{lem}

\begin{proof}Since the forgetful functor $\CMon(\CCat(\Cat))\ra \CCat(\Cat)$ detects limits and $\CCat(\Cat)\subset\Fun(\Delta^{\op},\Cat)$ is a full subcategory closed under limits, the Segal condition implies  that it suffices to show that the functors $\Ass(-),\BMod(-) \colon \CMon(\Cat) \to \Cat$ preserves limits. These functors factor as the composition of the forgetful functor $\CMon(\Cat)\ra \Opd$ which preserves limits since it admits a left adjoint (see \cref{sec:env}), with functors of the form $\Alg_\cO(-)\colon \Opd\ra \Cat$ for fixed operads $\cO\in\Opd$ where $\Alg_\cO(\cU)\subset \Fun_{\Fin_*}(\cO^\otimes, \cU^\otimes)$ for an operad $\cU$ is the full subcategory of those functors over $\Fin_*$ that are operad maps (in Lurie's notation: $\Alg_{\icat{A}\mathrm{ssoc}}(-)=\Ass(-)$ and $\Alg_\icat{BM}(-)=\BMod(-)$; see \cite[4.1.1.6, 4.3.1.11]{LurieHA}). To show that $\Alg_{\cO}(-)$ preserves limits, it suffices to show this for the functor $\Map_{\Cat}(\cD,\Alg_{\cO}(-))\colon \Opd\ra\cS$ for all $\cD\in\Cat$. The latter is equivalent to the functor that sends $(\cU^{\otimes}\ra\Fin_*)\in\Opd$ to $\Map_{\Opd}(\cO, \Fun(\cD,\cU^{\otimes})\times_{\Fun(\cD,\Fin_\ast)}\Fin_*)$ (c.f.\,2.2.5.4 loc.cit.). Since $\Map_{\Opd}(\cO,-)\colon \Opd\ra \cS$ preserves limits and the forgetful functor $\Opd\ra \Cat_{/\Fin_*}$ preserves and detects limits, it suffices to show that the functor $\Opd\ra \Cat_{/\Fin_*}$ that sends $\smash{(\cU^{\otimes}\ra\Fin_*)}$ to $\smash{\Fun(\cD,\cU^{\otimes})\times_{\Fun(\cD,\Fin_\ast)}\Fin_*\in\Cat_{/\Fin_*}}$ preserves limits. The latter is a composition 
	 \[\smash{\Opd\xra{\mathrm{forget}} \Cat_{/\Fin_*}\xra{\Fun(\cD,-)} \Cat_{/\Fun(\cD,\Fin_*)}\xra{\Fin_*\times_{\Fun(\cD,\Fin_*)}(-)}\Cat_{/\Fin_*}}\]
of limit-preserving functors, so it preserves limits itself and the claim follows.
\end{proof}

\begin{rem} \label{rem:morita-st-cartesian}The double category $\ALG(\cC)$ for any $\cC\in \CMon(\Cat)^{\cgr}$ satisfies the condition from \cref{rem:st-cartesian} since the forgetful functor $\BMod(\cC)\ra \Ass(\cC)\times \Ass(\cC)$ is a cartesian fibration as a result of \cite[4.3.2.2]{LurieHA}. Given an $(A,B)$-bimodule $M$, an $(A',B')$-bimodule $N$ and maps of associative algebras $\varsigma\colon A\ra A'$ and $\tau\colon B\ra B'$, the cartesian lift $(\varsigma,\tau)^*N\ra N$ of $(\varsigma,\tau)$ can be thought as given by $N$ with the $(A,B)$-bimodule structure induced by restriction along $(\varsigma,\tau)$.
\end{rem}

\subsubsection{Categories of cospans as Morita categories}\label{sec:cospan-alg}
Given a category $\cC$ with finite colimits, we can equip it with the cocartesian symmetric monoidal structure whose monoidal product $\sqcup$ is given by taking coproducts (see \cref{sec:cocart-oprightfib}). As explained in \cite[Section 6]{HaugsengMelaniSafronov}, the associated Morita category $\ALG(\cC^{\sqcup})\in\CMon(\CCat(\Cat))$ has in this case a simple description: it is the symmetric monoidal double category $\Cosp(\cC)\in\CMon(\CCat(\Cat))$ of \emph{cospans} in $\cC$: the category of objects is $\cC$, the mapping categories are overcategories $\Cosp(\cC)_{A,B}\simeq \cC_{A\sqcup B/}$, and the composition functors is given by taking pushouts (see also \cite[2.10]{KKDisc} for a summary). 

\begin{rem}\label{rem:remarks-on-cospans}Two remarks on this construction:
\begin{enumerate}[leftmargin=*]
	\item Strictly speaking, the functor $\ALG(-)$ as presented above is not applicable to $\cC^{\sqcup}$ since the latter need not be compatible with geometric realisations if $\cC$ is only assumed to have finite colimits. However there is a weaker condition on a monoidal category $\cC$ under which $\ALG(\cC)$ is defined: having \emph{good relative tensor products} (see \cite[Definition 4.18]{HaugsengMorita}). This is satisfied for $\cC^{\sqcup}$ is $\cC$ has finite colimits (see \cite[Remark 2.5.14]{HaugsengMelaniSafronov}).
	\item\label{enum:cospans-complete} By \cite[Corollary 8.5]{HaugsengSpans} the underlying double category $\Cosp(\cC)\in \CCat(\Cat)$ of a cospan category lies in the full subcategory $\Cat_{(\infty,2)}\subset \CCat(\Cat)$ described in \cref{sec:infty-two}.
\end{enumerate}
\end{rem}

\subsection{Algebras over operads and their modules}\label{sec:algebras-modules-intro}
We fix a symmetric monoidal category $\cC\in\CMon(\Cat)$ and an operad $\cO\in\Opd$ throughout this subsection.

\subsubsection{Algebras over operads}\label{sec:algebra-over-operads}
The category $\Alg_{\cO}(\cC)$ of \emph{$\cO$-algebras} in the symmetric monoidal category $\cC$ is defined as the category of operad maps from $\cO$ to the underlying operad of $\cC$ (see \cref{sec:operads}), i.e.\,as the full subcategory
\begin{equation}\label{equ:o-algebras}
\Alg_{\cO}(\cC)\subset \Fun_{\Fin_*}(\cO^\otimes,\cC^\otimes)
\end{equation} 
of the category of functors over $\Fin_*$ spanned by those functors that preserve cocartesian lifts of inert morphisms in $\Fin_*$. By construction, its core is the space of operad maps from $\cO$ to the underlying operad of $\cC$ which in turn agrees with the space of maps in $\CMon(\Cat)$ from $\Env(\cO)$ to $\cC$ by the adjunction from \cref{sec:env}, i.e.\,we have equivalences $\Alg_{\cO}(\cC)^\simeq\simeq\Map_{\Opd}(\cO,\cC)\simeq \Map_{\CMon(\Cat)}(\Env(\cO),\cC)$. This lifts to an equivalence of categories: precomposition with the counit $\cO\ra\Env(\cO)$ induces an equivalence of categories 
\begin{equation}\label{equ:universal-prop-env-on-cats}
	\Fun^\otimes(\Env(\cO),\cC)\simeq \Alg_{\cO}(\cC)
\end{equation} 
whose left-hand side is the category of symmetric monoidal functors, i.e.\,the full subcategory of $\Fun_{\Fin_*}(\Env(\cO)^\otimes,\cC^\otimes)$ on those functors that preserve all cocartesian lifts  \cite[2.2.4.9]{LurieHA}.

\begin{ex}
The categories $\Ass(\cC)$, $\BMod(\cC)$, $\LLMod(\cC)$, $\RRMod(\cC)$ of associative algebras and bi-, left-, or right-modules discussed in \cref{sec:bimod} are the special cases of \eqref{equ:o-algebras} where $\cO$ is the associative operad $\Assoc$ or one of the operads $\BM$, $\LM$, $\RM$ encoding bi-, left-, or right-modules respectively (see 4.1.1.3, 4.1.1.6, 4.2.1.7, 4.2.1.13, 4.2.1.36, 4.3.1.6, 4.3.1.12 loc.cit.).
\end{ex}

\subsubsection{Lax monoidal colimit-preserving extensions}\label{equ:lax-monoidal-colimi-extension}
Recall from \cref{sec:presheaves-yoneda} that the Yoneda embedding induces an equivalence $\Fun(\cC,\cD)\simeq \Fun^\colim(\PSh(\cC),\cD)$ when $\cD$ is cocomplete. This has a lax symmetric monoidal refinement: if $\cC$ and $\cD$ are symmetric monoidal and the monoidal product in $\cD$ preserves colimits in both variables, then the equivalence lifts to an equivalence  $\smash{\Alg^\colim_{\PSh(\cC)}(\cD) \simeq \Alg_\cC(\cD)}$,
where $\PSh(\cD)$ is equipped with the Day convolution symmetric monoidal structure (see \cref{sec:day} below), $\smash{\Alg^\colim_{\PSh(\cC)}(\cD)\subset \Alg_{\PSh(\cC)}(\cD)}$ is the full subcategory of those lax monoidal functors that preserve colimits, and the equivalence is induced by precomposition with the symmetric monoidal lift of the Yoneda embedding \cite[4.8.1.10, 4.8.1.12]{LurieHA}. In particular, for a lax symmetric monoidal functor $\varphi\colon \cC\ra \cD$, there is a unique colimit preserving extension $|-|_{\varphi}\colon \PSh(\cC)\ra\cD$ to a lax symmetric monoidal functor.

\subsubsection{Modules over algebras over operads}
For an $\cO$-algebra $A\in\Alg_{\cO}(\cC)$ in $\cC$, there is a category $\smash{\oMod_A^\cO(\cC)}$ of $\cO$-modules over $A$. There are two constructions of this category, one due to Lurie \cite[3.3.3.8]{LurieHA} and one due to Hinich \cite[Definition 5.2.1]{HinichRect}, and they agree for groupoid-coloured operads $\cO$ by Appendix B loc.cit., in particular B.1.2. For our purposes Hinich's construction is more convenient. To recall it, we write $\CM$ for the operadic nerve  \cite[2.1.1.27]{LurieHA} of the ordinary $2$-coloured operad in sets encoding pairs of a commutative monoid and a module over it (see e.g.\,\cite[Example 3.1.10 (1)]{HarpazNuitenPrasma}), and write $\cM\cO\coloneqq \CM\times \cO$ for its product in $\Opd$ with $\cO$ (as a result of \cref{lem:operad-limits}, products in $\Opd$ can be computed as pullbacks in $\Cat$ over $\Fin_*$). There are maps of operads $\cO\ra \cM\cO\ra \cO$ which compose to the identity: the second is induced by the projection $\CM \to \Fin_*$ and the first by the map $\Com\ra \CM$ that encodes forgetting the module. Precomposition induces functors $\smash{\Alg_\cO(\cC)\ra \Alg_{\cM\cO}(\cC)\ra \Alg_\cO(\cC)}$. The fibre 
\begin{equation}\label{equ:o-modules}
	\smash{\oMod_A^\cO(\cC)\coloneqq \fib_{A}\big(\Alg_{\cM\cO}(\cC)\ra\Alg_{\cO}(\cC)\big)\in\Cat},
\end{equation}
of the second functor at $A$ is category $\smash{\oMod_A^\cO(\cC)}$ of $\cO$-modules over $A$. It contains a preferred object $\smash{A\in \oMod^\cO_A(\cC)}$ induced by the image of $A$ under the first functor (informally speaking: $A$ is a $\cO$-module over itself). The functor $\smash{\Alg_{\cM\cO}(\cC)\ra\Alg_{\cO}(\cC)}$ is a cartesian fibration (e.g.\,by \cite[2.12.2.1]{HinichYoneda}), so the categories $\smash{\oMod_A^\cO(\cC)}$ extend to a functor $\smash{\oMod_{(-)}^\cO(\cC)\colon \Alg_{\cO}(\cC)^\op\ra\Cat}$. In particular, any map $\varphi\colon A\ra B$ in $\smash{\Alg_{\cO}(\cC)}$ yields a restriction functor $\smash{\varphi^*\colon \oMod_{B}^\cO(\cC)\ra \oMod_{A}^\cO(\cC)}$.

\begin{ex}Here are two examples of the above construction:\,
\begin{enumerate}[leftmargin=*]
\item For $\cO=\Assoc$, there is an equivalence $\smash{\oMod^{\Assoc}_A(\cC)\simeq\BMod_{A,A}(\cC)}$, i.e.\,an $\Assoc$-module over an associative algebra $A$ is an $(A,A)$-bimodule \cite[4.4.1.28]{LurieHA}.
\item For $\cO=\Com$, there is an equivalence $\smash{\oMod^{\Com}_A(\cC)\simeq\LLMod_{A}(\cC)}$, i.e.\,a $\Com$-module over a commutative algebra $A$ in $\cC$ is a left-module over $A$ \cite[4.5.1.4]{LurieHA}.
\end{enumerate}
\end{ex}

\noindent Under some assumptions on $\cC$ and $\cO$, the category $\smash{\oMod^{\cO}_A(\cC)}$ is equivalent to a category of left-modules in $\cC$ over a certain associative algebra in $\cC$ associated to $A$. The statement (which is stated in the special case $\cO=E_d$ in a slightly different form in \cite[5.5.4.16]{LurieHA} without proof) will involve the following construction: given an $\cO$-algebra $A$ in a presentable symmetric monoidal category $\cC$, we may view it via \eqref{equ:universal-prop-env-on-cats} as a symmetric monoidal functor $A\colon \Env(\cO)\ra \cC$. Since $\cC$ has colimits, this extends to a lax monoidal functor $|{-}|_A\colon \PSh(\Env(\cO))\ra \cC$ by the discussion in \cref{equ:lax-monoidal-colimi-extension}.

\begin{thm}\label{thm:ed-modules-as-leftmodules}For an operad $\cO$ such that $(\cO^\col)^\simeq$ is connected, there is an associative algebra $U_{\cO}$ in $\PSh(\Env(\cO))$ and a left $U_{\cO}$-module $M_\cO$, such that given any $\cO$-algebra $A$ in a presentable symmetric monoidal category $\cC$, there is an equivalence of categories
	\[\oMod_A^{\cO}(\cC) \simeq \LLMod_{|{U_\cO}|_{A}} (\cC).\]
such that
	\begin{enumerate}[leftmargin=*]
	\item $A$, viewed as a $\cO$-module over itself, is mapped to the left-module $|M_\cO|_{A}$ over $|U_\cO|_{A}$,
	\item the equivalence is compatible with postcomposition with (individual) symmetric monoidal left adjoints $g\colon \cC\ra \cC'$ between presentable symmetric monoidal categories.
	\end{enumerate}
\end{thm}

\begin{proof}There ought to be a clean $\infty$-categorical proof of this equivalence, but we will content ourselves with a more ad-hoc argument using model categories. We temporarily change our conventions for this proof: a ``category'' is now an ordinary category, an  ``operad'' is now an ordinary operad, etc.. We add $\infty$-prefixes to emphasise $\infty$-categorical objects.

\medskip

\noindent Fix a $\Sigma$-cofibrant (i.e.~$\Sigma_k$ acts freely on $\cat{O}(k)$ for all $k \geq 0$) $1$-coloured operad $\cat{O}$ in Kan complexes whose operadic nerve $N^\otimes(\cO)$ is equivalent to $\cO$ (use \cite[Corollary 1.2]{ChuHaugsengHeuts} to find a cofibrant coloured operad whose operadic nerve is equivalent to $\cO$ and pass to the suboperad generated by a single colour). Recall that any simplicial model category $\cat{C}$ induces an $\infty$-category $\cat{C}_\infty\in\Cat$ by either taking the coherent nerve of the $\Kan$-enriched category of cofibrant-fibrant objects in $\cat{C}$, or equivalently by taking the coherent nerve of the discrete category of cofibrant objects in $\cat{C}$ followed by inverting weak equivalences \cite[1.3.4.15, 1.3.4.20]{LurieHA}. By \cite[Theorems 1.1 and 2.8]{NikolausSagave}, any presentable symmetric monoidal $\infty$-category $\cC$ arises in this way from a simplicial combinatorial symmetric monoidal model category $\cat{C}$, and any symmetric monoidal left adjoint $g \colon \cC \to \cC'$ from a symmetric monoidal left Quillen functor $\cat{g} \colon \cat{C} \to \cat{C}'$. Moreover, by Theorem 2.5 (i) loc.cit., the model categories $\cat{C}$ appearing can be chosen so that every 1-coloured simplicial operad $\cat{O}$ is \emph{$\cat{C}$-admissible}, meaning that the category $\cat{Alg}_{\cat{O}}(\cat{C})$ of algebras in $\cat{C}$ over $\cat{O}$ carries the transferred model structure (the unique model structure in which weak equivalences and fibrations are detected by the forgetful functor to $\cat{C}$). Since $\cat{O}$ is $\Sigma$-free, it is projectively cofibrant as a symmetric sequence, so by \cite[Theorem 7.11]{PavlovScholbach} natural functor of $\infty$-categories	
$\smash{\cat{Alg}_{\cat{O}}(\cat{C})_\infty \ra \Alg_{O}(\cat{C}_\infty)}$
is an equivalence. We fix a cofibrant representative $\cat{A}\in \cat{Alg}_{\cat{O}}(\cat{C})$ for $A\in \Alg_{O}(\cat{C}_\infty)$ which in particular implies that the underlying object in $\cat{C}$ is cofibrant (combine that a cofibration is a retract of a transfinite composition of relative cell attachments with \cite[Lemma V.20.1.A]{FresseBook} which says that the underlying map of a relative cell attachment is a cofibration when $\cat{O}$ is $\Sigma$-cofibrant). We now prove the claim by constructing a zig-zag of equivalences of categories
\[
	\smash{\oMod^{\cO}_A(\cat{C}_\infty)  \xsla{\circled{1}} \cat{Mod}^{\cat{O}}_\cat{A}(\cat{C})_\infty \overset{\circled{2}}{\simeq}  \cat{LMod}_{\cat{U_O(A)}}(\cat{C})_\infty \xsra{\circled{3}}   \LLMod_{U_\cO(A)}(\cat{C}_\infty)\overset{\circled{4}}{\simeq}   \LLMod_{\lvert U_\cO \rvert_A}(\cat{C}_\infty)}
\]
 which is natural in left Quillen functors $\cat{g} \colon \cat{C} \to \cat{C}'$.
 
\begin{enumerate}[label=$\circled{\arabic{*}}$,leftmargin=0.5cm]
	\item The equivalence $\circled{1}$ is an instance of \cite[Proposition 4.3.2]{HarpazNuitenPrasma}. The source is the underlying $\infty$-category of the model category $\cat{Mod}^{\cat{O}}_\cat{A}(\cat{C})$ of strict $\cat{O}$-modules over $\cat{A}$ in $\cat{C}$ as defined e.g.\,in \cite[Definition 3.1.17]{HarpazNuitenPrasma} or in \cite[4.2.2]{FresseBook} (in loc.cit.\,strict $\cat{O}$-modules over $\cat{A}$ are called \emph{representations} of the $\cat{O}$-algebra $A$), equipped with the transferred model structure with respect to the forgetful functor to $\cat{C}$.
	\item Unwinding the definition of $\cat{Mod}^{\cat{O}}_\cat{A}(\cat{C})$, one sees that there is an associative algebra $\cat{U_O(A)}$ in $\cat{C}$ such that there is a strict isomorphism of categories $\smash{\cat{Mod}^{\cat{O}}_\cat{A}(\cat{C}) \cong \cat{LMod}_\cat{U_O(A)} (\cat{C})}$ compatible with the forgetful functors to $\cat{C}$, where the right-hand side is the category of strict left-modules over $\cat{U_O(A)}$ in $\cat{C}$ (see e.g.\,\cite[4.3.1, 4.3.2]{FresseBook}). Equipping both sides with the transferred model structure yields the equivalence $\circled{2}$. The algebra $\cat{U_O(A)}$ and the $\cat{U_O(A)}$-module $\cat{M_O(A)}$ that corresponds to $A$ as an $\cat{O}$-module over $\cat{A}$ can be described as follows: recalling that $\cat{S}$ denotes the category of simplicial sets, there are simplicial presheaves $\cat{U_O}\colon \cat{Env(O)}^\op\ra \cat{S}$ and $\cat{M_O}\colon \cat{Env(O)}^\op\ra \cat{S}$ on the simplicial envelope of the simplicial operad $\cat{O}$, given by $(\cat{U_O})(S) \coloneqq \cat{O}(S\sqcup\ast)$ and $(\cat{M_O})(S) \coloneqq \cat{O}(S)$ for finite sets $S$. The presheaf $\cat{U_O}$ carries a canonical strict associative algebra structure with respect to the Day convolution symmetric monoidal structure on presheaves, and $\cat{M_O}$ carries a left $\cat{U_O}$-action; both actions are induced by operadic composition at the input labelled by $\ast$. Viewing $\cat{A}$ as a symmetric monoidal functor $\cat{A}\colon \cat{Env(O)}\ra \cat{C}$, the associative algebra $\cat{U_O(A)}$ and the $\cat{U_O(A)}$-module $\cat{M_O(A)}$ are given by the simplicial coends $\cat{U_O(A)} = \cat{U_O} \otimes_{\cat{Env(O)}} \cat{A}$ and $\cat{M_O(A)}=\cat{M_O}\otimes_{\cat{Env(O)}} \cat{A}$ (its underlying object is isomorphic to $\cat{A}$), with algebra and module structure induced by that of $\cat{U_O}$ and $\cat{A}$. The underlying object of the coend $\cat{U_O(A)}$ is cofibrant (this can be seen e.g.\,by rewriting it as a relative composition product, c.f.~5.1.3 loc.cit., and applying Lemma 15.2.C loc.cit.~taking $M = \varnothing$, $N = \cat{U_O}$ which is $\Sigma$-cofibrant because $\cat{O}$ is, $A$ initial, and $B = \cat{A}$ which is cofibrant).
	\item This is an instance of \cite[4.3.3.17]{LurieHA} and uses that the underlying object of $\cat{U_O(A)}$ is cofibrant. 
	\item The strict coends $\cat{U_O(A)} = \cat{U_O} \otimes_{\cat{Env(O)}} \cat{A}$ and $\cat{M_O(A)}=\cat{M_O}\otimes_{\cat{Env(O)}} \cat{A}$ agree with the homotopy coends, using \cite[Theorem 15.2.A]{FresseBook}, that $\cat{U_O}$ and $\cat{M_O}$ are $\Sigma$-cofibrant, and that $\cat{A}$ is cofibrant. They hence model the corresponding $\infty$-categorical coends, which are in turn equivalent to $\lvert
	U_\cO\rvert_A$ and $\lvert M_\cO\rvert_A$ (see \cref{sec:presheaves-yoneda}). The algebra and module structure agree because the lax monoidal lift of the coend construction is unique, so the one constructed model-categorically agrees with the one constructed $\infty$-categorically.
	\qedhere
\end{enumerate}
\end{proof}

\subsection{(Op)lax natural transformations and lax limits}\label{sec:mate-calculus} On an informal level, an \emph{oplax} respectively \emph{lax natural transformation} $\eta\colon F\ra G$ between $\Cat$-valued functors $F,G\colon \cC\ra\Cat$ consists of squares
\begin{equation}\label{equ:oplax-diagram-informal}
	\begin{tikzcd}
	F(c)\arrow["F(\varphi)", swap,d]\rar{\eta_{c}}&G(c)\dar{G(\varphi)}\\
	F(d)\arrow[r,"\eta_{d}",swap]\arrow[shorten <=10pt,shorten >=10pt,Rightarrow]{ru}{\eta_\varphi} &G(d)
	\end{tikzcd}
	\quad \text{respectively}\quad
	\begin{tikzcd}
	F(c)\arrow["F(\varphi)", swap,d]\rar{\eta_{c}}&G(c)\dar{G(\varphi)}\arrow[shorten <=10pt,shorten >=10pt,Rightarrow,swap]{dl}{\eta_\varphi} \\
	F(d)\arrow[r,"\eta_{d}",swap]&G(d)
	\end{tikzcd}
\end{equation}
for each morphism $\varphi\colon c\ra d$ in $\cC$ where $\eta_\varphi$ is a natural transformation  $\eta_dF(\varphi)\ra G(\varphi)\eta_c$ respectively $G(\varphi)\eta_c\ra \eta_dF(\varphi)$. Following \cite[Section 1.3]{HHLN}, here is how to make this  precise: writing $\twoCat$ for the $2$-category of categories, an oplax (respectively lax) natural transformation $\eta\colon F\ra G$ between functors $F,G\colon \cC\ra\Cat$ is a $2$-functor $\cC \boxtimes [1] \to \twoCat_2$ (respectively $[1] \boxtimes \cC \to \twoCat_2$) where $(-)\boxtimes(-)$ is the Gray tensor product of $2$-categories, with equivalences between its restriction to $\cC \boxtimes \{0\}\simeq \cC$ respectively $\cC \boxtimes \{1\}\simeq \cC$ to $F$ respectively $G$ (see loc.cit.\,for further explanations and references). The relation to the informal description above is as follows: given a morphism in $\cC$ viewed as a functor $[1]\ra \cC$, a $2$-functor of the form $\cC \boxtimes [1] \to \twoCat_2$ (or $ [1]\boxtimes\cC\to \twoCat_2$) gives by precomposition a $2$-functor $[1] \boxtimes [1]\to \twoCat_2$. The latter corresponds to diagram as in \eqref{equ:oplax-diagram-informal} since the Gray tensor product $[1]\boxtimes [1]$ is the discrete $2$-category indicated by 
\[
	\begin{tikzcd}[row sep=0.7cm, column sep=0.7cm]
	\cdot\arrow[d]\rar&\cdot\dar\arrow[shorten <=10pt,shorten >=10pt,Rightarrow,swap]{dl} \\
	\cdot\rar&\cdot
\end{tikzcd}\]
There are categories $\Fun^\lax(\cC,\twoCat)$ and $\Fun^\oplax(\cC,\twoCat)$ of functors $\cC\ra \Cat$ and (op)lax natural transformations between them. These categories are uniquely characterised by equivalences 
\begin{equation}\label{equ:adjunction-oplax-cat}
\begin{split}
\Map_{\Cat}(-,\Fun^\lax(\cC,\twoCat))\simeq &\ \Map_{\Cat_2}(- \boxtimes \cC,\twoCat), \\ \Map_{\Cat}(-,
\Fun^\oplax(\cC,\twoCat))\simeq &\ \Map_{\Cat_2}(\cC \boxtimes -,\twoCat)
\end{split}
\end{equation}
where $\Cat_2$ is the category of $2$-categories. Both $\Fun^\lax(\cC,\twoCat)$ and $\Fun^\oplax(\cC,\twoCat)$ contain the usual functor category $\Fun(\cC,\Cat)$ as the wide subcategory on those (op)lax natural transformations for which all $2$-cells $\eta_\varphi$ are equivalences. Via the characterising equivalences, this corresponds to the fact that functors $\cC \times \cD\ra \Cat$ for $\cC,\cD\in\Cat$ are the same as $2$-functors  $\cC \boxtimes \cD\ra \twoCat$ that send all $2$-morphisms to equivalences, i.e.\,factor through the $2$-functor $\cC \boxtimes \cD\ra \cC\times\cD$ given by inverting all $2$-morphisms in $\cC \boxtimes \cD$. More generally, there is a $2$-functor $(\cC \times \cC') \boxtimes \cD \to \cC \times (\cC' \boxtimes \cD)$ that inverts all $2$-morphisms in the image of $(\cC \times \{c'\}) \boxtimes \cD$ for $c' \in \cC'$, so $2$-functors out of $(\cC \times \cC') \boxtimes \cD$ that map these $2$-morphisms to equivalences are the same as $2$-functors out of $\cC \times (\cC' \boxtimes \cD)$.
\subsubsection{The mate correspondence}There are wide subcategories $\smash{\Fun^{\oplax,L}(\cC,\twoCat)}$ of $\smash{\Fun^{\oplax}(\cC,\twoCat)}$, and $\smash{\Fun^{\lax,R}(\cC,\twoCat)}$ of $\smash{\Fun^{\lax,R}(\cC,\twoCat)}$, given by those lax respectively oplax transformations for which $\eta_c$ is a left or right adjoint for all $c\in\cC$, respectively. As a result of \cite[Corollary F]{HHLN}  there is an equivalence of categories 
\begin{equation}\label{equ:mate-equ}
 	\smash{\Fun^{\oplax,L}(\cC,\twoCat)}\simeq \smash{\Fun^{\lax,R}(\cC,\twoCat)}^\op
\end{equation} 
which is natural in $\cC\in\Cat$ with respect to precomposition. It is the identity on objects and sends an oplax natural transformation $\eta\colon F\ra G$ to a lax natural transformation $\overline{\eta}\colon G\ra F$ where $\overline{\eta}_c\colon G(c)\ra F(c)$ is for $c\in\cC$ the right adjoint to $\eta_c$ and $\overline{\eta}_\varphi$ is the Beck--Chevalley transformation
\begin{equation}\label{equ:mate-bc}
	F(\varphi) \overline{\eta}_c \ra  \overline{\eta}_d\eta_dF(\varphi) \overline{\eta}_c\ra  \bar{\eta}_dG(\varphi)\eta_c \overline{\eta}_c\ra  \overline{\eta}_dG(\varphi)
\end{equation}
whose outer maps are induced by the (co)units of the adjunctions $\overline{\eta}_c\dashv \eta_c$ and $\overline{\eta}_d\dashv \eta_d$ for $e=c,d$ and the middle one by $\eta_\varphi$. Diagrammatically, it turns
\[\begin{tikzcd}
	F(c) \dar[swap]{F(\varphi)} \rar{\eta_c} & G(c) \dar{G(\varphi)} \\
	F(d) \arrow[r,"\eta_{d}",swap] \arrow[shorten <=10pt,shorten >=10pt,Rightarrow]{ru}{\eta_\varphi} & G(d)
\end{tikzcd} \qquad \text{into} \qquad \begin{tikzcd}
	G(c) \dar[swap]{G(\varphi)} \rar{\overline{\eta}_c}& F(c) \arrow[shorten <=10pt,shorten >=10pt,Rightarrow,swap]{dl}{\overline{\eta}_\varphi}\dar{F(\varphi)}  \\
	G(d)  \arrow[r,"\overline{\eta}_{d}",swap]  & F(d).
\end{tikzcd}\] 

\subsubsection{Lax limits}\label{sec:laxlim}Given a category $\cC$, there is a functor $\laxlim_{\cC}\colon \Fun(\cC,\Cat)\ra\Cat$ that sends a functor $F\colon \cC\ra\Cat$ to its \emph{lax limit} $\laxlim_{\cC}(F)\in\Cat$. Roughly speaking, the lax limit of a $\Cat$-valued functor is a variant of the limit that takes the $2$-categorical structure of $\Cat$ into account. There are several equivalent models for the functor $\laxlim_{\cC} $: one way to define it is via weighted limits as in \cite[Definition 2.9]{GepnerHaugsengNikolaus}, another one by sending $F\colon \cC\ra\Cat$ to the category $\smash{\Fun_{\cC}(\cC,\int_{\cC}F)}$ of sections of its cocartesian unstraightening $\smash{\int_{\cC}F\ra \cC}$, as in Corollary 7.7 loc.cit.. 
\begin{ex}\label{ex:functor-cats-are-lax-limits}Functor categories are lax limits: for a category $\cD$, the cocartesian unstraightening of $\const_{\cD}\colon \cC\ra\Cat$ is $\pr_2\colon \cC\times\cD\ra \cC$ so we have $\smash{\laxlim_{\cC}(\const_\cD)\simeq \Fun_{\cC}(\cC,\cC\times\cD)\simeq \Fun(\cC,\cD)}$.
\end{ex}

\noindent The lax limit $\laxlim_{\cC}(F)$ of a functor turns out to not only be functorial in natural transformations, it is also functorial in lax natural transformations, i.e.\,$\laxlim_{\cC}\colon \Fun(\cC,\Cat)\ra\Cat$ extends to a functor $\laxlim_{\cC}\colon \Fun^\lax(\cC,\twoCat)\ra\Cat$. One way to implement this is to combine the description of $\laxlim_{\cC}(F)$ as sections of the cocartesian unstraightening with \cite[Theorem E]{HHLN}. In this model, it is also clear that given a functor $\varphi\colon \cC\ra\cF$, there is a canonical natural transformation of functors $\Fun^\lax(\cF,\twoCat)\ra\Cat$ from $\laxlim_{\cF}$ to $\laxlim_{\cC}\circ\varphi^*$. For instance, coming back to \cref{ex:functor-cats-are-lax-limits}, the components of the natural transformation $\laxlim_{\cF}$ to $\laxlim_{\cC}\circ\varphi^*$ at $\const_{\cD}\in \Fun^\lax(\cC,\twoCat)$ is given by the precomposition functor $\varphi^*\colon \Fun(\cF,\cD)\ra \Fun(\cC,\cD)$. There are similar notions of \emph{oplax limits} and \emph{(op)lax colimits}, but they will not appear in this work.

\subsection{Miscellaneous on right-fibrations} \label{sec:rfib}We record five facts on right-fibrations (to recall, right-fibrations are cartesian fibrations whose straightening takes values in groupoids). To state these facts, we fix a commutative square of categories
\begin{equation}\label{equ:BC-original-square}
	\begin{tikzcd}
	\cA\dar[swap]{\gamma}\rar{\alpha} & \cC\dar{\delta}\\
	\cB\rar{\beta} & \cD
	\end{tikzcd}
\end{equation}
and consider the square of space-valued presheaf categories $\PSh(-)=\Fun((-)^{\op},\cS)$ 
\begin{equation}\label{equ:BC-presheaf-square}
	\begin{tikzcd}
	\PSh(\cB)\dar[swap]{\gamma^*}\rar{\beta_!} & \PSh(\cD)\dar{\delta^*}\\
	\PSh(\cA)\rar{\alpha_!} & \PSh(\cC),
	\end{tikzcd}
\end{equation}
given vertically by restricting and horizontally by left Kan extending. The two compositions in this square are related by a \emph{Beck--Chevalley transformation} \begin{equation}\label{equ:BC-trafo}
	\alpha_! \gamma^* \lra \delta^* \beta_!
\end{equation} 
which can be obtained in two ways that agree as a result of the triangle identities: either as the composition $\alpha_! \gamma^* \ra \delta^*\delta_!\alpha_! \gamma^*\simeq \delta^*\beta_!\gamma_! \gamma^*\ra \delta^*\beta_!$ involving commutativity of \eqref{equ:BC-original-square}, the unit of the adjunction $\delta_!\dashv\delta^*$, and the counit of $\gamma_!\dashv\gamma^*$, or as the composition $\alpha_! \gamma^* \ra \alpha_! \gamma^*\beta^*\beta_!\simeq\alpha_! \alpha^*\delta^*\beta_! \ra \delta^*\beta_!$ involving commutativity of \eqref{equ:BC-original-square}, the unit of $\beta_!\dashv\beta^*$, and the counit of $\alpha_!\dashv\alpha^*$. If $\beta$ is right-fibration and the square \eqref{equ:BC-original-square} is cartesian, then \eqref{equ:BC-trafo} is an equivalence:

\begin{lem}\label{lem:right-fib-BC}If the square \eqref{equ:BC-original-square} of categories is cartesian and $\beta$ a right-fibration, then the square of presheaf categories \eqref{equ:BC-presheaf-square} commutes in that the Beck--Chevalley transformation \eqref{equ:BC-trafo} is an equivalence.
\end{lem}

\begin{proof}
Once rephrased via the straightening-unstraightening equivalence in terms of categories of right-fibrations as opposed to categories of presheaves, this follows from the fact that right-fibrations are stable under pullback \cite[p.~58]{LurieHTT} along with \cite[Proposition 9.7, 9.8]{GepnerHaugsengNikolaus}.
%
\end{proof}

Given a functor $\varphi\colon \cC\ra \cD$ and an object $d\in\cD$, we abbreviate $\cC_d\coloneqq \fib_{d}(\varphi)$ for the fibre. 

\begin{lem}\label{lem:criterion-mapping-space-equivalence}Fix $k\ge0$. If the square of categories \eqref{equ:BC-original-square} is cartesian and $\beta$ a right-fibration, then given objects $c,c'\in\cC$, the map between mapping spaces
\[\Map_{\cC}(c,c')\ra \Map_{\cD}(\delta(c),\delta(c'))\]
is $k$-connected if the following condition is satisfied: there exists a lift $\overline{c}'\in \cA_{c'}$ of $c'$ along $\alpha$ such that for all lifts $\overline{c}\in \cA_{c}$ of $c$ along $\alpha$, the map $\Map_{\cA}(\overline{c},\overline{c}')\ra \Map_{\cB}(\gamma(\overline{c}),\gamma(\overline{c}'))$ is $k$-connected.
\end{lem}

\begin{proof}
Consider the commutative diagram in $\cS$
\[
	\begin{tikzcd}
	\Map_{\cC}(c,c')\dar{\delta}\rar{\cA_{(-)}}&[5pt]\Map_{\cS}(\cA_{c'},\cA_{c})\rar{\ev_{\overline{c}'}}\dar{\gamma}&[10pt] \cA_{c}\dar{\gamma}\\
	\Map_{\cD}(\delta(c),\delta(c'))\rar{\cB_{(-)}}&\Map_{\cS}(\cB_{\delta(c')},\cB_{\delta(c)})\rar{\ev_{\gamma(\overline{c}')}}&\cB_{\delta(c')}.
	\end{tikzcd}
\]
Since the square of categories in the statement is cartesian, the rightmost vertical arrow is an equivalence, so to show the claim it suffices to show that the map on horizontal fibres between the two horizontal compositions are $k$-connected. For $\overline{c}\in \cA_{c}$, this map between horizontal fibres is the map $\gamma\colon \Map_{\cA}(\overline{c},\overline{c}')\ra \Map_{\cB}(\gamma(\overline{c}),\gamma(\overline{c}'))$, so it is $k$-connected by assumption.
\end{proof}

\begin{lem}\label{lem:left-kan-limit}Left Kan extension $\beta_!\colon \PSh(\cB)\ra\PSh(\cD)$ along a right-fibration $\beta\colon\cB\ra\cD$ preserves limits over weakly contractible categories.
\end{lem}

\begin{proof}
Writing $F_{\beta}\in\PSh(\cD)$ for the unstraightening of $\beta$, the functor $\beta_!$ is by \cite[Proposition 9.8]{GepnerHaugsengNikolaus} equivalent to the forgetful functor $\PSh(\cD)_{/F_{\beta}}\ra \PSh(\cD)$, so the claim follows from the fact that forgetful functors $p \colon \cC_{/c}\ra\cC$ of overcategories commute with weakly contractible limits. 
\end{proof}

For categories $\cC$ and $\cD$, we abbreviate $\cC^\cD\coloneqq\Fun(\cD,\cC)$ for the functor category.

\begin{lem}\label{lem:pullback-composition-rfib-fact}Given categories $\cE$ and $\cF$ with final objects, a functor $\varphi\colon \cE\ra\cF$ preserving the final objects, and a right-fibration $\beta\colon\cB\ra\cD$, the following commutative square of categories is cartesian
\[
	\begin{tikzcd}
	\PSh(\cB)^\cF\dar{(\beta_!)^\cF}\rar[swap]{\varphi^*}& \PSh(\cB)^\cE\dar{(\beta_!)^\cE}\\
	\PSh(\cD)^\cF\rar{\varphi^*}& \PSh(\cD)^\cE.
	\end{tikzcd}
\]
\end{lem}

\begin{proof}
By the argument from the proof of \cref{lem:left-kan-limit}, the square is equivalent to the left square in the commutative diagram of categories whose vertical arrows are the forgetful functors
\[
	\begin{tikzcd}[row sep=0.5cm]
	(\PSh(\cD)^\cF)_{/\const_{F_\beta}}\dar\rar{\varphi^*}& (\PSh(\cD)^\cE)_{/\const_{F_\beta}}\dar\rar{\ev_{e_t}}&\PSh(\cD)_{/F_\beta}\dar\\
	\PSh(\cD)^\cF\rar{\varphi^*}& \PSh(\cD)^\cE\rar{\ev_{e_t}}&\PSh(\cD);
	\end{tikzcd}
\]
here $e_t\in\cE$ is the terminal object. The right square is cartesian because $e_t$ is terminal and the outer rectangle is cartesian because $\varphi(e_t)\in\cF$ is terminal, so the left square is cartesian as well.\end{proof}

\begin{lem}\label{lem:test-fib-pullback}If $\alpha$ and $\beta$ in the square of categories \eqref{equ:BC-original-square} are right-fibrations, then the square is cartesian if and only if the functor $\fib_c(\alpha)\ra \fib_{\delta(c)}(\beta)$ on fibres is an equivalence for all $c\in \cC$.
\end{lem}

\begin{proof}
The task is to show that the functor $\cA\ra \cB\times_{\cD}\cC$ is an equivalence under the stated hypothesis. Since right-fibrations are stable under pullback this functor is a map of right-fibrations over $\cC$, so it is an equivalence if it is an equivalence on all fibres over objects $c\in \cC$ \cite[2.2.3.3]{LurieHTT}. The latter is equivalent to the hypothesis.
\end{proof}

\section{Groupoid-coloured operads and tangential structures}\label{sec:gc-operads}
Recall from \cref{sec:unital} that an operad $\cO\in\Opd$ is unital if $\Mul_{\cO}(\varnothing;c)$ is contractible for all $c\in \cO^\col$. The full subcategory $\Opd^\un\subset\Opd$ of unital operad contains two further full subcategories 
\begin{equation}\label{equ:operad-subcats}
	\Opd^\red \subset \Opd^\gc\subset\Opd^\un
\end{equation}
that will be important for us: $\Opd^\red$, is the category of \emph{reduced operads}---unital operads $\cO$ whose category of colours $\cO^\col$ is trivial, i.e.\,$\Mul_{\cO}(c;d)\simeq\ast$ for all colours $c$ and $d$---and $\Opd^\gc$, is the category of \emph{groupoid-coloured operads}---unital operads $\cO$ for which $\cO^\col$ is a groupoid. 

\medskip

\noindent In this section we explain an alternative point of view on groupoid-coloured operads as ``families of reduced operads indexed over a groupoid'' which motivates a theory of \emph{tangential structures} for operads that will be convenient for later applications to embedding calculus. But before getting into this, we note on passing that from the description of $\Env(\cO)$ in \cref{sec:env}, one can deduce the following characterisation of the subcategories \eqref{equ:operad-subcats}. 

\begin{lem}\label{lem:chara-unital-gc}A unital operad $\cO\in\Opd^\un$ 
\begin{enumerate}
	\item is groupoid-coloured if and only if the functor $\pi_{\cO}\colon \Env(\cO)\ra \Fin$ is conservative, and it
	\item is reduced if and only if $\pi_{\cO}\colon \Env(\cO)\ra \Fin$ is conservative and an equivalence on cores. 
\end{enumerate}
\end{lem}

\subsection{Groupoid-coloured operads as families of reduced operads}\label{sec:disintegration}The alternative perspective on groupoid-coloured operads as families of reduced operads is made precise in the following result, which we deduce from Lurie's theory of \emph{disintegration} and \emph{assembly} of operads \cite[2.3]{LurieHA}:

\begin{thm}\label{thm:operad-grp-type}
Given a functor $\theta\colon B\ra\Opd^\red$ whose domain $B$ is a groupoid, there is an equivalence of categories (here the colimit is taken in $\Opd$, not in $\Opd^\red$) 
\[
	(\colim _B \theta)^\col\simeq B.
\] 
Moreover, sending such functors $\theta$ to their colimits in $\Opd$ induces an equivalence 
\[
	\begin{tikzcd}[row sep=.3cm,column sep={3cm,between origins}]
	\int_{\cS}\Fun(-,\Opd^\red)\arrow[rr,pos=0.4,"\colim","\simeq"']\arrow[dr,"\pi",swap]& & \Opd^\gc\arrow[dl,"(-)^\col"]\\
	&\cS&
	\end{tikzcd}
\]
in ${\Cat}_{/\cS}$ between the cartesian unstraightening of $\Fun(-,\Opd^\red)\colon \cS^\op\ra \Cat$ and $(-)^\col\colon \Opd^\gc\ra \cS$.
\end{thm}

\begin{proof}The proof involves Lurie's notion of \emph{generalised operads} which are functors  $\cO^{\otimes}\ra\Fin_*$ that satisfy a weakening of the axioms of an operad \cite[2.3.2.1]{LurieHA}. In particular, the fibre $\smash{\cO^{\otimes}_{\langle 0\rangle}}$ over the singleton $\langle0\rangle=\{\ast\}\in\Fin_\ast$ for a generalised operad need no longer be trivial. The category of generalised operads $\smash{\Opd^\gn}$ is a subcategory of ${\Cat}_{/\Fin_\ast}$ which it contains $\Opd$ as a full subcategory. As preparation for the proof of the claim, we show that taking fibres $(-)_{\langle 0\rangle}\colon \Opd^\gn\ra\Cat$ is a cartesian fibration whose straightening has an explicit description. By 2.3.2.9 loc.cit., this functor has a fully faithful right adjoint $(-)\times \Fin_\ast\colon \Cat\hookrightarrow\Opd^\gn$ given by sending $\cC\in\Cat$ to the projection $\pr_2\colon \cC\times\Fin_\ast\ra\Fin_\ast$. Writing $\cE$ for the pullback of $\smash{\ev_{[1]}\colon (\Opd^{\gn})^{[1]}\ra \Opd^{\gn}}$ along $\smash{(-)\times\Fin_\ast\colon \Cat\ra \Opd^\gn}$, we consider the full subcategory $\smash{\overline{\cE}\subset \cE}$ spanned by those $(\cO^\otimes\ra \cC\times \Fin_*)\in \cE$ for which the map $\smash{\cO^\otimes_{\langle0\rangle}\ra \cC}$ on fibres over $\smash{\langle 0\rangle}$ is an equivalence. The composition $\smash{\overline{\cE}\subset \cE\hookrightarrow  (\Opd^{\gn})^{[1]}\ra\Opd^{\gn}}$ whose final map is evaluation at $0 \in [1]$ is an equivalence; its inverse induced by the counit of the adjunction $\smash{((-)_{\langle 0\rangle})\dashv((-)\times\Fin_\ast)}$. Since $\smash{\ev_{[1]}\colon (\Opd^{\gn})^{[1]}\ra \Opd^{\gn}}$ is the unstraightening of $\smash{(\Opd^{\gn})_{/(-)}\colon \Opd^{\gn}\ra\Cat}$, we see that $\smash{\Opd^{\gn}}\simeq \overline{\cE}\ra \Cat$ is the unstraightening of the functor that sends $\cC\in \Cat$ to the full subcategory $\Fam(\cC)\subset\smash{(\Opd^\gn)_{/\cC\times\Fin_\ast}}$ on the \emph{$\cC$-families of operads}, i.e.\,those maps $\cO^\otimes\ra \cC\times\Fin_\ast$ of generalised operads that are equivalences on fibres over $\langle0\rangle$ (see 2.3.2.10, 2.3.2.11 loc.cit.). In sum, this shows that there is an equivalence over $\Cat$ between $(-)_{\langle 0\rangle}\colon \Opd^\gn\ra \Cat$ and the cartesian unstraightening $\smash{\pi\colon \int_{\Cat}\Fam(-)\ra \Cat}$.

\medskip

\noindent If $\cC=B$ is a groupoid, then the proof of \cite[Proposition 2.2]{HorelKrannichKupers} shows that the unstraightening equivalence $\Fun(B,{\Cat}_{/\Fin_\ast})\simeq{\Cat}_{/B\times\Fin_\ast}$ restricts to an equivalence $\Fun(B,\Opd)\simeq \Fam(B)$. This equivalence is natural in $B$ with respect to precomposition on the left-hand side and taking pullbacks on the right-hand side, so combining this with the above discussion restricted to full subcategory $\cS\subset\Cat$ yields an equivalence over $\cS$ between $\smash{\int_\cS \Fun(-,\Opd)}$ and the full subcategory of $\Opd^{\gn}$ on those generalised operads whose fibre over $\langle 0\rangle$ is a groupoid. Restricting $\Opd$ to the full subcategory $\Opd^\red$ of reduced operads, and $\Fam(B)$ correspondingly to the full sub-category $\Fam^\red(B)$ on those families of operads whose fibres are reduced operads, we arrive at an equivalence $\smash{\int_\cS \Fun(-,\Opd^\red)\simeq \int_\cS \Fam^\red(-)\simeq \Opd^{\gn,\red}}$ over $\cS$ where $\smash{\Opd^{\gn,\red}\subset \Opd^{\gn}}$ is the full subcategory of \emph{reduced generalised operads} in the sense of \cite[2.3.4.2]{LurieHA}. 

\medskip

\noindent To relate this to the statement of the claim, we consider the left adjoint $\smash{\Assem\colon \Opd^\gn\ra\Opd}$ to the inclusion $\smash{\Opd\subset \Opd^\gn}$ from 2.3.3.3 loc.cit., which sends a generalised operad to its \emph{assembly}. By 2.3.4.4 loc.cit., it restricts to an equivalence $\smash{\Opd^{\gn,\red}\simeq \Opd^\gc}$. We will now give a description of the inverse of this equivalence. To this end, we consider the natural transformation $\gamma\colon (-)\times \Fin_*\ra  \inc\circ (-)^\sqcup$ of functors $\Cat\ra \Opd^\gn$ that sends a category $\cC$ to the map of generalised operads $\smash{\gamma\colon \cC\times\Fin_\ast\ra (\cC)^\sqcup}$ of 2.4.3.6 loc.cit.; here $(-)^\sqcup\colon \Cat\ra \Opd$ is the cocartesian operad functor from \cref{sec:cocart-oprightfib} and $\inc$ denotes the inclusion $\Opd\subset\Opd^\gn$. Precomposing $\gamma$ with the functor $(-)^\col\colon \Opd^\un\ra \Cat$, we get a natural transformation $\gamma\circ (-)^\col\colon (-)^\col\times\Fin_*\ra \inc\circ ((-)^\col)^\sqcup$ of functors $\Opd^\un\ra \Opd^\gn$. Pulling it back along the natural transformation $\inc\ra \inc\circ ((-)^\col)^\sqcup$ induced by the counit $\smash{\cO\ra(\cO^\col)^\sqcup}$ of the adjunction described in \cref{sec:cocart-oprightfib}, we obtain a natural transformation $G\ra \inc$ of functors $\Opd^\un\ra \Opd^\gn$ where $G\colon \Opd^\un\ra \Opd^{\gn}$ sends a unital operad $\cO$  to the pullback of generalised operads $\smash{\cO\times_{(\cO^\col)^\sqcup} (\cO^\col\times\Fin_\ast)\in\Opd^\gn}$. Note that because of the natural equivalence \begin{equation}\label{equ:assem-inverse-0}\smash{G(\cO)_{\langle 0 \rangle}=\big(\cO\times_{(\cO^\col)^\sqcup} (\cO^\col\times\Fin_\ast))_{\langle 0\rangle}\simeq \cO^\col},\end{equation} the functor $G\colon \Opd^\un\ra\Opd^\gn$ restricts to a functor of the form $G'\colon  \Opd^\gc\ra \Opd^{\gn,\red}$, and we claim that this is the inverse to the equivalence $\smash{\Assem\colon  \Opd^{\gn,\red}\simeq \Opd^\gc}$. To see this, note that as $\Assem$ is left adjoint to $\inc$, the natural transformation $G\ra \inc$ yields a natural transformation $\Assem\circ G\ra \id_{\Opd^\un}$ which restricts to a natural transformation $\Assem\circ G'\ra \id_{\Opd^\gc}$, so it suffices to show that the latter transformation is an equivalence, i.e.\,that for all $\cO\in \Opd^\gc$ the natural map of operads $\smash{\Assem(\cO\times_{(\cO^\col)^\sqcup} (\cO^\col\times\Fin_\ast))\ra \cO}$ is an equivalence. This is proved in the first paragraph of the proof of 2.3.4.4 loc.cit..

From the natural equivalence \eqref{equ:assem-inverse-0}, we see that the functor $\smash{(-)_{\langle 0\rangle}\colon \Opd^{\gn,\red}\ra\cS}$ translates under the equivalence $\smash{G'\colon \Opd^\gc \simeq \Opd^{\gn,\red}}$ to $\smash{(-)^\col\colon \Opd\ra \cS}$, so putting everything together we obtain a commutative diagram of categories
\begin{equation}\label{equ:functors-vs-generalised-operads}
\begin{tikzcd}[row sep=0.3cm]
\int_\cS \Fun(-,\Opd^\red)\arrow[drrr,"\pi",swap]&[-1cm]\simeq&[-1cm] \int_\cS \Fam^\red(-)&[-1cm]\simeq&[-1cm] \Opd^{\gn,\red}\arrow[r,"\simeq"',"\Assem"]&[10pt] \Opd^\gc\arrow[dll,"(-)^\col"]\\
&&&\cS
\end{tikzcd}
\end{equation}
whose horizontal functors are equivalences. To finish the proof, it suffices to show that the top row is induced by taking colimits in $\Opd$. This is part of the proof of \cite[Proposition 2.2]{HorelKrannichKupers}.
\end{proof}

\begin{rem}
The previous proof crucially relies on \cite[2.3.4.4]{LurieHA} which in turn relies on 2.4.3.6 loc.cit., which states that a certain functor $\gamma \colon \cC \times \Fin_\ast \to \cC^\sqcup$ that sends $(c,S\sqcup\ast)$ to $(c)_{i\in S}$ is an approximation in the sense of 2.3.3.6 loc.cit., by ``unwinding the definitions''. Some comments on this statement are in order (we use the notation and terminology of loc.cit.):
\begin{enumerate}[leftmargin=*]
	\item \label{enum:not-approx}As stated, the statement appears to be incorrect: consider the case $\cC = [1]$ and the unique morphism $(0,1) \to (1,1)=\gamma(1,\langle 2\rangle)$ in $[1]^\sqcup$ covering $\id_{\langle 2 \rangle}$. If $\gamma$ were an approximation, then part (2) in the definition of an approximation in 2.3.3.6 loc.cit.\,would imply that $(0,1) \to (1,1)$ is up to equivalence in the image of $\gamma$ which is not the case: the image of $\gamma$ does not even contain an object equivalent to $(0,1)$.
	\item For most applications of 2.4.3.6 in loc.cit.\,the weaker claim that $\gamma$ is a \emph{weak approximation} in the sense of 2.3.3.6 loc.cit.\,would suffice, but this is not the case either: consider the same example as in \ref{enum:not-approx}. If the category in (2') of 2.3.3.6 associated to the morphism $(0,1) \to (1,1)=\gamma(1,\langle 2\rangle)$ were nonempty, then there would be a factorisation of it as an inert map $(0,1)\ra \gamma(c,\langle m\rangle)$ for $c\in[1]$ and $m\ge0$ followed by a morphism $\gamma(c,\langle m\rangle)\ra \gamma(1,\langle 2\rangle)$ in the image of $\gamma$. As $(0,1) \to (1,1)$ covers $\id_{\langle 2\rangle}$, the underlying morphism of $(0,1)\ra \gamma(c,\langle m\rangle)$ is inert and injective, so we have $m=2$ and $c=1$. But there is no inert map $(0,1)\ra (1,1)$ since such a map would by definition be $\gamma$-cocartesian (see 2.1.2.3 loc.cit.), so in particular $\id_{(0,1)}$ would factor over $(1,1)$ which is not the case.
	\item However, if $\cC$ is a groupoid (which is the only case needed for the proof of \cite[2.3.4.4]{LurieHA} and thus for that of \cref{thm:operad-grp-type}), the functor $\gamma$ is an approximation as claimed: condition (1) in 2.3.3.6 loc.cit.\,holds since the morphisms $(\id_c \times \rho^i)$ in $\cC \times \Fin_\ast$ for $c\in \cC$ give the required locally cocartesian lifts along $\cC \times \Fin_\ast\ra \Fin_*$ whose images in $\cC^{\sqcup}$ are inert. For condition (2), we use that, since $\cC$ is a groupoid, any active morphism $X \to (c)_{i \in \langle n \rangle^\circ} = \gamma(c,\langle n \rangle)$ in $\cC^\sqcup$ with underlying active map $\varphi \colon \langle m \rangle \to \langle n \rangle$ is equivalent to a map in $\cC^\sqcup$ covering $\varphi$ of the form $(c)_{i \in \langle m \rangle^\circ} \to (c)_{i \in \langle n \rangle^\circ}$ given by the identities; this has a cartesian lift as required: $(\id_c \times \varphi)$.
\end{enumerate}
\end{rem}

\subsection{Tangential structures for operads}\label{sec:tangential-structures-operads}
Informally speaking, \cref{thm:operad-grp-type} says that groupoid-coloured operads $\cO$ are the same as reduced operad-valued functors $\theta\colon B \ra \Opd^\red$ out of groupoids, and that under this correspondence one has $\smash{B\simeq \cO^\col}$ and $\smash{\cO\simeq \colim_B\theta}$. We write $\smash{\theta_\cO\colon \cO^\col\ra \Opd^\red}$ for the functor associated to a groupoid-coloured operad $\cO$, and call $\smash{\theta_\cO(c)\in \Opd^\red}$ for a colour $c\in\cO^\col$ the \emph{reduced cover of $\cO$ at $c$}. A map $\varphi\colon \cO\ra\cP$ between groupoid-coloured operads is \emph{an equivalence on reduced covers} if the maps $\theta_\cO(c)\ra \theta_{\cP}(\varphi(c))$ are equivalences for all $c\in\cO^\col$. Given groupoid-coloured operads $\cO$ and $\cO'$, \cref{thm:operad-grp-type} in particular gives an equivalence 
\[
	\smash{\Map_{\Opd}(\cO,\cO')\simeq \Map_{\int_{\cS}\Fun(-,\Opd)}((\cO^\col,\theta_\cO),(\cU^\col,\theta_\cU))}
\]
and thus a fibre sequence
\begin{equation}\label{equ:maps-gc-operad}
	\smash{\Map_{\Fun(\cO^\col,\Opd)}(\theta_\cO, (\theta_{\cO'}\circ \varphi))\ra \Map_{\Opd}(\cO,\cO')\xlra{(-)^\col} \Map_{\cS}(\cO^\col,\cO'^\col)}
\end{equation}
in which we took fibres over a map $\varphi\in \Map_{\cS}(\cO^\col,\cO'^\col)$. Moreover, writing $\Opd^{\gc,\simeq_{\red}}\subset \Opd^{\gc}$ for wide subcategory on maps that are equivalences on reduced covers, we can combine \cref{thm:operad-grp-type} with the identification $\smash{\int_{\cS}\Map_{\cS}(-,\Opd^{\red,\simeq})\simeq \cS_{/\Opd^{\red,\simeq}}}$ to an equivalence
\begin{equation}\label{equ:gc-redequ-identification}
	\textstyle{\smash{\Opd^{\gc,\simeq_{\red}}\simeq \cS_{/\Opd^{\red,\simeq}}}}
\end{equation}
that sends $\smash{\cO\in \Opd^{\gc,\simeq_{\red}}}$ to $\smash{\theta_{\cO}\colon \cO^{\col}\ra \Opd^{\red,\simeq}}$. In particular, since any automorphism of a groupoid-coloured operad is an equivalence on reduced covers, we obtain an equivalence
\begin{equation}\label{equ:aut-gc-operad}
	\smash{  \Aut_{\Opd}(\cO)\simeq \Aut_{\cS_{/\Opd^\simeq}}(\theta_{\cO})\quad\text{ for all }\cO\in \Opd^{\gc}}.
\end{equation}
Closely related, and in analogy with the notion of a tangential structure for manifolds, we define:

\begin{dfn}\label{dfn:tangial-structures-opd} A \emph{tangential structure} for a reduced operad $\cU$ is a map $\theta\colon B\ra \BAut(\cU)$ for a groupoid $B$. The \emph{$\theta$-framed $\cU$-operad} for a tangential structure $\theta$ defined as the colimit
\[\cU^\theta\coloneqq \colim\big(B\xra{\theta} \BAut(\cU)\hookrightarrow \Opd\big)\in\Opd^{\gc}.\]
\end{dfn}

\noindent The colimit in \cref{dfn:tangial-structures-opd} can be made more concrete as follows. By the first part of \cref{thm:operad-grp-type}, the category of colours of $\cU^\theta$ is $(\cU^\theta)^\col\simeq B$ and the multi-operations are given by 
\begin{equation}\label{equ:multi-oeprations-framed}
	\textstyle{\Mul_{\cU^\theta}\big((c_s)_{s\in S};d\big)\simeq \Mul_{\cU}\big((*)_{s\in S};*\big)\times \bigsqcap_{s\in S}\Map_{B}\big(c_s,d\big)}
\end{equation} where the $c_s$ and $d$ are colours in $(\cU^\theta)^\col\simeq B$ and $\ast$ is the unique colour of $\cU$ (see \cref{lem:tangential-opd-multi-ops} below). As a result of the above discussion, for a fixed reduced operad $\cU$, any groupoid-coloured operad $\cO$ whose reduced covers are all equivalent to $\cU$, is equivalent to the $\theta$-framed operad $\cU^\theta$ for a unique tangential structure $\theta_{\cO}\colon \cO^\col \ra\BAut(\cU)$ for $\cU$. 

\begin{ex}\label{exam:framings-for-operads} There are standard examples of $\theta$-framed $\cU$-operads for a reduced operad $\cU$:
\begin{enumerate}
	\item The $(\varnothing \ra \BAut(\cU))$-framed $\cU$-operad is the initial operad, which in the model for $\Opd$ from \cref{sec:operads} corresponds to the inclusion $\{\ast\}\subset\Fin_\ast$.
	\item The $(\ast \ra\BAut(\cU))$-framed $\cU$-operad is $\cU$ itself.
	\item There is the $\smash{\id_{\BAut(\cU)}}$-framed $\cU$-operad $\smash{\cU^\id = \colim(\BAut(\cU) \hookrightarrow \Opd^\un)}$. As $\smash{\id_{\BAut(\cU)}}$ is terminal in the category $\smash{\cS_{/\BAut(\cU)}}$ of tangential structures for $\cU$, the equivalence \eqref{equ:gc-redequ-identification} implies that $\cU^\id$ is terminal in the full subcategory of $\smash{\Opd^{\gc,\simeq_{\red}}}$ spanned by those operads whose reduced covers are all equivalent to $\cU$. In particular, we have $\Aut(\cU^\id)\simeq \ast$.
\end{enumerate}
\end{ex}

\noindent The above discussion can also be generalised to take the truncations $\cU_{\le k}$ of $\cU$ from \cref{sec:truncation-operads} into account. As explained in \cref{sec:truncation-operads}, they fit into a tower $\cU_{\leq \bullet}$ of reduced operads and there is a corresponding tower $\BAut(\cU_{\leq \bullet})$ of groupoids. A \emph{tower of tangential structures} for $\cU$ is a map $\smash{\theta_\bullet \colon B_\bullet \to \BAut(\cU_{\leq \bullet})}$ in $\Tow(\cS)$. Such a map induces a tower $\smash{\cU^{\theta_\bullet}_{\leq \bullet}\in\Tow(\Opd^\gc)}$ featuring the $\theta_k$-framed $\cU_{\le k}$-operad. For $\smash{\theta_\bullet=\id_{\BAut(\cU_{\le \bullet})}}$ this yields a tower $\smash{\cU^\id_{\leq \bullet}}$ which is terminal in the subcategory of $\Tow(\Opd^{\gc})$ with objects those towers $\cO_\bullet$ for which the reduced covers of $\cO_k$ for $k\ge1$ are equivalent to $\cU_{\le k}$, and maps of towers that are levelwise equivalences on reduced covers.

\subsection{$\theta$-framed operads and cocartesian operads}\label{sec:cocartesian}
One of the simplest reduced operad is the commutative operad $\Com$. Since it is terminal in $\Opd$ (see \cref{sec:operads-intro}), we have $\BAut(\Com)\simeq\ast$, so tangential structures $\theta\colon B\ra\BAut(\Com)$ for $\Com$ are the same as groupoids $B$. The associated $B$-framed $\Com$-operad $\Com^B$ is also known under a different name: it is the cocartesian operad $B^\sqcup$.

\begin{lem}\label{lem:cocartesian-is-colimit}For a groupoid $X$, there is a natural equivalence of operads
\[X^\sqcup\simeq \colim_{X}(\const_{\Com})=\Com^X\]
where the right-hand side is the colimit of the constant functor $X\ra \Opd$ on the terminal operad. 
\end{lem}

\begin{proof}
Restricting the adjunction $(-)^\col\dashv(-)^\sqcup$ to the full subcategories $\Opd^\gc\subset \Opd^\un$ and $\cS\subset\Cat$, it suffices to show that $\smash{\colim_{(-)}(\const_{\Com})\colon \cS\ra \Opd^\gc}$ is right adjoint to $\smash{(-)^\col\colon \Opd^\gc\ra \cS}$, by the uniqueness of adjoints. By \cref{thm:operad-grp-type} this equivalent to showing that the projection $\smash{\pi\colon \int_{\cS}\Fun(-,\Opd^\red)\ra \cS}$ is left adjoint to the functor that sends a groupoid $B\in \cS$ to $(B,\const_\Com)$. This is a general fact: for a functor $\varphi\colon \cC\ra \Cat$ whose values have terminal objects, the unstraightening $\smash{\pi\colon \int_{\cC}\varphi\ra \cC}$ has a right adjoint given by sending $c\in \cC$ to $(c,t_{\varphi(c)})$ where $t_{\varphi(c)}\in\varphi(c)$ is the terminal object (combine the dual version of \cite[7.3.2.6]{LurieHA} with 7.3.2.1 and 7.3.2.2 loc.cit.).
\end{proof}

\noindent \cref{lem:cocartesian-is-colimit} has the following consequence: 

\begin{lem}\label{lem:pullbacks-groupoid-type} A map between groupoid-coloured operads $\varphi\colon \cO\ra \cP$ is an equivalence on reduced covers if and only if it is an operadic right-fibration in the sense of Definition \ref{dfn:oprightfib}.
\end{lem}

\begin{proof}
Since $\cO^\col$ and $\cP^\col$ are groupoids, any map between them is a right-fibration, so $\varphi$ is an operadic right-fibration if and only if the square in Definition \ref{lem:operadic-right-fib} is a pullback of operads.  Since $\cS\subset \Cat$ is closed under limits, and thus also $\smash{\Opd^\gc\subset\Opd}$, this is equivalent to the square being a pullback in $\Opd^\gc$. Using the equivalence $\Opd^\gc\simeq\int_{\cS}\Fun(-,\Opd^\red)$ together with \cref{lem:cocartesian-is-colimit}, the square is equivalent to a square in $\smash{\int_{\cS}\Fun(-,\Opd^\red)}$ of the form 
\[
	\begin{tikzcd}[column sep=1.5cm]
	(\cO^\col,\theta_\cO)\dar[swap]{(\id,t)}\rar{(\varphi^\col, \eta_\varphi)}\dar& (\cP^\col,\theta_\cP)\dar{(\id,t)}\dar\\
	(\cO^\col,\const_{\Com})\rar{(\varphi^\col, \id)}&(\cP^\col,\const_{\Com})
	\end{tikzcd}
\]
for a natural transformation $\smash{\eta_\varphi\colon \theta_\cO\ra \varphi^*\theta_\cP}$. Here the natural transformations indicated by $t$ are induced by the fact that $\smash{\Com\in\Opd^\red}$ is terminal. This a pullback in $\smash{\int_{\cS}\Fun(-,\Opd^\red)}$ if and only if $\eta_\varphi$ is an equivalence which is the definition of $\varphi$ being an equivalence on reduced covers. 
\end{proof}

\noindent \cref{lem:pullbacks-groupoid-type} allows us to establish the promised description \eqref{equ:multi-oeprations-framed} of the multi-operations in $\cU^\theta$:

\begin{lem}\label{lem:tangential-opd-multi-ops} Given a functor $\theta\colon B\ra\Opd^{\red}$, the space of multi-operations of the colimit $\colim_B\theta$ in $\Opd$ from a collection of colours $(c_s)_{s\in S}$ in $(\colim_B\theta)^\col\simeq X$ to a colour $d$ is given by
\[
	\textstyle{\Mul_{\colim_B\theta}((c_s)_{s \in S};d) \simeq \Mul_{\theta(d)}((\ast)_{s \in S};\ast) \times \bigsqcap_{s \in S} \Map_B(c_s,d).}
\]
\end{lem}

\begin{proof}For any unital operad $\cO$, the unit of the adjunction $(-)^\col\dashv(-)^\sqcup$ induces a natural map
\begin{equation}\label{equ:go-to-unary}
	\textstyle{p_\cO\colon \Mul_{\cO}((c_i)_{i\in S};d)\lra \Mul_{(\cO^\col)^\sqcup}(c_s;d)\simeq \bigsqcap_{s\in S}\Mul_{\cO}(c_s;d)}
\end{equation} 
which in terms of operadic compositions is induced by inserting the unique $0$-ary operations in all but one of the entries. For $c\theta\coloneqq\colim_B\theta$, we have $(c\theta)^\col\simeq B$, so the target of $p_{c\theta}$ is $ \sqcap_{s \in S} \Map_B(c_s,d)$, which is empty unless all $c_s\in B$ are in the same component as $d$, so the same holds for the domain. We may thus assume $c_s=d$ for all $s$ and consider the commutative diagram
\[\hspace{-0.2cm}
	\begin{tikzcd}[column sep=0.25cm]
	\sqcap_{s\in S}\Mul_{c\theta}(d;d)\times \Mul_{\theta(d)}((*)_{s\in S};*)\rar\dar{\id\times p_{\theta(d)}}&\sqcap_{s\in S}\Mul_{c\theta}(d;d)\times  \Mul_{c\theta}((d)_{s\in S};d) \rar\dar{\id\times p_{c\theta}}&\Mul_{c\theta}((d)_{s\in S};d)\dar{p_{c\theta}}\\
 	\sqcap_{s\in S}\Mul_{c\theta}(d;d)\times   \sqcap_{s\in S}\Mul_{\theta(d)}(*;*)\rar&\sqcap_{s\in S}\Mul_{c\theta}(d;d)\times    \sqcap_{s\in S}\Mul_{c\theta}(d;d) \rar&  \sqcap_{s\in S}\Mul_{c\theta}(d;d)
	\end{tikzcd}
\]
where the left horizontal maps are induced by the map $\theta(d)\ra c\theta$ induced by $\{d\}\subset B$ and the right ones by operadic composition. As $(c\theta)^\col\simeq B$, it suffices to show that the top horizontal composition is an equivalence. As $\cU$ is reduced, we have $\Mul_{\theta(d)}(\ast;\ast)\simeq\ast$, so the bottom horizontal composition is an equivalence and it is thus enough to show that the composition on vertical fibres is an equivalence. In view of the action of the group $ \sqcap_{s\in S}\Mul_{c\theta}(d;d)$ on the diagram by precomposition, it suffices to check this on fibres over the identity operations. In this case the map on fibres agrees with the map on vertical fibres of the square of spaces of multi-operations induced by the square of operads in Definition \ref{lem:operadic-right-fib} for $(\cO\ra\cP)=(\theta(d)\ra c\theta)$ induced by $\{d\}\subset X$. As $\theta(d)\ra c\theta$ is by construction an equivalence on reduced covers, this square of operads is a pullback by \cref{lem:pullbacks-groupoid-type} and thus induces a pullback on spaces of multioperations, so the claim follows.\end{proof}
 
\noindent A similar argument shows that taking $\theta$-framed operads for tangential structures $\theta$ commutes with truncations of operads as discussed in \cref{sec:truncation-operads}, in the following sense:

\begin{lem}\label{lem:truncated-tangential} Fix $k\ge1$.
\begin{enumerate}[leftmargin=*]
	\item \label{enum:truncated-tangential-i} A groupoid-coloured operad $\cO$ is $k$-truncated if and only if all its reduced covers are $k$-truncated.
	\item \label{enum:truncated-tangential-ii} Viewing $\Opd^{\le k, \un}$ as the full subcategory of $\Opd^{\un}$ of $k$-truncated operads, the equivalence of \cref{thm:operad-grp-type} restricts to an equivalence over $\cS$
	\[
		\smash{\textstyle{\int_{\cS}\Fun(-,\Opd^{\le k,\red})\simeq \Opd^{\le k,\gc}}}
	\] 
	where $\Opd^{\le k,\red}\coloneqq \Opd^{\le k,\un}\cap \Opd^{\red}$ and $\Opd^{\le k,\gc}\coloneqq \Opd^{\le k,\un}\cap \Opd^{\gc}$.
	\item \label{enum:truncated-tangential-iii} For a functor $\theta\colon B\ra\Opd^{\red}$ out of a groupoid, there is an equivalence (colimits are taken in $\Opd$)
	\[
		\big(\colim_{b\in B}\theta(b)\big)_{\le k}\simeq \colim_{b\in B}(\theta(d)_{\le k}).
	\]
	In particular, if $\theta$ is a tangential structure for a reduced operad $\cU$, we have \[(\cU^\theta)_{\le k}\simeq (\cU_{\le k})^{\theta_k}\] where $\theta_k$ is the composition of $\theta$ with the map $\BAut(\cU)\ra\BAut(\cU_{\le k})$ induced by truncation.
\end{enumerate}
\end{lem}

\begin{proof}Arguing as in the beginning of the proof of \cref{lem:tangential-opd-multi-ops}, to show  \ref{enum:truncated-tangential-i}, it suffices to prove that for each colour $d\in\cO^\col$ and finite set $S$, the map $\Mul_{\cO}((d)_{s\in S};d)\ra \lim_{S'\subseteq S,|S'|\le k}\Mul_{\cO}((d)_{s\in S'};d)$ is an equivalence if and only if the map $\Mul_{\theta_\cO(d)}((\ast)_{s\in S};\ast)\ra \lim_{S'\subseteq S,|S'|\le k}\Mul_{\theta_\cO(d)}((\ast)_{s\in S'};\ast)$ is an equivalence. Similarly to the previous proof, we consider the commutative square
\[
	\begin{tikzcd}
	\Mul_{\cO}((d)_{s\in S};d)\rar\dar{p_{\cO}}& \lim_{S'\subseteq S,|S'|\le k}\Mul_{\cO}((d)_{s\in S'};d)\dar{p_{\cO}}\\
	\bigsqcap_{s\in S}\Mul_{\cO}(d;d)\rar& \lim_{S'\subseteq S,|S'|\le k}\bigsqcap_{s\in S'}\Mul_{\cO}(d;d).
	\end{tikzcd}
\]
The bottom row is an equivalence, so the top row is an equivalence if and only if the map on vertical fibres over all basepoints is one. The action of the group $\bigsqcap_{s\in S}\Mul_{\cO}(d;d)$ on the square exhibits the latter to be equivalent to the map on vertical fibre over the identity basepoint being an equivalence. Arguing as in the previous proof, this map on fibres over the identity is given by the map $\Mul_{\theta(d)}((\ast)_{s\in S};\ast)\ra \lim_{S'\subset S,|S'|\le k}\Mul_{\theta(d)}((\ast)_{s\in S'};\ast)$, so the claim follows. Combining \ref{enum:truncated-tangential-i} with \cref{thm:operad-grp-type} immediately implies \ref{enum:truncated-tangential-ii}. From this \ref{enum:truncated-tangential-ii} follows, as it allows us to identify $\Opd^{\leq k,\gc}$ as the essential image of $\int_{\cS}\Fun(-,\Opd^{\le k,\red})$ under the equivalence of \cref{thm:operad-grp-type}. This implies \ref{enum:truncated-tangential-iii} by using that the equivalence in \cref{thm:operad-grp-type} is given by taking colimits.
\end{proof}

\subsection{Semidirect products of ordinary operads in spaces}\label{rem:semidirect} The notion of a $\theta$-framed $\cU$-operad for a reduced operad $\cU$ and a tangential structure $\theta\colon B\ra\BAut(\cU)$ (see \cref{dfn:tangial-structures-opd}) has a known analogue for classical operads in the $1$-category $\cat{S}$ of topological spaces or Kan complexes: for a group object $\cat{G}$ in $\cat{S}$ acting on an ordinary $1$-coloured symmetric operad $\cat{U}$ in $\cat{S}$, there is an associated $1$-coloured symmetric operad $\cat{U}\rtimes \cat{G}$ in $\cat{S}$, called the \emph{semidirect product of $\cat{U}$ and $\cat{G}$}, whose multi-operations are given by $(\cat{U}\rtimes \cat{G})(k)=\cat{U}(k)\times \cat{G}^k$ and whose operadic composition involves that of $\cat{U}$ and the $\cat{G}$-action on $\cat{U}$ (see \cite[Section 2]{SalvatoreWahl}). If $\cat{U}$ has weakly contractible $0$- and $1$-ary operations, then its operadic nerve $\cU\coloneqq N^{\otimes}(\cat{U})\in\Opd$ from \cite[2.1.1.27]{LurieHA} is a reduced operad in our sense and the action induces a tangential structure $\theta\colon \BG\ra\BAut(\cU)$ (see the proof of \cref{prop:semidirect-is-tangential}). As we shall we in the following proposition, the associated  $\theta$-framed $\cU$-operad turns out to be equivalent to the operadic nerve $\cU\rtimes G\coloneqq N^{\otimes}(\cat{U}\rtimes \cat{G})$ of the semidirect product (in the discrete case, this can also be extracted from \cite[Chapter 4]{Azam}).

\begin{prop}\label{prop:semidirect-is-tangential}In the notation above, if the spaces of $0$- and $1$-ary operations $\cat{U}(0)$ and $\cat{U}(1)$ are contractible, then there is an equivalence of operads $\cU\rtimes G\simeq \cU^{\theta}$.
\end{prop}

\begin{rem}In view of \cref{prop:semidirect-is-tangential}, specialising the fibre sequence \eqref{equ:maps-gc-operad} and the equivalence \eqref{equ:aut-gc-operad} to the operad $\cU\rtimes G$ recovers a result of Horel--Willwacher \cite[Theorem 1.1]{HorelWillwacher} (note that they work in the model category of ordinary $1$-coloured operads in $\cat{S}$ whose underlying $\infty$-category is not equivalent to a full subcategory of $\Opd$, but to one of $\Opd_{\ast/}$; see \cite[Section 2.3]{HorelKrannichKupers}).
\end{rem}

\begin{proof}[Proof of \cref{prop:semidirect-is-tangential}]We assume that $\cat{S}$ is the category of Kan-complexes; the case of topological spaces follows by first taking singular simplicial sets. Recall from \cite[2.1.1.22, 2.1.1.23]{LurieHA} that the operadic nerve $N^{\otimes}(\cat{O})\in\Opd$ of a coloured symmetric simplicial operad $\cat{O}$ in $\cat{S}$ is constructed as the coherent nerve $N(-)$ in the sense \cite[1.1.5.5]{LurieHTT} of a certain $\cat{S}$-enriched ordinary category $\cO^{\otimes}$ with a functor to $\Fin_*$, defined explicitly in terms of $\cat{O}$ \cite[2.1.1.22, 2.1.1.23]{LurieHA}.

\medskip

\noindent First we explain how the $\cat{G}$-action on $\cat{U}$ induces a functor $\theta\colon \BG\ra\BAut(\cU)$. This is not entirely obvious, since the operadic nerve considered as a functor from the $1$-category of coloured symmetric simplicial operads in $\cat{S}$ to the underlying $1$-category of the simplicial category whose coherent nerve is $\Opd$ (see \cite[2.1.4.1]{LurieHA}) is not simplicially enriched in evident way (c.f.~\cite[Section 6]{JoyalTierney}). Instead, we consider the $\cat{S}$-enriched category $\cat{U}^{\cat{G},\otimes}$ with objects finite pointed sets, morphism spaces $\cat{Map}_{\cat{U}^{\cat{G},\otimes}}(S\sqcup\ast,T\sqcup\ast)\coloneqq \sqcup_{\alpha\colon S\sqcup\ast \ra T\sqcup\ast}\cat{G}\times \smallprod_{t\in T}\cat{U}(\alpha^{-1}(t))$, and composition induced by the maps $\smash{\cat{G}\times [\sqcap_{j\in\beta^{-1}(i)}\cat{U}(\alpha^{-1}(j))]\times (\cat{G}\times \cat{U}(\beta^{-1}(i))\ra \cat{G}\times \cat{U}((\beta\circ\alpha)^{-1}(i))}$ giving by first letting the second copy of $\cat{G}$ act diagonally on $\sqcap_{j\in\beta^{-1}(i)}\cat{U}(\alpha^{-1}(j))$, followed by the map given by operad composition in $\cat{U}$ and multiplication in $\cat{G}$. Projection induces a functor $\cat{U}^{\cat{G},\otimes}\ra \Fin_*\times \cat{G}$ such that its composition $\cat{U}^{\cat{G},\otimes} \to \cat{G}$ with the projection to $\cat{G}$ is a cocartesian fibration of $\cat{S}$-enriched categories whose fibre is $\cat{U}^\otimes$. Applying $N(-)$ to this composition we obtain a cocartesian fibration with fibre $\cU\coloneqq N(\cat{U}^\otimes)$, the straightening gives the desired functor $\theta\colon \BG\ra\BAut(\cU)$.

\medskip

\noindent To show $\cU\rtimes G\simeq \cU^{\theta}$, first note that, by the proof of \cref{thm:operad-grp-type}, it suffices to show that the two generalised operads $\smash{\Assem^{-1}(\cU\rtimes G)}$ and $\smash{\Assem^{-1}(\cU^\theta)}$ are equivalent in $\Opd^\gn$ which we can check in $(\Cat_\infty)_{/\Fin_*}$, since $\Opd^\gn\hookrightarrow (\Cat_\infty)_{/\Fin_*}$ is a replete subcategory inclusion. To this end, note that by the description of $\Assem^{-1}(-)$ in the mentioned proof, there is a cartesian square in $\Cat_{/\Fin_*}$
\begin{equation}\label{equ:pullback-simplicial-operad}
	\begin{tikzcd}[row sep=0.5cm, column sep=0.5cm]
	\Assem^{-1}(\cU\rtimes G)\rar\dar &\cU\rtimes G\dar\\
	\Fin_*\times (\cU\rtimes G)^{\col}\rar{\gamma}& ((\cU\rtimes G)^{\col})^{\sqcup}.
	\end{tikzcd}
\end{equation}
This can be obtained as the coherent nerve of a square of $\cat{S}$-enriched categories over $\Fin_*$: firstly, for any $\cat{S}$-enriched category $\cat{C}$, the cocartesian operad $N(\cat{C})^\sqcup$ is the coherent nerve of the $\cat{S}$-enriched category $\cat{C}^\sqcup$ over $\Fin_*$ whose objects are sequences $(c_s)_{s\in S}$ of objects in $\cat{C}$ indexed by a finite set $S$ and whose morphism spaces are defined by the formulas \eqref{equ:maps-total-space-operads} and \eqref{equ:operations-cocartesian-operad}. In these terms, the functor $\gamma\colon \Fin_*\times N(\cat{C})\ra N(\cat{C})^\sqcup$ arises as $N(-)$ applied to the $\cat{S}$-enriched functor $\Fin_*\times \cat{C}\ra \cat{C}^\sqcup$ that sends $(S,c)$ to $(c)_{s\in S}$. Specialising this to $\cat{C}=(\cat{U} \rtimes \cat{G})^\col$, this gives the bottom row in \eqref{equ:pullback-simplicial-operad}. Moreover, by definition of the operadic nerve, we have $\cU\rtimes G\simeq N((\cat{U}\rtimes \cat{G})^\otimes)$ where $(-)^\otimes$ is as before. The projections $(\cat{U}\rtimes \cat{G})(k)=\cat{U}\times \cat{G}^k\ra \cat{G}^k$ induce a functor $(\cat{U}\rtimes \cat{G})^\otimes\ra  \cat{G}^\sqcup$ whose coherent nerve gives the right vertical map in \eqref{equ:pullback-simplicial-operad}. The latter functor is a fibration in the Bergner model structure on simplicial categories by inspection using the description of fibrations resulting from \cite[A.3.2.24, A.3.2.25, A.3.2.9]{LurieHTT}. Since this model category is left proper (see A.3.2.4 loc.cit.) and $N(-)$ is a Quillen equivalence to the Joyal model structure on simplicial sets whose underlying $\infty$-category is $\Cat$ (see 2.2.5.1 loc.cit.), we can compute the pullback \eqref{equ:pullback-simplicial-operad} in simplicial categories and obtain $\Assem^{-1}(\cU\rtimes G)\simeq N((\Fin_*\times\cat{G}) \times_{\cat{G}^\sqcup}(\cat{U}\rtimes \cat{G})^\otimes)$. Spelling out the definition, ones see that the strict pullback $(\Fin_*\times\cat{G}) \times_{\cat{G}^\sqcup}(\cat{U}\rtimes \cat{G})^\otimes$ is precisely $\cat{U}^{\cat{G},\otimes}$, so it remains to show that $\smash{\Assem^{-1}(\cU^\theta)}\in\Cat_{/\Fin_*}$ is equivalent to $N(\cat{U}^{\cat{G},\otimes})\in\Cat_{/\Fin_*}$. But by the proof of  \cref{thm:operad-grp-type}, $\smash{\Assem^{-1}(\cU^\theta)}\in(\Cat_\infty)_{/\Fin_*}$ is the unstraightening of $\theta\colon \BG\ra \BAut(\cat{U})\hookrightarrow \Cat_{/\Fin_*}$ which is $N(\cat{U}^{\cat{G},\otimes})$ by construction.
\end{proof}

\section{Presheaves, Day convolution, and a decomposition result}
The purpose of this section is twofold: we record general results on Day convolution symmetric monoidal structures on categories $\PSh(\cC)$ of space-valued presheaves on symmetric monoidal categories $\cC$ (Sections \ref{sec:day}--\ref{sec:modules-day}) and we establish a natural pullback decomposition of $\PSh(\cC)$ involving the category of presheaves on certain full subcategories of $\cC$ and their ``complements'' (\cref{sec:recollement}). 

\subsection{Day convolution}\label{sec:day} Given a symmetric monoidal category $\cC$, the category of space-valued presheaves $\PSh(\cC)=\Fun(\cC^{\op},\cS)$ carries a symmetric monoidal structure by Day convolution \cite[4.8.1.12, 4.8.1.13]{LurieHA} which is on objects given by 
\begin{equation}\label{equ:day-conv}
	(F\otimes G)(c)\simeq \underset{c\ra c'\otimes c''}{\colim}\,F(c')\times F(c''),
\end{equation} 
where the colimit is taken over the pullback $(\cC\times \cC)_{c/}$ of the monoidal product $\otimes \colon \cC\times \cC\ra \cC$ along the forgetful functor $\cC_{c/}\ra \cC$. Its unit is given by the representable presheaf $\Map_{\cC}(-,1_\cC)$ on the monoidal unit $1_\cC$.  There are several ways to characterise this symmetric monoidal structure: in $\CMon(\Cat)$ it can be characterised as the unique symmetric monoidal structure on $\PSh(\cC)$ that preserves small colimits in both variables and for which the Yoneda embedding can be made symmetric monoidal (see 4.8.1.12 loc.cit.), in $\CMon^\lax(\Cat)$ it can be characterised by the natural equivalence (see \cref{sec:lax} for the notation and 2.2.6.8 loc.cit.\,for a reference)
\[
	\Map_{\CMon^\lax(\Cat)}(-\times \cC^{\op},\cS)\simeq \Map_{\CMon^\lax(\Cat)}(-,\PSh(\cC)),
\] 
where $\times$ is the product in $\CMon^\lax(\Cat)$ (which is computed in ${\Cat}_{/\Fin_\ast}$, so given by taking pullbacks over $\Fin_\ast$). Combined with \eqref{equ:op-equivalence} and the Yoneda lemma, the second characterisation shows that there is a lift of the precomposition-functor $\PSh(-)\colon \Cat^\op\ra \Cat$ to a functor
\begin{equation}\label{equ:restriction-is-lax}
	\PSh(-)\colon \CMon^\oplax(\Cat)^\op\lra \CMon^\lax(\Cat).
\end{equation}
Since precomposition of presheaves along a functor has a left adjoint induced by taking left Kan extension, this functor lands in the subcategory $\CMon^{\lax,R}(\Cat)\subset \CMon^\lax(\Cat)$ so we may postcompose it with the equivalence \eqref{equ:mates} induced by taking left adjoints to obtain a lift
\begin{equation}\label{equ:left-kan-is-oplax}
	\PSh(-)\colon \CMon^\oplax(\Cat)\lra  \CMon^\oplax(\Cat)
\end{equation}
of the left Kan extension-functor $\PSh(-)\colon \Cat\ra \Cat$. 

\begin{ex}\label{ex:mapping-unit-in-is-lax}Inclusion of the unit in a symmetric monoidal category $\cC$ gives a symmetric monoidal functor $1\hookrightarrow \cC$. Applying \eqref{equ:left-kan-is-oplax} to it yields a lax monoidal lift of the evaluation functor $\ev_1\colon\PSh(\cC)\ra \cS$. Precomposing the latter with the symmetric monoidal Yoneda embedding, $y\colon \cC\ra\PSh(\cC)$ gives a lax  monoidal lift of the functor $\Map(1,-)\colon \cC\ra\cS$.
\end{ex}

\begin{lem}\label{lem:left-kan-monoidal}The functor \eqref{equ:left-kan-is-oplax} preserves monoidal functors, so restricts to an endofunctor on $\CMon(\Cat)$.\end{lem}

\begin{proof}We show the stronger claim that an oplax monoidal functor $F \colon \cC \to \cD$ is symmetric monoidal if and only if this holds for its image $F_!$ under \eqref{equ:left-kan-is-oplax}. To see this, note that as Day convolution preserves colimits in both variables and left Kan extension preserves colimits, the transformation $F_! \circ \otimes \to \otimes \circ (F_! \times F_!)$ given by the oplax structure is an equivalence if and only if its precomposition with the product of the Yoneda embeddings $y_\cC \times y_\cC$ of $\cC$ is. Using the naturality of the Yoneda embedding with respect to left Kan extension and that the Yoneda embedding is symmetric monoidal, tracing through the construction we see that the component of the natural transformation at $y_\cC(c) \times y_\cC(c')$ is given by the map of presheaves $y_\cD(-,F(c \otimes c')) \ra y_\cD(-,F(c) \otimes F(c'))$ induced by the oplax monoidality of $F$, and that the map $y_\cD(1_\cD)\to F_!(y_\cC({1}_\cD)) \simeq y_\cD(F({1}_\cC))$ is the value of $y_{\cD}$ at ${1}_\cD \to F({1}_\cC)$. Since the Yoneda embedding is conservative, this shows the claim.
\end{proof}

\subsection{Localising Day convolution}\label{sec:localise-day-convolution}
For a symmetric monoidal category $\cC$ and a full subcategory $\cC_0\subset \cC$, one can sometimes use the Day convolution symmetric monoidal structure on $\PSh(\cC)$ to construct a symmetric monoidal structure on $\PSh(\cC_0)$, even in situations where $\cC_0$ is not symmetric monoidal itself. To explain this, we consider the operad $\PSh(\cC)^\otimes\ra \Fin_\ast$ associated to the Day convolution structure (see \cref{sec:operads-as-symmon}), denote the fully faithful inclusion by $\iota\colon \cC_0\hookrightarrow \cC$ and write $\PSh(\cC_0)^\otimes\subset \PSh(\cC)^\otimes$ for the full subcategory on those collections of objects $(F_s)_{s\in S}$ (see \cref{sec:operads-intro}) for which each $F_s$ is in the essential image of the fully faithful right Kan extension $\iota_\ast\colon \PSh(\cC_0)\hookrightarrow\PSh(\cC)$. The restricted functor $\PSh(\cC_0)^{\otimes}\ra \Fin_\ast$ is again an operad, its category of colours is by construction equivalent to $\PSh(\cC_0)$, and the inclusion $\PSh(\cC_0)^{\otimes}\subset \PSh(\cC)^{\otimes}$ a morphism of operads \cite[p.\,193]{LurieHA}. In general, $\PSh(\cC_0)^{\otimes}\ra \Fin_\ast$ is not a cocartesian fibration, i.e.\,not a symmetric monoidal category, but there is the following condition under which it is. It involves the categories $(\cC \times \cC)_{c/}$ appearing in the colimit description \eqref{equ:day-conv} of the Day convolution.

\begin{lem}\label{lem:localise-day}If the inclusion $(\cC_0 \times \cC)_{c_0/} \to (\cC \times \cC)_{c_0/}$ is final for all objects $c_0 \in \cC_0$, then 
\begin{enumerate}
	\item\label{enum:localise-day-i} $\PSh(\cC_0)^\otimes\ra \Fin_\ast$ is a symmetric structure on $\PSh(\cC_0)$, 
	\item\label{enum:localise-day-ii} this symmetric monoidal structure is objectwise given by $F\otimes G\simeq\iota^*(\iota_*(F)\otimes\iota_*(G))$,
	\item\label{enum:localise-day-iii}$\PSh(\cC_0)$ is a presentable symmetric monoidal category in the sense of \cref{sec:presentable-monoidal},
	\item\label{enum:localise-day-iv} the inclusion $\PSh(\cC_0)^\otimes\subset \PSh(\cC)^\otimes$ is lax symmetric monoidal and refines the right Kan extension $\iota_\ast\colon \PSh(\cC_0)\hookrightarrow \PSh(\cC)$, and 
	\item\label{enum:localise-day-v} the oplax symmetric monoidal refinement of the restriction $\iota^\ast\colon \PSh(\cC)\ra \PSh(\cC_0)$ resulting from \ref{enum:localise-day-iv} and the equivalence \eqref{equ:mates} is symmetric monoidal.
\end{enumerate}
\end{lem}

\begin{proof}
During the proof we will repeatedly use that since $\iota$ is fully faithful the counit $\iota^*\iota_*\ra\id_{\PSh(\cC_0)}$ is an equivalence, the unit $\id_{\PSh(\cC)}\ra \iota_*\iota^*$ is a reflexive localisation in the sense of \cite[p.\,362]{LurieHTT}, and the right Kan extension $\iota_*\colon \PSh(\cC_0)\ra \PSh(\cC)$ is fully faithful.

\medskip

\noindent By construction the fibre of $\PSh(\cC_0)^\otimes\ra \Fin_\ast$ over $\langle 1\rangle$ is the essential image of $\iota_*$. The latter is equivalent to $\PSh(\cC_0)$, so to show \ref{enum:localise-day-i} it suffices to prove that $\PSh(\cC_0)^\otimes\ra \Fin_\ast$ is an operad and a cocartesian fibration. This follows from \cite[2.2.1.7, 2.2.1.9]{LurieHA} once we show that the finality condition ensures that if a map $\phi\colon F\ra G$ in $\PSh(\cC)$ is an equivalence after applying $\iota_*\iota^*$, then $\phi\otimes \id_H\colon F\otimes H\ra G\otimes H$ is for any $H\in \PSh(\cC)$ an equivalence after applying $\iota_*\iota^*$ as well. Since $\iota_*$ is fully faithful and equivalences on presheaves can be tested objectwise, this follows from showing that if $\phi\colon F\ra G$ has the property that $\phi(c_0)$ is an equivalence for all $c_0\in\cC_0$, then $(\phi\otimes\id_H)(c_0)$ is also an equivalence for all $c_0\in\cC_0$. By \eqref{equ:day-conv} the latter is equivalent to $\colim_{c_0 \to c' \otimes c''}(F(c') \times H(c'')) \ra \colim_{c_0 \to c' \otimes c''}(G(c) \times H(c''))$ being an equivalence for all $c_0\in\cC_0$. Here the colimits are taken over $(\cC \times \cC)_{c_0/}$. Since $\phi$ is an equivalence when restricted to $\cC_0$ by assumption, 
this follows from the assumption that $(\cC_0 \times \cC)_{c_0/} \to (\cC \times \cC)_{c_0/}$ is final.

\medskip

\noindent To show \ref{enum:localise-day-ii}, recall that the monoidal product $d\otimes d'$ in a symmetric monoidal category $\cD$ can be recovered from the corresponding cocartesian fibration $\cD^{\otimes}\ra \Fin_*$ as the target of the unique cocartesian lift with source $(d,d')$ of the unique active map $\{1,2,\ast\}\ra\{1,\ast\}$. Now observe that in $\PSh(\cC_0)^{\otimes}$ such a cocartesian lift with source $(\iota_*(F),\iota_*(G))$ is the composition of the cocartesian lift $(\iota_*(F),\iota_*(G))\ra \iota_*(F)\otimes \iota_*(G)$ in $\PSh(\cC)^{\otimes}$ followed by the unit $\iota_*(F)\otimes \iota_*(G)\ra \iota_*\iota^*(\iota_*(F)\otimes \iota_*(G))$. 

\medskip

\noindent Categories of presheaves are presentable \cite[5.5.1.1]{LurieHTT}, so by symmetry it suffices to show that $(-)\otimes G$ preserves colimits for any $G\in\PSh(\cC_0)$. Using \eqref{equ:day-conv}, the first part of \ref{enum:localise-day-ii} and the hypothesis, we first compute for $F\simeq \colim_i F_i$ in $\PSh(\cC_0)$ and $c_0\in \cC_0$
\[(F\otimes G)(c_0)\simeq \colim_{(c,c')\in (\cC\times \cC)_{c_0/}}(\iota_*(F)(c)\times \iota_*(G)(c'))\simeq \colim_{(c,c')\in (\cC_0\times \cC)_{c_0/}}(F(c)\times \iota_*(G)(c')) \]
which, since colimits commute with colimits, is equivalent to $\colim_i(\colim_{(c,c')\in (\cC_0\times \cC)_{c_0/}}(F_i(c)\times \iota_*(G)(c')))$. Applying the same reasoning backwards, the latter is equivalent to $\colim_i(F_i\otimes G)(c_0)$, and since colimits of presheaves are computed objectwise, the claim follows. 

\medskip

\noindent We observed above that $\PSh(\cC_0)^\otimes\subset \PSh(\cC)^\otimes$ is a morphism of operads, so it is (by definition) lax monoidal. On fibres over $\langle 1\rangle$ this functor is the full subcategory inclusion of the essential image of $\iota_\ast$ into $\PSh(\cC)$. As $\iota_*$ is fully faithful this inclusion is equivalent to $\iota_\ast$ itself, so \ref{enum:localise-day-iv} follows.

\medskip

\noindent The final point \ref{enum:localise-day-v} claims that for $F,G\in\PSh(\cC)$ the instance 
\[
	\iota^*\big(F\otimes G\big)\ra \iota^*\big(\iota_*(\iota^*(F))\otimes \iota_*(\iota^*(G))\big)\ra\iota^*\big(\iota_*(\iota^*(F)\otimes \iota^*(G))\big) \ra \iota^*(F)\otimes \iota^*(G)
\]
of the composition \eqref{equ:oplax-from-lax} is an equivalence. In fact all three maps are equivalences: for the first, this follows from the property we showed in the first part of the proof by factoring $F\otimes G\ra \iota_*(\iota^*(F))\otimes \iota_*(\iota^*(G))$ as $F\otimes G\ra \iota_*(\iota^*(F))\otimes G$ followed by $\iota_*(\iota^*(F))\otimes G\ra \iota_*(\iota^*(F))\otimes \iota_*(\iota^*(G))$. For the second map, note that by \ref{enum:localise-day-ii} we have $X\otimes Y\simeq \iota^*(\iota_*(X)\otimes\iota_*(Y))$ for $X,Y\in\PSh(\cC_0)$ and via this equivalence the map $\iota_*(X)\otimes \iota_*(Y)\ra \iota_*(X\otimes Y)$ coming from the lax monoidal structure is given by the unit $\id \ra \iota_*\iota^*$ applied to $\iota_*(X)\otimes \iota_*(Y)$. Specialising this to $X=\iota^*(F)$ and $Y=\iota^*(G)$, it follows that the map in question is given by applying $\iota^*$ to the unit $\id \ra \iota_*\iota^*$ at some object (namely $\iota_*(X)\otimes \iota_*(Y)$), which is indeed an equivalence since $\iota$ is fully faithful. Finally, the third map is an equivalence since the counit $\iota^*\iota_*\ra \id$ is an equivalence.
\end{proof}

\begin{rem}\cref{lem:localise-day} \ref{enum:localise-day-v} can also be deduced from \cite[2.2.1.9 (3)]{LurieHA}, by showing that the oplax structure on $\iota^*$ induced by the symmetric monoidal structure  from 2.2.1.9 loc.cit.\,agrees with the oplax structure on $\iota^*$ we used above (resulting from \eqref{equ:mates}). This can be seen by showing that both of them are obtained from the given lax monoidal structure on $\iota_*$ by taking parametrised left adjoints over $\Fin_\ast$ in the sense of \cite[3.1.1]{HHLN}.
\end{rem}

\subsection{Bimodules in categories of presheaves}\label{sec:modules-day} Given a symmetric monoidal category $\cC$  and associative algebras $A,B\in \Ass(\PSh(\cC))$ in the category of presheaves with respect to the Day convolution monoidal structure, there is a convenient description of the category $\BMod_{A,B}(\PSh(\cC))$ of $(A,B)$-bimodules in $\PSh(\cC)$ from \cref{sec:bimod}. To explain this, we consider the composition $(F_{AB}\circ y_\cC)\colon \cC\ra  \BMod_{A,B}(\PSh(\cC))$ of the Yoneda embedding $y_\cC$ of $\cC$ followed by the left adjoint $F_{AB}$ to the forgetful functor $U_{AB}\colon \BMod_{A,B}(\PSh(\cC))\ra \PSh(\cC)$ from \cref{sec:free-bimodules}. We denote the full subcategory given by the essential image of $(F_{AB}\circ y_\cC)$ (\textbf{f}ree bimodules on \textbf{rep}resentables) by 
\begin{equation}\label{equ:free-rep-bimodules}
	\smash{\BMod^{\frep}_{A,B}(\PSh(\cC))\subset\BMod_{A,B}(\PSh(\cC))}.
\end{equation} 
Since $\BMod_{A,B}(\PSh(\cC))$ has all colimits as a result of \cite[4.3.3.9]{LurieHA}, this full subcategory inclusion extends by 5.1.5.5 loc.cit.\,uniquely to a colimit preserving functor
\begin{equation}\label{equ:free-rep-bimod-colim-fun}
	\smash{\PSh(\BMod^{\frep}_{A,B}(\PSh(\cC)))\lra \BMod_{A,B}(\PSh(\cC))}
\end{equation}
which admits a right adjoint, given by the restricted Yoneda embedding, i.e.\,the Yoneda embedding of $\BMod_{A,B}(\PSh(\cC))$ followed by restriction along \eqref{equ:free-rep-bimodules}. 

\begin{lem}\label{lem:modules-as-presheaves}The functor \eqref{equ:free-rep-bimod-colim-fun} is an equivalence with the restricted Yoneda embedding as inverse.
\end{lem}

\begin{proof}The second part follows from the first since the restricted Yoneda embedding is right adjoint to the functor \eqref{equ:free-rep-bimodules}. To show that the latter is fully faithful, by \cite[5.1.6.10]{LurieHTT} it suffices to show that $F_{AB}(y_\cC(c))\in \BMod_{A,B}(\PSh(\cC))$ is completely compact for all $c\in\cC$ i.e.\,that the functor $\smash{\Map_{\BMod_{A,B}(\PSh(\cC))}(F_{AB}(y_\cC(c)),-)\colon \BMod_{A,B}(\PSh(\cC))\ra \cS}$ preserves colimits. By adjointness and the Yoneda embedding, this functor is the composition of the forgetful functor $U_{AB}\colon \BMod_{A,B}(\PSh(\cC))\ra \PSh(\cC)$ with the evaluation functor $\ev_c\colon \PSh(\cC)\ra \cC$ at $c$. Both of these preserve colimits: the former as result of \cite[4.3.3.9]{LurieHA} and the latter since colimits in presheaf categories are computed objectwise. It follows that \eqref{equ:free-rep-bimod-colim-fun} is fully faithful, so in order to show that it is an equivalence it suffices to prove that its right adjoint---the restricted Yoneda embedding---is conservative. This follows from the fact that the functors $\smash{\Map_{\BMod_{A,B}(\PSh(\cC))}(F_{AB}(y_\cC(c)),-)}$ are for $c\in\cC$ jointly conservative as a result of adjointness and the Yoneda lemma.\end{proof}

\begin{rem}The definition of $\smash{\BMod^{\frep}_{A,B}(\PSh(\cC))}$ and the functor \eqref{equ:free-rep-bimodules} makes sense for \emph{any} symmetric monoidal structure on a category $\PSh(\cC)$ of presheaves; it need not be induced from a symmetric monoidal structure on $\cC$ by Day convolution. Moreover, on a closer look, also the proof of \cref{lem:modules-as-presheaves} goes through for \emph{any} symmetric monoidal structure $\PSh(\cC)$, so it in particular also applies to $\PSh(\cC_0)$  fo full subcategories $\cC_0\subset \cC$ as in \cref{lem:localise-day}.\end{rem}

\subsection{Decomposing categories of presheaves} \label{sec:recollement} We now turn to a decomposition result for the category of space-valued presheaves on a category as a pullback in terms of categories of presheaves on certain full subcategories and their ``complements''. The subcategories we consider are:

\begin{dfn}\label{dfn:decom-pair}A category $\cC$ together with a full subcategory $\cC_0\subset\cC$ closed under equivalences is a \emph{decomposition pair} if the following two conditions are satisfied:
\begin{enumerate}[leftmargin=*]
	\item if the identity $\id_c$ of an object $c\in \cC$ factors through an object in $\cC_0$, then $c$ is contained in $\cC_0$,
	\item if a composition $g \circ f$ of two composable morphisms $f,g$ in $\cC$ factors through an object in $\cC_0$, then at least one of the morphisms $f$ or $g$ already factors through an object in $\cC_0$.
\end{enumerate}
\end{dfn}

\medskip

\noindent The conditions in this definition are designed exactly so that one can define a (potentially non-full) subcategory $\cC \backslash \cC_0$ of $\cC$ called the \emph{complement}, by restricting to objects that are not in $\cC_0$ and to morphisms that do not factor through objects in $\cC_0$. Writing $\iota\colon \cC_0\hookrightarrow \cC$ and $\nu \colon \cC\backslash\cC_0\hookrightarrow \cC$ for the inclusions we consider the functors four functors
\[
	\begin{tikzcd} 
	\PSh(\cC_0) \rar[shift left=1.2ex]{\iota_!} \rar[shift left=-1.2ex,swap]{\iota_*}& \PSh(\cC) \rar{\nu^*} \lar & \PSh(\cC \backslash \cC_0)\end{tikzcd}
\]
where the unlabelled left middle arrow is the restriction $\iota^*$ whose left and right adjoints $\iota_!$ and $\iota_*$ are given by left and right Kan extension. These functors feature in a pair of natural transformations 
\[
	\Lambda\coloneqq \big(\nu^*\iota_!\iota^*\ra \nu^*\ra \nu^*\iota_*\iota^*\big) ,
\]
of functors $\PSh(\cC) \to \PSh(\cC \backslash \cC_0)$, resulting from applying $\nu^*$ to the counit of $\iota_! \dashv \iota^*$ and the unit of $\iota^* \dashv \iota_*$. Similarly, there is a natural transformation of functors $\PSh(\cC_0) \to \PSh(\cC \backslash \cC_0)$
\[
	\Omega\coloneqq \big(\nu^*\iota_!\ra \nu^*\iota_*),
\]
obtained by applying $\nu^*$ to the zig-zag $\smash{\iota_!\xleftarrow{\simeq} \iota_!\iota^*\iota_*\ra \iota_*}$ induced by the counits of $\iota^* \dashv \iota_*$ and $\iota_! \dashv \iota^*$; the former is an equivalence since $\iota$ is fully faithful. Viewing $\Lambda$ and $\Omega$ as functors $\PSh(\cC) \ra \PSh(\cC\backslash \cC_0)^{[2]}$ and $\PSh(\cC_0) \ra \PSh(\cC\backslash \cC_0)^{[1]}$, the announced decomposition result reads as follows:

\begin{thm}\label{thm:recollement} For a decomposition pair $\cC_0 \subset \cC$, there is a commutative square in $\Cat$
\begin{equation}\label{eqn:recollement-square}
	\begin{tikzcd} \PSh(\cC) \dar[swap]{\iota^*}\rar{\Lambda} &\PSh(\cC\backslash \cC_0)^{[2]} \dar{(0\le 2)^*} \\
	\PSh(\cC_0) \rar{\Omega} & \PSh(\cC \backslash \cC_0)^{[1]}.
	\end{tikzcd}
\end{equation}
If for all $c,d \in \cC \backslash \cC_0$ the map
\[
	(\iota_!\circ \iota^*)\big(\Map_{\cC}(-,d)\big)(c)\sqcup \Map_{\cC\backslash\cC_0}(c,d)\lra \Map_{\cC}(c,d)
\]
induced by $\nu$ and the counit of $\iota_! \dashv \iota^*$, is an equivalence, then the square \eqref{eqn:recollement-square} is cartesian.
\end{thm}

\begin{rem}\,
\begin{enumerate}
	\item The proof of \cref{thm:recollement} presented below was explained to us by Rozenblyum; a similar proof and statement will appear in forthcoming joint work of him with Ayala and Mazel-Gee. Before learning about their argument, the authors knew of a different proof---based on resolving presheaves by representables---of the weaker statement that the functor from $\PSh(\cC)$ to the pullback of the remaining entries in \cref{thm:recollement} is fully faithful, but this proof did not show essential surjectivity. We are grateful to Ayala, Mazel-Gee, and Rozenblyum for allowing us to include their argument. Since the first version of our work appeared, a different proof of this result was also independently found by Haine, Ramzi, and Steinebrunner \cite[Theorem 5.12]{HaineRamziSteinebrunner}. Their proof is less direct, but also applies to categories of presheaves with values in presentable categories other than spaces.	
	\item The left Kan extension $(\iota_!\iota^*)(\Map_{\cC}(-,d))(c)$ in the statement can be equivalently described as the coend $\Map_{\cC}(c,-)\otimes_{\cC_0}\Map_{\cC}(-,d)$, so the condition in the theorem becomes that the canonical map $\Map_{\cC}(c,-)\otimes_{\cC_0}\Map_{\cC}(-,d)\ra \Map_{\cC}(c,d)$ is an equivalence onto the collection of path-components of morphisms that factor through an object in $\cC_{0}$.
\end{enumerate}
\end{rem}

\begin{proof}[Proof of \cref{thm:recollement} (after Ayala--Mazel-Gee--Rozenblyum)]
For the first part we need to provide a natural equivalence between the clock- and counterclockwise composition in the square. This arises by applying $\nu^*$ to a natural equivalence between the composition 
\[
	\smash{\iota_!\iota^* \overset{\epsilon}\lra \id \overset{\overline{\eta}}\lra \iota_*\iota^*}
\]
of the counit $\epsilon$ of $\iota_! \dashv \iota^*$ and the unit $\overline{\eta}$ of $\iota^* \dashv \iota_*$, and the zig-zag
\[
	\smash{\iota_!\iota^* \underset{\simeq}{\xleftarrow{\iota_! \overline{\epsilon}_{\iota^*}}} \iota_! \iota^* \iota_* \iota^* \xrightarrow{\epsilon_{\iota_* \iota^*}} \iota_* \iota^*},
\]
where $\overline{\epsilon}$ is the counit of $\iota^*\dashv \iota_*$. By the triangle identities, an inverse of $\smash{\iota_! \overline{\epsilon}}_{\iota^*}$ is given by $\iota_! \iota^*\overline{\eta}$, so it suffices to identify $\epsilon_{\iota_* \iota^*}\circ\iota_!\iota^*(\overline{\eta})$ with $\overline{\eta}\circ \epsilon$. But the structure of $\epsilon$ as a natural transformation provides this, since it gives natural identifications $\epsilon_d\circ \iota_!\iota^*(f)\simeq f\circ \epsilon_c$ for morphisms $f\colon c\ra d$, so in particular for the components of $\overline{\eta}$.

\medskip

\noindent Writing $\cP$ for the pullback of the square without the top-left corner, the task for the second part is to prove that the functor $\varphi\colon \PSh(\cC) \to \cP$ induced by the commutativity is an equivalence. To increase legibility in the following proof, we pretend that objects which come equipped with an equivalence between them, are actually equal. We begin by observing some properties of $\cP$ and $\varphi$:
\begin{enumerate}[leftmargin=0.6cm]
	\item\label{property-p-i} The category $\cP$ has all colimits: using that $\nu^*\iota_!$ is a left adjoint and $ \nu^*\iota_*$ a right adjoint, the colimit of a diagram $(X_i,(\nu^*\iota_!(X_i) \to Z_i \to \nu^*\iota_*(X_i)))_{i\in I}$ in $\cP$ indexed by some category $I$ where $X_i \in \PSh(\cC_0)$ and $(\nu^*\iota_!(X_i) \to Z_i \to \nu^*\iota_*(X_i)) \in \PSh(\cC \backslash \cC_0)^{[2]}$ is given by $(\colim\, X_i,(\nu^*\iota_!(\colim\, X_i)\simeq \colim\,(\nu^*\iota_!(X_i))\ra \colim\, Z_i\ra \nu^*\iota_*(\colim\, X_i)))$ where the final map is the composition of $\colim\, Z_i \to \colim\, \nu^*\iota_*(X_i)$ with the adjoint of $\iota^* \nu_!(\colim\, \nu^*\iota_*(X_i)) \simeq \colim(\iota^* \nu_!\nu^*\iota_*(X_i)) \to \colim\, X_i$ with respect to the adjunction $\iota^* \nu_!\dashv\nu^* \iota_*$.
	\item\label{property-p-ii} The description  of colimits in $\cP$ from \ref{property-p-i} shows that $\varphi\colon \PSh(\cC) \to \cP$ preserves colimits.
	\item \label{property-p-iii} There are functors \[\pi_1 \colon \cP \lra \PSh(\cC_0)\quad\text{and}\quad\pi_2 \colon \cP \ra \PSh(\cC\backslash \cC_0)\]
	given by $\pi_1(X,(\nu^*\iota_!(X) \to Z \to \nu^*\iota_*(X)))=X$ and $\pi_2(X,(\nu^*\iota_!(X) \to Z \to \nu^*\iota_*(X)))=Z$. They are jointly conservative and preserve colimits as a result of \ref{property-p-i}.
	\item\label{property-p-iv} The functors $\pi_1$ and $\pi_2$ from \ref{property-p-iii} admit left adjoints \[\lambda_1 \colon \PSh(\cC_0) \lra \cP\quad\text{and}\quad\lambda_2 \colon \PSh(\cC\backslash \cC_0) \lra \cP\] 	given by $\lambda_1(X)= (X,(\nu^*\iota_!(X) = \nu^*\iota_!(X) \to \nu^*\iota_*(X)))$ and $\lambda_2(Z)= (\iota^* \nu_!(Z),(\nu^*\iota_!\iota^*  \nu_!(Z)\ra \nu^*\iota_!\iota^* \nu_!(Z) \sqcup Z \to \nu^*\iota_*\iota^* \nu_!(Z)))$. They feature in natural transformations
	\begin{equation}\label{equ:lambda-nats}
		(\lambda_1\circ y_{\cC_0})\lra (\varphi \circ y_{\cC}\circ\iota) \quad\text{and}\quad(\lambda_2\circ y_{\cC\backslash\cC_0})\lra (\varphi\circ y_{\cC}\circ\nu)
	\end{equation}
	that are, via the adjunctions $\lambda_i\dashv \pi_i$, adjoint to the natural transformations $y_{\cC_0}\ra (\iota^*\circ y_{\cC}\circ\iota)$ and $y_{\cC\backslash\cC_0}\ra (\nu^*\circ y_{\cC}\circ\nu)$ that are induced by the inclusions of $\cC_0$ and $\cC\backslash\cC_0$ in $\cC$.
	\item\label{property-p-v} The natural transformations \eqref{equ:lambda-nats} are equivalences. Since the $\pi_i$ are jointly conservative, we may test this after applying the $\pi_i$ to both transformations. Applying them to the first transformation yields $y_{\cC_0}\ra (\iota^*\circ y_\cC\circ \iota)$ and $(\nu^*\circ \iota_!\circ y_{\cC_0})\ra (\nu^*\circ y_{\cC}\circ\iota)$ which are equivalences; the former since $\cC_0$ is a \emph{full} subcategory and the latter since the Yoneda embedding is compatible with left Kan extension. Applying the $\pi_i$ to the second transformation yields $(\iota^*\circ \nu_!\circ y_{\cC\backslash \cC_0})\ra (\iota^*\circ y_{\cC}\circ \nu)$ and $(\nu^*\circ \iota_! \circ \iota^*\circ y_{\cC}\circ \nu)\sqcup y_{\cC\backslash \cC_0}\ra (\nu^*\circ y_{\cC}\circ\nu)$ which are also equivalences, the former for the same reason as before and the latter by hypothesis. 
	\end{enumerate}	
To prove the claim, we write $\cR \subset \cP$ for the full subcategory spanned by the essential images of $(\lambda_1\circ y_{\cC_0})$ and $(\lambda_2\circ y_{\cC\backslash \cC_0})$. Since $\iota$ and $\nu$ are jointly essentially surjective, it follows from \ref{property-p-v} that the composition $(\varphi\circ y_{\cC})$ lands in $\cR \subset \cP$, so $(\varphi\circ y_{\cC})$ factors as a functor $\smash{\overline{\varphi\circ y_\cC}\colon \cC\ra\cR}$ followed by the inclusion $\cR\subset\cP$. Since $\cP$ has colimits by \ref{property-p-i} and $\varphi$ preserves those by \ref{property-p-ii}, this implies by taking colimit preserving extensions, that $\varphi$ agrees with the composition of the left Kan extension $\smash{(\overline{\varphi\circ y_\cC})_!\colon \PSh(\cC)\ra\PSh(\cR)}$ with the colimit-preserving extension $\ext\colon \PSh(\cR)\ra \cP$ of the inclusion $\cR\subset \cP$, so the claim would follow if we show that both $\smash{\overline{\varphi\circ y_\cC}}$ and $\ext$ are equivalences. 

\medskip

\noindent To show that $\ext$ is an equivalence, recall from \cite[5.1.6.2]{LurieHTT} that an object $e$ in a category $\cE$ is \emph{completely compact} if $\Map_\cE(c,-)\colon\cE\ra\cS$ preserves small colimits. Since, firstly, representables in a presheaf category are completely compact (by the Yoneda lemma and the fact that colimits of presheaves are computed pointwise) and, secondly, completely compact objects are preserved by left adjoints whose right adjoints preserve colimits, we conclude from \ref{property-p-iii} and \ref{property-p-iv} that all objects in $\cR\subset\cP$ are completely compact in $\cP$. This implies by 5.1.6.10 loc.cit. that $\ext\colon \PSh(\cR)\ra\cP$ of $\cR\subset \cP$ is fully faithful. Moreover, this functor has a right adjoint $\cP\ra \PSh(\cR)$ given the Yoneda embedding $y_{\cP}$ followed by restriction along $\cR\subset\cP$. Since $\pi_1$ and $\pi_2$ are jointly conservative by \ref{property-p-iii}, it follows that this right adjoint is conservative as well, so $\ext$ is an equivalence because any fully faithful functor with a conservative right adjoint is an equivalence.

\medskip

\noindent It remains to show that  $\smash{\overline{\varphi\circ y_\cC}\colon \cC\ra\cR}$ is an equivalence. It follows from \ref{property-p-v} that this functor is essentially surjective, so we are left to show that for $c,c'\in\cC$ the composition
\begin{equation}\label{equ:composition-recol-ff}
	\smash{\Map_{\cC}\big(c,c'\big)\xlra{y_\cC} \Map_{\PSh(\cC)}\big(y_\cC(c),y_\cC(c')\big)\xlra{\varphi}\Map_{\cP}\big(\varphi(y_\cC(c)),\varphi (y_\cC(c'))\big)}
\end{equation}
is an equivalence. We distinguish two cases: whether $c\in \cC_0$ or $c\in \cC\backslash\cC_0$.

\medskip

\noindent If $c\in \cC_0$, then we may test whether \eqref{equ:composition-recol-ff} is an equivalence by postcomposition with the equivalence $\Map_{\cP}(\varphi(y_\cC(c)),\varphi (y_\cC(c')))\ra \Map_{\cP}((\lambda_1\circ y_{\cC_0})(c)),(\varphi\circ y_\cC)(c'))$ induced with precomposition with the first map in \eqref{equ:lambda-nats} (which is an equivalence by \ref{property-p-v}). Using the adjunction $\lambda_1\dashv \pi_1$ the resulting map is equivalent to the natural map $\Map_{\cC}(c,c')\ra \Map_{\PSh(\cC_0)}(y_{\cC_0}(c),\iota^*(y_{\cC}(c')))$ which is an equivalence by the Yoneda lemma. The case $c\in \cC\backslash \cC_0$ works the same way, using the second  map in \eqref{equ:lambda-nats} and the adjunction $\lambda_2\dashv \pi_2$.
\end{proof}

\begin{ex}Fix a category $\cC$ and an object $\bar{c}\in \cC$ with the property that any object which admits both a morphism from $\bar{c}$ and one to $\bar{c}$, is equivalent to $\bar{c}$. Then the full subcategory $\cC_0\subset \cC$ of objects not equivalent to $\bar{c}$ is a decomposition pair to which \cref{thm:recollement} applies: for $c,c' \in \cC \backslash \cC_0$ we have $(\iota_!\circ \iota^*)(\Map_\cC(-,c'))(c)=\varnothing$ and $\Map_{\cC \backslash \cC_0}(c,c') \to \Map_\cC(c,c')$ is an equivalence. Under the additional assumption that every endomorphism in $\cC$ is invertible, the pullback \eqref{eqn:recollement-square} was in this case on the level of mapping spaces established by Göppl--Weiss in \cite[Theorem 4.1.2]{Goppl}.
\end{ex}

\subsubsection{Naturality of the decomposition}\label{sec:abstract-layer-naturality} In the remainder of this section, we will establish a result on the naturality of the square in \cref{thm:recollement} in the pair $\cC_0 \subset \cC$. The statement involves subcategories
\[
	\Cat^{\subset_\BC,\dec}\hookrightarrow\Cat^{\subset_\BC}\hookrightarrow \Cat^{[1]}
\]
defined as follows: $\Cat^{\subset_\BC}$ is the subcategory of $\Cat^{[1]}$ whose objects are full subcategory inclusions $\cC_0\subset \cC$ and whose morphisms are those commutative squares
\[\begin{tikzcd}[column sep=0.6cm, row sep=0.6cm] 
	\cD_0 \arrow[r,"j",hookrightarrow] \dar[swap]{\varphi_0} & \cD \dar{\varphi} \\
	\cC_0 \arrow[r,"\iota",hookrightarrow]  & \cC 
\end{tikzcd}\]
for which the Beck--Chevalley transformations
\begin{equation}\label{equ:BC-naturality}
	\varphi_{0!} j^* \lra \iota^* \varphi_!\quad\text{and}\quad \varphi_{!} j_* \lra \iota_* \varphi_{0!}
\end{equation}
are equivalences. Here the first transformation is the composition $\varphi_{0!} j^*\ra \iota^*\iota_!\varphi_{0!} j^*\simeq\iota^*\varphi_{!}j_! j^*\ra \iota^* \varphi_!$ induced by the unit of $\iota_!\dashv \iota^*$, the naturality of left Kan extensions as well as the counit of $j_!\dashv j^*$, and the second is the composition $\varphi_{!} j_* \ra \iota_* \iota^* \varphi_{!} j_*\simeq\iota_* \varphi_{0!} j^* j_* \ra \iota_* \varphi_{0!}$ induced by the unit of $\iota^*\dashv \iota_*$, where the equivalence is given by the first transformation as well as the counit of $j^*\dashv j_*$. 

\medskip

\noindent The second subcategory $\smash{\Cat^{\subset_\BC,\dec}}$ of $\Cat^{\subset_\BC}$ has decomposition pairs $\cC_0\subset\cC$ as objects and as morphisms squares as above with the additional property that $\varphi$ restricts to a functor $\varphi_{\neq0}\colon \cD\backslash \cD_0\ra \cC\backslash \cC_0$ such that the analogue of the first transformations in \eqref{equ:BC-naturality} with $j$ and $\iota$ replaced by the inclusions $\cD\backslash \cD_0\hookrightarrow\cD$ and $\cD\backslash \cD_0\hookrightarrow\cD$ is an equivalence. To state the naturally result for the square in \cref{thm:recollement}, note that the first part of the proof of this theorem produces for a full subcategory inclusion $\iota\colon\cC_0\hookrightarrow\cC$ a commutative square in $\Cat$
\begin{equation}\label{eqn:recollement-square-nat}
	\begin{tikzcd}[column sep=1.8cm] \PSh(\cC) \rar{(\iota_!\iota^* \to \id \to \iota_*\iota^*)} \dar[swap]{\iota^*} & \PSh(\cC)^{[2]} \dar{(0\le 2)^*} \\
	\PSh(\cC_0) \rar{(\iota_! \to \iota_*)} & \PSh(\cC)^{[1]}.
	\end{tikzcd}
\end{equation}

\begin{thm}\label{thm:recollement-nat} Assigning to a full subcategory inclusion $\cC_0 \subset \cC$ the square \eqref{eqn:recollement-square-nat}, and to a decomposition pair $\cC_0\subset\cC$ the square \eqref{eqn:recollement-square} extends to functors
\[
	 \Cat^{\subset_\BC} \lra \Cat^{[1] \times [1]}\quad\text{and}\quad \Cat^{\subset_\BC,\dec} \lra \Cat^{[1] \times [1]}
\]
which are on morphisms both induced by taking left Kan extensions.
\end{thm}

\noindent The proof of \cref{thm:recollement-nat} involves $2$-categorical concepts, especially (op)lax natural transformations. We refer to \cref{sec:mate-calculus} for a recollection of the relevant facts and the notation we use. For brevity, we often leave out the word ``natural'' in ``natural transformation'' in what follows. We write $[1]'\coloneqq [1]$ to notationally distinguish two copies of $[1]$. For a $2$-category $\mathbb{C}$, we call a $2$-functor $[1]\boxtimes [1]'\ra \mathbb{C}$ a \emph{lax square}. Diagrammatically, a lax square looks like the leftmost diagram in
\begin{equation}\label{equ:default-lax-square}
	\begin{tikzcd}
	a\rar{\alpha}\arrow[d,"\omega",swap]&b\dar{\beta}\arrow[shorten <=10pt,shorten >=10pt,Rightarrow]{dl}{\theta}\\
	c\arrow[r,swap,"\gamma"]&d
	\end{tikzcd}\quad\quad\quad
	\begin{tikzcd}
	a\arrow[r,"\beta\alpha",""{name=T,inner sep=1pt,below}]\arrow[d,"\id",swap]&d\dar{\id}\\
	a\arrow[r,swap,"\gamma\omega",""{name=B,inner sep=1pt,below}]&d
	\arrow[shorten <=8pt, shorten >=8pt, Rightarrow, from=T, to=B, "\theta"]
	\end{tikzcd}\quad \quad\quad
	\begin{tikzcd}
	a\arrow[r,"\alpha",""{name=T,inner sep=1pt,below}]\arrow[d,"\id",swap]&b\dar{\id}\\
	a\arrow[r,swap,"\alpha",""{name=B,inner sep=1pt,below}]&b
	\arrow[shorten <=8pt, shorten >=8pt, Rightarrow, from=T, to=B, "\id"],
	\end{tikzcd}
\end{equation}
We always indicate the $[1]$-direction horizontally and the $[1]'$-direction vertically. There are two constructions in a general  $2$-category $\mathbb{C}$ that we will use repeatedly:

\begin{enumerate}[label=(\Alph*)]
	\item \label{const:deform-square} Given a lax square $[1]\boxtimes [1]'\ra \mathbb{C}$, we can form a new lax square in $\mathbb{C}$ by precomposition with the unique $2$-functor $s\colon [1]\boxtimes [1]'\ra [1]\boxtimes [1]'$ that fixes the nontrivial $2$-morphism and is on objects given by $s(0,0)=(0,0)$, $s(1,0)=(1,1)$, $s(0,1)=(0,0)$, $s(1,1)=(1,1)$. This corresponds to turning the leftmost square in \eqref{equ:default-lax-square} into the middle one.
	\item \label{const:morphism-square} Given a morphism $\alpha\colon a\ra b$ viewed as a functor $[1]\ra\mathbb{C}$ , we can form a lax square by precomposition with the unique $2$-functor $e\colon [1]\boxtimes [1]'\ra [1]$ with $e(0,0)=0$, $s(1,0)=1$, $s(0,1)=0$, $s(1,1)=1$.  This corresponds to turning $\alpha$ into the rightmost square in \eqref{equ:default-lax-square}.
\end{enumerate}

\noindent We begin the proof of \cref{thm:recollement-nat} with four preparatory lemmas involving the functors 
\begin{equation}\label{equ:preshaef-functors-nat}
	\PSh(-_0)\colon \Cat^{\subset_\BC}\lra \Cat\quad \text{and} \quad \PSh(-)\colon  \Cat^{\subset_\BC}\lra \Cat
\end{equation} 
that send $\smash{(\cC_0{\subset}\cC)\in \Cat^{\subset_\BC}}$ to $\PSh(\cC_0)$ and to $\PSh(\cC)$ respectively, and are on morphisms induced by left Kan extensions. More precisely, they are the composition of the inclusion $\Cat^{\subset_\BC}\subset\Cat^{[1]}$ with the evaluation at $i\in[1]$ for $i=0,1$ respectively, followed by the functor $\PSh(-)\colon \Cat\ra\Cat$.

\begin{lem}\label{lem:kan-extension-upgrade}There are functors of the form $[1]\times \Cat^{\subset_\BC}\ra  \Cat$
\[
	\big(\PSh(-)\xrightarrow{\id}\PSh(-)\big),\   \big(\PSh(-_0)\xrightarrow{(-)_!}\PSh(-)\big),\   \big(\PSh(-)\xrightarrow{(-)^*}\PSh(-_0)\big),\   \big(\PSh(-_0)\xrightarrow{(-)_*}\PSh(-)\big)
\]
whose value at the morphisms $\smash{((0<1),\id_{\iota})}$ are given by the indicated functors, e.g.\,for $(-)_!$ by the left Kan extension $\iota_!\colon \PSh(\cC_0)\ra \PSh(\cC)$. 
\end{lem}

\begin{proof}The first functor is the composition of the projection $[1]\times \Cat^{\subset_\BC}\ra \Cat^{\subset_\BC}$ with the second functor in \eqref{equ:preshaef-functors-nat}. The second one is the adjoint of the composition of the inclusion $\Cat^{\subset_\BC}\subset \Cat^{[1]}$ followed by the endofunctor of $\Cat^{[1]}$ given by postcomposition with  $\smash{\PSh(-)\colon \Cat\ra \Cat}$. To construct the third one, we view the second functor $(-)_!$ as an oplax transformation $(-)_!\colon \Cat^{\subset_\BC}\boxtimes [1]\ra \twoCat$ whose components are $\iota_!$ so admit right adjoints $\iota^*$. We then apply \eqref{equ:mate-equ} to lift these right adjoints to a lax transformation $(-)^*\colon [1]\boxtimes\Cat^{\subset_\BC}\ra \twoCat$. Unpacking \eqref{equ:mate-bc}, the two-cells of $(-)^*$ are given by first transformation in \eqref{equ:BC-naturality}, so they are equivalences by assumption, i.e.\,$(-)^*$ gives a strict transformation, i.e.\,a functor $(-)^*\colon [1]\times\Cat^{\subset_\BC}\ra \Cat$ as required. To construct the fourth functor, we apply the same argument to $(-)^*$ in place of $(-)_!$, using that this time the occurring two-cells in the lax transformation are given by the second transformation in \eqref{equ:BC-naturality}.
\end{proof}

\begin{lem}\label{lem:units-coherent} There are $2$-functors of the form $([1]\boxtimes [1]')\times \Cat^{\subset_\BC}\ra \twoCat$
\[
	\big((-)_!(-)^*\xsra{\epsilon}\id\big),\quad  \big(\id\xsra{\eta}(-)^*(-)_!\big),\quad  \big((-)^*(-)_*\xsra{\overline{\epsilon}}\id\big),\quad \big(\id\xsra{\overline{\eta}}(-)_*(-)^*\big) 
\]
that agree, restricted to $([1]\boxtimes \{i\})\times \Cat^{\subset_\BC}\simeq [1]\times \Cat^{\subset_\BC}$ for $i=0,1$, with the indicated compositions of the functors from \cref{lem:kan-extension-upgrade}, and agree, restricted to $([1]\boxtimes [1]')\times \{\iota\}$, with the lax squares 
\[\hspace{-0.3cm}
\begin{tikzcd} \PSh(\cC) \arrow[r,"\iota_!\iota^*",""{name=T,inner sep=1pt,below}]\dar{\id} & \PSh(\cC) \dar{\id}  \\
	\PSh(\cC)  \arrow[r,"\id",""{name=B,inner sep=1pt,below},swap]& \PSh(\cC)
	\arrow[shorten <=8pt, shorten >=8pt, Rightarrow, from=T, to=B, "\epsilon"]
	\end{tikzcd}
	\begin{tikzcd} \PSh(\cC_0) \arrow[r,"\id",""{name=T,inner sep=1pt,below}]\dar{\id} & \PSh(\cC_0) \dar{\id}  \\
	\PSh(\cC_0)  \arrow[r,"\iota^*\iota_!",""{name=B,inner sep=1pt,below},swap]& \PSh(\cC_0)
	\arrow[shorten <=8pt, shorten >=8pt, Rightarrow, from=T, to=B, "\eta"]
	\end{tikzcd}
	\begin{tikzcd} \PSh(\cC_0) \arrow[r,"\iota^*\iota_*",""{name=T,inner sep=1pt,below}]\dar{\id} & \PSh(\cC_0) \dar{\id}   \\
	\PSh(\cC_0) \arrow[r,"\id",""{name=B,inner sep=1pt,below},swap]& \PSh(\cC_0)
	\arrow[shorten <=8pt, shorten >=8pt, Rightarrow, from=T, to=B, "\overline{\epsilon}"]
	\end{tikzcd}
	\begin{tikzcd} \PSh(\cC) \arrow[r,"\id",""{name=T,inner sep=1pt,below}]\dar{\id} & \PSh(\cC) \dar{\id}  \\
	\PSh(\cC)   \arrow[r,"\iota_*\iota^*",""{name=B,inner sep=1pt,below},swap]& \PSh(\cC)
	\arrow[shorten <=8pt, shorten >=8pt, Rightarrow, from=T, to=B, "\overline{\eta}"]
	\end{tikzcd}\]
featuring the (co)unit transformations of the adjunctions $\iota_!\dashv\iota^*\dashv\iota_*$.\end{lem}
 
\begin{proof}
Consider the unique functors $r_{\epsilon}\colon [1]'\times [1]\ra [1]$ with $r_{\epsilon}(i,j)=0$ if $(i,j)=(0,0)$ and $r_{\epsilon}(i,j)=1$ otherwise, and $r_{\eta}\colon [1]'\times [1]\ra [1]$ with $r_{\eta}(i,j)=1$ if $(i,j)=(1,1)$ and $r_{\eta}(i,j)=0$ otherwise. Precomposing $(-)_!\colon [1]\times \Cat^{\subset_\BC}\ra  \Cat$ from \cref{lem:kan-extension-upgrade} with these functors gives two functors  $ [1]'\times \Cat^{\subset_\BC}\times [1]\ra  \Cat$ that send $\smash{(\iota\colon \cC_0\subset\cC)\in \Cat^{\subset_\BC}}$ to the commutative squares
\begin{equation}\label{equ:tilt-left-kan}
	\begin{tikzcd}\ \PSh(\cC_0) \rar{\iota_!} \dar{\iota_!} & \PSh(\cC) \dar{\id} \\
	\PSh(\cC) \rar[swap]{\id}  & \PSh(\cC) 
	\end{tikzcd}
	\quad\text{respectively}\quad 	
	\begin{tikzcd}\ \PSh(\cC_0) \rar{\id} \dar{\id} & \PSh(\cC_0) \dar{\iota_!} \\	
	\PSh(\cC_0) \rar[swap]{\iota_!}  & \PSh(\cC).
	\end{tikzcd}
\end{equation}
Viewing these two functors as oplax transformations $([1]'\times \Cat^{\subset_\BC})\boxtimes [1]\ra  \twoCat$ and applying  \eqref{equ:mate-equ} (which corresponds in the above two squares to taking horizontal right adjoints) gives 2-functors $\widetilde{\epsilon}, \widetilde{\eta}\colon [1]\boxtimes  ([1]'\times \Cat^{\subset_\BC})\ra \twoCat$ that send $\smash{\iota\in \Cat^{\subset_\BC}}$ to the lax squares
\begin{equation}\label{equ:tilt-left-kan-mate}
	\begin{tikzcd} \PSh(\cC) \dar{\id}  \rar{\iota^*}& \PSh(\cC_0)\arrow[shorten <=10pt,shorten >=10pt,Rightarrow]{dl}{\epsilon}  \dar{\iota_!}  \\
	\PSh(\cC)  \rar[swap]{\id} & \PSh(\cC) \end{tikzcd}
	\quad\text{respectively}\quad
	\begin{tikzcd} \PSh(\cC_0) \rar{\id}\dar{\iota_!}  & \PSh(\cC_0) \arrow[shorten <=10pt,shorten >=10pt,Rightarrow]{dl}{\eta}\dar{\id} \\
	\PSh(\cC)  \rar[swap]{\iota^*} & \PSh(\cC_0).  \end{tikzcd}
\end{equation}
By naturality of \eqref{equ:mate-equ}, the restrictions of $\widetilde{\epsilon}$ to $(\{i\}\times \Cat^{\subset_\BC})\boxtimes [1]$ for $i=0,1$ agree with the 1-functors $(-)_!$ and $\id$ from \cref{lem:kan-extension-upgrade} respectively, and the same holds for $\widetilde{\eta}$ with $0$ and $1$ interchanged. In particular, these restrictions send all 2-morphisms to equivalences, so by the discussion in \cref{sec:mate-calculus} the 2-functors $\widetilde{\epsilon}$ and  $\widetilde{\eta}$ factor uniquely over the localisation $[1]\boxtimes ( [1]'\times \Cat^{\subset_\BC})\ra ([1]\boxtimes [1]')\times \Cat^{\subset_\BC}$ that inverts the $2$-morphisms in $[1]\boxtimes ( \{i\}\times \Cat^{\subset_\BC})$ for $i=0,1$. Applying the construction \ref{const:deform-square} yields the first two functors $\epsilon, \eta \colon ([1]\boxtimes [1]')\times \Cat^{\subset_\BC}\ra\twoCat$ in the claim. The exact same procedure but starting with $(-)^*$ instead of $(-)_!$ yields the second two functors.
\end{proof}

Note that we also have functors of the form $([1]\boxtimes [1]') \times \Cat^{\subset_\BC}\ra \twoCat$
\begin{equation}\label{equ:further-functors}
	\smash{\big((-)_!\xrightarrow{\id_{(-)_!}}(-)_!\big),\quad \big((-)^*\xrightarrow{\id_{(-)^*}}(-)^*\big),\quad \big((-)_*\xrightarrow{\id_{(-)_*}}(-)_*\big),}
\end{equation}
by applying \ref{const:morphism-square} to the functors in \cref{lem:kan-extension-upgrade}. Their values at $\smash{\iota\in \Cat^{\subset_\BC}}$ are the lax squares
\[
	\begin{tikzcd}
	\PSh(\cC_0) \arrow[r,"\iota_!",""{name=T,inner sep=1pt,below}]\arrow[d,"\id",swap]&\PSh(\cC)\dar{\id}\\
	\PSh(\cC_0) \arrow[r,swap,"\iota_!",""{name=B,inner sep=1pt,below}]&\PSh(\cC)
	\arrow[shorten <=8pt, shorten >=8pt, Rightarrow, from=T, to=B, "\id"],
	\end{tikzcd}\quad
	\begin{tikzcd}
	\PSh(\cC)\arrow[r,"\iota^*",""{name=T,inner sep=1pt,below}]\arrow[d,"\id",swap]&\PSh(\cC_0)\dar{\id}\\
	\PSh(\cC)\arrow[r,swap,"\iota^*",""{name=B,inner sep=1pt,below}]&\PSh(\cC_0)
	\arrow[shorten <=8pt, shorten >=8pt, Rightarrow, from=T, to=B, "\id"],
	\end{tikzcd}\quad
	\begin{tikzcd}
	\PSh(\cC_0) \arrow[r,"\iota_*",""{name=T,inner sep=1pt,below}]\arrow[d,"\id",swap]&\PSh(\cC)\dar{\id}\\
	\PSh(\cC_0) \arrow[r,swap,"\iota_*",""{name=B,inner sep=1pt,below}]&\PSh(\cC)
	\arrow[shorten <=8pt, shorten >=8pt, Rightarrow, from=T, to=B, "\id"].
	\end{tikzcd}
\]
By horizontally (i.e.\,in the $[1]$-direction) and vertically (i.e.\,in the $[1]'$-direction) composing and pasting these 2-functors and the 2-functors from \cref{lem:units-coherent}, we can build many new 2-functors. 

\begin{ex}\label{ex:pasting-examples}We spell out a few relevant examples of such compositions and pastings. Horizontal compositions and pastings are indicated notationally with a $\cdot$ or $\circ$ symbol, respectively, whereas vertical compositions and pastings are indicated by no symbol or an arrow, respectively.
\begin{enumerate}[leftmargin=0.7cm]
\item\label{equ:triangle-id-functor} Horizontally composing $\eta$ with $\id_{(-)_!}$, and $\id_{(-)_!}$ with $\epsilon$, and then vertically composing gives
\[
	\big((-)_!\xrightarrow{(\epsilon\cdot(-)_!)((-)_!\cdot\eta)}(-)_!\big)\colon ([1]\boxtimes [1]')\times \Cat^{\subset_\BC} \ra \twoCat.
\]
\item\label{equ:bottom-row-2cats} Horizontally composing $\overline{\epsilon}$ with $\id_{(-)_!}$, and $\id_{(-)_*}$ with $\overline{\epsilon}$, and then vertically pasting gives
\[
	\big((-)_!\xrightarrow{(-)_!\cdot\overline{\epsilon}}(-)_!(-)^*(-)_*\xrightarrow{\epsilon\cdot(-)_*}(-)_*\big)\colon ([1]\boxtimes K)\times \Cat^{\subset_\BC}\ra \twoCat,
\]
where $K\coloneqq (0\leftarrow 1\ra2)$ is the span category. Since for $\iota\in \Cat^{\subset_\BC}$ the counit transformation $\overline{\epsilon}\colon\iota^*\iota_*\ra \id$ is an equivalence as $\iota$ is fully faithful, this 2-functor descends to $\smash{([1]\boxtimes K')\times \Cat^{\subset_\BC}}$ where $K'$ is the localisation of $K$ resulting from inverting $0\leftarrow1$.
\item \label{equ:bottom-row-2cats-2}Vertically pasting $\id_{(-)^*}$ to itself and then horizontally pasting with \ref{equ:bottom-row-2cats} gives
\[
	\big((-)_!\circ(-)^*\xleftarrow{((-)_!\cdot\overline{\epsilon})\circ(-)^*}(-)_!(-)^*(-)_*\circ(-)^*\xrightarrow{(\epsilon\cdot(-)_*)\circ(-)^*}(-)_*\circ(-)^*\big),
\]
a 2-functor $([2]\boxtimes K')\times \Cat^{\subset_\BC}\ra \twoCat$.
\item \label{equ:top-row-2cats}Pasting $\epsilon$ and $\overline{\eta}$ vertically gives
\[
	\big((-)_!(-)^*\xsra{\epsilon}\id\xsra{\overline{\eta}}(-)_*(-)^*\big)\colon([1]\boxtimes [2])\times \Cat^{\subset_\BC}\ra \twoCat.
\]
\end{enumerate}

We encourage the reader to think of these pastings diagrammatically in terms of their values at objects $\iota\in \Cat^{\subset_\BC}$, e.g.\,\ref{equ:bottom-row-2cats-2} corresponds to the left diagram in
\begin{equation}\label{equ:massive-pasting}
	\begin{tikzcd}
	\PSh(\cC) \arrow[r,"\iota^*",""{name=TL,inner sep=1pt,below}]  & \PSh(\cC_0) \arrow[r,"\iota_*",""{name=T,inner sep=1pt,below}] & \PSh(\cC) \\
	\PSh(\cC)\uar{\id} \arrow[r,"\iota^*",""{name=ML,inner sep=1pt,below}] \dar[swap]{\id}&\PSh(\cC_0)\uar{\id} \arrow[r,"\iota_!\iota^*\iota_*",""{name=M,inner sep=1pt,below}]\dar[swap]{\id} & \PSh(\cC) \dar{\id}\uar[swap]{\id}  \\
	\PSh(\cC)   \arrow[r,"\iota^*",""{name=BL,inner sep=1pt,below},swap] &\PSh(\cC_0)  \arrow[r,"\iota_!",""{name=B,inner sep=1pt,below},swap] & \PSh(\cC),
	\arrow[shorten <=12pt, shorten >=6pt, Rightarrow, from=M, to=T, "\epsilon\cdot\iota_*",swap]
	\arrow[shorten <=8pt, shorten >=8pt, Rightarrow, from=M, to=B, "\iota_!\cdot\overline{\epsilon}"]
	\arrow[shorten <=12pt, shorten >=6pt, Rightarrow, from=ML, to=TL, "\id",swap]
	\arrow[shorten <=8pt, shorten >=8pt, Rightarrow, from=ML, to=BL, "\id"]
	\end{tikzcd}
	\quad \quad \quad\quad
	\begin{tikzcd}
	a\rar{\alpha}\arrow[d,"\omega",swap]&b\dar{\beta}\arrow[shorten <=10pt,shorten >=10pt,Rightarrow]{dl}{\theta}\\
	c\arrow[d,"\kappa",swap]\arrow[r,swap,"\gamma"]&d\dar{\delta}\arrow[shorten <=10pt,shorten >=10pt,Rightarrow]{dl}{\rho}\\
	e\rar{\epsilon}&f.
	\end{tikzcd}
\end{equation}
\end{ex}

\begin{lem}\label{lem:triangle-identities}Vertical compositions of horizontal compositions of $2$-functors in \cref{lem:units-coherent} with functors in \eqref{equ:further-functors} satisfy the usual triangle identities, e.g.\,\cref{ex:pasting-examples} \ref{equ:triangle-id-functor} is equivalent to $\id_{ (-)_!}$.
\end{lem}

\begin{proof}We explain the proof that \cref{ex:pasting-examples} \ref{equ:triangle-id-functor} is equivalent to $\id_{ (-)_!}$; the other triangle identities are similar. The claim will follows from a general observation: given a $2$-functor $[1]\boxtimes [2]\ra \mathbb{C}$ into a $2$-category $\mathbb{C}$, as indicated by the right-hand diagram in \eqref{equ:massive-pasting}, the two procedures to turn it into a $2$-functor of the form $[1]\boxtimes [1]'\ra \mathbb{C}$ indicated by the following schematic agree:
\[
	\begin{tikzcd}
	a\rar{\alpha}\arrow[d,"\omega",swap]&b\dar{\beta}\arrow[shorten <=10pt,shorten >=10pt,Rightarrow]{dl}{\theta}\\
	c\arrow[d,"\kappa",swap]\arrow[r,swap,"\gamma"]&d\dar{\delta}\arrow[shorten <=10pt,shorten >=10pt,Rightarrow]{dl}{\rho}\\
	e\rar{\epsilon}&f.
	\end{tikzcd}
	\overset{\ref{enum:procedure-a}}\rightarrow
	\begin{tikzcd}
	a \arrow[r,"\delta\beta\alpha",""{name=T,inner sep=1pt,below}] \arrow[d,"\id",swap]&d\dar{\id}\\
	a\arrow[d,"\id",swap]\arrow[r,swap,"\delta \gamma\omega",""{name=B,inner sep=1pt,below}] & f\dar{\id}\\
	a \arrow[r,swap,"\epsilon\kappa \omega",""{name=C,inner sep=1pt,below}] &f.
	\arrow[shorten <=8pt, shorten >=8pt, Rightarrow, from=T, to=B, "\delta\theta"]
	\arrow[shorten <=8pt, shorten >=8pt, Rightarrow, from=B, to=C, "\rho\omega"]
	\end{tikzcd} 
	\overset{\ref{enum:procedure-a}}\rightarrow
	\begin{tikzcd}
	a \arrow[r,"\delta\beta\alpha",""{name=T,inner sep=1pt,below}] \arrow[d,"\id",swap]&[20pt] d \dar{\id}\\[10pt]
	a\arrow[r,swap,"\epsilon\kappa\omega",""{name=B,inner sep=1pt,below}] & f
	\arrow[shorten <=6pt, shorten >=6pt, Rightarrow, from=T, to=B, "(\rho\omega)(\delta\theta)" description]
	\end{tikzcd}
	\overset{\ref{enum:procedure-b}}\leftarrow
	\begin{tikzcd}
	a\rar{\alpha}\arrow[d,"\kappa\omega",swap]&[20pt]b\dar{\delta\beta}\arrow[shorten <=7pt,shorten >=7pt,Rightarrow]{dl}[description]{(\rho \omega)(\delta \theta)}\\[10pt]
	e\arrow[r,swap,"\epsilon"]&f
	\end{tikzcd}
	\overset{\ref{enum:procedure-b}}\leftarrow
	\begin{tikzcd}
	a\rar{\alpha}\arrow[d,"\omega",swap]&b\dar{\beta}\arrow[shorten <=10pt,shorten >=10pt,Rightarrow]{dl}{\theta}\\
	c\arrow[d,"\kappa",swap]\arrow[r,swap,"\gamma"]&d\dar{\delta}\arrow[shorten <=10pt,shorten >=10pt,Rightarrow]{dl}{\rho}\\
	e\rar{\epsilon}&f.
	\end{tikzcd}
\]
In words, the two procedures are:
\begin{enumerate}[label=(\alph*),leftmargin=*]
	\item \label{enum:procedure-a} First postcompose the result of applying \ref{const:deform-square} to the upper square with \ref{const:morphism-square} applied to $\delta$, paste the result vertically with the result of applying the construction \ref{const:deform-square} to the lower square and precomposing horizontally with \ref{const:morphism-square} applied to $\omega$, and then vertically compose.
	\item \label{enum:procedure-b}First compose the two squares vertically and then apply \ref{const:deform-square}.
\end{enumerate}
In the universal case $\id\colon [1]\boxtimes [2]\ra [1]\boxtimes [2]=\mathbb{C}$, the two procedures correspond to two $2$-functors $[1]\boxtimes [1]'\ra[1]\boxtimes [2]$. These are $2$-functors in the ordinary sense of classical $2$-categories, so the above diagrammatic schematics is actually a \emph{proof} that they agree.

\medskip

\noindent The two procedures can be applied more generally to 2-functors $([1]\boxtimes [2])\times\cC\ra \mathbb{C}$ for $\cC\in\Cat$. With this in mind, we consider the functor $[2]\times \Cat^{\subset_\BC}\times [1]\ra \Cat$ resulting from vertically pasting the two functors with values the squares \eqref{equ:tilt-left-kan}. Viewing it as an oplax transformation  $([2]\times \Cat^{\subset_\BC})\boxtimes [1]\ra \twoCat$ and applying \eqref{equ:mate-equ} as in the proof of \cref{lem:units-coherent} yields $([1]\boxtimes [2])\times \Cat^{\subset_\BC}\ra \twoCat$ which sends $\iota\in  \Cat^{\subset_\BC}$ to the vertical pasting of the squares \eqref{equ:tilt-left-kan-mate}. As a result of the naturality of \eqref{equ:mate-equ}, applying procedure \ref{enum:procedure-a} and \ref{enum:procedure-b} to it gives \cref{ex:pasting-examples} \ref{equ:triangle-id-functor} and $\id_{ (-)_!}$ respectively, so they agree.
\end{proof}

\begin{lem}\label{lem:compatibility-square}There is a commutative square in $\Cat$
\[
	\begin{tikzcd}
	([1]\times \Cat^{\subset_\BC})\arrow["(0\le 2)\times\id",d,swap]\rar{\ref{equ:top-row-2cats}}& \Fun^\lax([2],\twoCat)\dar{(0\le 2)^*}\\
	([2]\times \Cat^{\subset_\BC})\rar{ \ref{equ:bottom-row-2cats-2}}&\Fun^\lax(K',\twoCat)\simeq \Fun^\lax([1],\twoCat)
	\end{tikzcd}
\]
whose horizontal arrows are induced by the indicated functors from \cref{ex:pasting-examples} and the adjunction \eqref{equ:adjunction-oplax-cat}, and the equivalence by $K'\simeq[1]$ given by inverting $0\leftarrow1$ and composing with $1\ra2$.
\end{lem}

\begin{proof}
The claim follows once we provide an equivalence of $2$-functors $[1]\boxtimes[1]'\times  \Cat^{\subset_\BC}\ra \twoCat$ 
\[
	\big((-)_!(-)^*\xrightarrow{\overline{\eta}\epsilon}(-)_*(-)^*\big)\simeq \big((-)_!(-)^*\xrightarrow{((-)_!\cdot\overline{\epsilon}\cdot(-)^*)^{-1}}(-)_!(-)^*(-)_*(-)^*\xrightarrow{\epsilon(-)_*(-)^*}(-)_*(-)^*\big).
\]
Using \cref{lem:triangle-identities}, this follows from the argument at the beginning of the proof of \cref{thm:recollement}.\end{proof}

\begin{proof}[Proof of \cref{thm:recollement-nat}]Consider the following diagram in $\Cat$
\[
	\begin{tikzcd}[row sep=0.2cm]
	[1]\times \Cat^{\subset_\BC}\arrow["(0\le 2)\times\id",dd,swap]\rar{\ref{equ:top-row-2cats}}& \Fun^\lax([2],\twoCat)\arrow[dd,"(0\le 2)^*"] \arrow[dr,"\laxlim_{[2]}",bend left=10,""{name=LL,inner sep=1pt,below}]&\\
	&&\Cat.\\
	{[2]\times \Cat^{\subset_\BC}}\rar{\ref{equ:bottom-row-2cats-2}}&\Fun^\lax(K',\twoCat)\simeq \Fun^\lax([1],\twoCat)\arrow[ur,"\laxlim_{[1]}",swap,bend right=10]\arrow[shorten <=8pt, shorten >=8pt, Rightarrow, from=LL]&\ 
	\end{tikzcd}
\]
The left square is the commutative square from  \cref{lem:compatibility-square}. The two bent functors are instances of the lax limit  discussed in \cref{sec:laxlim}. The natural transformation in the triangle is given by the naturality of lax limits discussed in that section. From the outer part of the diagram together with the transformation, we get a functor $([2]\cup_{[1]}([1]\times [1]'))\times \Cat^{\subset_\BC}\ra \Cat$ where the pushout is formed with respect to $(0\le 2)\colon [1]\ra[2]$ and $[1]=[1]\times\{1\}\subset [1]\times [1]'$. On $\iota\in \Cat^{\subset_\BC}$, this functor evaluates to the left diagram in (the outer square is the restriction to $[1] \times [1]'$, and the extension to the pushout gives the decomposition of the bottom row as a composition)
\begin{equation}\label{equ:glued-diagram-naturality}
	\begin{tikzcd}
	\PSh(\cC)^{[2]}\arrow[rr,"\iota_!\iota^*\ra\id\ra\iota_*\iota^*"]\dar{(0\le 2)^*}&&\PSh(\cC)^{[2]}\dar{(0\le 2)^*}\\
	\PSh(\cC)^{[1]}\arrow[r,"\iota^*=\iota^*"]&\PSh(\cC_0)^{[1]}\arrow[r,"\iota_!\ra \iota_*"]&\PSh(\cC)^{[1]}
	\end{tikzcd}\quad \begin{tikzcd}[column sep=1.5cm]
	\PSh(\cC)^{[2]}\arrow[r,"\iota_!\iota^*\ra\id\ra\iota_*\iota^*"]\arrow["\iota^*\circ 	(-)\circ (0\le 2)",d]&\PSh(\cC)^{[2]}\dar{(0\le 2)^*}\\
	\PSh(\cC_0)^{[1]}\arrow[r,"\iota_!\ra \iota_*"]&\PSh(\cC)^{[1]}.
	\end{tikzcd}
\end{equation}
In particular, composing the left vertical with the left bottom arrow we get a functor $[1]\times[1]'\times \Cat^{\subset_\BC}\ra \Cat$, which is on $\iota\in \Cat^{\subset_\BC}$ given by the right diagram in \eqref{equ:glued-diagram-naturality}. Horizontally precomposing it with the functor $[1]\times[1]'\times \Cat^{\subset_\BC}\ra \Cat$ that sends $\iota$ to the left square in
\begin{equation}\label{equ:final-squares-naturality}
	\begin{tikzcd}
	\PSh(\cC)=\PSh(\cC)^{[0]}\rar{(0\le 0\le 0)^*}\dar{\iota^*}&[10pt] \PSh(\cC)^{[2]}\arrow["\iota^*\circ (-)\circ (0\le2)",d]\\
	\PSh(\cC_0)=\PSh(\cC)^{[0]}\rar{(0\le 0\le 0)^*}&\PSh(\cC_0)^{[1]}
	\end{tikzcd}\quad\quad
	\begin{tikzcd}
	\PSh(\cC)^{[2]}\dar{(0\le 2)^*}\rar{\nu^*}&\PSh(\cC\backslash \cC_0)^{[2]}\dar{(0\le 2)^*}\\
	\PSh(\cC)^{[1]}\rar{\nu^*}&\PSh(\cC\backslash \cC_0)^{[1]}.
	\end{tikzcd}
\end{equation}
yields a functor $[1]\times[1]' \times \Cat^{\subset_\BC}\ra \Cat$ that sends $\iota\in \Cat^{\subset_\BC}$ to \eqref{eqn:recollement-square-nat}. By construction, the functoriality in the $\Cat^{\subset_\BC}$ variable is induced by taking left Kan extension, so this functor satisfies the first part of the claim. To construct the functor in the second claim cone restricts the functor from the first claim along $\Cat^{\subset_\BC,\dec}\hookrightarrow\Cat^{\subset_\BC}$ and horizontally postcomposes it with a functor $[1]\times[1]'\times \Cat^{\subset_\BC,\dec}\ra \Cat$ that sends $\iota$ to the right diagram in \eqref{equ:final-squares-naturality}. To construct the  latter, one first constructs a functor $[1]\times \Cat^{\subset_\BC,\dec}\ra \Cat$ that sends $(\cC_0\subset\cC)\in  \Cat^{\subset_\BC,\dec}$ to $\nu^*\colon \PSh(\cC)\ra\PSh(\cC\backslash \cC_0)$ as in the proof of \cref{lem:kan-extension-upgrade}; this uses the additional assumption on the Beck--Chevalley morphisms involving $\nu^*$ in the definition of $ \Cat^{\subset_\BC,\dec}$.
\end{proof}

\section{A calculus for right-modules over an operad}\label{sec:operadic-framework}In this section we develop the operadic framework outlined in the introduction and in particular prove \cref{bigthm:calc}. Throughout this section, we fix a unital operad $\cO$ and write as in the introduction \[\RMod_{k}(\cO)\coloneqq \PSh(\Env(\cO)_{\le k}),\]  for $1 \le k\le \infty$, where $\Env(\cO)_{\leq k}\subset \Env(\cO)$ is the full subcategory given as the preimage of the full subcategory $\Fin_{\leq k}\subset\Fin$ of finite sets of cardinality $\le k$ under $\pi_\cO\colon \Env(\cO)\ra\Fin$. As mentioned in the introduction, the inclusions $ \Env(\cO)_{\le 1}\subset  \Env(\cO)_{\le 2}\subset  \cdots$ induce a tower of categories
\begin{equation}
	\hspace{-0.3cm}\label{equation:tower-right-modules}\RMod_\bullet(\cO)=\Big(\RMod(\cO)=\RMod_\infty(\cO)\ra \cdots\ra \RMod_2(\cO)\ra\RMod_1(\cO)\Big)\in\Tow(\Cat).
\end{equation}
by restriction. We call $\RMod_k(\cO)$ the category of \emph{$k$-truncated right-modules over $\cO$} and the restriction functors in the tower \eqref{equation:tower-right-modules} \emph{truncations}. In this section we
\begin{enumerate}[leftmargin=*]
	\item[\ref{sec:sym-mon-tower}] construct a symmetric monoidal refinement of the tower \eqref{equation:tower-right-modules},
	\item[\ref{sec:naturality}] explain how \eqref{equation:tower-right-modules} and its symmetric monoidal refinement are functorial in $\cO$,
	\item[\ref{sec:unital-right mod}] discuss a \emph{unital} variant $\RMod^{\un}_\bullet(\cO)$ of $\RMod_\bullet(\cO)$;
	\item[\ref{sec:layers}] describe the \emph{first stage} $\RMod_1(\cO)$ explicitly and determine the \emph{layers} of the tower \eqref{equation:tower-right-modules} by expressing $\RMod_k(\cO)\ra \RMod_{k-1}(\cO)$ as pulled back from a simpler functor,
	\item[\ref{sec:smoothing-theory}] prove a \emph{smoothing theory} result: for certain maps of operads $\cO\ra\cU$ the map of towers $\RMod_\bullet(\cO)\ra \RMod_\bullet(\cU)$ is pulled back along the map between the first stages,
	\item[\ref{sec:morita-extension}] extend all previous points to the level of Morita symmetric monoidal double categories, and 
	\item[\ref{sec:proof-bigthm}] deduce \cref{bigthm:calc} from the above.
\end{enumerate}

\smallskip

\begin{center} \emph{During this section, we will freely use the concepts and notation introduced in \cref{sec:categorical-preliminaries}.}\end{center} 

\subsection{Symmetric monoidal refinement}\label{sec:sym-mon-tower}
The symmetric monoidal structure on $\Env(\cO)$ induces a Day convolution symmetric monoidal structure $\otimes$ on $\RMod(\cO)=\PSh(\Env(\cO))$; see \cref{sec:day}. The colimit formula for it from \eqref{equ:day-conv} can be simplified:

\begin{lem}\label{lem:day-con-env}For $X,Y\in \RMod(\cO)$ and $(c_i)_{i\in S}\in\Env(\cO)$, there is a natural equivalence 
\[
	\textstyle{(X\otimes Y)((c_i)_{i\in S})\simeq \bigsqcup_{S=S'\sqcup S''}X((c_i)_{i\in S'})\times Y((c_i)_{i\in S''})}.
\]
where the disjoint union runs over ordered partitions of the finite set $S$ into two parts.
\end{lem}

\begin{proof}
The indexing category $(\Env(\cO)\times\Env(\cO))_{(c_i)_{i\in S}/}$ for the colimit formula \eqref{equ:day-conv} for $(X\otimes Y)((c_i)_{i\in S})$ maps via the symmetric monoidal functor $\pi_\cO \colon \Env(\cO) \to \Fin$ to the corresponding category $(\Fin\times\Fin)_{S/}$ for finite sets. There is an equivalence $(\Fin\times\Fin)_{S/}\simeq \sqcup_{S=S'\sqcup S''}\Fin_{S'/}\times \Fin_{S''/}$ induced by decomposing the source of a map $S\ra T\sqcup T'$ into the preimages of $T$ and $T'$. Using \eqref{equ:mapping-space-env}, this lifts to a decomposition $(\Env(\cO)\times\Env(\cO))_{(c_i)_{i\in S}/}\simeq  \sqcup_{S=S'\sqcup S''}\Env(\cO)_{/(c_i)_{i\in S'}}\times\Env(\cO)_{/(c_i)_{i\in S''}}$ which in turn induces a decomposition of the colimit 
\[
	\textstyle{(X\otimes Y)((c_i)_{i\in S})\simeq \bigsqcup_{S=S'\sqcup S''} \colim_{\Env(\cO)_{/(c_i)_{i\in S'}}\times\Env(\cO)_{/(c_i)_{i\in S''}}}(X(-) \times Y(-)).}
\]
Since $\smash{(\id_{(c_i)_{i\in S'}}, \id_{(c_i)_{i\in S''}})\in {\Env(\cO)_{/(c_i)_{i\in S'}}\times\Env(\cO)_{/(c_i)_{i\in S''}}}}$ is a final object, the $(S=S'\sqcup S'')$-summand in the decomposition is given by $X((c_i)_{i\in S'})\times Y((c_i)_{i\in S''})$, so the claim follows.\end{proof}

As part of the following proposition, we show that the Day convolution structure on $\RMod(\cO)$ induces compatible symmetric monoidal structures on $\RMod_k(\cO)$ for all $k$.
\begin{prop}\label{prop:tower}\,
\begin{enumerate}
	\item\label{enum:tower-convergence} The tower $\RMod_\bullet(\cO)$ in $\Cat$ converges.
	\item\label{enum:tower-monoidal-tower} The tower $\RMod_\bullet(\cO)$ can be lifted to a converging tower in $\CMon(\Cat)$.
	\item \label{enum:day-truncated}Objectwise the symmetric monoidal structure on $\RMod_k(\cO)$ from \ref{enum:tower-monoidal-tower} is for all $1\le k\le \infty$ given as in \cref{lem:day-con-env}. It makes $\RMod_k(\cO)$ a presentable symmetric monoidal category.
	\item\label{enum:day-grtp} $\RMod_\bullet(\cO)\in \Tow(\CMon(\Cat))$ is contained in the subcategory $\Tow(\CMon(\Cat)^\cgr)$.
\end{enumerate}
\end{prop}

\begin{proof}
The union of full subcategories $\Env(\cO)=\bigcup_{i\ge1}\Env(\cO)_{\le k}$ expresses $\Env(\cO)$ as a filtered colimit in $\Cat$ (e.g.\,by \cite[03DE]{LurieKerodon}) which implies Item \ref{enum:tower-convergence} since $\PSh(-)\colon \Cat^{\op}\ra \Cat$ preserves limits (since it is right adjoint to $\Fun(-,\cS)^{\op}$). To prove Item \ref{enum:tower-monoidal-tower}, we consider the operad $\RMod(\cO)^{\otimes}\ra \Fin_\ast$ associated to the Day convolution structure on $\RMod(\cO)$ and the sequence of full subcategories \begin{equation}\label{equ:lax-tower}
	\cdots\subset \RMod_k(\cO)^{\otimes}\subset \RMod_{k+1}(\cO)^{\otimes}\subset \cdots \subset\RMod(\cO)^{\otimes}
\end{equation}
defined as in \cref{sec:localise-day-convolution}. By \cref{lem:localise-day} \ref{enum:localise-day-i} the restriction $\RMod_k(\cO)^{\otimes}\ra \Fin_\ast$ defines a symmetric monoidal structure on $\RMod_k(\cO)$ if the inclusions $(\Env(\cO)_{\le k}\times \Env(\cO))_{(c_s)_{s\in S}/}\ra (\Env(\cO)\times \Env(\cO))_{(c_s)_{s\in S}/}$ are final for all objects $(c_s)_{s\in S}\in \Env(\cO)_{\le k}$. Arguing as in the proof of \cref{lem:day-con-env}, this inclusion is equivalent to the inclusion given by $\bigsqcup_{S=S'\sqcup S''}(\Env(\cO)_{\le k})_{(c_s)_{s\in S'}/}\times\Env(\cO)_{(c_s)_{s\in S''}/} \subset \bigsqcup_{S=S'\sqcup S''} \Env(\cO)_{(c_s)_{s\in S'}/}\times\Env(\cO)_{(c_s)_{s\in S''}/}$. The latter is final if the individual inclusions in the disjoint union are final, and this is the case since each of them preserves the final object $(\id,\id)$ (note that $|S'|\le |S|\le k$). In view of \cref{lem:localise-day} \ref{enum:localise-day-iv} this lifts the tower $(\cdots \ra \RMod_k(\cO)\ra\RMod_{k+1}(\cO)\ra\cdots\ra \RMod(\cO))$ in $\Cat^{\op}$ induced by right Kan extensions along the inclusions to a tower in $\CMon^\lax(\Cat)^{\op}$. All functors in this tower admit left adjoints given by restriction, so this is actually a tower in $\CMon^{\lax,R}(\Cat)^{\op}$. Applying \eqref{equ:mates} we obtain a lift of \eqref{equation:tower-right-modules} to a tower in $\CMon^\oplax(\Cat)$. By \cref{lem:localise-day} \ref{enum:localise-day-v} this is even a tower in $\CMon(\Cat)$, so serves as a lift as promised in Item \ref{enum:tower-monoidal-tower}. The first part of Item \ref{enum:day-truncated} follows from Lemmas \ref{lem:localise-day} \ref{enum:localise-day-ii} and  \cref{lem:day-con-env} using that the counit $\iota^*\iota_*\ra \id$ is an equivalence since the inclusion $\iota\colon\Env(\cO)_{\le k}\hookrightarrow \Env(\cO)$ is fully faithful. The second part of Item \ref{enum:day-truncated} follows from \cref{lem:localise-day} \ref{enum:localise-day-iii} (or the formula in \cref{lem:day-con-env}). In view of the definition of the subcategory $\CMon(\Cat)^\cgr\subset \CMon(\Cat)$ in \cref{sec:cgr-subcat}, Item \ref{enum:day-grtp} is equivalent to showing that the symmetric monoidal categories $\RMod_k(\cO)$ are compatible with geometric realisations for all $k$ and that for all $\ell\le k$ the restriction functor $\RMod_k(\cO)\ra\RMod_{\ell}(\cO)$ preserve geometric realisations. The former holds since $\RMod_k(\cO)$ is even presentable symmetric monoidal by \ref{enum:day-truncated}, and the latter since the restriction functor has a right adjoint by right Kan extension.
\end{proof}

\subsection{Naturality of the tower}\label{sec:naturality}The tower $\RMod_\bullet(\cO)$ can be made functorial in the operad $\cO$:

\begin{thm}\label{thm:naturality}The construction $\cO\mapsto \RMod_\bullet(\cO)$ from \cref{prop:tower} extends to a functor 
\[
	\RMod_\bullet(-)\colon \Opd^\un \lra \Tow(\CMon(\Cat)^{\cgr})
\] 
which is on morphisms induced by taking left Kan extensions.
\end{thm}

\noindent The proof of \cref{thm:naturality} relies on the following lemma (which is, incidentally, the first place where we use unitality of $\cO$). In its statement and in the remainder of this work, we denote a map of operads $\varphi\colon \cO\ra \cU$ and its value under the functor $\Env(-)\colon \Opd\ra \Cat$ by the same symbol.

\begin{lem}\label{lem:change-of-operads-restriction}
For a map $\varphi\colon \cO\ra\cP$ of unital operads and $1\le \ell\le k\le \infty$, the squares in $\Cat$
\[
	\begin{tikzcd}
	\RMod_k(\cO)\arrow[d,swap,"\iota^\ast"]\rar{\varphi_!}&\RMod_k(\cP)\arrow[d,swap,"\iota^\ast"]\\
	\RMod_{\ell}(\cO)\rar{\varphi_!}&\RMod_{\ell}(\cP)
	\end{tikzcd}\quad\text{and}\quad \begin{tikzcd}
	\RMod_k(\cO)&\RMod_k(\cP)\arrow[l,"\varphi^*",swap]\\
	\RMod_{\ell}(\cO)\uar{\iota_\ast}&\RMod_{\ell}(\cP)\uar{\iota_\ast}\arrow[l,"\varphi^*",swap].
	\end{tikzcd}
\]
commute, i.e.\,the two Beck--Chevalley transformations $\iota^*\varphi_!\ra \varphi_!\iota^*$ and $\varphi^*\iota_* \ra \iota_*\varphi^*$ are equivalences.
\end{lem}

\noindent In the following and in subsequent proofs, if we are given a functor $\varphi\colon \cC\ra \cD$ and an object $d\in \cD$ we write $\cC_{/d}$ and $\cC_{d/}$ for the pullbacks $\cC\times_{\cD}(\cD_{/d})$ and $\cC\times_{\cD}(\cD_{d/})$, respectively.

\begin{proof}[Proof of \cref{lem:change-of-operads-restriction}]It suffices to show the claim for the first Beck--Chevalley transformation, since the second is obtained from the first by taking right adjoints. From the colimit formula for left Kan extensions (see e.g.\,\cite[02YC]{LurieKerodon}), we see that it suffices to show that for all $(c_s)_{s\in S}\in \Env(\cP)_{\le \ell}$, the inclusion
$\smash{((\Env(\cO)_{\le \ell})_{(c_s)_{s\in S}/})^\op\subset ((\Env(\cO)_{\le k})_{(c_s)_{s\in S}/})^\op}$ is cofinal. 
For this, it is by \cite[4.1.3.1]{LurieHTT} enough to check that for each $((c_s)_{s\in S}\ra (\varphi(d_t))_{t\in T})\in (\Env(\cO)_{\le k})_{(c_s)_{s\in S}/}$ the category
\[
	\smash{\big((\Env(\cO)_{\le \ell})_{(c_s)_{s\in S}/}\big)_{/((c_s)_{s\in S}\ra (\varphi(d_t))_{t\in T})}}
\]
is weakly contractible. To this end, consider the image $\smash{\overline{T}\subset T}$ of the underlying map of sets of $(c_s)_{s\in S} \ra (\varphi(d_t))_{t\in T}$. As a result of \cref{cor:unital-nonsurj}, there is a unique factorisation of this morphism as the composition of a morphism $(c_s)_{s\in S}\ra (\varphi(d_t))_{t\in \overline{T}}$ followed by the image under $\varphi$ of the inclusion $(d_t)_{t\in \overline{T}}\subseteq (d_t)_{t\in T}$. Moreover, the functors induced by $(d_t)_{t\in \overline{T}}\subseteq (d_t)_{t\in T}$
\[
	\Env(\cO)_{/(d_t)_{t\in\overline{T}}}\lra \Env(\cO)_{/(d_t)_{t\in T}}\quad\text{and}\quad \Env(\cP)_{/(\varphi(d_t))_{t\in\overline{T}}}\lra \Env(\cO)_{/(\varphi(d_t))_{t\in T}}
\]
are fully faithful with essential image those morphisms whose underlying map of finite sets has image in $\overline{T}\subseteq T$, so it follows that the functor induced by $(d_t)_{t\in \overline{T}}\subseteq (d_t)_{t\in T}$
\[
	\big((\Env(\cO)_{\le \ell})_{(c_s)_{s\in S}/}\big)_{/((c_s)_{s\in S}\ra (\varphi(d_t))_{t\in \overline{T}})}\lra \big((\Env(\cO)_{\le \ell})_{(c_s)_{s\in S}/}\big)_{/((c_s)_{s\in S}\ra (\varphi(d_t))_{t\in T})}
\]
is an equivalence. But $\lvert S\rvert\le \ell$, so $\lvert\overline{T}\rvert\le \ell$ and hence the identity on $((c_s)_{s\in S}\ra (\varphi(d_t))_{t\in \overline{T}})$ is an initial object in the source. In particular these categories are indeed weakly contractible.
\end{proof}

\begin{proof}[Proof of \cref{thm:naturality}]We first consider the composition 
\[
	\smash{\Opd^\un_\infty \xrightarrow{\Env(-)}\CMon(\Cat)\hookrightarrow \CMon^\oplax(\Cat)\xlra{\eqref{equ:restriction-is-lax}} \CMon^\lax(\Cat)^\op}
\] 
which sends an operad $\cO$ to the symmetric monoidal category $\RMod(\cO)$ and a morphism $\varphi\colon \cO\ra\cP$ to the lax monoidal functor $\varphi^\ast\colon \RMod(\cP)\ra \RMod(\cO)$ induced by precomposition. As a result of the second part of \cref{lem:change-of-operads-restriction}, the underlying functor $\varphi^\ast$ preserves for all $k$ the essential images of the right Kan extension functors $\iota_*\colon \RMod_k(\cP)\ra \RMod(\cP)$ and $\iota_*\colon \RMod_k(\cO)\ra \RMod(\cO)$, so the map of operads $\varphi^\ast\colon \RMod(\cP)^\otimes\ra \RMod(\cO)^\otimes$ preserves the tower of full subcategory inclusions \eqref{equ:lax-tower}. This lifts $\Opd^\un_\infty\ra  \CMon^\lax(\Cat)^\op$ to a functor  $\Opd^\un_\infty\ra \Tow(\CMon^\lax(\Cat)^\op)$. Since the functors $\iota_\ast\colon \RMod_k(\cO)\ra\RMod_{k+1}(\cO)$ and $\varphi^*\colon \RMod_k(\cP)\ra \RMod_k(\cO)$ are right adjoints for all $k$, this functor lands in the subcategory $\Tow(\CMon^{\lax,R}(\Cat)^\op)$, so yields by taking left adjoints via \eqref{equ:mates} a functor $\Opd^\un_\infty\ra  \Tow(\CMon^{\oplax}(\Cat))$. By construction, the latter sends $\cO\in\Opd^{\un}$ to the tower in $\Tow(\CMon(\Cat)^\cgr)$ from \cref{prop:tower}, and it sends a morphism $\varphi\colon \cO\ra \cP$ in $\Opd^{\un}$ to the map of tower consisting of the oplax symmetric monoidal functors $\varphi_!\colon \RMod_k(\cO)\ra\RMod_k(\cP)$ given as the image of the symmetric monoidal functor $\varphi\colon \Env(\cO)\ra \Env(\cP)$ under \eqref{equ:left-kan-is-oplax}. Since \eqref{equ:left-kan-is-oplax} sends (strong) symmetric monoidal functors to (strong) symmetric monoidal functors by \cref{lem:left-kan-monoidal}, it follows that the functor $\Opd^\un_\infty\ra  \Tow(\CMon^{\oplax}(\Cat))$ lands in $\Opd^\un_\infty \ra  \Tow(\CMon(\Cat))$, so it only remains to show that the left Kan extension $\varphi_!\colon \RMod_k(\cO)\ra \RMod_k(\cP)$ lies in $\CMon(\Cat)^\cgr$ for all $k$, i.e.\,preserves geometric realisations. But $\varphi_!$ is a left adjoint, so preserves all colimits.
\end{proof}

\subsection{Unital right-modules}\label{sec:unital-right mod}For some purposes, it is convenient to restrict to the full subcategory $\RMod_k^\un(\cO)\subset \RMod_k(\cO)$ of truncated right-modules that are \emph{unital}, i.e.\,their value at the monoidal unit $\varnothing\in\Env(\cO)_{\le k}$ is contractible. The following records various properties of this subcategory, including an alternative description as presheaves on the full subcategory $\Env^{\neq 0}(\cO)_{\le k}\subset \Env(\cO)_{\le k}$ on those objects that are \emph{not} equivalent to the unit $\varnothing\in \Env(\cO)_{\le k}$:

\begin{lem}\label{lem:unital-modules-properties}For all $1\le k\le \infty$, the full subcategory $\RMod_k^\un(\cO)\subset \RMod_k(\cO)$ satisfies the following:
\begin{enumerate}[leftmargin=*]
	\item\label{enum:unital-i} It is preserved by the truncations \eqref{equation:tower-right-modules}, so \eqref{equation:tower-right-modules} restricts to a tower $\RMod^\un_\bullet(\cO)\in\Tow(\Cat)$.
	\item\label{enum:unital-ii} It is closed under weakly contractible colimits, so in particular geometric realisations.
	\item\label{enum:unital-iii} It is preserved by left Kan extension $\RMod_k(\cO)\ra \RMod_k(\cP)$ along maps $\cO\ra\cU$ in $\Opd^\un$.
	\item\label{enum:unital-iv}It contains the monoidal unit of $\RMod_k(\cO)$ and is closed under monoidal products, so inherits a symmetric monoidal structure from $\RMod_k(\cO)$ by \cite[2.2.1.2]{LurieHA}.
	\item\label{enum:unital-v} The symmetric monoidal category from \ref{enum:unital-iv} is unital as an operad.
	\item\label{enum:unital-add} Restriction induces an equivalence of categories  $\RMod^\un_k(\cO)\simeq \PSh(\Env^{\neq 0}(\cO)_{\le k})$.
\end{enumerate}
\end{lem}

\begin{proof}Item \ref{enum:unital-i} is clear from the definition since the monoidal unit $\varnothing$ lies in $\Env(\cO)_{\le k}$ for all $k$. Item \ref{enum:unital-ii} follows from the fact that colimits in categories of presheaves are computed objectwise and that the colimit of a diagram in $\cS$ with contractible values which is indexed over a weakly contractible category is contractible. Item \ref{enum:unital-iii} follows by applying the argument in the proof of \cref{lem:change-of-operads-restriction} to the case to $\ell=0$ and $k=k$ since $\Env(\cO)_{\le 0}\simeq \ast$, so the restriction $\RMod_k(\cO)\ra\RMod_0(\cO)$ is equivalent to the evaluation $\ev_{\varnothing}\colon \RMod_k(\cO)\ra\cS$ whose fibre over $\ast\in\cS$ is $\RMod^{\un}_k(\cO)\subset \RMod_k(\cO)$. As with any Day convolution monoidal structure, the unit $\RMod_k(\cO)$ is the representable presheaf $E_{\varnothing}\coloneqq \Map_{\Env(\cO)_{\le k}}(\varnothing,-)$ on the unit, so it is unital since we assumed the operad $\cO$ to be unital. Together with an application of \cref{prop:tower} \ref{enum:day-truncated} this yields Item \ref{enum:unital-iv}. Item \ref{enum:unital-v} follows from the Yoneda lemma since $\Map_{\RMod^\un_k(\cO)}(E_{\varnothing},X)\simeq X(\varnothing)\simeq \ast$ for all $X\in \RMod^\un_k(\cO)$. 

\medskip

\noindent Regarding Item \ref{enum:unital-add}, we only explain the case $k=\infty$; the proof for other $k$ is the same. By general Kan extension considerations, it suffices to show that the unit $X\ra \iota_*\iota^*X$ of the adjunction between $\PSh(\Env^{\neq 0}(\cO))$ and $\RMod(\cO)$ given by restriction and right Kan extension along the inclusion $\iota\colon \Env^{\neq 0}(\cO)\hookrightarrow \Env(\cO)$ is an equivalence if and only if $X$ is reduced. By the limit formula for right Kan extensions, this unit is for $(c_i)_{i\in S}\in\Env(\cO)$ given by $X((c_i)_{i\in S})\ra \lim_{((c'_i)_{i\in S'}\ra (c_i)_{i\in S}, S'\neq \varnothing)}X((c'_i)_{i\in S'})$. The latter is an equivalence if $S\neq \varnothing$ since $\id_{(c_i)_{i\in S}}$ is in this case a terminal object for the indexing category, and it is an equivalence if $S=\varnothing$ if and only if $X$ is unital since in this case the target is contractible because the indexing category of the limit is empty since there are no maps from a nonempty to an empty set. 
\end{proof}

\begin{cor}\label{cor:tower-thms-unital}Theorems \ref{prop:tower} and \ref{thm:naturality} hold for $\RMod^\un_\bullet(\cO)$ in place of $\RMod_\bullet(\cO)$, except for the second part of \cref{prop:tower} \ref{enum:day-truncated}, on the presentable monoidality. Moreover, the inclusions \[\RMod^\un_\bullet(-)\subset \RMod_\bullet(-)\] extend to a natural transformation of functors $\Opd^\un\ra\Tow(\CMon(\Cat)^\cgr)$.
\end{cor}

\begin{proof}Using the properties established in \cref{lem:unital-modules-properties}, this is a direct consequence of that statements of Theorems \ref{prop:tower} and \ref{thm:naturality} for the larger category $\RMod_\bullet(\cO)$. For the convergence, one uses that the category $\RMod^\un_k(\cO)$ is the fibre at $\ast\in\cS$ of the evaluation $\ev_{\varnothing}\colon \RMod_k(\cO)\ra \cS$.
\end{proof}

\begin{rem}Note that if $\cO$ is not the initial operad, i.e.\,if $\cO^\col\neq\varnothing$, then $\RMod_k^\un(\cO)$ is \emph{not} presentable symmetric monoidal for any $1\le k\le\infty$: as a category of presheaves (see \cref{lem:unital-modules-properties} \ref{enum:unital-add}), it is presentable, but the monoidal structure does not preserve colimits in either of the variables. To see this, fix an object $c\in \cO^\col$ and a category $J$ which is not weakly contractible. For $C\coloneqq \const_*\in \RMod_k^\un(\cO)$ the constant presheaf with value a point, we write $\lvert-\rvert\colon \Cat\ra\cS$ for the left adjoint to the inclusion $\cS\subset\Cat$ (sometimes called the ``classifying space'') and compute
\[
	[\colim_{j\in J}(C\otimes C)](c)\simeq\colim_{j\in J}[(C\otimes C)(c)]\simeq\colim_{j\in J}[C(c)\times C(\varnothing) \sqcup C(\varnothing)\times C(c)]\simeq \lvert J\rvert\sqcup \lvert J\rvert,
\] 
where we used \cref{lem:day-con-env} and the fact that we can compute the values of a colimit in $\RMod_k^\un(\cO)$ at objects not equivalent to the unit as the colimit of the values, since the restriction $\RMod_k^\un(\cO)\ra \PSh(\Env^{\neq 0}(\cO)_{\le k})$ is an equivalence by  \cref{lem:unital-modules-properties} \ref{enum:unital-add}, so preserves colimits. On the other hand
\[
	([\colim_{j\in J}(C)]\otimes C)(c)\simeq\big([\colim_{j\in J}(C)](c)\times C(\varnothing)\big) \sqcup \big([\colim_{j\in J}(C)](\varnothing)\times C(c)\big)\simeq \lvert J \rvert \sqcup *\not\simeq \lvert J \rvert\sqcup \lvert J \rvert.
\]
\end{rem}

\subsection{The layers of the tower}\label{sec:layers} One purpose of the tower $\RMod_\bullet(\cO)$ is that it decomposes the category $\RMod(\cO)$ into ``simpler pieces'', namely $\RMod_1(\cO)$ (the \emph{bottom layer}) and the ``differences'' between $\RMod_k(\cO)$ and $\RMod_{k-1}(\cO)$ for $k>1$ (the \emph{higher layers}). This section identifies these ``simpler pieces'', beginning with the bottom layer:

\begin{thm}[The bottom layer]\label{thm:bottom-layer} Restriction along $\cO^\col\subset \Env(\cO)$ induces an equivalence \[
	\RMod^\un_1(\cO)\simeq\PSh(\cO^\col)
\]
of symmetric monoidal categories, where $\RMod^\un_1(\cO)$ is equipped with the Day convolution structure resulting from \cref{cor:tower-thms-unital} and $\PSh(\cO^\col)$ with the cocartesian symmetric monoidal structure.
\end{thm}

\begin{proof} \cref{lem:unital-modules-properties} \ref{enum:unital-add} implies that the restriction $\RMod^\un_1(\cO)\ra \PSh(\Env^{\neq 0}(\cO)_{\le 1})$ is an equivalence. As $\Env^{\neq0}(\cO)_{\le 1}\simeq \cO^\col$ as a result of \eqref{equ:mapping-space-env}, it only remains to show that the Day convolution structure on $\RMod^\un_1(\cO)$ is cocartesian, for which it suffices to show that the induced monoidal structure on the homotopy category is given by the coproduct (see \cref{sec:cocart-oprightfib}). This follows from specialising the formula in \cref{lem:day-con-env} to $S\in\Fin$ with $|S|=1$ (this step only works for unital right-modules).
\end{proof}

\subsubsection{The higher layers}\label{sec:higher-layers-nomodule}
Our next goal is to explain what data is needed to reconstruct the category $\RMod_{k}(\cO)$ from $\RMod_{k-1}(\cO)$. This will be an application of the decomposition result for presheaf categories from \cref{thm:recollement}, but for convenience we will introduce the objects and notation involved from scratch: writing $\iota\colon \Env(\cO)_{\le k-1}\hookrightarrow \Env(\cO)_{\le k}$ for the usual inclusion and $\nu\colon \cO^\col\wr\Sigma_k\hookrightarrow \Env(\cO)_{\le k}$ for the inclusion of the wreath product resulting from \eqref{equ:operad-wreath-subcat}, we consider functors
\begin{equation}\label{equ:layers-functors}
	\begin{tikzcd} 
	\RMod_{k-1}(\cO) \rar[shift left=1.2ex]{\iota_!} \rar[shift left=-1.2ex,swap]{\iota_*}& \RMod_{k}(\cO) \rar{\nu^*} \lar & \PSh( \cO^\col\wr\Sigma_k)
	\end{tikzcd}
\end{equation}
where the unlabelled leftmost middle arrow is the restriction $\iota^*$, whose left and right adjoints $\iota_!$ and $\iota_*$ are given by left and right Kan extensions respectively. The counit of $\iota_! \dashv \iota^*$ and the unit of $\iota^* \dashv \iota_*$ induce natural transformations $\nu^*\iota_!\iota^*\ra \nu^*\ra \nu^*\iota_*\iota^*$ which we view as a functor of the form
\[
	\Lambda\coloneqq \big(\nu^*\iota_!\iota^*\ra \nu^*\ra \nu^*\iota_*\iota^*\big)\colon\RMod_k(\cO)\lra \PSh(\cO^\col\wr\Sigma_k)^{[2]}.
\]
Similarly, there are natural transformations $\nu^*\iota_!\xsla{\simeq}\nu^*\iota_!\iota^*\iota_*\ra\nu^*\iota_*$ induced by the counits of $\iota^* \dashv \iota_*$ and $\iota_! \dashv \iota^*$ (the former is an equivalence as $\iota$ is fully faithful), which we view as a functor
\[
	\Omega\coloneqq \big(\nu^*\iota_!\ra \nu^*\iota_*\big)\colon\RMod_k(\cO)\lra \PSh(\cO^\col\wr\Sigma_k)^{[1]}.
\]

\begin{thm}[The higher layers]\label{thm:layers-tower}Fix $1<k<\infty$.
\begin{enumerate}
	\item\label{enum:layer-pullback} There is a commutative square in $\Cat$
	\begin{equation}\label{equ:recol-discs}
		\begin{tikzcd}
		\RMod_k(\cO) \rar{\Lambda} \dar[swap]{\iota^*} & \PSh(\cO^\col\wr\Sigma_k)^{[2]} \dar{(0\le 2)^*} \\
		\RMod_{k-1}(\cO) \rar{\Omega} & \PSh(\cO^\col\wr\Sigma_k)^{[1]}
		\end{tikzcd}
	\end{equation}
	which is cartesian.
	\item\label{enum:layer-pullback-gc} If $\cO$ is groupoid-coloured, then also the commutative square
	\begin{equation}\label{equ:recol-discs-simpler}
		\begin{tikzcd}
		\RMod_k(\cO) \rar{\Lambda^c} \dar[swap]{\iota^*} &\cS^{[2]}\dar{(0\le 2)^*}\\
		\RMod_{k-1}(\cO) \rar{\Omega^c} &\cS^{[1]}
		\end{tikzcd}
	\end{equation}
	resulting from extending \eqref{equ:recol-discs} to the right using $\colim\colon \PSh(\cO^\col\wr\Sigma_k)\ra\cS$, is cartesian.
	\item\label{enum:layer-pullback-red} Both \ref{enum:layer-pullback} and \ref{enum:layer-pullback-gc} are also valid with $\RMod^\un_\bullet(\cO)$ in place of $\RMod_\bullet(\cO)$.
	\end{enumerate}
\end{thm}

\begin{proof}
Item \ref{enum:layer-pullback} follows from an application of \cref{thm:recollement}, once we show that $\Env(\cO)_{\le k-1}\subset \Env(\cO)_{\le k}$ is a decomposition pair in the sense of \cref{dfn:decom-pair}, that it satisfies the assumption in \cref{thm:recollement}, and that it has the property that the complementary subcategory $\Env(\cO)_{\le k}\backslash \Env(\cO)_{\le k-1}\subset \Env(\cO)_{\le k}$ is given by $\cO^\col\wr\Sigma_k\simeq \Env(\cO)\times_{\Fin}\Fin_k^\simeq \subset \Env(\cO)_{\le k}$. The first condition in \cref{dfn:decom-pair} follows by applying $\pi_\cO\colon \Env(\cO)\ra \Fin$ and the second condition follows from the factorisation of morphisms  $\alpha\colon (c_s)_{s\in S}\ra (d_t)_{t\in T}$ in $\Env(\cO)_{\le k}$ through the inclusions $\smash{(d_t)_{t\in \im(\overline{\alpha})}\subseteq (d_t)_{t\in T}}$ resulting from \cref{cor:unital-nonsurj}; here $\overline{\alpha}\colon S\ra T$ is the underlying map of finite sets. Similarly, since a selfmap of finite sets of cardinality $k$ is an equivalence if and only if it does not factor through a set of lower cardinality, the claim that $\Env(\cO)_{\le k}\backslash \Env(\cO)_{\le k-1}$ agrees with $\cO^\col \wr \Sigma_k \simeq \Env(\cO)\times_{\Fin}\Fin_k^\simeq $ is equivalent to showing that a morphism $\alpha\colon (c_s)_{s\in S}\ra (d_t)_{t\in T}$ in $\Env(\cO)_{\le k}$ with $\lvert S\rvert=\lvert T\rvert=k$ factors through an object in $\Env(\cO)_{\le k-1}$ if and only if this is the case for the underlying maps of finite sets. This again follows from the factorisation of $\alpha$ through $(d_t)_{t\in \im(\overline{\alpha})}\subseteq (d_t)_{t\in \im(\overline{\alpha})}$. This leaves us with verifying the assumption of \cref{thm:recollement}, namely that the natural map 
\begin{equation}\label{equ:recollement-assumption-rmod}
	\begin{tikzcd}[row sep=0.3cm]
	(\iota_!\circ \iota^*)\big(\Map_{\Env(\cO)_{\le k}}(-,(d_t)_{t\in T})\big)((c_s)_{s\in S})\sqcup \Map_{\cO^\col \wr \Sigma_k}((c_s)_{s\in S},(d_t)_{t\in T}) \dar \\ \Map_{\Env(\cO)_{\le k}}((c_s)_{s\in S},(d_t)_{t\in T})
	\end{tikzcd}
\end{equation}
is an equivalence for all $(c_s)_{s\in S}, (d_t)_{t\in T}\in\Env(\cO)$ with $|S|=|T|=k$. To show this, we observe, firstly, that the second summand in the source of the map in the hypothesis of \cref{thm:recollement} is by definition the collection of those components of the target that map to isomorphisms in $\Fin$ (or equivalently, to surjections, as $|S|=|T|=k$), and, secondly, that the map from the first summand lands in the subspace of the target consisting of those components whose underlying map of finite sets is not surjective; we henceforth indicate this subspace by adding a superscript $(-)^\nsur $. As a result, the map \eqref{equ:recollement-assumption-rmod} is an equivalence if and only if its restriction 
\[
	(\iota_!\circ \iota^*)\big(\Map_{\Env(\cO)_{\le k}}(-,d_t)_{t\in T})\big)((c_s)_{s\in S})\lra \Map^{\nsur }_{\Env(\cO)_{\le k}}((c_s)_{s\in S},(d_t)_{t\in T})
\]
is an equivalence. Since ${\Map^{\nsur }_{\Env(\cO)_{\le k}}(-,(d_t)_{t\in T})\subset \Map_{\Env(\cO)_{\le k}}(-,(d_t)_{t\in T})}$ is a subfunctor and an equivalence when restricted to $\Env(\cO)_{\le k-1}$, so it suffices to showing that the counit 
\begin{equation}\label{equ:counit-nonsur}
	\smash{(\iota_!\circ \iota^*)\big(\Map^{\nsur }_{\Env(\cO)_{\le k}}(-,(d_t)_{t\in T})\big)\lra\Map^{\nsur }_{\Env(\cO)_{\le k}}(-,(d_t)_{t\in T}).}
\end{equation} 
is an equivalence in $\PSh(\Env(\cO)_{\le k})$. To this end, we first show that the map in $\PSh(\Env(\cO)_{\le k})$
\begin{equation}\label{equ:colim-nonsur}
	\smash{\underset{T'\subsetneq T}\colim\,\Map_{\Env(\cO)_{\le k}}(-,(d_t)_{t\in T'})\lra\Map^{\nsur }_{\Env(\cO)_{\le k}}(-,(d_t)_{t\in T})}
\end{equation}
induced by the cubical diagram $(d_t)_{t\in \bullet}$ from \cref{cor:unital-nonsurj} is an equivalence. By the second part of \cref{cor:unital-nonsurj}, the map $\smash{\Map_{\Env(\cO)_{\le k}}((c'_{s})_{s\in S'},(d_t)_{t\in T'})\ra\Map^{\nsur }_{\Env(\cO)_{\le k}}((c'_{s})_{s\in S'},(d_t)_{t\in T})}$ is for all $(c'_{s})_{s\in S'}\in \Env(\cO)_{\le k}$ and $T'\subsetneq T$ an equivalence onto the subspace 
\[
	\smash{\Map^{\nsur }_{\Env(\cO)_{\le k}}((c'_{s})_{s\in S'},(d_t)_{t\in T})_{T'}\subset \Map^{\nsur }_{\Env(\cO)_{\le k}}((c'_{s})_{s\in S'},(d_t)_{t\in T})}
\] 
of those components whose induced map in $\Fin$ has image in $T'$. As $T'\subsetneq T$ varies, these components exhaust the whole space because the non-surjectivity condition, so it suffices to show that the colimit is given by the union of these collections of components indexed by proper subsets $T'\subsetneq T$. For this, we note that this union has the property that the intersection of the components corresponding to a subset $T'\subsetneq T$ with those corresponding to $T''\subsetneq T$ are the ones corresponding to $T'\cap T''\subsetneq T$, which says that the cubical diagram of collections of components indexed by the poset of subsets of $T$ is strongly cartesian. For strongly cartesian cubical diagrams consisting of inclusions of components the colimit of the restriction to the poset of proper subsets of $T$ is given by the union (this follows e.g.\,from \cite[Lemma 5.10.4]{MunsonVolic}), so \eqref{equ:counit-nonsur} is indeed an equivalence.

\medskip

\noindent Since $\iota_!$ and $\iota^*$ are left adjoints and thus preserve colimits, the colimit decomposition \eqref{equ:colim-nonsur} shows that \eqref{equ:counit-nonsur} is an equivalence if the component of the counit of the adjunction $i_! \dashv i^*$ at the presheaf $\smash{\Map_{\Env(\cO)_{\le k}}(-,(d_t)_{t\in T'})}$ is for all $T'\subsetneq T$ an equivalence. But since $\lvert T'\rvert\le k-1$, applying $\iota^*(-)$ to this presheaf gives the representable presheaf $\smash{\Map_{\Env(\cO)_{\le k-1}}(-,(d_t)_{t\in T'})}$, so the claim follows from the fact that left Kan extensions preserves representable presheaves. This concludes the proof of \ref{enum:layer-pullback}.
 
 \medskip

\noindent Item  \ref{enum:layer-pullback-gc} follows by pasting the pullback from \ref{enum:layer-pullback} with that in \cref{lem:pullback-composition-rfib-fact} for $\varphi\coloneqq ((0\le 2)\colon [1]\ra[2])$, $\cB\coloneqq \cO^\col \wr \Sigma_k$, and $\cD=\ast$, using that $\cO^\col\wr\Sigma_k$ is a groupoid since $\cO^\col$ is assumed to be one, so the map to a point $\cO^\col\wr\Sigma_k\ra\ast$ is a right-fibration.
Finally, Item \ref{enum:layer-pullback-red} follows by pasting the pullbacks in \ref{enum:layer-pullback} and \ref{enum:layer-pullback-gc} with the square that expresses $\RMod^\un_k(\cO)$ as the pullback of the restriction functor $\iota^*\colon \RMod_{k}(\cO)\ra \RMod_{k-1}(\cO)$ along the inclusion $\RMod^\un_{k-1}(\cO)\subset \RMod_{k-1}(\cO)$.
\end{proof}

\begin{rem}\label{rem:mod-out-ocol-square}Occasionally it is convenient to factor \eqref{equ:recol-discs-simpler} into a pasting of two commutative squares
\begin{equation}\label{equ:recol-discs-simpler-factorisation}
		\begin{tikzcd}
		\RMod_k(\cO) \rar{\Lambda^c} \dar[swap]{\iota^*} &\PSh(\Fin^\simeq_{k})^{[2]}\rar{\colim}\dar{(0\le 2)^*} &\cS^{[2]}\dar{(0\le 2)^*}\\
		\RMod_{k-1}(\cO) \rar{\Omega^c} &\PSh(\Fin^\simeq_{k})^{[1]}\rar{\colim}&\cS^{[1]}
		\end{tikzcd}
	\end{equation}
	whose left horizontal arrows are obtained from $\Lambda^c$ and $\Omega^c$ in \eqref{equ:recol-discs} using the left Kan extension $\PSh(\cO^\col\wr\Sigma_k)\ra \PSh(\Fin_k^\simeq)$ along the projection $\cO^\col\wr\Sigma_k\ra \Fin_k^\simeq$. In \cref{thm:layers-tower} \ref{enum:layer-pullback-gc} we showed that if $\cO$ is groupoid-coloured, then the outer square is a pullback. The same argument also shows that the left-hand square is a pullback if $\cO$ is groupoid-coloured.
\end{rem}

\subsubsection{A simplified description of the right Kan extension} \label{sec:simplified-coboundaries} 
The right Kan extension $\iota_\ast\colon\RMod_{k-1}(\cO)\ra \RMod_{k}(\cO)$ featuring in the definition of $\Lambda$ and $\Omega$ in the description of the higher layers from \cref{thm:layers-tower} has a more explicit description in terms of the cubical diagrams $(c_s)_{s\in\bullet}$ from \cref{cor:unital-nonsurj}. This is a special case of the following identification of the right Kan extension along any of the inclusions $\iota\colon \Env(\cO)_{\le \ell}\hookrightarrow \Env(\cO)_{\le k}$ for $1\le \ell\le k\le \infty$:

\begin{lem}\label{lem:simplified-right-kan}
For $X\in \RMod_\ell(\cO)$ and $(c_{s})_{s\in S}\in\Env(\cO)_{\le k}$, there is an equivalence 
\[
	\smash{\iota_*(X)\big((c_{s})_{s\in S}\big)\simeq \lim_{S'\subseteq S, \lvert S'\rvert \leq \ell}X\big((c_{s})_{s\in S'}\big)}
\]
induced by the cubical diagram $(c_s)_{s\in\bullet}$ from \cref{cor:unital-nonsurj}. This is natural in $X$ and $(c_{s})_{s\in S}$.
\end{lem}

\begin{proof}
We write $\mathrm{P}_{\leq \ell}(S)$ for the poset of subsets of $S$ of cardinality $\leq \ell$ ordered by inclusion. The cubical diagram $(c_s)_{s\in\bullet}$ induces functor $\mathrm{P}_{\leq \ell}(S)_{/S}\ra (\Env(\cO)_{\le \ell})_{/(c_s)_{s\in S}}$. From the limit formula for right Kan extension, we see that the claim follows from showing that this functor is cofinal, for which it is by \cite[4.1.3.1]{LurieHTT} enough to prove that for $(\alpha\colon (d_t)_{t\in T}\ra (c_s)_{s\in S})\in (\Env(\cO)_{\le \ell})_{/(c_s)_{s\in S}}$ the category $\mathrm{P}_{\leq \ell}(S)_{((d_t)_{t\in T}\ra (c_s)_{s\in S})/}$ is weakly contractible. This is the category of factorisations $(d_t)_{t\in T}\ra (c_s)_{s\in T'}\subset (c_s)_{s\in S}$ of $\alpha$ for subsets $T'\subset S$ of cardinality $\leq \ell$ which is indeed is weakly contractible since it has in view of \cref{cor:unital-nonsurj} an initial object $(d_t)_{t\in T}\ra (c_s)_{s\in \im(\overline{\alpha})}\subset (c_s)_{s\in S}$ where $\overline{\alpha}\colon T\ra S$ is the map of finite sets induced by $\alpha$. This implies the first part of the claim. The second part follows from the naturality of the cubical diagram $(c_s)_{s\in\bullet}$ from \cref{cor:unital-nonsurj}.
\end{proof}

\begin{ex}[The second layer]\label{ex:snd-layer-abstract}When specialised to $k=2$ and restricting to unital right-modules the pullback square \eqref{equ:recol-discs} has the form 
\begin{equation}\label{equ:recol-discs-simpler-factorisation}
		\begin{tikzcd}
		\RMod^\un_2(\cO) \rar{\Lambda^c} \dar[swap]{\iota^*} &\PSh(\cO^\col\wr\Sigma_2)^{[2]}\dar{(0\le 2)^*}\\
		\PSh(\cO^\col)\rar{\Omega^c} &\PSh(\cO^\col\wr\Sigma_2)^{[1]};
		\end{tikzcd}
	\end{equation}
	here we used \cref{thm:bottom-layer}. Using \cref{lem:simplified-right-kan} and rewriting the relevant left Kan extension as a coend, one sees that $\Lambda^c$ sends $X\in \RMod^\un_2(\cO)$ to the composition in $\PSh(\cO^\col\wr\Sigma_2)$ that on a pair of colours  $(c,d)\in \cO^\col\wr\Sigma_2$ is given by the following composition induced by the functoriality of $X$
	\[\Map_{\Env(\cO)}(c\sqcup d,-)\otimes_{\cO^\col}X \lra X(c\sqcup d)\lra X(c)\times  X(d)
	\]
\end{ex}

\subsubsection{Naturality of the layers} The identification of the higher layers in \cref{thm:layers-tower} can be made natural in operadic right-fibrations in the sense of Definition \ref{dfn:oprightfib}, i.e.\,it extends to a functor on the wide subcategory $\smash{\Opd^{\rf}\subset \Opd^{\un}}$ on the operadic right-fibrations:

\begin{thm}\label{thm:naturality-of-layers}The squares \eqref{equ:recol-discs} and \eqref{equ:recol-discs-simpler} in \cref{thm:layers-tower} extend to functors $\smash{\Opd^{\rf} \ra \Cat^{[1] \times [1]}}$. On morphisms, these extensions are induced by left Kan extension.
\end{thm}

\begin{rem}[Naturality in the tangential structure]By \cref{lem:pullbacks-groupoid-type}, the category $\Opd^{\rf}$ contains the subcategory $\smash{\Opd^{\gc,\simeq_\red}}$ of $\smash{\Opd^{\un}}$ consisting of groupoid-coloured operads and maps between them that are equivalences on reduced covers (see \cref{sec:tangential-structures-operads}). The later is equivalent to the overcategory $\smash{\cS_{/\Opd^{\red,\simeq}}}$ by \eqref{equ:gc-redequ-identification}, so for a fixed reduced operad $\cU$ we obtain functors $\smash{\cS_{/\BAut(\cU)}\ra \Cat^{[1] \times [1]}}$ that send a tangential structure $\theta\colon B\ra \BAut(\cU)$ in the sense of \cref{dfn:tangial-structures-opd} to the pullbacks from \cref{thm:layers-tower} for the $\theta$-framed $\cU$-operad $\cU^{\theta}$. Informally speaking, this says that the layers are natural in the tangential structure.
\end{rem}

 The key to proving \cref{thm:naturality-of-layers} is the following lemma:

\begin{lem}\label{lem:auxiliarly-naturality-layers}For an operadic right-fibration $\varphi\colon \cO\ra\cP$, both squares in the diagram of categories 
\[
	\begin{tikzcd}
	\RMod_{k-1}(\cO)\rar{\iota_*}\dar{\varphi_!}& \RMod_{k}(\cO)\dar{\varphi_!}\rar{\nu^*}& \PSh(\cO^\col\wr{\Sigma_k})\dar{\varphi_!}\\
	\RMod_{k-1}(\cP)\rar{\iota_*}& \RMod_{k}(\cP)\rar{\nu^*}& \PSh(\cP^\col\wr{\Sigma_k})
	\end{tikzcd}
\]
commute in the sense that the following Beck--Chevalley transformations are equivalences
\[
	\varphi_!\iota_* \lra \iota_*\iota^*\varphi_!\iota_* \simeq \iota_*\varphi_!\iota^*\iota_* \simeq \iota_*\varphi_!\quad\text{and}\quad\varphi_!\nu^*\lra\nu^*\varphi_!,
\]
where the left one is induced by the (co)unit of $\iota^* \dashv \iota_*$ and the equivalence from \cref{lem:change-of-operads-restriction}.\end{lem}

\begin{proof}Throughout the proof, for a finite set $S\in\Fin_{\le k}$ we write $j_S\colon \Env(\cO)_S\hookrightarrow \Env(\cO)_{\le k}$ for the inclusion of the fibre of $\pi_{\cO}$ at $S$. If $|S|\le k-1$ we occasionally implicitly consider $j_S$ to have target $\Env(\cO)_{\le k-1}$ as opposed $\Env(\cO)_{\le k}$. Writing $\underline{k}\coloneqq \{1,\ldots,k\}$ we first consider the pullbacks 
\begin{equation}\label{equ:right-fib-pullback-naturality-layers}
	\begin{tikzcd}
	\cO^\col\wr{\Sigma_k} \dar{\varphi}\rar[hook]{\nu} & \Env(\cO)_{\le k}\dar{\varphi}\\
	\cP^\col\wr{\Sigma_k} \rar[hook]{\nu} & \Env(\cP)_{\le k}.
	\end{tikzcd}\quad\text{and}\quad
	\begin{tikzcd}
	\Env(\cO)_{\underline{k}}\dar{\varphi}\rar[hook]{j_{\underline{k}}} & \Env(\cO)_{\le k}\dar{\varphi}\\
	\Env(\cP)_{\underline{k}} \rar[hook]{j_{\underline{k}}} & \Env(\cP)_{\le k}.
	\end{tikzcd}
\end{equation}
whose vertical maps are right-fibrations as pullbacks of $\Env(\cO)\ra\Env(\cP)$ which is a right-fibration by assumption. Applying \cref{lem:right-fib-BC} to the left pullback square shows that the second transformation in the claim is an equivalence. Using that $\iota$ and $j_{\underline{k}}$ are jointly essentially surjective, in order to show that the first transformation is an equivalence, it suffices to prove that the two transformations obtained applying $\iota^*$ and $\smash{j_{\underline{k}}^*}$ to it, $\iota^*\varphi_!\iota_*\ra \iota^*\iota_*\iota^*\varphi_!\iota_*\simeq \iota^*\iota_*\varphi_!$ and $\smash{j_{\underline{k}}^*\varphi_!\iota_*\ra j_{\underline{k}}^*\iota_*\varphi_!}$, are equivalences. For the first, this follows from an application of the triangle identities and the fact that the counit of $\iota^*\dashv \iota_*$ is an equivalence since $\iota$ is fully faithful. For the second, we first apply \cref{lem:right-fib-BC} to the right-hand pullback in \eqref{equ:right-fib-pullback-naturality-layers} to rewrite $\smash{j_{\underline{k}}^*\varphi_!\iota_*\ra j_{\underline{k}}^*\iota_*\varphi_!}$ as a transformation of the form $\smash{\varphi_!j_{\underline{k}}^*\iota_*\ra j_{\underline{k}}^*\iota_*\varphi_!}$. Then we use \cref{lem:simplified-right-kan} to write $\smash{j_{\underline{k}}^*\iota_*}$ as a limit of functors $\smash{j_{\underline{k}}^*\iota_*\simeq \lim_{S\subsetneq \ul{k}}(\gamma_S^*\iota^*_S)}$ where $\smash{\gamma_{S}\colon \Env(\cO)_{\ul{k}}\ra\Env(\cO)_{S}}$ is the functor which on objects sends $\smash{(c_s)_{s\in \underline{k}}}$ to $\smash{(c_s)_{s\in S}}$ (from \eqref{equ:mapping-space-env} one sees that $\smash{\Env(\cO)_{\ul{k}}}\simeq (\cO^\col)^{\underline{k}}$ and this functor is given by projection onto the factors corresponding to $S$). Combining this with \cref{lem:left-kan-limit} and the fact that the poset of proper subsets of $\underline{k}$ has an initial object, we are left to show that the natural transformations $\varphi_!\gamma_S^*\iota^*_S \ra \gamma_S^*\iota^*_S\varphi_!$ are equivalences for all $S\subsetneq \underline{k}$. This follows from an application of \cref{lem:right-fib-BC} to the pasted pullback
\[
	\begin{tikzcd}
	\Env(\cO)_{\ul{k}}\dar{\varphi}\arrow[r,"\gamma_{S}"]&[5pt] \Env(\cO)_{S}\dar{\varphi}\rar["j_{S}",r,hook]&\Env(\cO)_{\le k-1}\dar{\varphi}\\
	\Env(\cP)_{\ul{k}}\arrow[r,"\gamma_{S}"]&\Env(\cP)_{S}\rar["j_{S}",r,hook]&\Env(\cP)_{\le k-1}
	\end{tikzcd}
\]
whose vertical maps are right-fibrations since they are pulled back from $\varphi\colon \Env(\cO)\ra\Env(\cP)$.
\end{proof}

\begin{proof}[Proof of \cref{thm:naturality-of-layers} ]
Assigning an operad $\cO$ to the inclusion $\iota\colon \Env(\cO)_{\le k-1}\subset \Env(\cO)_{\le k}$ gives a functor $\smash{\Opd^{\un} \ra\Cat^{[1]}}$. Its restriction to $\Opd^{\rf}$ lands in the subcategory $\smash{\Cat^{\subset_\BC,\dec}}$ from \cref{sec:abstract-layer-naturality} since for an operadic right-fibration $\varphi$ the relevant Beck--Chevalley transformations $\varphi_!\iota^*\ra \iota^*\varphi_!$, $\varphi_!\iota_*\ra \iota_*\varphi_!$ and $\varphi_!\nu^*\ra \nu^*\varphi_!$ are equivalences: the first by \cref{lem:change-of-operads-restriction} and the second and third by \cref{lem:auxiliarly-naturality-layers}. Postcomposing the resulting functor $\smash{\Opd^{\rf} \ra\Cat^{\subset_\BC,\dec}}$ with the second functor $\Cat^{\subset_\BC,\dec}\ra \Cat^{[1]\times [1]}$ from \cref{thm:recollement-nat} extends the squares \eqref{equ:recol-discs} to a functor $\Opd^{\rf}\ra  \Cat^{[1]\times [1]}$. By construction its restriction to the right vertical maps in \eqref{equ:recol-discs} is the composition \vspace{-0.2cm}
\[\Opd^{\rf}\xra{(-)^\col\wr\Sigma_k} \Cat\xra{\PSh(-)}\Cat\xra{(-)^{[2]}\xlra{(0\le 2)^*}(-)^{{[1]}}}\Cat^{[1]},\] so we can paste it with the composition \vspace{-0.2cm}
\[
	\Opd^{\rf}\xra{(-)^\col\wr\Sigma_k} \Cat\simeq\Cat_{/*} \xra{\PSh(-)}\Cat_{/\cS} \xra{(-)^{[2]} \xlra{(0\le 2)^*}(-)^{[1]}}\Cat^{[1]}_{/(\cS^{[2]}\ra\cS^{[1]})} \xra{\text{forget}}\Cat^{[1]\times [1]}
\]
to obtain a functor $\Opd^{\rf}\ra  \Cat^{[1]\times [1]}$ that extends the squares \eqref{equ:recol-discs-simpler}.
\end{proof}

\subsection{Smoothing theory for right-modules}\label{sec:smoothing-theory}The natural identification of the layers of $\RMod_\bullet(\cO)$ in Theorems \ref{thm:layers-tower} and \ref{thm:naturality-of-layers} implies in particular that for any operadic right-fibration $\varphi\colon \cO\ra \cP$, the map of towers $\RMod_\bullet(\cO)\ra\RMod_\bullet(\cP)$ resulting from \cref{thm:naturality} consists of pullbacks:

\begin{prop}\label{prop:smoothing-theory-rmod-step}For an operadic right-fibration $\varphi\colon \cO\ra\cP$, the square in $\CMon(\Cat)$
\[
	\begin{tikzcd}
	\RMod_{k}(\cO)\rar\dar[swap]{\iota^*}\rar{\varphi_!}&\RMod_{k}(\cP)\dar{\iota^*}\\
	\RMod_{\ell}(\cO)\rar{\varphi_!}&\RMod_{\ell}(\cP).
	\end{tikzcd}
\]
is cartesian for any $1\le \ell\le k\le \infty$. The same applies for the subcategories of unital right-modules. 
\end{prop}
\begin{proof}By taking limits and pasting, we can assume $k<\infty$ and $\ell=k-1$. Since the forgetful functor $\CMon(\Cat)\ra\Cat$ detects pullbacks, it suffices to prove that the square in the claim is a pullback of categories. The functoriality of \eqref{equ:recol-discs} established in \cref{thm:naturality-of-layers} yields a commutative cube
\[
	\begin{tikzcd}[row sep=0.1cm, column sep=0.3cm]
	&\RMod_k(\cP)\arrow[dd]\arrow[rr,"\Lambda",pos=0.75]&&[20pt]\PSh(\cP^\col\wr{\Sigma_k})^{[2]}\arrow[dd]\\
	\RMod_k(\cO)\arrow[dd]\arrow[ur]\arrow[rr,crossing over,"\Lambda",pos=0.75]&&\PSh(\cO^\col\wr{\Sigma_k})^{[2]}\arrow[ur]&\\
	&\RMod_{k-1}(\cP)\arrow[rr,"\Omega",pos=0.75]&&\PSh(\cP^\col\wr{\Sigma_k})^{[1]}\\
	\RMod_{k-1}(\cO)\arrow[ur]\arrow[rr,"\Omega",pos=0.75]&&\PSh(\cO^\col\wr{\Sigma_k})^{[1]} \arrow[ur]\arrow[uu,leftarrow,crossing over]&
	\end{tikzcd}
\]
whose diagonal arrows are induced by left Kan extension along $\varphi$, the two leftmost vertical arrows by restriction along $\iota$, and the two rightmost vertical arrows by restriction along $(0\le 2)\colon [1]\ra[2]$. The front, the back, and the right face in this cubes are cartesian: the former two by \cref{thm:layers-tower} \ref{enum:layer-pullback} and the latter by an application of \cref{lem:pullback-composition-rfib-fact}, using that $\cO^\col\wr{\Sigma_k}\ra \cP^\col\wr{\Sigma_k}$ is a right-fibration because it is pulled back from the right-fibration $\Env(\cO)\ra  \Env(\cP)$ along $\Fin_{\le k}^\simeq \subset \Fin$. This implies that the left face in the cube is cartesian as well, as claimed.
\end{proof}

Setting $\ell=1$ and invoking \cref{thm:bottom-layer}, \cref{prop:smoothing-theory-rmod-step} implies:

\begin{thm}[Smoothing theory]\label{thm:smoothing-theory-rmod} For an operadic right-fibration $\varphi\colon \cO\ra\cP$, the square 
\[
	\begin{tikzcd}
	\RMod_{k}^\un(\cO)\rar\dar[swap]{\iota^*}\rar{\varphi_!}&\RMod_{k}^\un(\cP)\dar{\iota^*}\\
	\PSh(\cO^\col)\rar{\varphi_!}&\PSh(\cP^\col).
	\end{tikzcd}
\]
is a pullback in $\CMon(\Cat)$ for any  $1\le k\le \infty$. 
 \end{thm}
 
\begin{rem}\label{rem:why-smoothing-thy}We think of \cref{thm:smoothing-theory-rmod} as a ``smoothing theory'' for right-modules because it is for certain choices of $\cO$ (variants of the $E_d$-operad) closely related to a result in the theory of high-dimensional manifolds that goes under the same name; see \cref{sec:smoothing-embcalc}.
\end{rem}

\subsection{Extension to Morita categories}\label{sec:morita-extension}
We now extend the results from the previous subsection from the categories of (truncated) right-modules to their Morita double categories. 
\subsubsection{A tower of Morita categories}\label{sec:tower-morita-general}
Applying the functor $\ALG(-)$ from \eqref{equ:alg-functor-com} to the converging tower $\RMod_\bullet(\cO)\in\Tow(\CMon(\Cat)^{\cgr})$ from \cref{prop:tower} and its reduced variant from \cref{cor:tower-thms-unital} results in converging towers of symmetric monoidal Morita double categories 
\[
	\ModInf_\bullet(\cO)\coloneqq \ALG(\RMod_\bullet(\cO))\quad \text{and}\quad \ModInf^\un_\bullet(\cO)\coloneqq \ALG(\RMod^\un_\bullet(\cO)),
\]
related by a map $\ModInf^\un_\bullet(\cO)\ra \ModInf_\bullet(\cO)$ of towers (use \cref{lem:limits-of-morita} for the convergence). Moreover the naturality of $\RMod_\bullet(-)$ and $\RMod^\un_\bullet(-)$ of  \cref{thm:naturality} and \cref{cor:tower-thms-unital} extends these constructions to functors of the form \begin{equation}\label{equ:naturality-morita-operadic-tower}
 	\smash{\ModInf^\un_\bullet(-), \ModInf_\bullet(-)\colon \Opd^\un \lra \Tow(\CMon(\CCat(\Cat)))}
\end{equation} 
which are related by a natural transformation that on each entry of the tower is a levelwise full subcategory inclusion.
 
\medskip

\noindent Applying $\ALG(-)$ to the symmetric monoidal equivalence from \cref{thm:bottom-layer} and using the identification of the Morita category of a cocartesian symmetric monoidal category in terms of the double category of cospans from \cref{sec:cospan-alg} results in the following identification of $\ModInf_1^\un(\cO)$:

\begin{thm}[The bottom Morita layer] \label{thm:bottom-layer-morita} Restriction along $\cO^{\col}\subset \Env(\cO)$ induces an equivalence 
\[
	\ModInf_1^\un(\cO)\simeq \Cosp(\PSh(\cO^{\col}))
\]
of symmetric monoidal double categories.
\end{thm}

\noindent Before discussing the higher Morita layers to the next subsection, we observe that the smoothing theory for right-modules from \cref{thm:smoothing-theory-rmod} directly implies the following variant on the level of Morita categories, simply by applying $\ALG(-)$ and using \cref{lem:limits-of-morita}:

\begin{thm}[Morita Smoothing theory] \label{thm:smoothing-theory-morita} For an operadic right-fibration $\varphi\colon \cO\ra\cP$, the square
\[
	\begin{tikzcd}
	\ModInf_{k}^\un(\cO)\rar\dar{\iota^*}\rar{\varphi_!}&\ModInf_{k}^\un(\cP)\dar{\iota^*}\\
	\Cosp(\PSh(\cO^\col))\rar{\varphi_!}&\Cosp(\PSh(\cP^\col))
	\end{tikzcd}
\]
is a pullback of symmetric monoidal double categories for any $1\le k \le \infty$, so in particular induces pullbacks on the level of bimodule categories and on the level of mapping spaces of the latter.
\end{thm}

Instead of going from the $k$th stage immediately to the first, one may also consider the intermediate squares of symmetric monoidal double categories 
\begin{equation}\label{equ:smoothing-theory-morita-only-one-step}
	\begin{tikzcd}
	\ModInf_{k}^\un(\cO)\rar\dar{\iota^*}\rar{\varphi_!}&\ModInf_{k}^\un(\cP)\dar{\iota^*}\\
	\ModInf_{k-1}^\un(\cO)\rar{\varphi_!}&\ModInf_{k}^\un(\cP)
\end{tikzcd}
\end{equation}
which are pullbacks under the same assumptions of \cref{thm:smoothing-theory-morita}, by the gluing property for pullbacks. Note this version is also valid for the bigger categories of non-unital right-modules; we only passed to unital right-modules before to identify the first stages in terms of cospans.

\begin{ex}[Change of tangential structure]\label{sec:terminal-operadic}
Given a reduced operad $\cU$ and a map 
\[
	\begin{tikzcd}[row sep=0.3cm,column sep=0.3cm]
	A\arrow[rr,"\varphi"]\arrow[dr,"\theta",swap]&&B\arrow[dl,"\nu"]\\
	&\BAut(\cU)&
	\end{tikzcd}
\]
of tangential structures for $\cU$ in the sense of \cref{sec:tangential-structures-operads}, the induced map of operads $\varphi\colon \cU^\theta \ra \cU^{\nu}$ is an operadic right-fibration by \cref{lem:pullbacks-groupoid-type}. \cref{thm:smoothing-theory-morita} thus in particular implies that the tower $\ModInf^\un_\bullet(\cU^{\theta})$ is pulled back from $\ModInf^\un_{\bullet}(\cU^{\nu})\ra \Cosp(\cS_{/B})$ along the map $\Cosp(\cS_{/A})\ra \Cosp(\cS_{/B})$ between the first stages induced by $\varphi$; here we used the straightening--unstraightening equivalence $\PSh(X)\simeq \cS_{/X}$ for groupoids $X$. In particular, setting $\nu=\id_{\BAut(\cU)}$, we see that for any choice of $\theta$, the tower $\ModInf^\un_\bullet(\cU^{\theta})$ is pulled back from $\ModInf^\un_{\bullet}(\cU^{\id})\ra \Cosp(\cS_{/\BAut(\cU)})$.

\medskip

\noindent 
More generally, given a tower of tangential structures $\smash{\theta_\smallsquare \colon B_\smallsquare\ra \BAut(\cU_{\le \smallsquare})}$ as in \cref{sec:tangential-structures-operads}, we get a tower $\smash{\cU_{\le \smallsquare}^{\theta_\smallsquare}}$ of unital operads to which we can apply $\smash{\ModInf^\un_\bullet(-)}$ to obtain a tower of towers of symmetric monoidal double categories ${\ModInf^\un_\bullet(\cO^{\smallsquare}_{\le \smallsquare})}$  which is pulled back from the tower of towers $\smash{\ModInf^\un_\bullet(\cO^{\id_{\BAut(\cO_{\le \smallsquare})}}_{\le \smallsquare})}$ along the map of towers $\smash{\Cosp(\cS_{/B_\smallsquare})\ra \Cosp(\cS_{/\BAut(\cU_{\le \smallsquare})})}$ induced by the tower of tangential structures $\theta_\smallsquare$.
\end{ex}

\subsubsection{The higher Morita layers}\label{sec:higher-morita-layers}Fixing associative algebras $A,B \in \Ass(\RMod(\cO))$ viewed as objects of the double category $\ModInf(\cO)$, the tower of double categories $\ModInf_\bullet(\cO)$ induces on mapping categories between $A$ and $B$ (see Sections \ref{sec:monoids} and \ref{sec:morita}) a tower of categories of bimodules
\begin{equation}\label{equ:bimodule-tower-rmod}
	\ModInf_\bullet(\cO)_{A,B} = \BMod_{A_\bullet,B_\bullet}(\RMod_\bullet(\cO))
\end{equation}
where $A_k,B_k\in \Ass(\RMod_k(\cO))$ are the $k$-truncations of the $A$ and $B$ (to ease the notation we often omit the $k$-subscripts). For $A=B$ the monoidal unit, this recovers the tower $\RMod_\bullet(\cO)$ whose layers we described in \cref{thm:layers-tower} (see \cref{sec:bimod}). In this section we extend this description to nontrivial $A$ and $B$. To do so we proceed similarly to \cref{sec:layers} and consider functors
\begin{equation}\label{equ:module-layers-functors}
	\begin{tikzcd} 
	\ModInf_{k-1}(\cO)_{A,B} \rar[shift left=1.2ex]{\iota_{AB!}} \rar[shift left=-1.2ex,swap]{\iota_{AB*}}& \ModInf_{k}(\cO)_{A,B} \rar{\nu^* \ U_{AB}} \lar &[0.3cm] \PSh(\cO^\col\wr\Sigma_k)
	\end{tikzcd}
\end{equation}
where the left middle unlabelled arrow is the functor $\smash{\iota_{AB}^*\colon  \ModInf_{k}(\cO)_{A,B}\ra  \ModInf_{k-1}(\cO)_{A,B}}$ featuring in \eqref{equ:bimodule-tower-rmod}, $\iota_{AB!}$ and $\iota_{AB*}$ are its left and right adjoints (which exist by \cref{lem:left-adjoint-bimodules} and \cref{prop:tower} \ref{enum:day-truncated}), $\nu\colon \cO^\col\wr\Sigma_k\hookrightarrow \RMod_k(\cO)$ is as in \cref{sec:layers}, and $U_{AB}\colon \ModInf_{k}(\cO)_{A,B}\ra \RMod_k(\cO)$ is the forgetful functor. Analogously to the non-bimodule situation of \cref{sec:layers} there are functors
\begin{align*}
	\Lambda_{AB}&\coloneqq\big(\nu^*U^{\ }_{AB}\iota^{\ }_{AB!}\iota_{AB}^*\ra \nu^*U^{\ }_{AB}\ra \nu^*U^{\ }_{AB}\iota^{\ }_{AB*}\iota_{AB}^*\big)\colon \ModInf_{k}(\cO)_{A,B}\lra \PSh(\cO^\col\wr\Sigma_k)^{[2]} && \text{and} \\
	\Omega_{AB}&\coloneqq\big(\nu^*U^{\ }_{AB}\iota^{\ }_{AB!}\ra \nu^*U^{\ }_{AB}\iota^{\ }_{AB*}\big)\colon \ModInf_{k-1}(\cO)_{A,B}\lra \PSh(\cO^\col\wr\Sigma_k)^{[1]}
\end{align*}
that fit into a commutative square of categories
\begin{equation}\label{equ:recol-discs-modules}\begin{tikzcd}
	\ModInf_{k}(\cO)_{A,B} \rar{\Lambda_{AB}} \dar[swap]{\iota_{AB}^*} & \PSh(\cO^\col\wr\Sigma_k)^{[2]} \dar{(0\le 2)^*} \\
	\ModInf_{k-1}(\cO)_{A,B}  \rar{\Omega_{AB}} & \PSh(\cO^\col\wr\Sigma_k)^{[1]}.\end{tikzcd}
	\end{equation}
The construction of $\Lambda_{AB},\Omega_{AB}$ and the square goes as in \cref{sec:layers} by replacing the role of \eqref{equ:layers-functors} with \eqref{equ:module-layers-functors} (for $\Omega_{AB}$ this uses that $\iota_{AB*}$ is fully faithful which follows e.g.\,from \cref{lem:bimodule-tower-is-presheaf-tower} below).

\begin{thm}[The higher Morita layers]\label{thm:layers-tower-module}Fix $1<k<\infty$ and associative algebras $A,B\in \Ass(\RMod(\cO))$ whose underlying right-modules are unital in the sense of \cref{sec:unital-right mod}.
\begin{enumerate}
	\item\label{enum:layer-pullback-module} The square of categories \eqref{equ:recol-discs-modules} is cartesian.\item\label{enum:layer-pullback-gc-module} If $\cO$ is groupoid-coloured, then also the commutative square of categories
	\begin{equation}\label{equ:recol-discs-simpler-module}
		\begin{tikzcd}
		\ModInf_{k}(\cO)_{A,B}\rar{\Lambda^c_{AB}} \dar[swap]{\iota_{AB}^*} 	&\cS^{[2]}\dar{(0\le 2)^*}\\
		\ModInf_{k-1}(\cO)_{A,B} \rar{\Omega^c_{AB}} &\cS^{[1]}
		\end{tikzcd}
	\end{equation}
	obtained by extending \eqref{equ:recol-discs-modules} to the right using functor $\colim\colon \PSh(\cO^\col\wr\Sigma_k)\ra\cS$ is cartesian.
	\item\label{enum:layer-pullback-red-module} Both \ref{enum:layer-pullback} and \ref{enum:layer-pullback-gc} are also valid with $\ModInf^\un_{\bullet}(\cO)$ in place of $\ModInf_{\bullet}(\cO)$.
	\end{enumerate}
\end{thm}

\begin{rem}The condition in \cref{thm:layers-tower-module} that the underlying right-modules of of $A$ and $B$ are unital is not essential, but the statement becomes slightly more involved without that assumption: the category $\cO^\col \wr \Sigma_k$ needs to be replaced by the pullback $\smash{\BMod^\frep_{A,B}(\RMod(\cO))_{\leq k} \times_{(\Fin_{*_L\sqcup *_R})} \Fin_k^\simeq}$ that appears in the proof of \cref{thm:layers-tower-module} below.
\end{rem}

Before proving \cref{thm:layers-tower-module}, we describe the functors $\smash{(U_{AB}\iota^{\ }_{AB*})}$ and $\smash{(U_{AB}\iota^{\ }_{AB!})}$ featuring in the definition of $\smash{\Lambda_{AB}}$ and $\smash{\Omega_{AB}}$ in terms of the variants $\iota_*,\iota_!\colon \RMod_{k-1}(\cO)\ra  \RMod_{k}(\cO)$ of these functors in the non-bimodule situation of \cref{sec:layers}:

\begin{lem}\label{lem:match-latch-module-simpler}For $M\in\ModInf_{k-1}(\cO)_{A,B}$, there are natural equivalences 
\[
	(U_{AB}\  \iota^{\ }_{AB*})(M)\simeq (\iota_*\ U_{AB})(M)\quad\text{and}\quad(U^{\ }_{AB}\  \iota^{\ }_{AB!})(M)\simeq \big\lvert[p]\mapsto A_k\otimes\iota_!\big(A_{k-1}^{\otimes p}\otimes M\otimes B_{k-1}^{\otimes p}\big)\otimes B_k \big\rvert
\]
where the rightmost object is the realisation of the indicated simplicial object in $\RMod_k(\cO)$ which is induced from the monadic bar construction $\Barc_{U_{AB}F_{AB}}(U_{AB}F_{AB},U_{AB}(M))_\bullet$ of \cref{sec:free-bimodules}.
\end{lem}

\begin{proof}
The first equivalence follows from applying right adjoints to the equivalence $F_{AB}\iota^*\simeq \iota_{AB}^* F_{AB}$ of functors $\RMod_k(\cO)\ra \ModInf_{k-1}(\cO)_{A,B}$ resulting from  \cref{lem:monoidal-functors-preserve-free-bimodules}. For the second equivalence, we write $M$ as the realisation of the monadic bar construction $\Barc_{U_{AB}F_{AB}}(U_{AB}F_{AB},U_{AB}(M))_\bullet\in\Fun(\Delta^{\op},\ModInf_{k-1}(\cO)_{A,B})$ with $p$-simplices $(F_{AB}U_{AB})^{p+1}(M)$. Since $\smash{U_{AB}\ \iota_{AB!}}$ preserves colimits as a composition of left adjoints, $\smash{U_{AB}\ \iota_{AB!}}(M)$ is the realisation of $\smash{U_{AB}\ \iota_{AB!}}(\Barc_{U_{AB}F_{AB}}(U_{AB}F_{AB},U_{AB}(M))_\bullet)$. Combining the equivalence $\smash{\iota_{AB!}F_{AB}\simeq F_{AB}\iota_!}$ resulting from taking left adjoints in $\smash{U_{AB}\ \iota^*_{AB}\simeq \iota^*\ U_{AB}}$ with the description of the monadic bar construction after applying $U_{AB}$ from \cref{sec:free-bimodules}, we see that this simplicial object has $p$-simplices as in the claim: $U_{AB}\ \iota_{AB!}(F_{AB}U_{AB})^{p+1}(M)\simeq U_{AB}F_{AB}\iota_!\ U_{AB}(F_{AB}U_{AB})^{p}(M)\simeq A_k\otimes\iota_!(A_{k-1}^{\otimes p}\otimes M\otimes B_{k-1}^{\otimes p})\otimes B_k$.
\end{proof}

As for \cref{thm:layers-tower}, the proof of \cref{thm:layers-tower-module} is an application of the decomposition result for presheaf categories \cref{thm:recollement}. To apply it, we rewrite \eqref{equ:bimodule-tower-rmod} in terms of the full subcategory $\smash{\BMod^{\frep}_{A,B}(\RMod(\cO))\subset \BMod_{A,B}(\RMod(\cO))}$ from \cref{sec:modules-day}: the essential image of the composition
\[
	\smash{\Env(\cO)\xlhra{y}\RMod(\cO)\xlra{F_{AB}}\BMod_{A,B}(\RMod(\cO)).}
\] 
Restricting to the essential images of the subcategories $\Env(\cO)_{\le k}\subset \Env(\cO)$ for varying $k$ yields a tower $\smash{\BMod^{\frep}_{A,B}(\RMod(\cO))_{\le \bullet}}$ of full subcategory inclusions which in turn induces a tower $\smash{\PSh(\BMod^{\frep}_{A,B}(\RMod(\cO))_{\le \bullet})}$ of presheaf categories by restriction. The latter is equivalent to \eqref{equ:bimodule-tower-rmod}:

\begin{lem}\label{lem:bimodule-tower-is-presheaf-tower}
There is an equivalence of towers of categories \begin{equation}\label{equ:tower-of-bimodules-is-presheaf-tower}
	\ModInf_{\bullet}(\cO)_{A,B}\simeq \PSh(\BMod^{\frep}_{A,B}(\RMod(\cO))_{\le \bullet})
\end{equation}
which for $\bullet=\infty$ agrees with the equivalence from \cref{lem:modules-as-presheaves}.
\end{lem}

\begin{proof}
Both towers are contained in the wide subcategory $\Cat^R$ of $\Cat$ on those functors that are right adjoints (for the left one this uses \cref{lem:left-adjoint-bimodules} and for the right one the left adjoints are given by left Kan extension), so in view of the equivalence $\Cat^R\simeq \Cat^{L,\op}$ induced by taking adjoints, we may equivalently show that there is an equivalence between the towers of left adjoints
\begin{equation}\label{equ:left-adjoint-module-tower}
	\BMod_{A_\bullet,B_\bullet}(\RMod_\bullet(\cO))^L\simeq \PSh(\BMod^{\frep}_{A,B}(\RMod(\cO))_{\le \bullet})^L
\end{equation}
in $\Cat^{L,\op}$. Now consider the following diagram of categories for $1\le k\le \ell\le \infty$
\begin{equation}\label{equ:diagram-free-inclusion-module}
	\begin{tikzcd}
	\Env(\cO)_{\le k}\dar{\subset}\rar{y}&\RMod_k(\cO)\dar{\iota_!}\rar{F_{A_kB_k}}&[5pt]\BMod_{A_k,B_k}(\RMod_k(\cO))\dar\\
	\Env(\cO)_{\le \ell}\rar{y}&\RMod_\ell(\cO)\rar{F_{A_\ell B_\ell}}&\BMod_{A_\ell,B_\ell}(\RMod_\ell(\cO)),
	\end{tikzcd}
\end{equation}
whose rightmost vertical map is the map in the left-hand tower of \eqref{equ:left-adjoint-module-tower}. Both of the squares in \eqref{equ:diagram-free-inclusion-module} commute: the left one by the universal property of $y$ and the right one because the corresponding square of right adjoints commutes. The essential image of the rows are by definition the subcategories indicated by a $\frep$-subscript, so by commutativity the left-hand tower of \eqref{equ:left-adjoint-module-tower} preserves these subcategories. By taking colimit preserving extensions we thus obtain a map \begin{equation}\label{equ:bimods-as-presheaves-for-rmods}
	\PSh(\BMod_{A_\bullet,B_\bullet}^{\frep}(\RMod_\bullet(\cO))^L)\xlra{\simeq}\BMod_{A_\bullet,B_\bullet}(\RMod_\bullet(\cO))^L
\end{equation} 
of towers in $\Tow(\Cat^{L,\op})$ which is an equivalence by \cref{lem:modules-as-presheaves}. For \eqref{equ:left-adjoint-module-tower}, it thus suffices to show 
\begin{equation}\label{equ:subcat-module-tower}
	\BMod_{A_\bullet,B_\bullet}^{\frep}(\RMod_\bullet(\cO))^L\simeq \BMod^{\frep}_{A,B}(\RMod(\cO))_{\le \bullet}
\end{equation}
Since the value of at $\bullet=\infty$ of the towers in \eqref{equ:subcat-module-tower} agree (by definition) and all maps in both towers are fully faithful (for the right one by definition and for the left one as a result of \cref{lem:left-adjoint-bimodules}), we may view both towers as towers of full subcategories of $\smash{\BMod_{A,B}^{\frep}(\RMod(\cO))}$. It thus suffices to show that the essential image of the map from $k$ to $\infty$ in the left tower agrees with the value at $k$ of the right tower. This follows from commutativity of  \eqref{equ:diagram-free-inclusion-module} for $k=\infty$ and $\ell=k$.
\end{proof}

\begin{proof}[Proof of \cref{thm:layers-tower-module}] Items \ref{enum:layer-pullback-gc-module} and \ref{enum:layer-pullback-red-module} follow from \ref{enum:layer-pullback-module} as in the proof of \cref{thm:layers-tower-module}. In view of \cref{lem:bimodule-tower-is-presheaf-tower} and the uniqueness of adjoints, to deduce \ref{enum:layer-pullback-module} from \cref{thm:recollement} it suffices to show
\begin{enumerate}[label=(\alph*),leftmargin=*]
	\item\label{enum:recollement-module-a} $\BMod^{\frep}_{A,B}(\RMod(\cO))_{\le k-1}\subset \BMod^{\frep}_{A,B}(\RMod(\cO))_{\le k}$ is a decomposition pair,
	\item\label{enum:recollement-module-b} there is a commutative square in $\Cat$ with vertical equivalences
	\vspace{-0.15cm} 
	\[\hspace{1cm}
		\begin{tikzcd}[ar symbol/.style = {draw=none,"\textstyle#1" description,sloped},
		equivalent/.style = {ar symbol={\simeq}}, row sep=0.4cm, column sep=0.6cm]
		\ModInf_{k}(\cO)_{A_{k},B_{k}} \rar{\nu^*\ U_{AB}}\arrow[d,equivalent,"\eqref{equ:tower-of-bimodules-is-presheaf-tower}",swap] & \PSh(\cO^\col\wr\Sigma_k)\arrow[d,equivalent] \\ \PSh(\BMod^{\frep}_{A,B}(\RMod(\cO))_{\le k})\rar{\text{res}}& \PSh(\BMod^{\frep}_{A,B}(\RMod(\cO))_{\le k}\backslash \BMod^{\frep}_{A,B}(\RMod(\cO))_{\le k-1}),
		\end{tikzcd}
	\]
	\item\label{enum:recollement-module-c} the assumption of \cref{thm:recollement} is satisfied for the composition pair in \ref{enum:recollement-module-a}.
\end{enumerate}
We first consider the commutative diagram of categories
\begin{equation}\label{equ:map-free-mods-to-finite-sets}
	\hspace{-0.3cm}
	\begin{tikzcd}[row sep=0.5cm, column sep=0.4cm]
	\Env(\cO)\arrow[r,hookrightarrow,"y"]\arrow[d,"\Env(t)","\pi_\cO"']&\RMod(\cO)\dar{t_!}\rar{F_{AB}}&\BMod_{A,B}(\RMod(\cO))\dar{t_!}\\
	\Fin=\Env(\Com)\arrow[r,hookrightarrow,"y"]\arrow[d,"F_{\ast_L\ast_R}","\ast_L\sqcup (-) \sqcup \ast_R"']&\RMod(\Com)\dar{F_{\ast_L\ast_R}}\rar{F_{t_!A,t_!B}}&\BMod_{t_!A,t_!B}(\RMod(\Com))\arrow[dl,"(\ast_L)\otimes_{t_!A}(-)\otimes_{t_!B}(\ast_R)", bend left=10]\\
	\Fin_{\ast_L\sqcup\ast_R/}=\BMod_{\ast_L,\ast_R}(\Fin)\arrow[r,"y",hookrightarrow]&\BMod_{\ast_L,\ast_R}(\RMod(\Com)).
	\end{tikzcd}
\end{equation}
Here $t\colon \cO\ra \Com$ is the map to the terminal operad and $\ast_L,\ast_R$ stand for copies of the terminal object in the respective categories. Commutativity of the top left, bottom left, and top right square is provided by the universal property and the symmetric monoidality of the Yoneda embedding (see \cref{sec:day}) as well as the symmetric monoidality of $t_!$ from \cref{lem:left-kan-monoidal}. Commutativity of the triangle is induced by the evident commutativity of the triangle of right adjoints, using that the diagonal functor is left adjoint to restricting bimodule structures along the algebra maps from $t_!A$ and $t_!B$ to the terminal algebras $*$ (combine 4.6.2.17, 4.6.3.11, and 4.3.2.8 loc.cit.). By commutativity, the composition of the two leftmost vertical functors induce a functor
\begin{equation}\label{equ:map-to-bipointedsets}
	\BMod^\frep_{A,B}(\RMod(\cO))\lra \BMod^\frep_{\ast_1,\ast_2}(\RMod(\Com))\simeq \Fin_{\ast_L\sqcup\ast_R}
\end{equation}
where $\Fin_{\ast_L\sqcup\ast_R}\subset \Fin_{\ast\sqcup\ast/}$ is the essential image of the leftmost bottom vertical functor (i.e.\,the full subcategory on the injective maps $\ast_L\sqcup\ast_R\ra S$); the equivalence in \eqref{equ:map-to-bipointedsets} is induced by the bottom left square of \eqref{equ:map-free-mods-to-finite-sets}. Unpacking the definition, one sees that \eqref{equ:map-to-bipointedsets} sends $F_{AB}(y((c_s)_{s\in S}))$ to $\ast_L\sqcup S\sqcup \ast_L$ on objects and is on morphisms with respect to the equivalence
\begin{align}
	\begin{split}\label{equ:free-bimodule-maps}&\Map_{\BMod_{A,B}(\RMod(\cO))}(F_{AB}(y((c_s)_{s\in S})),F_{AB}(y((d_t)_{t\in T}))) \\
	&\qquad \simeq \textstyle{\bigsqcup_{S=S_L\sqcup S_M\sqcup S_R}A((c_s)_{s\in S_L})\times \Map_{\Env(\cO)}((c_s)_{s\in S_M},(d_t)_{t\in T})\times B((c_s)_{s\in S_R})}
	\end{split}
\end{align}
induced by the adjunction $F_{AB}\dashv U_{AB}$, the Yoneda lemma, and \cref{lem:day-con-env} given as follows: $S_k$ is sent to $\ast_k$ for $k=L,R$ and $S_M$ is sent to $T$ via the map induced by applying $\pi_\cO\colon \Env(\cO)\ra\Fin$ to the middle factor. In particular, the full subcategory $\smash{\BMod^{\frep}_{A,B}(\RMod(\cO))_{\le k}}$ is the preimage under \eqref{equ:map-to-bipointedsets} of the full subcategory $\smash{\Fin_{\ast_L\sqcup \ast_R}^{\le k}\subset \Fin_{\ast_L\sqcup \ast_R}}$ on sets of the form $\ast_L\sqcup S\sqcup \ast_R$ with $|S|\le k$. With this in mind, the proof of Items \ref{enum:recollement-module-a} and \ref{enum:recollement-module-c} is a straight-forward adaption of the proof of the corresponding properties in \cref{thm:layers-tower} by replacing the role of $\pi_\cO\colon \Env(\cO)\ra\Fin$ in that proof with \eqref{equ:map-to-bipointedsets} and repeatedly using the adjunction $F_{AB}\dashv U_{AB}$. This leaves us with showing \ref{enum:recollement-module-b}. Checking the definition, the complement of the decomposition pair from \ref{enum:recollement-module-a} is the subcategory of $\BMod^\free_{A,B}(\RMod(\cO))$ with objects those $F_{AB}(y((c_s)_{s\in S})$ with $|S|=k$ and morphisms given by the summand in \eqref{equ:free-bimodule-maps} with $S_L=S_R=\varnothing$. Since $A(\varnothing)\simeq B(\varnothing)\simeq\ast$ as we assumed $A$ and $B$ to be unital, it follows that the upper row in \eqref{equ:map-free-mods-to-finite-sets} induces an equivalence of categories from $\cO^\col\wr\Sigma_k\simeq\Env(\cO)\times_{\Fin}\Fin_k^\simeq$ to the complement, so it suffices to show that restriction along this equivalence makes the diagram in \ref{enum:recollement-module-b} commute. Writing $\cC$ for the complement, we consider the commutative diagram of categories
\[
	\begin{tikzcd}
	\cO^{\col}\wr\Sigma_k\arrow[d,"F_{AB}\circ y_{\Env(\cO)}\circ \nu"',"\simeq"]\rar{\nu}&\Env(\cO)_{\le k}\rar{F_{A_kB_k}\circ y_{\Env(\cO)_{\le k}}}\dar[swap]{F_{AB}\circ y_{\Env(\cO)}}&[1cm]\BMod^\frep_{A_k,B_k}(\RMod_k(\cO))\arrow[d,"\iota_{AB!}","\simeq"']\\
	\cC\arrow[r,"\inc",hookrightarrow]&\BMod^\frep_{A,B}(\RMod(\cO))_{\le k}\arrow[r,equal]&\BMod^\frep_{A,B}(\RMod(\cO))_{\le k}
	\end{tikzcd}
\]
whose right-hand square involving the equivalence from \eqref{equ:subcat-module-tower}. Applying $\PSh(-)\colon \Cat^\op\ra\Cat$ to this, the left vertical equivalence and the bottom composition become the respective maps in the diagram in \ref{enum:recollement-module-b}. The right vertical equivalence differs from the respective equivalence in  \ref{enum:recollement-module-b} by precomposition with the inverse of \eqref{equ:bimods-as-presheaves-for-rmods} which is by \cref{lem:modules-as-presheaves} given by the restricted Yoneda embedding. We are thus left to identify the two functors $\smash{\BMod_{A_k,B_k}(\RMod_k(\cO))\ra \PSh(\cO^\col\wr\Sigma_k)}$ given by $\smash{(\nu^*\circ y^*\circ F_{A_kB_k}^*\circ y)}$ and $\smash{(\nu^*\circ U_{A_kB_k})}$. But $(y^*\circ F_{A_kB_k}^*\circ y)\simeq U_{A_kB_k}$ by the Yoneda lemma and the adjunction $F_{A_kB_k}\dashv U_{A_kB_k}$, so we are done.
\end{proof}

\subsubsection{Naturality of higher Morita layers}\label{sec:naturality-higher-morita} We now extend the naturality of the layers from \cref{thm:naturality-of-layers} to the level of bimodule categories by making the squares \eqref{equ:recol-discs-modules} and \eqref{equ:recol-discs-simpler-module} functorial in maps of triple $(\cO,A,B)\ra (\cP,A',B')$ given by a map of unital operads $\varphi\colon \cO\ra \cP$ and equivalences of algebras $\varphi_!(A)\simeq A'$ and $\varphi_!(B)\simeq A'$ in $\RMod(\cP)$. To make this precise, we consider the composition
\begin{equation}\label{equ:preliminary-functor-nat-bimodule-cats}
	\Opd^\un \xrightarrow[\ref{cor:tower-thms-unital}]{\RMod^\un(-)} \CMon(\Cat)^\cgr \xrightarrow{\Ass(-)^{\times 2}} \Cat
\end{equation}
from which we obtain a functor $\smash{(\Ass(\RMod^\un(-))^\simeq)^{\times 2}\colon \Opd^\rf\ra \Cat}$ by restricting to the subcategory $\smash{\Opd^\rf\subset \Opd^{\un}}$ and passing to the subfunctor taking value in the cores of the values of \eqref{equ:preliminary-functor-nat-bimodule-cats}. The cocartesian unstraightening of this functor is the category of triples $(\cO,A,B)$ on which we make the squares from \cref{thm:naturality-of-layers} functorial:

\begin{thm}\label{thm:naturality-of-layers-module}The squares \eqref{equ:recol-discs-modules} and \eqref{equ:recol-discs-simpler-module} in \cref{thm:layers-tower-module} extend to functors of the form
\[
	\smash{\textstyle{\int_{\Opd^\rf} (\Ass(\RMod^\un(-))^\simeq)^{\times 2} \lra \Cat^{[1] \times [1]}}}
\] 
which are on morphisms induced by taking left Kan extensions.
\end{thm}

\begin{proof}
The extension of the square \eqref{equ:recol-discs-simpler-module} to a functor follows from that of the square \eqref{equ:recol-discs-modules} as in the proof of \cref{thm:naturality-of-layers}, so we may focus on the square \eqref{equ:recol-discs-modules}. Recall from the proof of \cref{thm:layers-tower-module} that this square is the special case of the square \eqref{eqn:recollement-square} for the decomposition pair \begin{equation}\label{equ:decomp-pair-naturality-bimodule}
	\smash{\BMod_{A,B}^\frep(\RMod(\cO))_{\le k-1}\subset \BMod_{A,B}^\frep(\RMod(\cO))_{\le k}}
\end{equation} 
in the sense of \cref{dfn:decom-pair}, so using the second part of \cref{thm:recollement-nat}, it suffices to upgrade this composition pair to a functor $\smash{\textstyle{\int_{\Opd^\rf} (\Ass(\RMod^\un(-))^\simeq)^{\times 2} \ra \Cat^{\subset_{\BC},\dec}}}$ which is on morphisms induced by left Kan extension. For this we consider the composition
\begin{equation}\label{equ:main-functor-bimodule-naturality}
	\textstyle{\int_{\Opd^\un} \Ass(\RMod^\un(-))^{\times 2} \lra \int_{\Mon(\Cat)^\cgr} \Ass(-)^{\times 2} \xrightarrow[\eqref{equ:bimodule-naturality-functor}]{\BMod_{-,-}(-)} \Cat}
\end{equation}
which sends a triple $(\cO,A,B)$ to $\BMod_{A,B}(\RMod(\cO)) \simeq \ModInf(\cO)_{A,B}$ and is on morphisms given by sending a morphism $(\varphi\colon \cO\ra \cP,  \varphi_!(A)\ra A', \varphi_!(B)\ra B')$ in the source of \eqref{equ:main-functor-bimodule-naturality} to the top row in \vspace{-0.15cm}
\[
	\begin{tikzcd}[row sep=0.5cm]
	\BMod_{A,B}(\RMod(\cO))\rar{\varphi_!}&\BMod_{\varphi_!(A),\varphi_!(B)}(\RMod(\cP))\rar{A'\otimes_{ \varphi_!(A)}(-)\otimes_{\varphi_!(B)}B'}&[1.5cm]\BMod_{A',B'}(\RMod(\cP))\\
	\RMod(\cO)\rar{\varphi_!}\uar{F_{AB}}&\RMod(\cP)\uar{F_{\varphi_!(A)\varphi_!(B)}}\arrow[ur,"{F_{A'B'}}", bend right=10]&\\
	\Env(\cO)\rar{\varphi}\uar{y}&\Env(\cP)\uar{y}&
	\end{tikzcd}
\]
This diagram commutes: the upper left square by \cref{lem:monoidal-functors-preserve-free-bimodules}, the lower left one by the naturality of the Yoneda embedding with respect to left Kan extension, and  the upper right triangle by the same argument as for the commutativity of the triangle in \eqref{equ:map-free-mods-to-finite-sets}. The essential images of the maps from the bottom row to the top row are by definition the $\frep$-subcategories, so we see that the functor \eqref{equ:main-functor-bimodule-naturality} restricts to a functor $\BMod_{-,-}^\frep(\RMod(-))$ with values $\BMod_{A,B}^\frep(\RMod(\cO))$. Using the notation from the proof of \cref{thm:layers-tower-module}, the source of \eqref{equ:main-functor-bimodule-naturality} has a terminal object given by $(\Com,\ast_L,\ast_R)$, so since $\smash{\BMod^{\frep}_{\ast_L,\ast_R}(\RMod(\Com))\simeq \Fin_{\ast_L\sqcup\ast_R}}$, the functor $\smash{\BMod_{-,-}^\frep(\RMod(-))}$ has a preferred lift to a functor with values in $\smash{\Cat_{/\Fin_{\ast_L\sqcup\ast_R}}}$. Now recall from the proof of \cref{thm:layers-tower-module} that the subcategories \eqref{equ:decomp-pair-naturality-bimodule} as well as their complements are given by preimages of certain fixed subcategories of $\Fin_{\ast_L\sqcup\ast_R}$ that are independent of the operad and the algebras, so the functor $\BMod_{-,-}^\frep(\RMod(-))$ preserves \eqref{equ:decomp-pair-naturality-bimodule} and their complement, and thus in particular induces a functor $\smash{\int_{\Opd^\un} (\Ass(\RMod^\un(-))^\simeq)^{\times 2} \ra \Cat^{[1]}}$ whose values are the pairs \eqref{equ:decomp-pair-naturality-bimodule} and which is on morphisms induced by left Kan extension. It thus suffices to justify that this functor lands in the subcategory $\Cat^{\subset_{\BC},\dec}\subset \Cat^{[1]}$ once we restrict to the subcategory $\smash{\int_{\Opd^\rf} (\Ass(\RMod^\un(-))^\simeq)^{\times 2}\subset \int_{\Opd^\un} \Ass(\RMod^\un(-))^{\times 2}}$. This can be done by adapting the proof of \cref{lem:auxiliarly-naturality-layers} in the case $A=B=\varnothing$, by replacing the role of $\Fin$ with $\Fin_{\ast_L\sqcup \ast_R}$ and using that the functor $\smash{\varphi_!\colon \BMod^\frep_{A,B}(\RMod(\cO))\ra \BMod^\frep_{\varphi_!(A),\varphi_!(B)}(\RMod(\cP))}$ is a right-fibration if $\varphi$ is an operadic right-fibration, since it is a cartesian fibration whose fibres are groupoids, which one shows by directly checking the definition using \eqref{equ:free-bimodule-maps}.
\end{proof}

\subsection{Proof of \cref{bigthm:calc}} \label{sec:proof-bigthm} \cref{bigthm:calc} is merely a summary of some of the results in the previous subsections: the tower \eqref{equ:tower-intro-un} results from \cref{lem:unital-modules-properties} \ref{enum:unital-i}, the claimed convergence, monoidality, and naturality properties follow from \cref{prop:tower}, \cref{thm:naturality}, and \cref{cor:tower-thms-unital}, the part regarding the first layer is \cref{thm:bottom-layer}, the claim on the higher layers is contained in Theorems \ref{thm:layers-tower} and \ref{thm:naturality-of-layers}, and the smoothing theory is \cref{thm:smoothing-theory-rmod}. The tower of Morita categories is obtained by applying the functor $(-)^{(\infty,2)}\colon \CCat(\Cat)\ra \Cat_{(\infty,2)}$ from \cref{sec:infty-two} to the tower of double categories $\ModInf^\un_\bullet(\cO)$ from \cref{sec:tower-morita-general}. On the level of double categories, extensions of the properties to the tower $\ModInf^\un_\bullet(\cO)$ of Morita categories have been established in \cref{sec:morita-extension} from which similar extensions on the level of $(\infty,2)$-categories follow by applying the functor $(-)^{(\infty,2)}$. For the smoothing theory pullback, this uses that cospan double categories lie in the subcategory $\Cat_{(\infty,2)}\subset \CCat(\Cat)$ by \cref{rem:remarks-on-cospans} \ref{enum:cospans-complete} and that $(-)^{(\infty,2)}$ preserves pullbacks in $\CCat(\Cat)$ whose lower right corner lie in $\Cat_{(\infty,2)}$ by property \ref{item:oo-2-pullbacks} in \cref{sec:infty-two}.  

\section{Variants of embedding calculus}\label{sec:emb-calc}
As outlined in the introduction, the tower $\RMod_\bullet(\cO)$ from \cref{sec:operadic-framework} and its Morita category refinement $\ModInf_\bullet(\cO)$ is for certain operads $\cO$ closely related to Goodwillie--Weiss embedding calculus \cite{WeissEmbeddings,WeissEmbeddingsErratum,GoodwillieWeiss}, suggests new variants of embedding calculus, and allows for various applications of our general results from \cref{sec:operadic-framework}. This section makes this precise. We will
\begin{enumerate}[leftmargin=*]
	\item[\ref{sec:ed-variants}] discuss the tower $\RMod_\bullet(E_d^\theta)$ for tangential structures $\theta\colon B\ra \BAut(E_d)$ for the $E_d$-operad,
	\item[\ref{sec:embcalc-without-boundary}] relate the tower $\RMod_\bullet(E_d^o)$ for a tangential structure $o\colon \BO(d)\ra\BAut(E_d)$ to Goodwillie--Weiss' calculus for smooth embeddings and use another tangential structure $t\colon \BTop(d)\ra\BAut(E_d)$ to introduce a version of embedding calculus for topological embeddings,
	\item[\ref{sec:relation-embcalc}] utilise the towers on Morita categories $\ModInf_\bullet(E_d^o)$ and $\ModInf_\bullet(E_d^t)$ to extend smooth and topological embedding calculus to the level of bordism categories, discuss manifold calculus,
	\item[\ref{sec:embcalc-props}] discuss some properties of embedding calculus,
	\item[\ref{sec:layers-embcalc}] identify the layers of embedding calculus in terms of frame bundles and configuration spaces,
	\item[\ref{sec:smoothing-embcalc}] deduce smoothing theory results for embedding calculus,
	\item[\ref{sec:particle}] introduce \emph{particle embedding calculus} and tangential structures for embedding calculus,
	\item[\ref{sec:tk-equivalences}] give conditions under which $T_k$-self-maps are equivalences,
	\item[\ref{sec:delooping}] prove a delooping result for embedding calculus, and
	\item[\ref{sec:positive-codim}] end with a discussion of embedding calculus in positive codimension.
\end{enumerate}

\subsection{The $E_d$-operad and its variants}\label{sec:ed-variants}As mentioned in the introduction, the operads $\cO$ for which we relate the tower $\ModInf_\bullet(\cO)$ to embedding calculus are instances of the $\theta$-framed $E_d$-operad $E_d^\theta$ for tangential structures $\theta\colon B\ra \BAut(E_d)$. Let us recall how these operads arise.

\subsubsection{The $\theta$-framed $E_d$-operad}\label{sec:ed-operads-general} Fix $d\ge0$. The operad $E_d$ \emph{of little $d$-discs} (also called the operad of \emph{of little $d$-cubes} or the \emph{$E_d$-operad}) is the operadic nerve \cite[2.1.1.27]{LurieHA} of the ordinary operad in the category of topological spaces with a single colour, spaces of operations the spaces $\Emb^\rec(S\times (-1,1)^d,(-1,1)^d)$ of rectilinear embeddings $S\times (-1,1)^d\hookrightarrow (-1,1)^d$ for finite sets $S$, and composition induced by the composition of embeddings, see e.g.\,5.1.0.4 loc.cit.. Since there is unique embedding $\varnothing\hookrightarrow (-1,1)^d$, the operad $E_d$ is unital, and since the space of rectilinear embeddings $(-1,1)^d\hookrightarrow (-1,1)^d$ is contractible, it is also reduced (in the sense of \cref{sec:tangential-structures-operads}). Given a tangential structure $\theta\colon B\ra \BAut(E_d)$, we can form the \emph{$\theta$-framed $E_d$-operad} (see  \cref{dfn:tangial-structures-opd})
\begin{equation}\label{equ:theta-framed-ed}
	\smash{E_d^{\theta}\coloneqq \colim\big(B\xsra{\theta}\BAut(E_d)\subset \Opd\big)\in\Opd^{\gc}}.
\end{equation}
This construction generalises several known variants of the $E_d$-operad, in particular the operads
\[\smash{E_d^{\oo}}\in\Opd\quad\text{and}\quad \smash{E_d^{\ot}}\in\Opd,\]
obtained by replacing the role of the spaces $\Emb^\rec(S\times (-1,1)^d,(-1,1)^d)$ of rectilinear embeddings in the above construction of the $E_d$-operad with the spaces $\Emb^\oo(S\times \bfR^d,\bfR^d)$ of all smooth embeddings (with the smooth topology) or with the spaces $\Emb^\ot(S\times \bfR^d,\bfR^d)$ of all topological embeddings (with the compact-open topology). To see that $E_d^{\oo}$ and $E_d^{\ot}$ are indeed special cases of \eqref{equ:theta-framed-ed}, note that these operads are unital, for the same reason as $E_d$, and that the forgetful maps $\smash{\Emb^\oo(S\times \bfR^d,\bfR^d)\ra \Emb^\ot(S\times \bfR^d,\bfR^d)}$ induce a map $\smash{E_d^{\oo}}\ra \smash{E_d^{\ot}}$ of operads which agrees upon applying $(-)^{\col}$ with the coherent nerve of the forgetful map of topological monoids $\smash{\Emb^{\oo}(\bfR^d,\bfR^d)\ra \Emb^{\ot}(\bfR^d,\bfR^d)}$. Since the inclusion maps 
\begin{equation}\label{equ:emb-is-top}
	\smash{\oO(d)\xra{\simeq} \Emb^{\oo}(\bfR^d,\bfR^d)}\quad\text{and}\quad \smash{\Top(d)\xra{\simeq} \Emb^{\ot}(\bfR^d,\bfR^d)}
\end{equation}
of the topological groups of orthogonal transformations (with the subspace topology of $\smash{\bfR^{d^2}}$) and the group of homeomorphisms of $\bfR^d$ (with the compact-open topology) are both weak equivalences (this follows from taking derivatives for the first and from the Kister--Mazur theorem \cite[Theorem 1]{Kister} for the second), this identifies the functor $(E_d^{\oo})^\col\ra (E_d^{\ot})^\col$ induced by $E_d^\oo\ra\E_d^\ot$ with the coherent nerve $\BO(d)\ra\BTop(d)$ of the forgetful map $\oO(d)\ra\Top(d)$ of topological groups. In particular $\smash{E_d^{\oo}}$ and $\smash{E_d^{\ot}}$ are groupoid-coloured. Since their reduced covers in the sense of \cref{sec:tangential-structures-operads} are equivalent to the $E_d$-operad and the map $\smash{E_d^{\oo}}\ra \smash{E_d^{\ot}}$ is an equivalence on reduced covers (combine the argument in the proof of \cite[5.4.2.9]{LurieHA} with the commutativity of \eqref{equ:functors-vs-generalised-operads}), it follows from the discussion in \cref{sec:tangential-structures-operads} that $E_d^{\ot}$ is the $\ot$-framed $E_d$-operad for a tangential structure 
\begin{equation}\label{equ:top-acts-on-ed}
	\ot\colon\BTop(d)\lra \BAut(E_d),
\end{equation}
and that $E_d^{\oo}$ is the $\oo$-framed $E_d$-operad for the tangential structure $\oo\colon\BO(d)\ra \BAut(E_d)$ obtained by precomposing \eqref{equ:top-acts-on-ed} with the forgetful map $\BO(d)\ra\BTop(d)$.

\begin{rem}\,
\begin{enumerate}[leftmargin=*]
	\item The operad $E_d^\ot$ also arises as the operadic nerve $N(-)$ of a semidirect product $\mathsf{E}_d\rtimes \Top(d)$ in the sense of \cref{rem:semidirect}, for a strict action of the topological group $\Top(d)$ on a certain model $\mathsf{E}_d$ of the $E_d$-operad as an ordinary operad in topological spaces (see below), and similarly $E_d^\oo$ is the operadic nerve of the semidirect product $\mathsf{E}_d\rtimes \oO(d)$ for the restriction of the action along $\oO(d)\ra\Top(d)$. Using \cref{thm:operad-grp-type} and \cref{prop:semidirect-is-tangential}, this in particular identifies the tangential structure \eqref{equ:top-acts-on-ed} as being induced by the strict $\Top(d)$-action on $\mathsf{E}_d$.  To construct the model $\mathsf{E}_d$ on which $\Top(d)$ acts, one considers for finite sets $S$ the map of topological spaces $\smash{\Emb^\ot(S\times\bfR^d,\bfR^d)\ra \bigsqcap_{s\in S}\Emb^\ot(\bfR^d,\bfR^d)}$ that forgets that embeddings are disjoint. Using a Moore paths construction, their homotopy fibres over the identities $(\id_{\bfR^d})_{s\in S}$ assemble to an ordinary operad $\mathsf{E}_d$ in topological spaces which is easily seen to be equivalent to $E_d$ (by using $(-1,1)^d\cong \bfR^d$ and that any rectilinear self-embedding $(-1,1)^d$ is canonically isotopic to the identity) and on which $\Top(d)$ acts strictly by conjugating embeddings. The composition maps $\Emb^\ot(S\times\bfR^d,\bfR^d)\times \sqcap_{s\in S}\Emb(\bfR^d,\bfR^d)\ra \Emb^\ot(S\times\bfR^d,\bfR^d)$ then induce an equivalence between $\mathsf{E}_d\rtimes \Top(d)$ and the topological operad described above whose operadic nerve is $E^t_d$, so one obtains an equivalence on operadic nerves as claimed.
	\item Any topological $d$-manifold $M$ gives rise to a tangential structure $\smash{\theta_M\colon M\ra\BAut(E_d)}$ by composing a classifier for the topological tangent bundle $M\ra\BTop(d)$ with the tangential structure $\smash{\ot\colon \BTop(d)\ra\BAut(E_d)}$. The associated $\smash{\theta_M}$-framed $E_d$-operad $\smash{E_d^{\theta_M}\in \Opd^{\gc}}$ is denoted $\mathbb{E}_M$ in \cite[5.4.5.1]{LurieHA}. It can be thought as an operad whose colours are points in $M$ and whose reduced cover at $m\in M$ is the operad of little $d$-discs in the tangent space of $M$ at $m$. Despite the notation, the operad $\mathbb{E}_M$ is \emph{very} different from the right module $E_M\in\RMod(E_d^{\ot})$ that we will considered later in this section; for instance we have $\smash{\Aut(E_d^{\theta_M})}\simeq \Aut_{\cS_{/\BAut(E_d)}}(M)$ by \eqref{equ:aut-gc-operad} whereas $\smash{\Aut_{\RMod(E_d^{\ot})}(E_M)}$ is often very close to the group $\smash{\Homeo(M)}$ of homeomorphism of $M$, by a combination of the main results of \cite{KKDisc} and \cref{thm:sdisc-smooth-vs-top} below.\end{enumerate}
\end{rem}

\subsubsection{Right-modules over the $\theta$-framed $E_d$-operad}\label{sec:rmod-framed-ed}Specialising the setting of \cref{sec:operadic-framework} to the unital operad $\smash{E_d^{\theta}}$ for a tangential structure $\theta\colon B\ra \BAut(E_d)$ gives a tower $\smash{\RMod_\bullet(E^{\theta}_d)}$ of symmetric monoidal categories of (truncated) right-modules over $\smash{E^{\theta}_d}$ which in turn induces a tower $\smash{\ModInf_\bullet(E^{\theta}_d)}$ of symmetric monoidal Morita double categories. For the two examples \eqref{equ:theta-framed-ed}, this can be described more explicitly. The map induced by $\smash{E_d^{\oo}\ra E_d^{\ot}}$ on symmetric monoidal envelopes is equivalent to the functor of symmetric monoidal categories $\DiscInf^{\oo}_d\ra\DiscInf^{\ot}_d$ obtained as the coherent nerve of the forgetful functor between the two topologically enriched symmetric monoidal categories with objects finite sets, morphism spaces the embedding spaces $\Emb^{\oo}(S\times\bfR^d,T\times \bfR^d)$ and $\Emb^{\ot}(S\times\bfR^d,T\times \bfR^d)$, respectively, and symmetric monoidal structures given by disjoint union (this follows by combining the construction of $\smash{E_d^{\oo}\ra E_d^{\ot}}$ from \cref{sec:ed-operads-general} with the explicit model for the monoidal envelope in \cite[2.2.4]{LurieHA}). In particular, we have an equivalence of towers of categories 
\begin{equation}\label{equ:rmod-as-discs}
	\RMod_{\bullet}(E^{c}_d)\simeq \PSh(\DiscInf^{c}_{d,\le \bullet})\quad\text{for } c\in\{\oo,\ot\}
\end{equation} 
where the right-hand side is the tower induced by restriction along the tower of inclusions of full subcategories $\DiscInf^{c}_{d,\le \bullet}$ of $\DiscInf^{c}_{d}$ obtained by restricting the cardinality of the occurring finite sets $S$. 

\subsection{Embedding calculus, without boundary conditions}\label{sec:embcalc-without-boundary} The categories $\DiscInf_d^{c}$ arise as full subcategories of the coherent nerves $\ManInf_d^{c}$ of the topologically enriched categories of smooth or topological (indicated by $c\in\{\oo,\ot\}$) $d$-manifolds without boundary $M,N,\ldots$ and spaces of smooth or topological embeddings $\Emb^{c}(M,N)$ between them (with the smooth or the compact-open topology, respectively). We may thus consider the diagram of categories
\begin{equation}\label{equ:emb-calc-cat-no-bdy}
	\begin{tikzcd}[row sep=0.55cm]
	\ManInf_d^{\oo}\dar{\delta}\rar{y}&\PSh(\ManInf_d^{\oo})\dar{\delta_!}\rar{\iota^*}&\PSh(\DiscInf_{d,\le \bullet}^{\oo})\dar{\delta_!}\simeq \RMod_\bullet(E_d^{\oo})\\
	\ManInf_d^{\ot}\rar{y}&\PSh(\ManInf_d^{\ot})\rar{\iota^*}&\PSh(\DiscInf_{d,\le \bullet}^{\ot})\simeq\RMod_\bullet(E_d^{\ot})
	\end{tikzcd}
\end{equation}
where $y$ is the respective Yoneda embedding, $\iota$ is the respective inclusion $\DiscInf_d^{c}\subset \ManInf_d^{c}$, and $\delta$ stands for the functors that forgets smoothness. We denote the top and bottom composition by
\begin{equation}\label{equ:e-functor-no-bordism-cat}
	E\colon \ManInf_d^{c}\lra \RMod_\bullet(E_d^{c})\quad\text{for }c\in\{\oo,\ot\}
\end{equation}
and their values at a $d$-manifold $M\in \ManInf_d^{c}$ by $E_M\coloneqq \Emb^{c}(-,M)\in \RMod_k(E_d^{c})$ for $1\le k\le \infty$.

\begin{lem}\label{lem:compatibility-o-to-t-underlying-cat}The diagram \eqref{equ:emb-calc-cat-no-bdy} lifts to a commutative diagram of categories.
\end{lem}

\begin{proof}Commutativity of the left square is provided by the naturality of the Yoneda embedding. The rightmost vertical map is a map of towers by the naturality of $\RMod_\bullet(-)$ from \cref{thm:naturality}. It thus suffices to lift the right-hand square for $\bullet=\infty$ to a commutative square which in turn follows from showing that the Beck--Chevalley transformation $\delta_!\iota^*\ra \iota^*\delta_!$ of functors $\PSh(\ManInf_d^{\oo})\ra \RMod(E_d^{\ot})$ is an equivalence. The class of presheaves for which it is an equivalence is closed under colimits, since all functors involved are left adjoints. Left Kan extension preserves representable presheaves, so the class contains $\Emb^{\oo}(-,S\times \bfR^d)\in \PSh(\ManInf_d^\oo)$ for $S\in\Fin$. It thus suffices to show that any $\smash{X\in  \PSh(\ManInf_d^\oo)}$ is a colimit of presheaves of this form. Any presheaf is a colimit of representables, so we may assume $X=\Emb^\oo(-,M)$ for $\smash{M\in \ManInf_d^\oo}$. But $\smash{\Emb(-,M)\simeq\colim_{D\in \cO_\infty(M)}\Emb(-,D)}$ where $\cO_{\infty}(M)$ is the poset of open subsets $D\subset M$ diffeomorphic to $S\times  \bfR^d$ for $S\in \Fin$, by a well-known descent argument (as e.g.\,in the proof of \cite[Lemma 6.4]{KnudsenKupers}), so the claim follows.
\end{proof}

For $M,N\in\ManInf^{\oo}_d$, the upper row of \eqref{equ:emb-calc-cat-no-bdy} induced on mapping spaces a tower of spaces
\begin{equation}\label{equ:emb-calc-no-bdy-o}
 	\smash{\Emb^{\oo}(M,N)\ra \Map_{\RMod_{\bullet}(E^{o}_d)}(E_M,E_N)\eqcolon T_\bullet\Emb^{\oo}(M,N)}
\end{equation}
under the space $\Emb^{\oo}(M,N)$ of smooth embeddings. It was shown in \cite{BoavidaWeiss} that this tower agrees with the embedding calculus Taylor tower for the space $\Emb^{\oo}(M,N)$ of smooth embeddings as introduced in \cite{WeissEmbeddings} (alternatively, this is a special case of \cref{prop:morita-gives-embcalc} below). Replacing $\oo$-subscripts with $\ot$ subscripts, we obtain an analogous tower
\begin{equation}\label{equ:emb-calc-no-bdy-t}
	\smash{\Emb^{\ot}(M,N)\ra \Map_{\RMod_{\bullet}(E^{\ot}_d)}(E_M,E_N)\eqcolon T_\bullet\Emb^{\ot}(M,N)}
\end{equation}
under the space $\Emb^{\ot}(M,N)$ of embeddings between topological $d$-manifolds $M,N\in\ManInf^{\ot}_d$.

\subsection{Extension to bordism categories}\label{sec:relation-embcalc}We will now explain that the towers $\ModInf_\bullet(E^{c}_d)$ of Morita categories for $c\in\{\oo,\ot\}$ allow for an enhancement of the functors giving rise to embedding calculus for manifolds without boundary \eqref{equ:e-functor-no-bordism-cat}, to the level of bordism categories, which in particular give generalisations of \eqref{equ:emb-calc-no-bdy-o} and \eqref{equ:emb-calc-no-bdy-t} for spaces of boundary-fixing embeddings between manifolds with boundary, as well as coherent gluing maps. As indicated in the introduction, this is an extension of a construction from our previous work \cite{KKDisc}, so we first recall relevant aspects of the latter.

\subsubsection{Recollection from \cite{KKDisc} and extension to topological manifolds}\label{sec:e-for-topological-mfds}
In \cite[Section 3-4]{KKDisc} we constructed a functor of symmetric monoidal double categories (our notation differs slightly from that in loc.cit.: we write $\ncBordInf^{\oo}(d)$ and $\ModInf(E^{\oo}_d)$ instead of $\ncBordInf(d)$ and $\icat{M}\mathrm{od}(d)$) 
\begin{equation}\label{equ:E-functor}
	\smash{E\colon \ncBordInf^{\oo}(d)\lra\ModInf(E^{\oo}_d)}
\end{equation} 
which can be informally summarised as follows: the source $\ncBordInf^\oo(d)$ is a noncompact smooth $d$-dimensional bordism double category: the category of objects consists of smooth (potentially noncompact) $(d-1)$-manifolds without boundary and smooth embeddings between them, the mapping categories consist of (potentially noncompact) smooth $d$-dimensional bordisms and smooth embeddings fixing the boundaries between them, the composition functors are induced by gluing bordisms, and the symmetric monoidal structure is by taking disjoint unions. The functor \eqref{equ:E-functor} assigns to a $(d-1)$-manifold $P\in \ncBordInf^{\oo}(d)$ the presheaf $E_{P\times I}\coloneqq\Emb(-,P\times I)\in \PSh(\DiscInf_d^{\oo})=\RMod(E_d^{\oo})$ for $I\coloneqq[0,1]$, equipped with the algebra structure with respect to Day convolution induced by ``stacking'' $P\times I$ to itself, and to a bordism $W\colon P\leadsto Q$ the presheaf $E_W\coloneqq\Emb(-,W)$ with the $(E_{P\times I},E_{Q\times I})$-bimodule structure induced by ``stacking'' $P\times I\sqcup Q\times I$ to $W$. 

\medskip

\noindent The construction of \eqref{equ:E-functor} also goes through for \emph{topological} manifolds. Since was not discussed in loc.cit., we shall make good for it now:

\begin{thm}\label{thm:extension-e-topological}There is a commutative square of symmetric monoidal double categories 
\[
	\begin{tikzcd}	\ncBordInf^{\oo}(d)\rar{E}\dar&\ModInf(E_d^{\oo})\dar\\
	\ncBordInf^{\ot}(d)\rar{E}&\ModInf(E_d^{\ot})
	\end{tikzcd}
\]
where
\begin{enumerate}[leftmargin=*]
	\item\label{enum:extension-topo-i} $\ncBordInf^{c}(d)$ for $c\in\{\oo,\ot\}$ is a bordism category of potentially noncompact smooth (respectively topological) $(d-1)$-manifolds and the left vertical map is induced by forgetting smooth structures,
	\item\label{enum:extension-topo-ii} the right vertical map is induced by the map $E_d^{\oo}\ra E_d^{\ot}$ of operads from \cref{sec:ed-operads-general}, and 
	\item\label{enum:extension-topo-iii} on mapping categories from $\varnothing$ to itself, the square agrees with the outer square in \eqref{equ:emb-calc-cat-no-bdy} for $\bullet=\infty$, 
	\item\label{enum:extension-topo-iv} the horizontal functors land in the levelwise full subcategories $\ModInf^{\un}(E_d^{c})\subset \ModInf(E_d^{c})$.
\end{enumerate}
\end{thm}

\begin{proof}
This follows mostly from a minor variant of the construction of \eqref{equ:E-functor} for smooth manifolds in \cite[Section 3]{KKDisc}. To explain this, recall from Section 3\,loc.cit.\,that \eqref{equ:E-functor} arises as follows: first one constructs a composition in $\CMon(\Fun(\Delta^{\op},\Cat))$
\begin{equation}\label{equ:construction-e-smooth}
	\smash{\ncBordInf^{\oo}(d)\xlra{E^\geo} \overline{\ALG}(\ManInf^{\oo}_d)\xlra{y} \overline{\ALG}(\PSh(\ManInf^{\oo}_d))\xlra{\iota^*} \overline{\ALG}(\PSh(\DiscInf^{\oo}_d)).}
\end{equation}
This involves the symmetric monoidal categories  $\DiscInf^{\oo}_d\subset \ManInf^{\oo}_d$ from above and the functor $\smash{\overline{\ALG}(-)\colon \CMon(\Cat)\ra \CMon(\Fun(\Delta^{\op},\Cat))}$
given by assigning $\cC\in  \CMon(\Cat)$ to its \emph{pre-Morita category} $\smash{\overline{\ALG}(\cC)\in  \CMon(\Fun(\Delta^{\op},\Cat))}$ which, if $\cC$ is compatible with geometric realisations, contains $\ALG(\cC)\in  \CMon(\CCat(\Cat))\subset  \CMon(\Fun(\Delta^{\op},\Cat))$ as a levelwise full subcategory (see 2.9 loc.cit.). The functor $\smash{E^\geo}$  is an explicit geometric construction, the second and third functor in \eqref{equ:construction-e-smooth} results from applying $\smash{\overline{\ALG}(-)}$ to the Yoneda embedding of $\ManInf^{\oo}_d$ and the restriction along $\DiscInf^{\oo}_d\subset \ManInf^{\oo}_d$. To obtain \eqref{equ:E-functor} from \eqref{equ:construction-e-smooth}, one shows that the latter lands in the levelwise full subcategory $\smash{\ModInf(E_d^\oo)\simeq \ALG(\PSh(\DiscInf^{\oo}_d))\subset\overline{\ALG}(\PSh(\DiscInf^{\oo}_d))}$. All this can be mimicked for topological manifolds (we add $\ot$-superscripts for the topological variants): as for $\DiscInf_d^{\ot}\subset \ManInf_d^{\ot}$ above, for the definition of $\ncBordInf^{\ot}(d)$ and the construction of the analogue of $\smash{E^\geo}$, one literally replaces all occurrences of smooth manifolds with topological manifolds in \cite[3.\circled{1}--\circled{6}]{KKDisc}, and all spaces of smooth embeddings between smooth manifolds with the corresponding spaces of topological embeddings. The analogue of the second and third functor in \eqref{equ:construction-e-smooth} is defined as in the smooth case, and the proof that the composition lands in $\smash{\ModInf(E_d^\ot)\simeq \ALG(\PSh(\DiscInf^{\ot}_d))\subset\overline{\ALG}(\PSh(\DiscInf^{\ot}_d))}$ can be copied almost verbatim, except that in the proof of Proposition 3.6 loc.cit.\,one has to replace the frame bundle $\Fr(M)$ with the topological frame bundle $\Fr^t(M)$ of germs of topological embeddings of $\bfR^d$ into $M$.

\medskip

\noindent To obtain the vertical maps and the commutativity, one constructs a commutative diagram \vspace{-0.1cm}
\[
	\begin{tikzcd}[row sep=0.25cm]\ncBordInf^{\oo}(d)\dar\rar{E^\geo}& 		\overline{\ALG}(\ManInf^{\oo}_d)\rar{y}\dar& \overline{\ALG}(\PSh(\ManInf^{\oo}_d))\rar{\iota^*}\dar& \overline{\ALG}(\PSh(\DiscInf^{\oo}_d))\dar\\
	\ncBordInf^{\ot}(d)\rar{E^\geo}& \overline{\ALG}(\ManInf^{\ot}_d)\rar{y}& \overline{\ALG}(\PSh(\ManInf^{\ot}_d))\rar{\iota^*}& \overline{\ALG}(\PSh(\DiscInf^{\ot}_d))
	\end{tikzcd}\vspace{-0.1cm}
\]
as follows: the first two vertical functors are given by forgetting smoothness and make the first square commute by construction. The second square is obtained from the functoriality of $\smash{\overline{\ALG}(-)}$ and the naturality of the symmetric monoidal Yoneda embedding \cite[4.8.1.12, 4.8.1.13]{LurieHA}. The final square is obtained by applying $\smash{\overline{\ALG}(-)}$ to a commutative square of symmetric monoidal categories whose underlying square of categories is the right-hand square in \eqref{equ:emb-calc-cat-no-bdy} for $\bullet=\infty$. Both $\iota$ and $\delta$ are symmetric monoidal, so by the discussion in \cref{sec:day} the vertical functors in the right-hand square in \eqref{equ:emb-calc-cat-no-bdy} for $\bullet=\infty$ are symmetric monoidal and the horizontal ones lax symmetric monoidal. But by the argument in \cite[Lemma 3.10]{KKDisc} they are actually (strong) monoidal, so since the forgetful functor $\CMon(\Cat)\ra\Cat$ is conservative, to lift \eqref{equ:emb-calc-cat-no-bdy} to a commutative square in $\CMon(\Cat)$ it suffices that the Beck--Chevalley transformation $\delta_!\iota^*\ra \iota^*\delta_!$ of underlying categories is an equivalence. We already proved this in \cref{lem:compatibility-o-to-t-underlying-cat}, so the construction of the square in the claim is finished. Items \ref{enum:extension-topo-i}-\ref{enum:extension-topo-ii} hold by construction and by using \eqref{equ:rmod-as-discs}. Item \ref{enum:extension-topo-iii} also follows from the construction, combined with the general fact that mapping categories in Morita categories are categories of bimodules and that the category of bimodules over the unit algebras is the underlying category (see \cref{sec:morita-functor}). Item \ref{enum:extension-topo-iv} follows from the description of the functor \eqref{equ:E-functor} together with the observation that $\Emb^{c}(\varnothing,M)=\ast$ for $c\in\{\oo,\ot\}$ and any manifold $M$.
\end{proof}

\subsubsection{Embedding calculus on bordism categories} \label{sec:sm-top-emb-calc} 
The map of towers $\ModInf_\bullet(E_d^{\oo})\ra \ModInf_\bullet(E_d^{\ot})$ induced by $E_d^{\oo}\ra E_d^{\ot}$ and \cref{thm:bottom-layer} agrees with the right vertical map in \cref{thm:extension-e-topological} for $\bullet=\infty$, so we obtain a commutative diagram of symmetric monoidal double categories\vspace{-0.1cm}
\[\hspace{-0.8cm}
	\begin{tikzcd}[sep=small,row sep=0.25cm]\ncBordInf^{\oo}(d)\rar{E}\dar& \ModInf^\un(E^{\oo}_d)=\ModInf_\infty^\un(E^{\oo}_d)\dar\rar&\cdots\rar& \ModInf_2^\un(E^{\oo}_d)\rar\dar& \ModInf_1^\un(E^{\oo}_d)\overset{\ref{thm:bottom-layer-morita}}{\simeq}\Cosp(\cS_{/\BO(d)})\dar\\
	\ncBordInf^{\ot}(d)\rar{E}& \ModInf^\un(E^{\ot}_d)=\ModInf^\un_\infty(E^{\ot}_d)\rar&\cdots\rar& \ModInf^\un_2(E^{\ot}_d)\rar& \ModInf^\un_1(E^{\ot}_d)\overset{\ref{thm:bottom-layer-morita}}{\simeq}\Cosp(\cS_{/\BTop(d)});
	\end{tikzcd}\vspace{-0.1cm}
\]
here the identification of the bottom stages uses \cref{thm:bottom-layer} together with the (unstraightening) equivalence $\PSh(\cG)\simeq \cS_{\cG}$ for groupoids $\cG$ and $(E_d^\oo)^\col\simeq\BO(d)$ as well as $(E_d^\ot)^\col\simeq\BTop(d)$ (see \cref{sec:ed-operads-general}). When transposing the diagram and passing to mapping spaces of mapping categories this induces for smooth $(d-1)$-manifolds $P$ and $Q$ without boundary and bordisms $M,N\colon P\leadsto Q$ between them (all potentially noncompact) a commutative diagram
\[
	\hspace{-0.3cm}
	\begin{tikzcd}[row sep=0.25cm, column sep=0.3cm, ar symbol/.style = {draw=none,"\textstyle#1" description,sloped},
	equivalent/.style = {ar symbol={\simeq}}]
	&[-0.55cm]&[-0.55cm]\Emb_\partial^\oo(M,N)\arrow[d,equivalent]\rar&\Emb_\partial^\ot(M,N)\arrow[d,equivalent]&[-0.55cm]&[-0.55cm]\\[-0.2cm]
	&&\Map_{\ncBordInf^{\oo}(d)_{P,Q}}(M,N)\rar\dar&\Map_{\ncBordInf^{\ot}(d)_{P,Q}}(M,N) 	\dar&&\\
	T_\infty\Emb^{\oo}_\partial(M,N)&\coloneqq&\Map_{\ModInf(E^{\oo}_d)_{E_{P\times I},E_{Q\times I}}}(E_M,E_{N})\rar\dar&\Map_{\ModInf(E^{\ot}_d)_{E_{P\times I},E_{Q\times I}}}(E_M,E_{N})\dar&\eqcolon&T_\infty\Emb^{\ot}_\partial(M,N)\\
	&&\vdots\dar&\vdots\dar&&\\
	T_2\Emb^{\oo}_\partial(M,N)&\coloneqq&\Map_{\ModInf_2(E^{\oo}_d)_{E_{P\times I},E_{Q\times I}}}(E_M,E_{N})\rar\dar&\Map_{\ModInf_2(E^{\ot}_d)_{E_{P\times I},E_{Q\times I}}}(E_M,E_{N})\dar&\eqcolon&T_2\Emb^{\ot}_\partial(M,N)\\
	T_1\Emb^{\oo}_\partial(M,N)&\coloneqq&\Map_{\ModInf_1(E^{\oo}_d)_{E_{P\times I},E_{Q\times I}}}(E_M,E_{N})\arrow[d,equivalent,swap,"\ref{prop:bottom-manifold-layer}"]\rar&\Map_{\ModInf_1(E^{\ot}_d)_{E_{P\times I},E_{Q\times I}}}(E_M,E_{N})\arrow[d,equivalent,swap,"\ref{prop:bottom-manifold-layer}"]&\eqcolon&T_1\Emb^{\ot}_\partial(M,N)\\[-0cm]
	&&\Map^{/\BO(d)}_{\partial}(M,N)\rar&\Map^{/\BTop(d)}_{\partial}(M,N)&&
	\end{tikzcd}
\]
of spaces, where the bottom most two spaces are the mapping spaces in the category $(\cS_{/B})_{P\sqcup Q/}$ of spaces over $B\in\{\BO(d),\BTop(d)\}$ and under $\partial M\cong P\sqcup Q\cong \partial N$. The two columns of this diagram recover \eqref{equ:emb-calc-no-bdy-o} and \eqref{equ:emb-calc-no-bdy-t} in the case $P=Q=\varnothing$. The long vertical compositions from the embedding spaces to the bottom mapping spaces are induced by taking (topological) derivatives (see \cref{prop:bottom-manifold-layer}). Note that the right column is also defined if $P,Q,M$, and $N$ are not smooth. 

\medskip

\noindent In \cite[Theorem 4.5]{KKDisc} we showed that the map $\Emb^{\oo}_\partial(M,N)\ra T_\infty\Emb^{\oo}_\partial(M,N)$ in the left column is equivalent to the limit of the classical embedding calculus Taylor tower for the space $\Emb^{\oo}_\partial(M,N)$ of smooth embeddings fixing the boundary from \cite{WeissEmbeddings}. Replacing the use of the poset $\cU$ of open subsets of $M$ that are unions of an open collar of the boundary and a disjoint union of open discs, with the poset $\cU_k$ for $1\le k\le \infty$ where the cardinality of the discs is bounded by $k$, the same argument shows that this identification extends to the level of towers:

\begin{prop}\label{prop:morita-gives-embcalc} The left column in the above diagram is equivalent to Weiss' embedding calculus Taylor tower for the space $\Emb^{\oo}_\partial(M,N)$ of smooth embeddings fixed on the boundary from \cite{WeissEmbeddings}.
\end{prop}

\noindent For spaces of topological embeddings, embedding calculus has not considered in the literature yet. We \emph{define} the topological embedding calculus Taylor tower as the right column in the above diagram (as mentioned above, the manifolds need not be smooth for this). Other possibilities would be to adapt to the topological setting the definition from \cite{WeissEmbeddings} in terms of universal polynomial approximations, or the definition from \cite{BoavidaWeiss} in terms of homotopy sheafifications; these constructions would lead to the same result, by a variant of the argument for \cref{prop:morita-gives-embcalc}.

\begin{rem}[Embeddings of triads]\label{rem:triads} Fix $c\in\{\oo,\ot\}$. As already hinted at in the introduction, the above setting can also be applied to the space of embeddings $\Emb^{c}_{\partial_0}(M,N)$ between $d$-manifolds $M$ and $N$ that agree with a given embedding $e_{\partial_0}\colon \partial_0M\hookrightarrow \partial N$ on a codimension $0$ submanifold $\partial_0M\subset M$: since the inclusion $M'\coloneqq \interior(M)\cup \interior(\partial_0M)\subset M$ and $N'\coloneqq \interior(N)\subset\interior(e_{\partial_0}(\partial_0M))\subset N$ are isotopy equivalences relative to $\interior(\partial_0M)$, and we have $\partial M'=  \interior(\partial_0M)\cong\partial N'$ induced by $e_{\partial_0}$, we have $\smash{\Emb^{c}_{\partial_0}(M,N)\simeq \Emb^{c}_{\partial}(M',N')}$,
so to apply embedding calculus to $\Emb^{c}_{\partial_0}(M,N)$, one can instead apply it to the equivalent space $\Emb^{c}_{\partial}(M',N')$ which fits into the above setup. 
\end{rem}

\begin{rem}[Composition and gluing]\label{rem:comp-and-gluing} By construction, for $d$-dimensional bordisms $M,N\colon P\leadsto Q$, the embedding calculus tower $\Emb^c_{\partial}(M,N)\ra T_\bullet\Emb^c_\partial(M,N)$ is induced by the tower of double categories $\ncBordInf^c(d)\ra\ModInf_\bullet(E_d^{c})$. This in particular equips embedding calculus with coherent composition and gluing maps. For example, given bordisms $W,V\colon Q\leadsto R$, specialising the composition functor in a double category (see \cref{sec:monoids}) we obtain a map of towers from $\Emb^c_{\partial}(M,N)\times \Emb^c_{\partial}(V,W)\ra T_\bullet\Emb^c_\partial(M,N)\times T_\bullet\Emb^c_\partial(V,W)$ to $\Emb^c_{\partial}(M\cup_QV,N\cup_QW)\ra T_\bullet\Emb^c_\partial(M\cup_QV,N\cup_QW)$ that extends the map on embedding spaces given by gluing embeddings along $Q$. For another example, if $M$ arises as a composition of bordisms $M=M_1\cup_{L}M_2\colon P\leadsto Q$, we have a map of towers from $\Emb^c_{\partial}(M,N)\ra T_\bullet\Emb^c_\partial(M,N)$ to $\Emb^c_{\partial}(M_1\backslash L,N)\ra T_\bullet\Emb^c_\partial(M_1\backslash L,N)$ that extends the map on embedding spaces given by restriction along the inclusion $M_1\backslash L\subset M$, by first gluing on $Q\times[0,1)\colon Q\leadsto \varnothing$ and then using precomposition with the inclusion $M_1\backslash L\subset M\cup_{Q}[0,1)$ considered as a morphism in $\ncBordInf(d)_{P,\varnothing}$.
\end{rem}

\subsubsection{Manifold calculus}\label{sec:manifold-calc}Fix $c\in\{\oo,\ot\}$. Embedding calculus can be generalised to \emph{manifold calculus}: given a presheaf $F \in \PSh(\ncBordInf^c(d)_{P,Q})$, its \emph{manifold calculus tower} is a tower \begin{equation}\label{equ:manifold-calc-tower-f}
	\smash{F\lra T^c_\infty F\lra \cdots \lra T^c_2F\lra T^c_1F}
\end{equation}
of presheaves on $\ncBordInf^c(d)_{P,Q}$ under $F$, defined as follows: since $\smash{\ModInf^c(d)_{E_{P \times I},E_{Q \times I}}}$ is cocomplete as a result of \cite[4.3.3.9]{LurieHA}, the functor $\smash{E\colon \ncBordInf^c(d)_{P,Q}\ra \ModInf^c_k(d)_{E_{P \times I},E_{Q \times I}}}$ has a colimit-preserving extension $\smash{\lvert-\rvert_{E}}$ to $\smash{\PSh(\ncBordInf^c(d)_{P,Q})}$ (see \cref{sec:presheaves-yoneda}) which agrees with the colimit-preserving extension in the untruncated case $k=\infty$ followed by $k$-truncation. This gives a tower \begin{equation}\label{equ:manifol-calc-tower-cat}
	\lvert-\rvert_{E}\colon \PSh(\ncBordInf^c(d)_{P,Q})\ra\ModInf_\bullet^c(d)_{E_{P \times I},E_{Q \times I}}
\end{equation} 
of categories under $\PSh(\ncBordInf^c(d)_{P,Q})$. Each functor in \eqref{equ:manifol-calc-tower-cat} has a right-adjoint given by the Yoneda embedding followed by restriction along $E$. The tower \eqref{equ:manifold-calc-tower-f} is then defined by the components of the counits of these adjunctions at $F$ (note that $|y(-)|_E\simeq E_{(-)}$): 
\[
	F(-)\simeq \Map_{\PSh(\ncBordInf^c(d)_{P,Q})}(y(-),F)\xra{\lvert-\rvert_{E}}\Map_{\ModInf_\bullet^c(d)_{E_{P \times I},E_{Q \times I}}}(E_{(-)},\lvert F\rvert_{E}).
\]
Applied to $F\coloneqq \Emb^c_\partial(-,N) = y(N)$ for $N \in \ncBordInf^c(d)_{P,Q}$, this recovers the embedding calculus tower for $\Emb^c_\partial(M,N)$ as discussed above. In \cref{rem:mfd-calc-comp} below, we see that manifold calculus defined this way agrees with  the previously considered version by Boavida de Brito--Weiss \cite{BoavidaWeiss}.

\subsubsection{Comparison to Boavida de Brito--Weiss' model with boundary}\label{sec:comparison-pedro-michael}
There is a different construction of the two columns of the bottom diagram in \cref{sec:sm-top-emb-calc}: consider the full subcategory $\smash{(\DiscInf^{c}_{d})_{P,Q}\subset\ncBordInf^{c}(d)_{P,Q}}$ for $c\in\{\oo,\ot\}$ spanned by those bordisms that are diffeomorphic or homeomorphic (depending on $c$) relative to the boundary to $\smash{P\times[0,1)\sqcup S\times \bfR^d\sqcup Q\times (-1,0]}$ for some finite sets $S$. This has a filtration by full subcategories $\smash{(\DiscInf^{c}_{d,\le \bullet})_{P,Q}}$ by bounding the cardinalities of the finite sets $S$, so we obtain a tower of categories under $\ncBordInf^{c}(d)_{P,Q}$ 
\begin{equation}\label{equ:pedro-michael-model-cats}
	\smash{(\iota^*\circ y)\colon \ncBordInf^{c}(d)_{P,Q}\ra\PSh((\DiscInf^{c}_{d,\le \bullet})_{P,Q})}
\end{equation} 
given by the Yoneda embedding $y$ followed by the restrictions $\iota$ along the full subcategory inclusions. For $c=\oo$, this tower was considered in \cite{BoavidaWeiss} where they showed in particular that it agrees on mapping spaces with the embedding calculus tower from \cite{WeissEmbeddings} and thus by \cref{prop:morita-gives-embcalc} also with the tower on mapping spaces induced by our $\smash{E\colon \ncBordInf^c(d)_{P,Q}\ra \ModInf_\bullet(E_d^{c})_{E_{P\times I},E_{Q\times I}}}$. For $c=\oo$ and $k=\infty$ we showed in \cite[Proposition 4.8]{KKDisc} that this equivalence can be lifted to an equivalence of categories. The same holds for $c=\ot$ and generalises to the level of towers:

\begin{lem}\label{lem:tower-bdbw-boundary} For $c\in\{\oo,\ot\}$, there is an equivalence 
\[
	\ModInf_\bullet(E_d^{c})_{E_{P\times I},E_{Q\times I}}\simeq \PSh((\DiscInf^{c}_{d,\le \bullet})_{P,Q})
\] 
of towers of categories under $\ncBordInf^c(d)_{P,Q}$.\end{lem}

\begin{proof}
Recall that $\ModInf_\bullet(E_d^{c})_{E_{P\times I},E_{Q\times I}}$ is a tower of bimodule categories $\BMod_{{E_{P\times I},E_{Q\times I}}}(\RMod_\bullet(E_d^{c}))$ (see \cref{sec:morita-functor}). Viewing $c(P)\coloneqq P\times[0,1)$ and $c(Q)\coloneqq (-1,0]\times Q$ as bordisms $P\leadsto \varnothing$ and $\varnothing\leadsto Q$, we consider the diagram of categories
\[
	\begin{tikzcd}[ar symbol/.style = {draw=none,"\textstyle#1" description,sloped},
	equivalent/.style = {ar symbol={\simeq}}]
	\DiscInf_d^c\dar[swap]{c(P)\cup (-)\cup c(Q)}\dar&[-1cm]\subset&[-1cm] \ncBordInf^c(d)_{\varnothing,\varnothing}\rar{E}\dar[swap]{c(P)\cup (-)\cup c(Q)}&\BMod_{{E_{\varnothing},E_{\varnothing}}}(\RMod(E_d^{c}))\dar[swap]{E_{c(P)}\otimes (-)\otimes E_{c(Q)}}\\
	(\DiscInf_d^c)_{P,Q}&\subset&\ncBordInf^c(d)_{P,Q}\rar{E}&\BMod_{{E_{P\times I},E_{Q\times I}}}(\RMod(E_d^{c}))
	\end{tikzcd}
\]
where the left square commutes by definition and the right one because $E$ is a functor of double categories. The composition along the bottom row is full faithful by the argument in the proof of \cite[Theorem 4.5]{KKDisc}. The argument in that proof also shows that the composition $\DiscInf_d^c\ra \BMod_{{E_{P\times I},E_{Q\times I}}}(\RMod(E_d^{c}))$ sends a collection of discs $D\in \DiscInf_d^c\simeq\Env(E_d^{c})$ to the free bimodule $F_{E_{P\times I},E_{Q\times I}}(E_D)$ in the sense of \cref{sec:free-bimodules} on the representable presheaf on $D$, so the essential image of this composition is the full subcategory $\smash{\BMod_{{E_{P\times I},E_{Q\times I}}}^{\frep}(\RMod(E_d^{c}))\subset \BMod_{{E_{P\times I},E_{Q\times I}}}(\RMod(E_d^{c}))}$ of \cref{sec:modules-day}. Together with the fully faithfulness of the bottom row, this implies that the functor $E$ restricts to an equivalence of towers of full subcategory inclusions
\begin{equation}\label{equ:step-in-bimodule-filtration-as-discs}
	(\DiscInf_{d,\le \bullet}^c)_{P,Q}\xlratwo{E}{\simeq}\BMod_{{E_{P\times I},E_{Q\times I}}}^{\frep}(\RMod(E_d^{c}))_{\le \bullet}
\end{equation} 
where the right hand tower is defined as in the discussion around \cref{lem:bimodule-tower-is-presheaf-tower}. Abbreviating $\RMod(E_d^c)\coloneqq \cR^c$, we consider the diagram of categories
\[\hspace{-0.4cm}
	\begin{tikzcd}[column sep=0.5cm, row sep=0.1cm]
	&&[-1.5cm]\ncBordInf^c(d)_{P,Q}\arrow[lld,"E",swap,bend right=7]\arrow[rrd,"\iota^* y", bend left=7]&[-1.5cm]&\\
	\BMod_{{E_{P\times I},E_{Q\times I}}}(\cR^c)\rar{y}&\PSh(\BMod_{{E_{P\times I},E_{Q\times I}}}(\cR^c))\arrow[rr,"j^*"]&&\PSh(\BMod_{{E_{P\times I},E_{Q\times I}}}^\frep(\cR^c))\arrow[r,"E^*","\simeq"']&\PSh((\DiscInf_{d}^c)_{P,Q}).
	\end{tikzcd}
\]
The composition of the first two arrows in the bottom row is an equivalence by \cref{lem:modules-as-presheaves} and extends to an equivalence of towers by \cref{lem:bimodule-tower-is-presheaf-tower}. The final arrow is given by restriction along the equivalence \eqref{equ:step-in-bimodule-filtration-as-discs}, so also extends to the towers. It thus suffices to provide commutativity data for the triangle. Rewriting the bottom row using $E^* j^*=\iota^* E^*$, this was done in the proof of \cite[Theorem 4.5]{KKDisc} for $c=\oo$ and the same argument applies if $c=\ot$.
\end{proof}

\begin{rem}[Manifold calculus comparison]\label{rem:mfd-calc-comp}Under the equivalence of \cref{lem:tower-bdbw-boundary}, the tower \eqref{equ:manifol-calc-tower-cat} under $\PSh( \ncBordInf^c(d)_{P,Q})$ agrees with the colimit-preserving extension of \eqref{equ:pedro-michael-model-cats} which, by uniqueness of colimit-preserving extensions, agrees with the tower $\smash{\PSh(\ncBordInf^c(d)_{P,Q})\ra \PSh((\DiscInf_{d,\le \bullet}^c)_{P,Q})}$ under $\smash{\PSh(\ncBordInf^c(d)_{P,Q})}$ induced by restriction along $\iota_k\colon (\DiscInf_{d,\le k}^c)_{P,Q}\subset \ncBordInf^c(d)_{P,Q}$. The manifold calculus tower \eqref{equ:manifold-calc-tower-f} for $F\in\PSh( \ncBordInf^c(d)_{P,Q})$ as defined above thus agrees with the tower $F\ra (\iota_\bullet)_*(\iota_\bullet)^*(F)$ induced by the counits of the adjunction by restriction and right Kan extension. For $c=\oo$, this is how manifold calculus is defined in \cite{BoavidaWeiss}.
\end{rem}

\subsection{Some properties of embedding calculus}\label{sec:embcalc-props} There are a number of known formal properties for smooth embedding calculus (see e.g.\,\cite[Sections 2.2.3, 2.3]{KKsurfaces}). Essentially all of them also hold for the topological version. We explicitly mention two of the most important ones in the following; others, e.g.~those in loc.cit., follow from these. We fix $c\in\{\oo,\ot\}$.

\subsubsection{Descent for Weiss $k$-covers}\label{sec:descent}For a bordism $M\colon P\leadsto Q$ between manifolds $P$ and $Q$ without boundary (smooth or topological, depending on $c$), write $\cO(W)$ for poset of open subsets of $W$ that contain a neighbourhood of the boundary $\partial M=P \sqcup Q$, ordered by inclusion. Viewing such open subsets as bordisms $P\leadsto Q$ gives a functor $\cO(M)\ra (\ncBordInf^c(d)_{P,Q})_{/M}$  and thus a functor $\smash{\cO(W)\ra (\ModInf^c(d)_{E_{P\times I},E_{Q\times I}})_{/E_W}}$ by postcomposition with $E$. Recall that a subposet of $\cO(M)$ is a \emph{Weiss $k$-cover} for $1\le k\le \infty$ if any subset of cardinality $\leq k$ is contained in an element of $\cO(M)$. A \emph{Weiss $k$-cover} is \emph{complete} if it contains a Weiss $k$-cover of any finite intersection of its elements. The second part of the following statement involves the manifold calculus tower \eqref{equ:manifold-calc-tower-f}.

\begin{lem}[Descent] \label{lem:tk-descent} For $1\le k\le \infty$, a nonempty bordism $M\in\ncBordInf^c(d)_{P,Q}$ and a nonempty complete Weiss $k$-cover $\cU\subset \cO(M)$, the diagram $\{E_{U}\}_{U\in\cU}\ra E_M$ in $\ModInf^c_k(d)_{E_{P\times I},E_{Q\times I}}$ is a colimit diagram, and the diagram $T_k^cF(M)\ra \{T_k^cF(U)\}_{U\in\cU}$ a limit diagram for all $F\in\PSh(\ncBordInf^c(d)_{P,Q})$.
\end{lem}

\begin{proof}
For $c=\oo$ and $k=\infty$, the first part was deduced in the proof of \cite[Proposition 4.3]{KKDisc} from a well-known descent argument for configuration spaces. The same argument applies for any $k$ and for $c=\ot$. The second part follows from the first by the universal property of the colimit, since we have $\smash{T_k^cF(-)=\Map_{\ModInf_k^c(d)_{E_{P\times I},E_{Q\times I}}}(E_{(-)},|F|_E)}$; see \cref{sec:manifold-calc}.
\end{proof}

\subsubsection{Isotopy extension}\label{sec:topological-isotopy-extension} 
There is an isotopy extension theorem in the context of embedding calculus, which describes under some conditions the fibres of maps between limits of the smooth embedding calculus tower that are induced by restricting along a submanifold (see \cite[Theorem 4.10]{KKDisc} and \cite[Theorem 6.1, Remark 6.2 (iv)-(v)]{KnudsenKupers}). The same statements hold for the limit of the topological embedding calculus tower, since the proof in the smooth case only uses naturality properties, descent for Weiss $\infty$-covers and the usual parametrised isotopy extension theorem for spaces of smooth topological embeddings of codimension $0$, which are all are available in the topological setting: the required naturality is provided by \cref{thm:extension-e-topological}, the descent property by \cref{sec:descent}, and parametrised isotopy extension for topological embeddings by \cite{EdwardsKirby}.

\subsection{The layers of embedding calculus} \label{sec:layers-embcalc}Given that smooth and topological embedding calculus is encoded by the tower $\ModInf^\un_\bullet(E_d^c)$ of symmetric monoidal double categories under the noncompact bordism category $\ncBordInf^c(d)$ via the functor $E$ (for $c\in\{\oo,\ot\}$), its layers are in principle described by the abstract layer identifications for the towers $\ModInf^\un_\bullet(\cO)$ for general unital operads $\cO$ from Theorems \ref{thm:bottom-layer-morita} and \ref{thm:layers-tower-module}. However, for many purposes a more geometric description in terms of manifold-theoretic data is preferable. We will work out such a description in this subsection. This will in particular recover Weiss' description of the layers of the classical embedding calculus tower in terms of relative section spaces over bundles over configuration spaces \cite{WeissEmbeddings}.

\subsubsection{The bottom layer}As we already explained in \cref{sec:sm-top-emb-calc}, the map to the bottom layer of the topological embedding calculus tower on bordism categories has the form 
\begin{equation}\label{equ:E-functor-to-bottom-layer}
	E\colon \ncBordInf^{\ot}(d)\lra \ModInf_1^\un(E_d^\ot)\simeq  \Cosp(\cS_{/\BTop(d)}).
\end{equation}
The following shows that it is essentially given by ``taking topological tangent bundles'':
\begin{prop}\label{prop:bottom-manifold-layer}
The functor \eqref{equ:E-functor-to-bottom-layer} has the following properties:

\begin{enumerate}[leftmargin=*]
	\item it sends a $(d-1)$-manifold $P$ without boundary viewed as an object of $\ncBordInf^{\ot}(d)$ to the classifier $(P\ra \BTop(d))\in\cS_{/\BTop(d)}$ of its once stabilised topological tangent bundle,
	\item it sends a bordism $M\colon P\leadsto Q$ viewed as an object in the mapping category $\ncBordInf^{\ot}(d)_{P,Q}$ to the cospan $(P\ra M\la Q)\in (\cS_{/\BTop(d)})_{(P\sqcup Q)/}$ over $\BTop(d)$ induced by the tangent bundles,
	\item on mapping spaces of mapping categories 
	\[
		\hspace{0.6cm}\Map_{\ncBordInf^{\ot}(d)_{P,Q}}(M,N)\simeq \Emb^{\ot}_\partial(M,N)\ra \Map_\partial^{/\BTop(d)}(M,N)\coloneqq \Map_{(\cS_{/\BTop(d)})_{(P\sqcup Q)/}}(M,N)
	\]
	it is induced by taking topological derivatives.
\end{enumerate}
The same holds for the smooth version when replacing $c=\ot$ by $c=\oo$, $\BTop(d)$ by $\BO(d)$ and topological tangent bundles and derivatives by smooth ones.
\end{prop}

\begin{proof}We begin with the following general recollections on topological tangent bundles:
\begin{enumerate}[leftmargin=*, label=(\alph*)]
	\item One model of the tangent classifier of the topological tangent bundle of a $d$-manifold $M$ is given by the following zig-zag
	\begin{equation}\label{equ:topological-tangent-bundle}
		M\xleftarrow{\simeq}\Emb^\ot(\bfR^d,M)_{\Top_0(d)}\xrightarrow{\simeq}\Emb^\ot(\bfR^d,M)_{\Top(d)}\ra\BTop(d),
	\end{equation}
	where $\Top_0(d)\subset\Top(d)$ is the topological subgroup of origin-preserving homeomorphisms. Composing with an appropriate translation gives a homotopy inverse to the inclusion, so $\Top(d)_0 \simeq\Top(d)$. This explains the middle equivalence between the orbits with respect to the action by precomposition. The leftmost map is induced by evaluation at $0$, which is a fibration by isotopy extension, and is an equivalence since the fibre at $m\in M$ of $\ev_0\colon \Emb^\ot(\bfR^d,M)\ra M$ is equivalent to the space of embeddings of $\bfR^d$ into a $\bfR^d$-neighbourhood of $m\in M$ which is in turn equivalent to $\Top(d)$ by the Kister--Mazur theorem.
	\item In this model, the topological derivative of a codimension $0$-embedding $e\colon M\hookrightarrow N$ is given by the map between the zig-zag \eqref{equ:topological-tangent-bundle} for $M$ to the one for $N$, induced by postcomposition with $e$.
	\item A model for the once-stabilised tangent bundle of a $(d-1)$-manifold $P$ is given by applying the above to $P\times(0,1)$ and using $P\times (0,1)\simeq P$ by the projection.
\end{enumerate}
With these models for the topological tangent bundles and the derivative, the claim follows by chasing through the construction of the functor $E$ and the fact that the straightening equivalence $\PSh(\cG)\simeq \cS_{/\cG}$ for groupoids is given by sending $X\in \PSh(\cG)$ to $\colim_\cG(X)\ra \colim_\cG(\ast)\simeq \cG$. The argument in the smooth case is the same; one simply replaces $\Top(d)$ by $\oO(d)$ and all spaces of topological embeddings by their smooth analogous.
\end{proof}

\subsubsection{The higher layers, without boundary}For topological $(d-1)$-manifolds $P$ and $Q$ without boundary, each consecutive functor in the tower $\ModInf_\bullet(E_d^\ot)_{E_{P\times I},E_{E\times I}}$ under $\ncBordInf^\ot(d)_{P,Q}$ fits by \cref{thm:layers-tower-module} \ref{enum:layer-pullback-gc} into a pullback square
\begin{equation}\label{equ:geometric-layer-diagram-bordism}
	\begin{tikzcd}[column sep=1.5cm, row sep=0.5cm]
	\ModInf_k(E_d^\ot)_{E_{P\times I},E_{E\times I}}\dar\rar{\Lambda^c_{E_{P\times I}E_{Q\times I}}}&\cS^{[2]}\dar{(0\le 2)^*}\\
	\ModInf_{k-1}(E_d^\ot)_{E_{P\times I},E_{E\times I}}\rar{\Omega^c_{E_{P\times I}E_{Q\times I}}}&\cS^{[1]}.
	\end{tikzcd}
\end{equation}
Our goal of this and the following subsection is to give a geometric description of the composition of the top horizontal arrow with $E\colon \ncBordInf^\ot(d)_{P,Q}\ra \ModInf_k(E_d^\ot)_{E_{P\times I},E_{E\times I}}$. This will in particular identify the fibres of the left vertical functor in \eqref{equ:geometric-layer-diagram-bordism} in terms of configuration space data (see \cref{cor:layer-description-explicit}). We first deal with the special case $P=Q=\varnothing$, where \eqref{equ:geometric-layer-diagram-bordism} becomes the pullback 
\begin{equation}\label{equ:geometric-layer-diagram-single}
	\begin{tikzcd}[row sep=0.5cm]
	\RMod_k(E_d^\ot)\dar\rar{\Lambda^c}&\cS^{[2]}\dar{(0\le 2)^*}\\
	\RMod_{k-1}(E_d^\ot)\rar{\Omega^c}&\cS^{[1]}
	\end{tikzcd}
\end{equation}
from \cref{thm:layers-tower} \ref{enum:layer-pullback-gc}.

\begin{rem}[Smooth layers]We focus on topological embedding calculus in this subsection, because the corresponding results for the smooth version follows from the topological case if $P$ and $Q$ are smooth: by the commutativity of the square in \cref{thm:extension-e-topological} and the naturality of the higher Morita layers from \cref{thm:naturality-of-layers-module} applied to the map of operads $\smash{E_d^\oo\ra E_d^\ot}$, the composition of the smooth version of the top horizontal arrow in \eqref{equ:geometric-layer-diagram-bordism} with $\smash{E\colon \ncBordInf^\oo(d)_{P,Q}\ra \ModInf_k(E_d^\oo)_{E_{P\times I},E_{E\times I}}}$ agrees with the corresponding composition composition in the topological case, precomposed with the forgetful map $\smash{\ncBordInf^\oo(d)_{P,Q}\ra \ncBordInf^\ot(d)_{P,Q}}$ followed by the composition.\end{rem}

The geometric description of the higher layers we  will give involves point-set topological constructions for which we work temporarily in a self-enriched convenient category $\cTop$ of topological spaces. Given a manifold $M$ without boundary, we consider its \emph{$k$th ordered configuration space}
\[
	F_k(M) \coloneqq \Emb(\ul{k},M)=M^k\backslash \Delta_k(M),\ \  \text{where}\ \  \Delta_k(M) \coloneqq \{(x_1,\ldots,x_k)\in M^k \mid \text{$x_i = x_j$ for some $i \neq j$}\} 
\]
is the \emph{thick diagonal}; here $\underline{k}\coloneqq\{1,\ldots,k\}$. The \emph{homotopy link} (see e.g.\,\cite[2.1]{QuinnLink})
\[\partial^h F_k(M)\coloneqq \holink(\Delta_k(M)\subset M^k)\in\cTop\] 
of the thick diagonal is the subspace $\partial^h F_k(M) \subset \Map([0,1],M^k)$ of paths $\gamma\colon [0,1]\ra M^k$ for which $\gamma^{-1}(\Delta_k(M))=\{0\}$. Evaluation at $1\in[0,1]$ gives the left map in a sequence of $\Sigma_k$-equivariant maps
\begin{equation}\label{equ:mfd-match-latch-before-orbits}
	\partial^h F_k(M) \xlra{\ev_1} F_k(M)\lra \holim_{S\subsetneq \underline{k}}F_S(M)
\end{equation}
whose right-hand map to the homotopy limit of the punctured cubical diagram $\underline{k}\supsetneq S\mapsto F_S(M)\eqcolon \Emb(S,M)$ is induced by forgetting points. Writing $\partial^h C_k(M) \coloneqq \partial^h F_k(M)_{\Sigma_k}$ and $C_k(M)=F_k(M)_{\Sigma_k}$ for the strict orbits, the sequence \eqref{equ:mfd-match-latch-before-orbits} induces maps in $\cTop$
\begin{equation}\label{equ:mfd-match-latch}
	\partial^h C_k(M) \lra C_k(M) \lra[\holim_{S\subsetneq \underline{k}}F_S(M)]_{h\Sigma_k}.
\end{equation}
This construction is functorial in codimension $0$ embeddings, so \eqref{equ:mfd-match-latch} gives rise to a topologically enriched functor $\Man^\ot_d\ra \cTop^{[2]}$ out of the topologically enriched category $\Man^\ot_d$ of topological $d$-manifolds without boundary and codimension $0$ embeddings between them. Upon taking coherent nerves, this gives a functor $\ManInf^\ot_d\ra \cS^{[2]}$. This agrees with the top composition in \eqref{equ:geometric-layer-diagram-single}:

\begin{prop}\label{lem:latch-match-manifolds} For a $d$-dimensional manifold $M$ without boundary, there is an equivalence in $\cS^{[2]}$
	\begin{equation}\label{latch-match-manifolds}
		\Lambda^c(E_M)\simeq \big(\partial^h C_k(M)\ra C_k(M)\ra [\holim_{S\subsetneq\underline{k}} F_S(M)]_{h \Sigma_k}\big).
	\end{equation}
which is natural in $M$, i.e.\,features in an equivalence of functors $\ManInf^\ot_d \to \cS^{[2]}$.
\end{prop}

\begin{rem}In \cref{rem:alternative-models-layers} below, we explain alternative point-set models for $\partial^h C_k(M)\ra C_k(M)$.
\end{rem}

\begin{ex}\label{ex:boundary-for-discs}We discuss a key example before turning to the proof of \cref{lem:latch-match-manifolds}, namely $M=S\times \bfR^d$ for $S\in \Fin$. In this case the map $\ev_1\colon \partial^hC_k(S\times \bfR^d)\ra C_k(S\times \bfR^d)$ lands in the collection of components $\smash{C^{\ninj}_k(S\times \bfR^d)\subset C_k(S\times \bfR^d)}$ of \emph{non-injective configurations}, i.e.\,those configurations for which one of the disjoint summands in $S\times \bfR^d$ contains at least two points (note that the points are still disjoint, and ``non-injective'' refers to path components). The induced map
\begin{equation}\label{equ:latch-noninj}
	\ev_1\colon \partial^hC_k(S\times \bfR^d)\xlra{\simeq} C^{\ninj}_k(S\times \bfR^d)
\end{equation}
turns out to be a homotopy equivalence, since it has a homotopy inverse given by sending a non-injective ordered configuration $C\subset S\times \bfR^d$ to the straight-line paths from the origin of each $\bfR^d$-summand containing a point in $C$ to $C$: this is a strict right-inverse to $\ev_1$ and also a homotopy left inverse in view of the following homotopy of Alexander-trick type:
\[
	0,1] \times \partial^h C_k(S \times \bfR^d) \ni (s,\gamma) \mapsto \left(t \mapsto \begin{cases} 	(1-s)\cdot \gamma(\frac{t}{1-s}) & \text{if $t \in [0,1-s]$} \\
	t\cdot \gamma(1) & \text{if $t \in [1-s,1]$}\end{cases} \right)\in \partial^h C_k(S \times \bfR^d).
\]
This equivalence is natural in embeddings between $S\times\bfR^d$ for different $S$, so induces an equivalence $\smash{(\partial^hC_k(-)\ra C_k(-))\simeq (C_k^{\ninj}(-)\subset C_k)}$ of functors $\DiscInf_d^\ot\ra \cS^{[1]}$.
\end{ex}

\begin{proof}[Proof of \cref{lem:latch-match-manifolds}]We abbreviate  $\smash{\DiscInf_{d,\le k}^{\ot}\coloneqq \DiscInf_{\le k}}$. Going through the definition, we see that the inclusion $\smash{\nu\colon (E_d^{\ot})^\col\wr\Sigma_k\hookrightarrow \Env(E_d^{\ot})_{\le k}}$ from \cref{sec:higher-layers-nomodule} agrees with respect to the equivalence $\smash{\Env(E_d^{\ot})\simeq \DiscInf^{\ot}}$ (see \cref{sec:rmod-framed-ed}) with the inclusion $\smash{\BAut(\underline{k}\times\bfR^d)\hookrightarrow \DiscInf_{\le k}}$ of the automorphisms of $\underline{k}\times\bfR^d$ in $\DiscInf_{\le k}$, so the functor $\smash{\Lambda^c(-)\colon \PSh(\DiscInf_{\le k}) \ra\cS^{[2]}}$ sends $X$ to 
\begin{equation}\label{equ:match-latch-spelled-out}
	\big((\iota_!\iota^*X)(\underline{k}\times\bfR^d)_{\Aut(\underline{k}\times\bfR^d)}\ra X(\underline{k}\times\bfR^d)_{\Aut(\underline{k}\times\bfR^d)}\ra (\iota_\ast\iota^*X)(\underline{k}\times\bfR^d)_{\Aut(\underline{k}\times\bfR^d)}\big)
\end{equation}
where $\smash{\iota\colon \DiscInf_{\le k-1}\hookrightarrow \DiscInf_{\le k}}$ is the inclusion and $\smash{(-)_{\Aut(\underline{k}\times\bfR^d)}}$ denotes taking orbits. As a result of the Kister--Mazur theorem, the group $\Aut(\underline{k}\times\bfR^d)$ is equivalent to the wreath product $\smash{\Top(d)^{k}\rtimes \Sigma_k}$, so we can compute the $\Aut(\underline{k}\times\bfR^d)$-orbits by first taking $\Top(d)^k$-orbits and then taking orbits of to the residual $\Sigma_k$-action. Combining this with 
 \cref{lem:simplified-right-kan}, we get a natural equivalence
\[
	\textstyle{\big(X(\underline{k}\times\bfR^d)\ra (\iota_*\iota^*X)(\underline{k}\times\bfR^d)\big)_{\Aut(\underline{k}\times\bfR^d)}\simeq \big(X(\underline{k}\times\bfR^d)_{\Top(d)^k}\ra \lim_{S\subsetneq \underline{k}}(X(S\times\bfR^d)_{\Top(d)^S})\big)_{\Sigma_k}}
\]
For $X=E_M=\Emb(-,M)\in\PSh(\DiscInf_{\le k})$ and a finite sets $S$, evaluating at the centres gives an equivalence (again as a result of the Kister--Mazer theorem) \begin{equation}\label{equ:unthicken-configurations}
	\Emb(S\times\bfR^d,M)_{\Top(d)^S}\simeq F_S(M)=\Emb(S,M).
\end{equation} 
Combining this with the above, we get a natural equivalence
\[
	\textstyle{\big(E_M(\underline{k}\times\bfR^d)\ra (\iota_*\iota^*E_M)(\underline{k}\times\bfR^d)\big)_{\Aut(\underline{k}\times\bfR^d)}\simeq \big(C_k(M)\ra [\lim_{S\subsetneq \underline{k}}F_S(M)]_{\Sigma_k}\big)}
\]
which identifies the right part of \eqref{equ:match-latch-spelled-out} as claimed.
To identify the left part of \eqref{equ:match-latch-spelled-out}, we construct a natural commutative diagram with horizontal equivalences
\[
	\hspace{-0.4cm}
	\begin{tikzcd}[column sep=0.1cm,ar symbol/.style = {draw=none,"\textstyle#1" description,sloped}, equivalent/.style = {ar symbol={\simeq}}, row sep=0.4cm]
	(\iota_!\iota^*E_M)(\underline{k}\times\bfR^d)_{\Aut(\underline{k}\times\bfR^d)}\dar\rar[equivalent]&(\iota_!\iota^*E_{(-)})(\underline{k}\times\bfR^d)_{\Aut(\underline{k}\times\bfR^d)}\otimes_{\DiscInf_{\le k}}E_M\dar\rar[equivalent] &\partial^hC_k(-)\otimes_{\DiscInf_{\le k}}E_M\dar\rar[equivalent]&\partial^hC_k(M)\dar\\
	E_M(\underline{k}\times\bfR^d)_{\Aut(\underline{k}\times\bfR^d)}\rar[equivalent]&E_{(-)}(\underline{k}\times\bfR^d)_{\Aut(\underline{k}\times\bfR^d)}\otimes_{\DiscInf_{\le k}}E_M\rar[equivalent]&C_k(-)\otimes_{\DiscInf_{\le k}}E_M\rar[equivalent]&C_k(M).
	\end{tikzcd}
\]
where the two leftmost vertical maps are induced by the counit of $\iota_!\dashv\iota^*$ and the two rightmost ones by the evaluation $\ev_1\colon \partial^hC_k(-)\ra C_k(-)$. The identification between the first two columns follows from the fact that left Kan extension can be computed via a coend and the compatibility of coends with colimits. To identify the second column with the third, we construct equivalences
\[
	\smash{\big(\iota_!\iota^*E_{(-)})(\underline{k}\times\bfR^d)\ra E_{(-)}(\underline{k}\times\bfR^d)\big)_{\Aut(\underline{k}\times\bfR^d)}\simeq \big(C^\ninj_k(-)\subset C_k(-)\big)\simeq \big(\partial^hC_k(-)\ra C_k(-)\big)}
\]
of functors $\DiscInf_{\le k}\ra \cS^{[1]}$ where $\smash{C^\ninj_k(S\times\bfR^d)\subset C_k(S\times\bfR^d)}$ is defined as in \cref{ex:boundary-for-discs}. The second of these equivalences was explained in \cref{ex:boundary-for-discs}. In view of \eqref{equ:unthicken-configurations}, the first equivalence follows from showing that the counit $\smash{(\iota_!\iota^*E_{S\times \bfR^d})(\underline{k}\times\bfR^d)\ra E_{S\times \bfR^d}(\underline{k}\times\bfR^d)=\Emb(\underline{k}\times\bfR^d,S\times \bfR^d)}$  induces an equivalence onto the components whose map on path-components is not injective. If $|S|\le k-1$, then these are all components, so there is nothing to show since $\iota_!E_{S\times\bfR^d}\simeq E_{S\times\bfR^d}$ in this case. For $|S|=k$, we already established this as part of the proof of \cref{thm:layers-tower} (see \eqref{equ:counit-nonsur}; note that injectivity and surjectivity are the same for maps of finite sets of equal cardinality).

\medskip

\noindent This leaves us with the task to identify the final two columns in the above diagram. To do so, we work in a convenient category of topological spaces $\cTop$ and model the coends by the thick geometric realisations $\| \cat{B}_\bullet(\cat{E}_M,\Disc_{\le k},\partial^hC_k(-))\|$ and $\| B_\bullet(\cat{E}_M,\Disc_{\le k},C_k(-))\|$ of the respective bar constructions where $\cat{E}_M=\Emb(-,M)$ is an enriched presheaf on the topologically enriched category $\Disc_{\le k}$ with objects $S\times\bfR^d$ for $S\in\Fin_{\le k}$ and morphisms codimension $0$ embeddings, in the strict sense (cf.\,\cite[Section 4.4.1]{KKsurfaces}). It suffices to show that the two natural augmentations
\begin{equation}\label{equ:augmentation-config}
	\|\cat{B}_\bullet(\cat{E}_M,\Disc_{\le k},C_k(-))\|\ra C_k(M)\quad\text{and}\quad\| \cat{B}_\bullet(\cat{E}_M,\Disc_{\le k},\partial^hC_k(-))\|\ra \partial^hC_k(M)
\end{equation}
are equivalences. To this end, we consider the symmetric powers $\Sym_k(M) \coloneqq \smash{M^k_{\Sigma_k}}$ which contain $C_k(M)$ as a subspace, and consider the analogous augmentation $\|\cat{B}_\bullet(\cat{E}_M,\Disc_{\leq k},\Sym_k(-))\| \ra \Sym_k(M)$.
By a standard argument (c.f..~\cite[Proposition 4.5]{KKsurfaces} or \cite[Section 4]{KlangKupersMiller}) this is a Serre microfibration with weakly contractible fibres, so in fact an Serre fibration \cite[Lemma 2.2]{WeissClassify} and thus a weak equivalence. Being a Serre fibration with weakly contractible fibres is preserved by strict pullbacks, so by pulling back along $C_k(M) \subset \Sym_k(M)$, it follows that the first map in \eqref{equ:augmentation-config} is a weak equivalence. For the second, we consider the map of augmented semisimplicial spaces 
\[
	\begin{tikzcd}[row sep=0.3cm]
	\cat{B}_\bullet(\cat{E}_M,\Disc_{\leq k},\partial^h C_k(-)) \rar \dar & \cat{B}_\bullet(\cat{E}_M,\Disc_{\leq k},\Sym_k(-)) \dar \\
	\partial^h C_k(M) \rar{\ev_0} & \Sym_k(M) 
	\end{tikzcd}
\]
where the top map is induced by $\ev_0$. Setting $\cat{B}_\bullet \coloneqq \cat{B}_\bullet(\cat{E}_M,\Disc_{\leq k},\Sym_k(-)) \times_{\Sym_k(M)}\partial^h C_k(M)$, this factors as a horizontal composition of squares
\begin{equation}\label{eqn:microfib-diagram} 
	\begin{tikzcd}[row sep=0.3cm]
	\cat{B}_\bullet(\cat{E}_M,\Disc_{\leq k},\partial^h C_k(-)) \rar \dar &[-10pt] \cat{B}_\bullet \rar \dar &[-10pt] \cat{B}_\bullet(\cat{E}_M,\Disc_{\leq k},\Sym_k(-)) \dar \\
	\partial^h C_k(M) \rar[equal] &	\partial^h C_k(M) \rar{\ev_0} & \Sym_k(M).
	\end{tikzcd}
\end{equation}
On thick geometric realisation the right-hand square becomes a strict pullback which follows from the general fact that if $X_\bullet \to X_{-1}$ is an augmented semisimplicial space and $Y_{-1} \to X_{-1}$ is any continuous map, then the induced map $\|X_\bullet \times_{X_{-1}} Y_{-1}\| \to \|X_\bullet\| \times_{X_{-1}} Y_{-1}$ is a homeomorphism since the pullback functor $\cTop_{/X_{-1}}\ra \cTop_{/Y_{-1}}$ and the forgetful functor $\cTop_{/Y_{-1}}\ra \cTop$ both preserve colimits; the former since it is a left adjoint (see e.g.\,\cite[2.1.3]{MaySigurdsson}). Since being a Serre fibration with weakly contractible fibres is preserved by strict pullbacks, we conclude that the middle vertical map in \eqref{eqn:microfib-diagram} is an equivalence on thick geometric realisations, so it suffices to show that the top left horizontal map is an equivalence on thick geometric realisations.

\medskip

\noindent We will show that $\cat{B}_\bullet(\cat{E}_M,\Disc_{\leq k},\partial^h C_k(-)) \to \cat{B}_\bullet$ is in fact a levelwise weak equivalence. By composing embeddings, we can think of a point in the topological space of $0$-simplices of its domain as a collection of paths $\gamma \in \partial^h C_k(M)$ and a collection of $\leq k$ embedded discs in $M$ containing $\gamma$. More generally, a point in the $p$-simplices in the domain is a $\gamma$ and $p+1$ collections of discs, each nested inside the next and the innermost containing $\gamma$. In the target, the discs are only required to contain the starting point of $\gamma$. So if we are given a commutative diagram
\[
	\begin{tikzcd}[column sep=1cm,row sep=0.3cm] S^{i-1} \dar \rar & \cat{B}_p(\cat{E}_M,\Disc_{\leq k},\partial^h C_k(-)) \dar \\
	D^i \rar \arrow[dashed]{ru} & \cat{B}_p = \cat{B}_p(\cat{E}_M,\Disc_{\leq k},\Sym_k(-)) \underset{\Sym_k(M)}\times \partial^h C_k(M),\end{tikzcd}
\]
there exists by compactness an $\epsilon>0$ such that we restrict the paths $\gamma_t$ for $t \in D^i$ to $[0,\epsilon] \subset [0,1]$, then they will be contained entirely in the discs (see \cref{fig:lifting-microfib}). Restricting all $\gamma_t$ to $[0,(1-s)+s\epsilon]$ for $s \in [0,1]$ and reparametrising, we obtain a homotopy of commutative diagrams from the original one to one that admits a dotted lift. This shows that the top-left horizontal map in \eqref{eqn:microfib-diagram} is a levelwise equivalence, which finishes the proof.\end{proof}

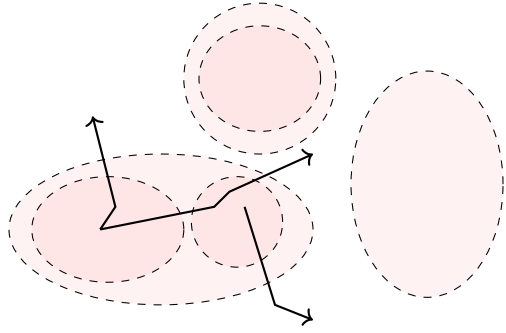
\begin{figure}
	\begin{tikzpicture}
		\draw [dashed,fill=red!5!white] (0,0) ellipse (2cm and 1cm);
		\draw [dashed,fill=red!5!white] (3.5,.6) ellipse (1cm and 1.5cm);
		\draw [dashed,fill=red!5!white] (1.3,2) ellipse (1cm and 1cm);
		\draw [dashed,fill=red!10!white] (-.7,0) ellipse (1cm and .7cm);
		\draw [dashed,fill=red!10!white] (1,0.1) ellipse (.6cm and .6cm);
		\draw [dashed,fill=red!10!white] (1.3,2) ellipse (.8cm and .7cm);
		\draw [thick,->] (-.8,0) -- (-.6,.3) -- (-.9,1.5);
		\draw [thick,->] (-.8,0) -- (.7,.3) -- (.9,.5) -- (2,1);
		\draw [thick,->] (1.1,0.3) -- (1.5,-1) -- (2,-1.2);
	\end{tikzpicture}
	\caption{The directed paths indicate an element $\gamma\in \partial^h C_3(M)$ starting at an element $\gamma(-)\in\Sym_3(M)$, which has been lifted to $\cat{B}_1(\cat{E}_M,\Disc_{\leq 3},\Sym_3(-))$ as indicated by the coloured discs. To lift $\gamma$ to $\cat{B}_1(\cat{E}_M,\Disc_{\leq 3},\partial^h C_3(-))$, we shrink the domain of the directed paths until their image lies in the innermost discs.}
	\label{fig:lifting-microfib}
\end{figure}

\subsubsection{The higher layers, with boundary}\label{sec:higher-layers-bdy}We now describe the top composition in \eqref{equ:geometric-layer-diagram-bordism} in the more general case where $P$ and $Q$ need not be empty. Given a bordism $M\colon P\leadsto Q$ of topological manifolds, we modify the definition of $\Delta_k(M)$ and $\partial^h F_k(M)$ to
\begin{align*}
	\Delta_{PQ,k}(M) &\coloneqq \big\{(x_1,\ldots,x_k)\in M^k \mid x_i = x_j\text{ for some }i \neq j,\text{ or } x_i \in \partial M=P \cup Q\text{ for some }i\big\}\in\cTop\\
	\partial^h_{PQ} F_k(M)&\coloneqq \holink(\Delta_{PQ,k}(M)\subset M^k)=\big\{\gamma\in \Map([0,1],M^k)\mid \gamma^{-1}(\Delta_{PQ,k}(M))=\{0\}\big\}\in\cTop.
\end{align*}
Evaluation at $1$ gives a $\Sigma_k$-equivariant map $\smash{\partial^h_{PQ} F_k(M) \to F_k(M)}$ which on strict $\Sigma_k$-orbits induces
\[
	\partial^h_{PQ} C_k(M) \coloneqq \partial^h_{PQ} F_k(M)_{\Sigma_k} \lra F_k(M)_{\Sigma_k}=C_k(M).
\]
The following generalises \cref{lem:latch-match-manifolds}: 

\begin{prop}\label{lem:latch-match-bordism}For a $d$-dimensional bordism $M \colon P \leadsto Q$, there is an equivalence in $\cS^{[2]}$
\[
	\Lambda^c_{E_{P\times I},E_{Q\times I}}(E_M)\simeq \big(\partial^h_{PQ} C_k(M)\ra C_k(M)\ra [\holim_{S\subsetneq\underline{k}}F_S(M)]_{\Sigma_k}\big)
\]
which is natural in the bordism $M$, i.e.~features in an equivalence of functors $\ncBordInf^\ot(d)_{P,Q} \to \cS^{[2]}$.
\end{prop}

\begin{rem}\label{rem:alternative-models-layers}There are several other point-set models for the map $\smash{\partial^h_{PQ} C_k(M)}\ra C_k(M)$ for a (potentially noncompact) bordism $M\colon P\leadsto Q$:
\begin{enumerate}[leftmargin=*]
	\item \label{enum:latch-fm-bdy}For smooth bordisms $M\colon P\leadsto Q$ one can make use of the \emph{Fulton--MacPherson bordification} $\FM_k(M)$ which is smooth manifold with corners that contains $F_k(M)$ as an open subspace (see \cite{SinhaFM} for a detailed discussion in the case where $\partial M=\varnothing$, and \cite[Section 5.3]{KRWAlgebraic} for an explanation of how to generalise this to $\partial M\neq\varnothing$). Its interior is $F_k(\interior(M))$, it is natural in smooth embeddings, it comes with a free $\Sigma_k$-action extending the evident action on $F_k(M)$, and it is compact if and only if $M$ is compact. By construction, the inclusion $F_k(M)\ra M^k$ extends to a map $\gamma\colon \FM_k(M)\ra M^k$ with $\gamma^{-1}(\Delta_{PQ,k}(M))=\partial \FM_k(M)$. Writing $\mathrm{CM}_k(M)\coloneqq \FM_k(M)_{\Sigma_k}$ for the strict quotient, there is an equivalence of pairs
	\begin{equation}\label{equ:latching-fm-compactification}
		\big(\partial^h_{PQ} C_k(M)\ra C_k(M)\big)\simeq \big(\partial 	\mathrm{CM}_k(M)\subset \mathrm{CM}_k(M)\big)
	\end{equation}
	that is natural in smooth codimension $0$ embeddings. This follows by combining the proof of \cref{lem:latch-match-manifolds} with the $\Sigma_k$-orbits of the equivalence from \cite[Proposition 4.5]{KKsurfaces}.

	\item\label{enum:latch-fm-bdy-topological} The Fulton--MacPherson bordification does not extend to topological manifolds \cite{KupersFM}, but there are other manifold models for the map $\smash{\partial^h_{PQ} C_k(M)\ra C_k(M)}$ if the bordism $M\colon P\leadsto Q$ is not necessarily smooth. For instance, we have
	\begin{equation}\label{equ:latching-regular-nghbd}
		\big(\partial^h_{PQ} C_k(M)\ra C_k(M)\big)\simeq \big(E\backslash 	(\Delta_{PQ,k}(M)_{\Sigma_k})\subset C_k(M)\big)
	\end{equation}
	where $E\subset M^k_{\Sigma_k}$ is a \emph{regular open neighbourhood} of $\Delta_{PQ,k}(M)_{\Sigma_k}\subset M^k_{\Sigma_k}$ in the sense of \cite{Siebenmann,SiebenmannGuillouHahl}. Such a regular neighbourhood $E$ exists by \cite[Proposition 1.6]{SiebenmannGuillouHahl} using Th\'eor\'eme 5.15, Lemme 5.16 loc.cit.~and the observation that $\smash{\Delta_{PQ,k}(M)_{\Sigma_k} \subset M^k_{\Sigma_k}}$ is locally triangulable. To verify the latter we may take $\smash{M = \bfR^d}$ or $\smash{M = [0,\infty) \times \bfR^{d-1}}$ (though it is convenient to double the latter to $\bfR^d$ with $C_2$-action given by reflection in the first coordinate), then apply the existence result for equivariant triangulations for smooth actions of finite groups on smooth manifolds \cite[Theorem]{Illman} to either the action of $\Sigma_k$ or $\Sigma_k \wr C_2$ on $M^k$, take the union of those simplices that have non-trivial isotropy group, and pass to the quotient. The regular neighbourhood is unique up to isotopy equivalence fixing $\Delta_{PQ,k}(M)_{\Sigma_k}$ by \cite[Th\'eor\'eme 2.2]{SiebenmannGuillouHahl}. To establish the claimed equivalence, one uses Proposition 5.12 loc.cit.~and argues as for $\partial^h C_k(S \times \bfR^d) \simeq C_k(S \times \bfR^d)^\ninj$ in \cref{ex:boundary-for-discs}.

	\medskip

	\noindent For the map $\smash{\partial^h_{PQ} F_k(M)\ra F_k(M)}$ \emph{before} taking $\Sigma_k$-orbits, one can do better than \eqref{equ:latching-regular-nghbd}: this map turns out to be equivalent to the boundary inclusion of a manifold, and this manifold can be chosen to be compact if $M$ is compact. To see this, we may assume $d\ge4$ and $k\ge2$, since for $d\le 3$ all bordisms are smoothable and we may use the Fulton--MacPherson bordification describe above, and for $k=1$ we have $\smash{(\partial^h_{PQ} F_k(M)\ra F_k(M))\simeq (\partial M\subset M)}$. In the other cases, the pair $\Delta_{PQ,k}(M)\subset M^k$ satisfies the hypotheses of \cite[3.1.1]{QuinnI}, since it is a closed locally triangulable subspace (see above) of codimension $\geq 4$ of a manifold in dimension $\geq 8$ containing the boundary, so has empty interior, is an ANR by \cite[Theorem 0.6.1]{DavermanVenema}, and 1-LC by Proposition 1.3.3 loc.cit.. As a result, $\Delta_{PQ,k}(M)\subset M^k$ has a \emph{mapping cylinder neighbourhood $N$} in the sense of \cite[p.~285]{QuinnI}: this is a codimension $0$ submanifold $N\subset M^k$ that is closed as a subspace and contains $\smash{\Delta_k^{PQ}(M)}$, together with a map $\pi\colon \partial N\ra \smash{\Delta_k^{PQ}(M)}$ such that $N\cong \cyl(\pi)$ relative to $\smash{\Delta_k^{PQ}(M)}$. The latter implies the second equivalence in
	\begin{equation}\label{equ:mapping-cylinder-pairs-latching}\hspace{0.3cm}
		\smash{\big(\partial^h_{PQ} F_k(M)\ra F_k(M)\big)\simeq \big(\interior(N)\backslash \Delta_{PQ,k}(M)\subset F_k(M)\big)\simeq \big(\partial(M^k\backslash \interior(N))\subset M^k\backslash \interior(N)\big)}
	\end{equation} 
	and the first equivalence follows as for \eqref{equ:latching-regular-nghbd} (in fact, the middle pair in \eqref{equ:mapping-cylinder-pairs-latching} is isotopy equivalent to the preimage of the right-hand side of \eqref{equ:latching-regular-nghbd} under the quotient map $F_k(M)\ra C_k(M)$, by the above mentioned uniqueness for regular open neighbourhoods.) Note that if $M$ is compact, then so is the right-hand side of \eqref{equ:mapping-cylinder-pairs-latching}. The reason why we cannot apply the same strategy to obtain a similar model for $\smash{\partial^h_{PQ} C_k(M)\ra C_k(M)}$ as the boundary inclusion of a manifold without additional work is that in contrast to \cite[Proposition 1.6]{SiebenmannGuillouHahl} (which only provides an \emph{open} neighbourhood), the result \cite[3.1.1]{QuinnI} assumes the surrounding space $M^k$ to be a manifold (which $M^k_{\Sigma_k}$ is typically not).
\end{enumerate}
\end{rem}

\begin{ex}\label{ex:boundary-for-discs-bimodule}Before turning to the proof of \cref{lem:latch-match-bordism}, we discuss a generalisation of \cref{ex:boundary-for-discs} for bordisms: consider a bordism $M\in\ncBordInf^t(d)_{P,Q}$ of the form 
\begin{equation}\label{equ:free-bordisms}
	\smash{M=\Big(P\times[0,1)\sqcup S\times\bfR^d\sqcup (-1,0]\times Q\Big)\colon P\leadsto Q\quad\text{for }S\in\Fin.}
\end{equation}
In this case the map $\smash{\ev_1\colon \partial^h_{PQ} C_k(M)\ra C_k(M)}$ lands in the components $\smash{C^{\ninj}_{PQ,k}(M)\subset C_k(M)}$ of those configurations for which one point lies in $\partial M=P\times\{0\}\cup \{0\}\times Q$ or one of the copies of $\bfR^d$ contains at least two points. A minor variant of the argument in \cref{ex:boundary-for-discs-bimodule} (use the straight-line paths in $P\times [0,1)\sqcup (-1,0]\times Q$ to the boundary) shows that the induced map \begin{equation}\label{equ:latch-noninj-bdy}
	\smash{\ev_1\colon \partial^h_{PQ}C_k(S\times \bfR^d)\xlra{\simeq} C^{\ninj}_{PQ,k}(M)}
\end{equation}
is homotopy equivalence. This gives rise to an equivalence $\smash{(\partial^h_{PQ}C_k(-)\ra C_k(-))\simeq (C_{k,PQ}^{\ninj}(-)\subset C_k)}$ of functors $\smash{(\DiscInf^{\ot}_d)_{P,Q}\ra \cS^{[1]}}$ where $\smash{(\DiscInf^{\ot}_d)_{P,Q}\subset \ncBordInf^{\ot}(d)_{P,Q}}$ is the full subcategory spanned by the bordisms of the form \eqref{equ:free-bordisms}.\end{ex}

\begin{proof}[Proof of \cref{lem:latch-match-bordism}]Adopting the notation from \cref{sec:higher-morita-layers} and abbreviating $A\coloneqq E_{P\times I}$, $B\coloneqq E_{Q\times I}$, arguing as in the first part of the proof of \cref{lem:latch-match-manifolds} and using the first equivalence from \cref{lem:match-latch-module-simpler}, we see that $\Lambda^c_{E_{P\times I},E_{Q\times I}}\colon \BMod_{A,B}(\RMod_k(E_d^{\ot}))\ra\cS^{[2]}$ sends $X$ to 
\[
	\big((U_{AB}\iota_{AB!}^{\ }\iota_{AB}^*)(X)(\underline{k}\times\bfR^d)_{\Aut(\underline{k}\times\bfR^d)}\ra U_{AB}(X)(\underline{k}\times\bfR^d)_{\Aut(\underline{k}\times\bfR^d)}\ra \iota_*\iota^*U_{AB}(X)(\underline{k}\times\bfR^d)_{\Aut(\underline{k}\times\bfR^d)}\big).
\]
For $X=E_{M}\in \BMod_{A,B}(\RMod_k(E_d^{\ot}))$, we have $U_{AB}(E_M)=\Emb(-,M)\in \PSh(\DiscInf_{\le k})$, so we can use the identification of the right-hand map with $C_k(M)\ra [\holim_{S\subsetneq\underline{k}}F_S(M)]_{\Sigma_k}$ from the proof of \cref{lem:latch-match-manifolds}. This leaves us with identifying the left hand map with $\smash{\partial^h_{PQ} C_k(M)\ra C_k(M)}$. As part of the proof of \cref{thm:layers-tower-module}, we identified this map with respect to the equivalence \[\smash{\BMod_{E_{P\times I},E_{P\times I}}(\RMod_\bullet(E_d^{\ot})\simeq \PSh(\BMod^{\frep}_{E_{P\times I},E_{Q\times I}}(\RMod(E_d^{\ot}))_{\le\bullet})}\] from \cref{lem:bimodule-tower-is-presheaf-tower} (induced by the restricted Yoneda embedding) with the map
\begin{equation}\label{equ:latch-j
	}j_{!}j^{*}(X)(\underline{k}\times\bfR^d)_{\Aut(\underline{k}\times\bfR^d)}\lra X(\underline{k}\times\bfR^d)_{\Aut(\underline{k}\times\bfR^d)}
\end{equation}
induced by the counit of $j_!\dashv j^*$ where $j$ is the inclusion 
\[
	\BMod^{\frep}_{E_{P\times I},E_{Q\times I}}(\RMod(E_d^{\ot}))_{\le k-1}\subset \BMod^{\frep}_{E_{P\times I},E_{Q\times I}}(\RMod(E_d^{\ot}))_{\le k}.
\] 
In the proof of \cref{lem:tower-bdbw-boundary} we gave geometric description of this inclusion: the functor $\smash{E}$ restricts to an equivalence of towers $\smash{\DiscInf^{\ot}_{P,Q,\le \bullet}\simeq \BMod^{\frep}_{E_{P\times I},E_{Q\times I}}(\RMod(E_d^{\ot}))_{\le \bullet}}$; see \eqref{equ:step-in-bimodule-filtration-as-discs}. Once rephrased like this, the identification of \eqref{latch-match-manifolds} for $\smash{X=E_M=\Emb_{\partial}(-,M)\in \PSh(\DiscInf^{\ot}_{P,Q})}$ follows by making minor adaptations to the argument in the proof of \cref{lem:latch-match-manifolds} to take the collars $P\times[0,1)$ and $(-1,0]\times Q$ into account (the role of $\iota$ is played by $j$ and the role of \eqref{equ:latch-noninj} by \eqref{equ:latch-noninj-bdy}).
\end{proof}

\begin{cor}\label{cor:layer-description-explicit} Fix a topological $d$-dimensional bordism $M\colon P\leadsto Q$.
\begin{enumerate}[leftmargin=*]
	\item\label{enum:layer-corollary-geometric} For another bordism $N\colon P\leadsto Q$ and a point $\smash{e\in T_{k-1}\Emb^\ot_\partial(M,N)}$, we have a natural equivalence 	
	\[
		\hspace{0.7cm}\fib_e\left(\begin{tikzcd}[row sep=0.3cm]T_{k}\Emb^\ot_\partial(M,N)\dar\\ T_{k-1}\Emb^\ot_\partial(M,N)\end{tikzcd}\right)\simeq \left\{\begin{tikzcd}[row sep=0.3cm, column sep=0.3cm]
		\partial^h_{PQ} C_k(M) \rar\dar & C_k(N)\dar\\
		C_k(M)\rar\arrow[ur,dashed]& \left[\lim_{S\subsetneq \underline{k}} F_S(N)\right]_{\Sigma_k}
		\end{tikzcd}\right\}
	\]
	of spaces. Here the right-hand side is the space of lifts in the commutative square with vertical arrows as in \cref{lem:latch-match-bordism} and horizontal arrows induced by $e$. 
	\item\label{enum:layer-corollary-general} More generally, for $\smash{X\in \ModInf_{k}(E_d^{\ot})_{E_{P\times I},E_{Q\times I}}}$ and $\smash{e\in \Map_{ \ModInf_{k-1}(E_d^{\ot})_{E_{P\times I},E_{Q\times I}}}(E_M,X)}$, we have natural equivalence 	
	\[
		\hspace{0.3cm}\fib_e\left(\begin{tikzcd}[row sep=0.3cm]\Map_{ \ModInf_{k}(E_d^{\ot})_{E_{P\times I},E_{Q\times I}}}(E_M,X)\dar\\ \Map_{ \ModInf_{k-1}(E_d^{\ot})_{E_{P\times I},E_{Q\times I}}}(E_M,X)\end{tikzcd}\right)\simeq \left\{\begin{tikzcd}[row sep=0.3cm, column sep=0.3cm]
		\partial^h_{PQ} C_k(M) \rar\dar & X(\underline{k}\times\bfR^d)_{\Top(d)\wr\Sigma_k}\dar\\
		C_k(M)\rar\arrow[ur,dashed]& \left[\lim_{S\subsetneq \underline{k}} X( S\times\bfR^d)\right]_{\Top(d)\wr\Sigma_k}
		\end{tikzcd}\right\}
	\]
	of spaces. Here the right-hand side is the space of lifts in the commutative square with left vertical arrow as above, right vertical arrow induced by restriction, and horizontal arrows induced by $e$. 
\end{enumerate}
The analogous statements hold in the smooth setting.
\end{cor}

\begin{proof}We abbreviate $A\coloneqq E_{P\times I}$ and $B\coloneqq E_{P\times I}$ for brevity. The first item is the special case $X=E_N$ of the second item. To prove the second item, we take mapping spaces from $E_M$ to $X$ in the pullback square in \eqref{equ:geometric-layer-diagram-bordism} and using the induced equivalence on vertical fibres together with \eqref{equ:geometric-layer-diagram-bordism} and the definition of $\Lambda^c_{AB}$ and $\Omega^c_{AB}$ from \cref{sec:higher-morita-layers} to see the fibre in question is equivalent to the space of fillers in the following diagram of spaces
\begin{equation}\label{equ:layer-filler-symmetric}
	\begin{tikzcd}[row sep=0.3cm]
	\partial^h_{PQ} C_k(M)\dar\rar & \underset{\BTop(d)\wr\Sigma_k}\colim\,(\nu^*U_{AB}\iota_{AB!}\iota^*_{AB}(X))\dar\\
	C_k(M)\dar\rar[dashed] & \underset{\BTop(d)\wr\Sigma_k}\colim\,(\nu^*U_{AB}(X))\dar\\
	\left[\lim_{S\subsetneq \underline{k}} F_S(M)\right]_{\Sigma_k}\rar&\underset{\BTop(d)\wr\Sigma_k}\colim\,(\nu^*U_{AB}\iota_{AB*}\iota^*_{AB}(X)).
\end{tikzcd}
\end{equation}
Combining \cref{lem:match-latch-module-simpler} with \cref{lem:simplified-right-kan} and tracing through the definition, we see that the right bottom entry is equivalent to $\smash{ [\lim_{S\subsetneq \underline{k}} X( S\times\bfR^d)]_{\Top(d)\wr\Sigma_k}}$ and the right middle entry to $X(\underline{k}\times\bfR^d)_{\Top(d)\wr\Sigma_k}$. Rewriting the lifting problem in \eqref{equ:layer-filler-symmetric} less symmetrically by composing the upper right and lower right part of the diagram, the claim follows. The proof in the smooth case is identical.
\end{proof}

\begin{rem}[Comparison to Weiss' layer description]
The fibres of $C_k(N)\ra [\lim_{S\subsetneq \underline{k}} F_S(N)]_{\Sigma_k}
$ are the total homotopy fibres of the cubical diagram $\underline{k}\supset S\mapsto F_S(N)$, so by taking pullbacks of the outer lower right triangle in the square in \cref{cor:layer-description-explicit} \ref{enum:layer-corollary-geometric}, we see that the space of lifts in this result is equivalent to a space of relative sections of a map to $C_k(M)$ whose fibres are the total homotopy fibres of $\underline{k}\supset S\mapsto F_S(M)$, relative to a given section defined on $\smash{\partial^h_{PQ} C_k(M)}$. By \cref{rem:alternative-models-layers} \ref{enum:latch-fm-bdy-topological} the latter can be informally described as the subspace of the unordered configuration space $C_k(M)$ of configurations that ``are near the thick diagonal or near the boundary''; this uses that any open neighbourhood of $\smash{\Delta_{PQ,k}(M)}$ contains a regular open neighbourhood. This recovers Weiss' original identification of the fibres of the maps $T_{k}\Emb_\partial(M,N)\ra T_{k-1}\Emb_\partial(M,N)$ from \cite[Theorem 9.2, Example 10.3]{WeissEmbeddings}, by an independent proof.
\end{rem}

\begin{rem}[Layers in manifold calculus]\label{rem:layers-manifold-calculus}\cref{cor:layer-description-explicit} \ref{enum:layer-corollary-general} also yields a description of the layers in manifold calculus: recall from \cref{sec:manifold-calc} that the manifold calculus tower for a presheaf $F\in \PSh(\ncBordInf^\ot(d)_{P,Q})$ for $c\in\{\oo,\ot\}$ is given by the tower of mapping spaces induced by $\smash{\ModInf_{\bullet}(E_d^{c})_{E_{P\times I},E_{Q\times I}}}$ from $E_M$ to $\lvert F\rvert_E$, so applying \cref{cor:layer-description-explicit} \ref{enum:layer-corollary-general} with $X=\lvert F\rvert_E$ and using the natural equivalence $\lvert F\rvert_E(S\times\bfR^d)\simeq F(P\times [0,1)\sqcup S\times\bfR^d\sqcup Q\times(-1,0])$ resulting from the commutativity of the final diagram in the proof of \cref{lem:tower-bdbw-boundary}, we see that the fibres of the map $T_k^cF(M)\ra T_{k-1}^cF(M)$ are spaces of sections relative to $\smash{\partial^h_{PQ} C_k(M)}$ of a map to $C_k(M)$ whose fibres are the total homotopy fibres of the cubical diagram $\smash{\underline{k}\supset S\mapsto F\big(P\times [0,1)\sqcup S\times\bfR^d\sqcup (-1,0]\times Q\big)}$.
\end{rem}

\begin{ex}[Second layer]We conclude this subsection with a discussion of the second layer. We will not rely on the geometric description resulting from \cref{lem:latch-match-manifolds} for this, but instead work with the abstract layer description from \cref{sec:layers} directly. We discuss the case without boundary and for topological manifolds for simplicity, but the smooth case and the case with boundary is similar. The variant of the layer-pullback square we discuss is the left square in \cref{rem:mod-out-ocol-square} which, when restricted to unital right-modules and using the equivalence $\RMod^\un_1(E_d^t)\simeq\cS_{/\BTop(d)}$ from \cref{thm:bottom-layer}, specialises to a pullback square of categories
\begin{equation}\label{equ:geometric-layer-diagram-single}
	\begin{tikzcd}[row sep=0.5cm]
	\RMod^\un_2(E_d^t)\dar\rar{\Lambda^c}&\PSh(B\Sigma_2)^{[2]}\dar\\
	\cS_{/\BTop(d)}\rar{\Omega^c\dar}&\PSh(B\Sigma_2)^{[1]}.
	\end{tikzcd}
\end{equation}
Using that $(E_d^t)^\col\wr\Sigma_2$ is equivalent to $B(\Top(d)^{2}\rtimes\Sigma_2)$ (see the proof of \cref{lem:latch-match-manifolds}), we see from \cref{ex:snd-layer-abstract} that the top row $\Lambda^c$ sends $X\in \RMod_2(E_d^\ot)$ to the clockwise composition in the following commutative diagram of $\Sigma_2$-spaces induced by the functoriality of $X$
\begin{equation}\label{equ:second-layer-square}\begin{tikzcd}
\big(\Emb^t(\bfR^d\sqcup\bfR^d,\bfR^d)_{\Top(d)^{2}}\big)\times_{\Top(d)}X(\bfR^d)\rar\dar&X(\bfR^d\sqcup\bfR^d)_{\Top(d)^{\times 2}}\dar\\
\big(\Emb^t(\bfR^d,\bfR^d)_{\Top(d)}\big)^{2}\times_{\Top(d)}X(\bfR^d)\simeq X(\bfR^d)_{\Top(d)}\rar{\Delta}&X(\bfR^d)_{\Top(d)}\times X(\bfR^d)_{\Top(d)}
\end{tikzcd}
\end{equation}
where $\Delta$ is the diagonal. The $\Sigma_2$-action on this square is induced by swapping the two disjoint summands in $\bfR^d\sqcup \bfR^d$. The top-left corner can be simplified by using $\Top_0(d)\simeq \Top(d)$ as in the proof of \cref{prop:bottom-manifold-layer} and considering the $\Sigma_2\times\Top_0(d)$-equivariant equivalences
\[
\Emb^t(\bfR^d\sqcup\bfR^d,\bfR^d)_{\Top(d)^{\times 2}}\xlla{\simeq}\Emb^t(\bfR^d\sqcup\bfR^d,\bfR^d)_{\Top_0(d)^{\times 2}}\xlra{\simeq} F_2(\bfR^d)
\]
where the first one is induced by the inclusion $\smash{\Top_0(d){\subset}\Top(d)}$ and the second one by restriction along $\{0\}\sqcup\{0\}\subset \bfR^d\sqcup\bfR^d$. As a $\Top_0(d)$-space we moreover have an equivalence $F_2(\bfR^d)\simeq \bfR^d\backslash\{0\}$ by sending $x\in \bfR^d\backslash\{0\}$ to $(0,x)\in F_2(\bfR^d)$. Using these equivalences and writing $\lvert X\rvert\coloneqq X(\bfR^d)_{\Top(d)}$, we that the left vertical map in \eqref{equ:second-layer-square} is given by the spherical fibration $\pi\colon S(\xi_{\lvert X\rvert})\ra \lvert X\rvert$ associated to the $\bfR^d$-bundle $\xi_{\lvert X\rvert}$ over $\lvert X\rvert$ induced by the $\Top(d)$-action on $X(\bfR^d)$, with the fibrewise $\Sigma_2$-action given by permuting points in $F_2(\bfR^d)\simeq \bfR^d\backslash\{0\}$. The commutative square \eqref{equ:second-layer-square} thus takes the form 
\[\begin{tikzcd}
S(\xi_{\lvert X\rvert})\rar\arrow[d,"\pi",swap]&X(\bfR^d\sqcup\bfR^d)_{\Top(d)^{\times 2}}\dar\\
\lvert X\rvert\rar{\Delta}&\lvert X\rvert^2.
\end{tikzcd}
\]
Informally speaking, this discussion shows that $\RMod_2^\un(E_d^t)$ is equivalent to the category of
\begin{enumerate}[leftmargin=*]
\item\label{enum:snd-layer-space} spaces $\lvert X\rvert $ with an $\bfR^d$-bundle $\xi_{\lvert X\rvert}$ over it, together with 
\item\label{enum:snd-layer-ext} an extension in $\Sigma_2$-spaces of the composition $(\Delta\circ\pi)\colon S(\xi_{\lvert X\rvert})\ra \lvert X\rvert \ra \lvert X\rvert^2$ of the projection of the associated spherical fibration with the diagonal, to a commutative square in $\Sigma_2$-spaces.
\end{enumerate}
For the $2$-truncated presheaf $E_M$ of a $d$-manifold $M$, \ref{enum:snd-layer-space} remembers the underlying homotopy type of $M$ and the topological tangent bundle $TM$ over it (see \cref{prop:bottom-manifold-layer}), and \ref{enum:snd-layer-ext} remembers the factorisation of $(\Delta\circ \pi)\colon S(TM)\ra M$ over the $2$-point configuration space $F_2(M)=M\times M\backslash \Delta$ induced by the choice of a $\Sigma_2$-equivariant tubular neighbourhood of $\Delta\subset M\times M$ (see \cite[Section 4]{KKM}). If $M$ is closed, then the square in \ref{enum:snd-layer-ext} is a pushout and yields the Poincaré embedding structure on the diagonal of the underlying Poincaré complex of $M$ induced by the manifold structure on $M$.
\end{ex}

\subsection{Smoothing theory for embedding calculus} \label{sec:smoothing-embcalc}
Consider the commutative diagram of symmetric monoidal double categories for $1\le k\le\infty$
\begin{equation}\label{equ:smoothing-theory-diagram}
	\begin{tikzcd}[row sep=0.3cm]
	\ncBordInf^{\oo}(d)\dar\rar{E}&\ModInf_k^{\un}(E^{\oo}_d)\dar\rar&\Cosp(\cS_{/\BO(d)})\dar\\
	\ncBordInf^t(d)\rar\rar&\ModInf_k^{\un}(E^{\ot}_d)\rar&\Cosp(\cS_{/\BTop(d)}).
	\end{tikzcd}
\end{equation}
It follows from the general smoothing theory result  for the towers $\ModInf_k^{\un}(\cO)$ for unital operads $\cO$ (see \cref{thm:smoothing-theory-morita}) that the right square is a pullback. As part of \cref{thm:smoothing-theory-emb-calc} below, we will see that classical smoothing theory for manifolds on the other hand implies that the outer square induces for $d\neq4$ pullbacks on all mapping categories, so the same is true for the left-hand square.

\begin{thm}[Smoothing theory for embedding calculus]\label{thm:smoothing-theory-emb-calc}\, 
\begin{enumerate}[leftmargin=*]
	\item \label{enum:smoothing-theory-emb-calc-i} The right-hand square in \eqref{equ:smoothing-theory-diagram} is a pullback of symmetric monoidal double categories.
	\item \label{enum:smoothing-theory-emb-calc-ii} For $d \neq 4$ the outer (and hence the left) square in \eqref{equ:smoothing-theory-diagram} induce pullbacks on all mapping categories.
\end{enumerate}
\end{thm}

\begin{proof} As already pointed out above, \cref{enum:smoothing-theory-emb-calc-i} follows by specialising \cref{thm:smoothing-theory-morita} to the map of groupoid-coloured operads $E_d^{\oo}\ra E_d^{\ot}$ (which is an operadic right-fibration as a result of \cref{lem:pullbacks-groupoid-type}) and using the straightening equivalence $\PSh(\cG)\simeq \cS_{/\cG}$ for groupoids $\cG$. 

\medskip

\noindent For \ref{enum:smoothing-theory-emb-calc-ii} we have to show that the outer square induces a pullback on mapping categories between smooth $(d-1)$-manifolds $P$ and $Q$ without boundary, thought of as objects in $\ncBordInf^{\oo}(d)$. Essential subjectivity of the map from the top-left corner to the pullback of the remaining corners is equivalent to showing that for every topological $d$-manifold $M$ with boundary $P\sqcup Q$ any lift of its tangent classifier $M \to \BTop(d)$ along $\BO(d)\ra\BTop(d)$ extending the lifts on $\partial M\cong P\sqcup Q$ induced by the given smooth structure on $P$ and $Q$ is induced by a smooth structure on $M$ that extending the given one on $\partial M$. For $d > 4$ this is a consequence of the first part of \cite[Theorem IV.10.1]{KirbySiebenmann}, and for $d < 4$ there is nothing to prove since $\BO(d) \to \BTop(d)$ is an equivalence (\cite[V.5.5.9]{KirbySiebenmann} for $d < 3$, and \cite[p.~605]{HatcherSmale} for $d = 3$) and every manifold has admits a smooth structure extending any given one on the boundary (combine \cite[Chapters 8, 35]{Moise} with \cite[Theorem 5.3, p.\,122-123]{HirschMazur}). Fully faithfulness is equivalent to proving that for smooth $d$-manifolds $M$ and $N$ with identified boundary $\partial M\cong\partial N\cong P\sqcup Q$, the following square of embedding spaces spaces is a pullback
\begin{equation}\label{equ:embeddings-smoothing-theory}
	\begin{tikzcd}[column sep=0.5cm, row sep=0.5cm] 
	\Emb^\oo_\partial(M,N) \rar \dar & \Emb^t_\partial(M,N) \dar \\
	\Map^{/\BO(d)}_{\partial}(M,N) \rar & \Map^{/\BTop(d)}_{\partial}(M,N);
	\end{tikzcd}
\end{equation}
here the bottom row is as in \cref{sec:sm-top-emb-calc}. The first step is to reduce to a situation where $M$ is compact: pick an exhaustion $K_0\subset K_1\subset \cdots \subset M$ by compact submanifolds transverse to $\partial M$ such that the intersections $\partial_0 K_i\coloneqq \partial M\cap K_i\subset \partial K_i$ give an exhaustion of $\partial M$ by compact submanifolds. Consider the squares induced by restriction
\[
	\begin{tikzcd}[column sep=0.5cm, row sep=0.5cm]
	\Emb^\oo_{\partial}(M,N) \rar{\simeq} \dar & \lim_i \Emb^\oo_{\partial_0K_i}(K_i,N) \dar \\
	\Emb^\ot_{\partial}(M,N) \rar{\simeq} & \lim_i \Emb^\ot_{\partial_0 K_i}(K_i,N) \end{tikzcd} \quad \begin{tikzcd}[column sep=0.5cm, row sep=0.5cm] 
	\Map^{/\BO(d)}_{\partial}(M,N) \rar{\simeq} \dar & \lim_i \Map^{/\BO(d)}_{\partial_0K_i}(K_i,N) \dar \\
	\Map^{/\BTop(d)}_{\partial}(M,N) \rar{\simeq} & \lim_i \Map^{/\BTop(d)}_{\partial_0K_i}(K_i,N) \end{tikzcd}
\]
where the sub- or superscript $\partial_0 K_i$ indicates that we are considering spaces of embeddings (or maps over $\BO(d)$ or $\BTop(d)$) that agree on $\partial_0 K_i\subset \partial M$ with the respective embedding (or map) induced by the identification $\partial M\cong \partial N$. All horizontal maps in these squares are equivalences: for the ones in the right square this follows by using that the union of the $K_i$ respectively the $\partial_0 K_i$ express $K_i$ respectively $\partial_0K_i$ as a colimit in $\cS_{/\BO(d)}$ and $\cS_{/\BTop(d)}$, for those in the left square one can argue as follows: modelling these embedding spaces by simplicial sets, restriction along $K_{i-1}\subset K_i$ is a Kan fibration by isotopy extension, so the limit $ \lim_i \Emb^c_{\partial_0K_i}(K_i,N)$ for $c\in\{\oo,\ot\}$ can be computed as the strict limit of simplicial sets, which is isomorphic to $\Emb^c_{\partial_0K_i}(K_i,N)$ as a simplicial set. This reduces the claim to showing that \eqref{equ:embeddings-smoothing-theory} is a pullback after replacing $M$ by $K_i$ for a fixed $i$ and the $\partial$-subscripts with $\partial_0 K_i$. The latter can be further reduced to the case $\partial_0 K_i=\varnothing$, by choosing a collar $\partial_0K_i\times [0,1]\subset K_i$ and using the fibre sequences (the upper one uses isotopy extension again)
\[
	\begin{gathered}\Emb^c_{\partial_0K_i}(K_i,N)\lra\Emb^c(K_i,N)\lra \Emb^c(\partial_0K_i\times [0,1],N)\quad\text{for }c\in\{\oo,\ot\}\\
	\Map^{/\BG}_{\partial_0K_i}(K_i,N)\lra \Map^{/\oB G}(K_i,N)\lra \Map^{/\BG}(\partial_0 K_i\times I,N)\quad\text{for }G\in \{\oO(d),\Top(d)\}.\end{gathered}
\]
It remains to show that the version of the square \eqref{equ:embeddings-smoothing-theory} where the source manifold $M$ is compact and there are no boundary conditions, is a pullback, which is a well-known consequence of smoothing theory (see e.g.\,\cite[Corollary 3.2]{BurgheleaLashofHomotopy}).\end{proof}

Writing $\smash{\BordInf^{c}(d)\subset \ncBordInf^{c}(d)}$ for the levelwise full subcategory obtained by restricting to \emph{compact} bordisms for $c\in\{\oo,\ot\}$ (see \cite[3.\circled{7}.1, 4.1.2]{KKDisc} for the case $c=\oo$), we have a pullback decomposition $\smash{\BordInf^{\oo}(d)\simeq \BordInf^t(d)\times_{\ncBordInf^t(d)}\ncBordInf^{\oo}(d)}$ in $\CMon(\CCat(\Cat))$ which we may combine with the pullback in \cref{thm:smoothing-theory-emb-calc} to conclude:

\begin{cor}\label{thm:smoothing-theory-emb-calc-compact}For $1\le k\le \infty$, the commutative square of symmetric monoidal double categories
\[
	\begin{tikzcd}[row sep=0.3cm]
	\BordInf^{\oo}(d)\dar\rar &\ModInf_k(E^{\oo}_d)\dar\\
	\BordInf^t(d)\rar\rar&\ModInf_k(E^{\ot}_d)
	\end{tikzcd}
\]
induces for $d\neq4$ a pullback on all mapping categories.
\end{cor}

\subsubsection{Topological $\DiscInf$-structure spaces} \label{sec:s-disc}\cref{thm:smoothing-theory-emb-calc-compact} has a consequence for the $\DiscInf$-structure spaces studied in \cite{KKDisc}. To explain this, recall from \cite[4.5.1]{KKDisc} that for a closed smooth $(d-1)$-manifold $P$, the \emph{$\DiscInf$-structure space} $\smash{S^{\DiscInf,\oo}_P(X)}$ of a right-$E_{P\times I}$-module $X\in\RMod(E^{\oo}_d)$ is 
\[
	\smash{S^{\DiscInf,\oo}_P(X)\coloneqq \fib_{X}\big(E\colon \BordInf^{\oo}_{P}\ra\ModInf_k(E^{\oo}_d)^{\simeq }_{E_{P\times I}}\big)\in\cS}.
\]
This includes the $\DiscInf$-structure space $S_\partial^\oo(M)$ of a compact smooth $d$-manifold $M$ from 4.5.2 loc.cit.\,as the case $\smash{S_\partial^\oo(M)=S^{\DiscInf,\oo}_{\partial M}(E_M)}$. The same make sense in the topological category: for a closed topological $(d-1)$-manifold $P$ and a right-$E_{P\times I}$-module $X\in\RMod(E^{\ot}_d)$ we set
\[
	\smash{S^{\DiscInf,\ot}_P(X)\coloneqq \fib_{X}\big(E\colon \BordInf^{\ot}(d)_{P}\ra\ModInf_k(E^{\ot}_d)^{\simeq }_{E_{P\times I}}\big)\in\cS}
\] 
and $\smash{S^{\DiscInf,\oo}_\partial(M)\coloneqq S^{\DiscInf,\oo}_{\partial M}(E_M)}$ for  compact topological $d$-manifolds $M$. Commutativity of the square in \cref{thm:smoothing-theory-emb-calc-compact} induces comparison maps on horizontal fibres \[\smash{S^{\DiscInf,\oo}_P(X)\lra S^{\DiscInf,\ot}_P(X)}\] which are equivalences for $d\neq4$ as a result of \cref{thm:smoothing-theory-emb-calc-compact}. In particular, we get:

\begin{cor}\label{thm:sdisc-smooth-vs-top} For a compact smooth manifold $M$ of dimension $d \neq 4$, the comparison map
\[
	S_{\partial}^{\DiscInf,\oo}(M)\lra S_{\partial}^{\DiscInf,\ot}(M)
\]
between the smooth and topological $\DiscInf$-structure spaces of $M$ is an equivalence.
\end{cor}

\begin{rem}\, 
\begin{enumerate}[leftmargin=*]
	\item One consequence of \cref{thm:sdisc-smooth-vs-top} is that $S_{\partial}^{\DiscInf,\oo}(M)$ is independent of the smooth structure of $M$ for $d\neq 4$, but this was already known as a result of the $2$-type invariance of \cite[Theorem A]{KKDisc} (see the discussion below the statement of the theorem).
	\item It follows from \cite[Theorem B]{KnudsenKupers} that the map in \cref{thm:sdisc-smooth-vs-top} is for $d=4$ not always an equivalence (c.f.\,Item (iii) of the remark following \cite[Theorem C]{KKDisc}).
\end{enumerate}
\end{rem}

\subsection{Particle embedding calculus and tangential structures}\label{sec:particle} We have considered two tangential structures for the $E_d$-operad in this section so far, $\oo\colon \BO(d)\ra\BAut(E_d)$ and $\ot\colon \BTop(d)\ra\BAut(E_d)$, due to their relationship to manifold theory via the functors $\smash{E\colon \ncBordInf^c(d)\ra \ModInf_\bullet(E_d^{c})}$ for $c \in \{\oo,\ot\}$. One of the strengths of our setup, however, is that it allows other tangential structures for $E_d$ (and its truncations) that have no direct geometric analogue. To explain this, we develop the following commutative diagram of towers of symmetric monoidal double categories 
\begin{equation}\label{equ:big-particle-calculus-diagram}
	\begin{tikzcd}[row sep=0.3cm, column sep=0.3cm]
	|[label={[label distance=-2mm]-45:\ulcorner_{\text{mc},d\neq 4}}]|\ncBordInf^\oo(d)\rar\dar \arrow[dr, phantom, "\circled{1}"]&\ncBordInf^\ot(d)\dar &&\\
	|[label={[label distance=-2mm]-45:\ulcorner}]|\ModInf^\un_\bullet(E_d^\oo)\arrow[dr, phantom, "\circled{2}"]\dar\rar&|[label={[label distance=-2mm]-45:\ulcorner}]|\ModInf^\un_\bullet(E_d^\ot)\arrow[dr, phantom, "\circled{3}"]\rar\dar &|[label={[label distance=-2mm]-45:\ulcorner}]|\ModInf^\un_\bullet(E_d^p)\arrow[dr, phantom, "\circled{4}"]\dar\rar&\ModInf^\un_\bullet(E_{d,\le \bullet}^p)\dar\\
	\Cosp(\cS_{/\BO(d)})\rar&\Cosp(\cS_{/\BTop(d)})\rar &\Cosp(\cS_{/\BAut(E_d)})\rar&\Cosp(\cS_{/\BAut(E_{d,\le \bullet})})
\end{tikzcd}
\end{equation}
hinted at in the introduction. Here is the construction:
\begin{itemize}[leftmargin=0.5cm]
	\item Whenever an entry in \eqref{equ:big-particle-calculus-diagram} does not feature a $\bullet$-sign, it is considered as a constant tower.
	\item The square $\circled{1}$ is the content of \cref{thm:extension-e-topological} and it induces a pullback (of towers, i.e.\,levelwise) on all mapping categories as long as $d\neq 4$ by \cref{thm:smoothing-theory-emb-calc}. 
	\item The first three entries in the middle row are obtained by applying $\smash{\ModInf^\un_\bullet(-)}$ from \eqref{equ:naturality-morita-operadic-tower} to the operad maps $\smash{E_d^{\oo}\ra E_d^{\ot}\ra E_d^{p}\eqcolon E_d^{\id}}$ induced by the changes of tangential structures 
	\[
		\begin{tikzcd}[column sep=1cm, row sep=0.5cm]\BO(d)\arrow[dr,"\oo",swap]\rar& 	\BTop(d)\arrow[d,"\ot"]\rar&\BAut(E_d)\arrow[dl,"\id"]\\
		&\BAut(E_d)&
	\end{tikzcd}\] 
	for $E_d$. On bottom layers, the maps between these three towers induce the first three bottom vertical arrows (see \cref{thm:bottom-layer-morita}). The resulting squares $\circled{2}$-$\circled{3}$ are pullbacks by \cref{thm:smoothing-theory-morita}.
	\item The square $\circled{4}$ is obtained by also taking the truncations $\smash{E_{d,\le k}\coloneqq (E_d)_{\le k}}$ of the $E_d$-operad into account (see \cref{sec:truncation-operads}), namely as follows: first one considers the tower \[
		\smash{E_{d,\le \smallsquare}^p\eqcolon E_{d,\le \smallsquare}^{\id}}\in\Tow(\Opd^{\gc})
	\] 
	of operads featuring the $\smash{\id_{\BAut(E_{d,\le k})}}$-framed $\smash{E_{d,\le k}}$-operad $\smash{E^\id_{d,\le k}}$ for $1\le k\le \infty$ (see Sections \ref{sec:tangential-structures-operads}). On categories of colours, this tower gives the sequence of spaces induced by truncation
	\[
		\BAut(E_d)\simeq\BAut(E_{d,\le \infty})\lra \cdots\lra\BAut(E_{d,\le 2})\lra \BAut(E_{d,\le 1}).
	\]
	Then one applies $\ModInf_\bullet^\un(-)$ to $E_{d,\le \smallsquare}^p$ get a double tower 
	\[
		\smash{\ModInf_\bullet^\un(E_{d,\le \smallsquare}^{\id})}\in\Tow(\Tow(\CMon(\CCat(\Cat))))
	\]
	from which the top right corner in $\circled{4}$ is obtained by restricting to the diagonal $\bullet=\smallsquare$, the bottom right corner by restricting to $\bullet=1$, the top left one by restricting to $\smallsquare=\infty,$ and the bottom left one by restricting to $(\bullet,\smallsquare)=(1,\infty)$. The square $\circled{4}$ is a pullback: apply \cref{thm:smoothing-theory-morita} for fixed $\bullet=k$ to the change of tangential structure $p_k\colon \BAut(E_d)\ra \BAut(E_{d,\le k})$ for the $E_{d,\le k}$-operad and use that the map $\smash{E_d^{p}\ra E^{p_k}_{d,\le k}}$ induces an equivalence on $\ModInf_\bullet(-)$ for $\bullet\le k$, since the map $\smash{\Env(E_d^p)_{\le k}\ra\Env(E^{p_k}_{d,\le k})_{\le k}}$ is an equivalence by \cref{rem:restricted-env-depends-on-truncation} and \cref{lem:truncated-tangential}.
\end{itemize}
For smooth bordisms $M,N\coloneqq P\leadsto Q$, the diagram \eqref{equ:big-particle-calculus-diagram} induces on mapping spaces of mapping categories a diagram of towers in spaces consisting of levelwise pullbacks
\begin{equation}\label{equ:big-particle-calculus-diagram-mapping-spaces}
	\hspace{-.2cm}\begin{tikzcd}[row sep=0.1cm, column sep=0.18cm]
	|[label={[label distance=-2mm]-45:\ulcorner_{d\neq 4}}]|\Emb^{\oo}_\partial(M,N)\rar\dar \arrow[dr, phantom, "\circled{1}"]&\Emb^{\ot}_\partial(M,N)\dar &&\\
	|[label={[label distance=-2mm]-45:\ulcorner}]|T_\bullet\Emb^{\oo}_\partial(M,N)\arrow[dr, phantom, "\circled{2}"]\dar\rar&|[label={[label distance=-2mm]-45:\ulcorner}]|T_\bullet\Emb^{\ot}_\partial(M,N)\arrow[dr, phantom, "\circled{3}"]\rar\dar &|[label={[label distance=-2mm]-45:\ulcorner}]|T_\bullet\Emb^{p_\bullet}_\partial(M,N)\arrow[dr, phantom, "\circled{4}"]\dar\rar&T_\bullet\Emb^{p}_\partial(M,N)\dar\\
	\Map_\partial^{/\BO(d)}(M,N)\rar&\Map_\partial^{/\BTop(d)}(M,N)\rar &\Map^{/\BAut(E_{d})}_\partial(M,N)\rar&\Map^{/\BAut(E_{d,\le 	\bullet})}_\partial(M,N).
	\end{tikzcd}
\end{equation}
Note that, without the leftmost column, this diagram also exists for \emph{topological} bordisms.

\medskip

\subsubsection{Particle embedding calculus}\label{sec:particle-calculus-subsubsection}The rightmost upper tower $\smash{\ModInf^\un_\bullet(E_{d,\le \bullet}^p)}$ in \eqref{equ:big-particle-calculus-diagram} (as well as the induced tower $\smash{T_\bullet\Emb^{p}_\partial(M,N)}$ on mapping spaces of mapping categories featuring in \eqref{equ:big-particle-calculus-diagram-mapping-spaces}), plays a special role in our theory, since they are in a sense, the universal towers built from tangential structures for the $E_d$-operad and its truncations (see \cref{sec:terminal-operadic} for a precise statement). We refer to the towers $\smash{\ModInf^\un_\bullet(E_{d,\le \bullet}^p)}$ and $\smash{T_\bullet\Emb^{p}_\partial(M,N)}$ as \emph{particle embedding calculus}, since they are, roughly speaking, constructed purely in terms of configuration space data and no longer involve the homotopy type of $\BO(d)$ or $\BTop(d)$.
 
\begin{rem}[Configuration categories]\label{rem:config-categories}As mentioned in the introduction, particle embedding calculus is likely to be closely related to Boavida de Brito--Weiss' theory of \emph{configuration categories}. For example, the pullback square obtained by horizontally composing $\circled{2}$-$\circled{4}$ in \eqref{equ:big-particle-calculus-diagram-mapping-spaces} is very reminiscent of pullback squares in their theory \cite[Theorem 6.4]{BoavidaWeissConfiguration}. The precise relationship between particle embedding calculus and configuration categories is subject of future work.
\end{rem}

\subsubsection{Other tangential structures}\label{sec:towers-tangential-structures-for-emb}
Instead of the identity tangential structures in the rightmost column of \eqref{equ:big-particle-calculus-diagram}, we can consider any tower $\theta_\bullet\colon B_\bullet\ra\BAut(E_{d,\le \bullet})$ of tangential structures for $E_d$ as in \cref{sec:tangential-structures-operads} and form the tower $\smash{\ModInf_\bullet^\un(E^{\theta_\bullet}_{d,\le \bullet})}$. It fits into a diagram of pullbacks of towers
\begin{equation}\label{equ:bordism-cats-with-tang-structures}
	\begin{tikzcd}[row sep=0.2cm, column sep=0.5cm]
	|[label={[label distance=-2mm]-45:\ulcorner}]|  \ncBordInf^{o^*\theta}(d)\dar\rar&|[label={[label distance=-2mm]-45:\ulcorner}]|  \ncBordInf^{t^*\theta}(d)\rar\dar& |[label={[label distance=-2mm]-45:\ulcorner}]|  \ModInf_\bullet^\un(E^{\theta_\bullet}_{d,\le \bullet})\rar\dar&\Cosp(\cS_{/B_\bullet})\dar\\
	\ncBordInf^{o}(d)\rar&\ncBordInf^{t}(d)\rar&\ModInf_\bullet^\un(E^p_{d,\le \bullet})\rar&\Cosp(\cS_{/\BAut(E_{d,\le \bullet})})
	\end{tikzcd}
\end{equation}
whose rightmost square is an instance of the discussion in \cref{sec:terminal-operadic}. To explain the left part of this diagram, recall that a \emph{tangential structure for $d$-manifolds} is a map $\nu\colon X\ra\BTop(d)$. Given a $d$-manifold $M$, a $\nu$-structure on $M$ is a lift of its tangent classifier along $\nu$. The bordism category of manifolds with $\nu$-structure is the pullback (cf.\,\cite[Definition 3.11]{KKDisc}) 
\[
	\smash{\ncBordInf^{\ot}(d)^{\nu}\coloneqq \ncBordInf^{c}(d)\times_{\Cosp(\cS_{/\BTop(d)})}\Cosp(\cS_{/X})}.
\] 
The same discussion applies smoothly, where one starts with a map to $\BO(d)$. The bordism categories appearing in \eqref{equ:bordism-cats-with-tang-structures} are the special cases where $\nu$ is the pullback of the map $\theta\colon B\ra \BAut(E_d)$ (the value of $\theta_\bullet$ at $\bullet=\infty$) along $\ot\colon\BTop(d)\ra \BAut(E_d)$ (or $\oo\colon\BO(d)\ra \BAut(E_d)$ in the smooth case). The left two squares in \eqref{equ:bordism-cats-with-tang-structures} are induced by the universal property of the pullback, and the fact that they are pullbacks follows by using that $\Cosp(\cS_{/-})\colon \cS\ra\CMon(\CCat(\Cat)))$ preserves pullbacks.

\medskip

\noindent Given a $d$-dimensional bordism $M\colon P\leadsto Q$, a choice of lift $\ell_M\colon M\ra B$ of the composition of the tangent bundle $M\ra\BTop(d)$ with $\ot\colon \BTop(d)\ra\BAut(E_d)$ along $\theta\colon B\ra\BAut(E_d)$ in particular induces similar lifts $\ell_{P}$, $\ell_{Q}$ involving the $1$-stabilised tangent bundles of $P$ and $Q$. In particular, this lifts $P$, $Q$ from $\ncBordInf^\ot(d)$ to objects $(P,\ell_P)$ and $(Q,\ell_Q)$ in $\smash{\ncBordInf^{\ot^*\theta}(d)}$, and $M$ from $\smash{\ncBordInf^\ot(d)_{P,Q}}$ to $\smash{(M,\ell_M)\in \ncBordInf^{\ot^*\theta}(d)_{(P,\ell),(Q,\ell)}}$. Given another bordism $N\colon P\leadsto Q$ with a lift $\ell_N$ and an identification of the induced lift for $P$ and $Q$ with the lifts $\ell_P$ and $\ell_Q$ induced by $\ell_M$, we consider the images $\smash{E_{(P\times I,\ell_P)}}$, $\smash{E_{(Q\times I,\ell_Q)}}$, $\smash{E_{(M,\ell_M)}}$ and $\smash{E_{(N,\ell_N)}}$ under 
\begin{equation}\label{equ:e-with-tangential-structures-on-mapping-cat}
	\smash{E\colon \ncBordInf^{\ot^*\theta}(d)_{(P,\ell_P),(Q,\ell_Q)}\ra \ModInf_\bullet^\un(E^{\theta_\bullet}_{d,\le \bullet})_{E_{(P\times I,\ell_P)},E_{(Q\times I,\ell_Q)}}}
\end{equation}
and write
\begin{equation}\label{equ:emb-calc-with-tang-structure}
	T_\bullet\Emb_\partial^{\theta_\bullet}((M,\ell_M),(N,\ell_N))\coloneqq \Map_{\ModInf_\bullet^\un(E^{\theta_\bullet}_{d,\le \bullet})_{E_{(P\times I,\ell_P)},E_{(Q\times I,\ell_Q)}}}(E_{(M,\ell_M)},E_{(N,\ell_N)})
\end{equation} 
for the tower of mapping spaces between these objects. If the tangential structures $\ell_M$ and $\ell_N$ are clear from the context, we omit them in the notation. Note that the $k$th entry in \eqref{equ:emb-calc-with-tang-structure} only depends on the $k$th stage $\theta_k\colon B_k\ra\BAut(E_{d,\le k})$ of the tower of tangential structures.

\begin{ex} Here are some special cases of \eqref{equ:emb-calc-with-tang-structure}:
\begin{enumerate}[leftmargin=*]
	\item The towers in the middle row of \eqref{equ:big-particle-calculus-diagram-mapping-spaces} correspond to the towers $\theta_\bullet\colon B_\bullet\ra \BAut(E_{d,\le \bullet})$ of tangential structures given by $\BO(d)\ra \BAut(E_{d,\le \bullet})$, $\BTop(d)\ra \BAut(E_{d,\le \bullet})$, $\BAut(E_d)\ra \BAut(E_{d,\le \bullet})$, and $\BAut(E_{d,\le \bullet})\ra \BAut(E_{d,\le \bullet})$, induced by $\oo,\ot,$ and the truncation maps.
	\item Any map $\theta\colon B\ra \BAut(E_{d,\le k})$ for a fixed $k$ can viewed as a tower $B_\bullet \ra \BAut(E_{d,\le \bullet})$ of tangential structures by $B_n=B_k$ for $n\le k$ and $B_n=B_k\times_{\BAut(E_{d,\le k})} \BAut(E_{d,\le n})$ otherwise.
	\item Precomposing $\theta_\bullet \colon \BO(d)\ra \BAut(E_{d,\le \bullet})$ with a map $B\ra\BO(d)$ yields embedding calculus with tangential structures common in the theory of smooth manifolds (the topological case is similar) such as framings $B=\ast$, orientations $B=\BSO(d)$, or spin structures $B=\BSpin(d)$. Note that for framings $\smash{\fr\colon \ast\ra \BAut(E_{d,\le \bullet})}$, we have $\smash{E_{d,\le \bullet}^\fr\simeq E_{d,\le \bullet}}$.
\end{enumerate}
\end{ex}  

\begin{rem}[Layers with tangential structures]\label{sec:layers-tangential-structures} The layers in $\smash{\ModInf_\bullet^\un(E^{\theta_\bullet}_{d,\le \bullet})}$ can be analysed as follows: the bottom layer is given as $\smash{\ModInf_1^\un(E^{\theta_1}_{d,\le 1})\simeq\Cosp(\cS_{/B_1})}$ by \cref{thm:bottom-layer-morita}. Regarding the higher layers, one notes that the functor from the $k$th stage to the $(k-1)$st factors as a composition
\[
	\ModInf_k^\un(E^{\theta_k}_{d,\le k})\lra\ModInf_{k-1}^\un(E^{\theta_k}_{d,\le k})\simeq \ModInf_{k-1}^\un(E^{\theta_{k}}_{d,\le k-1})\lra \ModInf_{k-1}^\un(E^{\theta_{k-1}}_{d,\le k-1})
\] 
where the second functor fits into a pullback with $\Cosp(\cS_{/B_k})\ra \Cosp(\cS_{/B_{k-1}})$ by \cref{thm:smoothing-theory-morita} and the first functor fits (on mapping categories) into a pullback with $(0\le 2)^*\colon \cS^{[2]}\ra \cS^{[1]}$ by \cref{thm:layers-tower-module}, involving the functor \begin{equation}\label{equ:layer-lambda-with-tangential-structures}
	\smash{\Lambda^c_{E_{(P\times I,\ell_P)},E_{(Q\times I,\ell_Q)}}\colon \ModInf_k^\un(E^{\theta_k}_{d,\le k})_{E_{(P\times I,\ell_P)},E_{(Q\times I,\ell_Q)}}\lra \cS^{[2]}}.
\end{equation}
Moreover, it follows from the naturality of the higher Morita layers in \cref{sec:naturality-higher-morita} with \eqref{equ:e-with-tangential-structures-on-mapping-cat} agrees with the forgetful functor $\smash{\ncBordInf^{\ot}(d)_{P,Q}}$ followed by the functor $\smash{\Lambda^c_{E_{P\times I},E_{Q\times I}}}$ we described geometrically in \cref{lem:latch-match-bordism}. In particular, combining this discussion with \cref{cor:layer-description-explicit}, we obtain an equivalence
\begin{equation}\label{equ:first-layer-tangential-structures}
	T_1\Emb_\partial^{\theta_1}((M,\ell_M),(N,\ell_N))\simeq \Map^{/B_1}_{\partial}(M,N)
\end{equation}
and natural fibre sequences of the form
\[
	\hspace{-0.45cm}\left\{\begin{tikzcd}[row sep=0.3cm, column sep=0cm]		\partial^h_{PQ} C_k(M) \rar\dar & C_k(N)\dar\\		C_k(M)\rar\arrow[ur,dashed]& \left[\lim_{S\subsetneq \underline{k}} F_S(N)\right]_{\Sigma_k}
	\end{tikzcd}\right\}
	\ra \fib_e\left(\begin{tikzcd}[row sep=0.3cm]T_{k}\Emb^{\theta_k}_\partial((M,\ell_M),(N,\ell_N))\dar\\ T_{k-1}\Emb^{\theta_{k-1}}_\partial((M,\ell_M),(N,\ell_N))\end{tikzcd}\right)\ra \Omega\left\{\begin{tikzcd}[row sep=0.5cm, column sep=0.3cm]
	P\sqcup Q \rar{\ell_N}\arrow[d,hook] & B_k\dar\\		M\rar[swap]{\ell_M}\arrow[ur,dashed]& B_{k-1}
	\end{tikzcd}\right\}
\]
where the loop space is based at $\ell_M$ and the horizontal maps in the leftmost lifting problem depend on the point $e$ we took fibres over. 

\medskip

\noindent In the example of particle embedding calculus from \cref{sec:particle-calculus-subsubsection}, the map $B_k\ra B_{k-1}$ is given by the map $\BAut(E_{d,\le k})\ra \BAut(E_{d,\le k-1})$ induced by truncation. The latter map  can be analysed using work of Göppl--Weiss \cite{Goppl} (see also \cite[Section 2.2.5]{KKPontryagin}).
\end{rem}

\subsection{$T_k$-self-maps are equivalences}\label{sec:tk-equivalences} As a sample application of the geometric layer identification for embedding calculus from \cref{sec:layers-embcalc}, we show that the monoid $T_k\Emb^{\ot}_\partial(M,M)$ with respect to  composition is grouplike for all $1\le k\le \infty$ under some assumptions on $M$. This is reminiscent of the fact that self-embeddings of compact manifolds that fix the boundary are diffeomorphisms. We do this in the more general context of embedding calculus with tangential structures (see \cref{sec:towers-tangential-structures-for-emb}), so fix a tangential structure $\theta\colon B\ra \BAut(E_{d,\le k})$ and a $\theta$-structure $\ell_M$ on $M$ (i.e\,a lift of the composition $M\ra\BTop(d)\ra \BAut(E_{d,\le k})$).

\begin{thm}\label{thm:tk-equivalences}Fix $1\le k\le \infty$, a $d$-dimensional bordism $M\colon P\leadsto Q$, a tangential structure $\theta\colon B\ra\BAut(E_{d,\le k})$, and a $\theta$-structure $\ell_M$ on $M$. If $M$ is compact, connected, orientable, has nonempty boundary, and $\pi_1(M)$ is polycyclic-by-finite, then the monoid $T_k\Emb_\partial^{\theta}((M,\ell_M),(M,\ell_M))$ is group-like.
\end{thm}

The proof involves the following minor generalisation of \cite[Proposition 1(a)]{Hausmann}:

\begin{lem}\label{lem:hausmann-gen} Let $(X,\partial X)\in\cS^{[1]}$ be an oriented finite Poincar\'e duality pair such that $X$ is connected, $\partial X \neq \varnothing$, the group $\pi_1(X)$ is Hopfian, and the group ring $\bfZ[\pi_1(X)]$ Noetherian. Then any self-map $(\varphi,\varphi_\partial)\colon (X,\partial X)\ra  (X,\partial X)$ is an equivalence if and only if the map $\varphi_\partial\colon \partial X\ra \partial X$ is an equivalence.\end{lem}

\begin{proof}It suffices to prove that $\varphi$ is an equivalence, assuming that $\varphi_\partial$ is an equivalence. If $\varphi_\partial$ is an equivalence, then it has degree $\pm1$, so the same holds for $(\varphi,\varphi_\partial)$, since $\partial X\neq\varnothing$ and $X$ is connected. This implies that $\varphi_* \colon \pi_1(X) \to \pi_1(X)$ is surjective by the argument in \cite[Proposition 1.2]{Browder} (which goes through for pairs), so an isomorphism since $\pi_1(X)$ was assumed to be Hopfian. Writing $\Lambda \coloneqq \bfZ[\pi_1(X)]$, this implies $\oH_*(X,\varphi^*\Lambda) \cong \oH_*(X;\Lambda)$ as $\Lambda$-modules. Moreover, the map $\oH_*(\varphi;\Lambda)\colon \oH_k(X,\varphi^*\Lambda)\ra \oH_k(X;\Lambda)$ induced by $\varphi$ is for all $k$ split surjective by Poincaré duality (see e.g.\,\cite[Lemma 2.2]{WallSCM}). By finiteness of $X$, the $\Lambda$-module $\oH_k(X;\Lambda)$ is finitely generated, so as any surjection of finitely generated modules over Noetherian rings is an isomorphism, we conclude that $\oH_*(\varphi;\Lambda)$ is an isomorphism, so $\varphi$ is an equivalence by the relative Hurewicz theorem.
\end{proof}

\begin{rem}\label{ex:malcev-hall}A class of groups $\pi$ that satisfy the hypotheses of \cref{lem:hausmann-gen} (i.e.\,$\pi$ is Hopfian and $\bfZ[\pi]$ noetherian) are polycyclic-by-finite groups. Indeed, since these groups contain a polycyclic subgroup $\pi'$ of finite index, they are finitely generated and residually finite, and therefore Hopfian (see e.g.\,\cite[p.\,307, 310]{Magnus}). For the fact that the group ring $\bfZ[\pi]$ is Noetherian, see Theorem 1 (and the subsequent discussion) in \cite{Hall}.
\end{rem}

\begin{proof}[Proof of \cref{thm:tk-equivalences}]
We adopt notation from Sections \ref{sec:layers-embcalc} and \ref{sec:particle}. First we show that a self-map of ${\Lambda^c_{E_{(P\times I,\ell_{P})},E_{(Q\times I,\ell_{P})}}(E_{(M,\ell_M)})\in\cS^{[2]}}$ is an equivalence if the self-map of ${\Omega^c_{E_{(P\times I,\ell_{P})},E_{(Q\times I,\ell_{P})}}(E_{(M,\ell_M)})\in\cS^{[1]}}$ induced by composition is an equivalence. In view of \cref{lem:latch-match-bordism} and the discussion in \cref{sec:layers-tangential-structures}, we have an equivalence in $\cS^{[2]}$
\[
	\Lambda^c_{E_{(P\times I,\ell_{P})},E_{(Q\times I,\ell_{P})}}(E_{(M,\ell_M)})\simeq \big(\partial^h_{PQ}C_k(M)\ra C_k(M)\ra[\holim_{S\subsetneq \underline{k}}F_S(M)]_{\Sigma_k}\big).
\] 
The right-hand sequence was obtained from $\smash{(\partial^h_{PQ}F_k(M)\ra F_k(M)\ra\holim_{S\subsetneq \underline{k}}F_S(M))}$ by taking $\Sigma_k$-orbits, so it suffices to show that the self-map of the sequence before taking these orbits is an equivalence if the self-maps of the outer two spaces are. In fact, we will show the stronger statement that a self-map $(\varphi,\varphi_\partial)$ of $\smash{(F_k(M),\partial^h_{PQ}F_k(M))}$ is an equivalence if $\varphi_\partial$ is one. The pair $\smash{(F_k(M),\partial^h_{PQ}F_k(M))}$ is equivalent to the boundary inclusion of a compact manifold by the discussion at the end of \cref{rem:alternative-models-layers} \ref{enum:latch-fm-bdy-topological}, so in particular a finite Poincaré pair. For $d\ge3$, we can then deduce the claim as an application of \cref{lem:hausmann-gen} and \cref{ex:malcev-hall}, since $\pi_1(F_k(M))\cong  \pi_1(M)^k$ is polycyclic-by-finite if $\pi_1(M)$ is. For $d\le 2$, the manifold $M$ is either an interval, a disc, or an annulus. In these cases, $F_k(M)$ is aspherical and has Hopfian fundamental group (by the same reasoning as in \cref{ex:malcev-hall}, using that it is a finite-index subgroup of a braid group, which is finitely generated and residually finite (see e.g.\,\cite[p.~307]{Magnus}), so it suffices to show that $\pi_1(\varphi)$ is surjective. This holds by the first part of the proof of \cref{lem:hausmann-gen}.

\medskip

\noindent We now prove the claim by induction on $k$. For $k=1$ the claim follows from \eqref{equ:first-layer-tangential-structures} since any self-map of $M$ under $\partial M\cong P\sqcup Q$ is an equivalence by the same argument as above. Assuming the claim for $k-1$, we consider the pullback square of monoids resulting from \cref{thm:layers-tower-module}
\[
	\begin{tikzcd}[row sep=0.3cm]
	T_{k}\Emb^{\theta}_\partial((M,\ell_M),(M,\ell_M))\dar\rar&\End_{\cS^{[2]}}(\Lambda^c_{E_{(P\times I,\ell_{P})},E_{(Q\times I,\ell_{P})}}(E_{(M,\ell_M)}))\dar\\
	T_{k-1}\Emb^{\theta}_\partial((M,\ell_M),(M,\ell_M))\rar&\End_{\cS^{[1]}}(\Omega^c_{E_{(P\times I,\ell_{P})},E_{(Q\times I,\ell_{P})}}(E_{(M,\ell_M)})).
	\end{tikzcd}
\]
By induction the lower left corner is group-like, which implies together with the first part of the proof that the top horizontal map has image in the units. Since taking units preserves pullbacks, it follows that the top left corner is group-like, which finishes the induction.
\end{proof}

\subsection{A delooping result for embedding calculus}\label{sec:delooping} In this subsection, we utilise our perspective on embedding calculus---as a special case of our general calculus for right-modules over an operad---to prove an example of a delooping result in the context of embedding calculus. To state the precise result, fix $1\le k\le \infty$ and a tangential structure $\theta\colon B\ra\BAut(E_{d,\le k})$. The standard framing on $D^d$ inherited from the inclusion on $\bfR^d$ induces a $\fr$-structure $\ell_\st$ on $D^d$ in the sense of \cref{sec:towers-tangential-structures-for-emb} where $\fr\colon \ast\ra \BO(d)$ is the tangential structure encoding framings. A choice of basepoint $\ast\in\fib(\theta)$ induces a map of tangential structures $\fr\ra \oo^*\theta$ for $E_{d,\le k}$, and thus in particular a $\oo^*\theta$-structure $\ell_\ast$ on $D^d$. The main purpose of this subsection is to prove the following delooping result:

\begin{thm}\label{thm:delooping}For $1\le k\le \infty$, a tangential structure $\theta\colon B\ra\BAut(E_{d,\le k})$, and a basepoint $\ast\in \fib(\theta)$, there is an equivalence of $E_1$-spaces
\[
	\smash{T_k\Emb^\theta_\partial\big((D^d,\ell_\ast),(D^d,\ell_\ast)\big) \simeq \Omega^{d+1}\fib(B\xra{\theta} \BAut(E_{d,\le k}))}
\]
where the $E_1$-space structures are given by composition and the loop space structure, respectively.
\end{thm}

\begin{rem}\label{rem:relation-to-existing-delooping}As mentioned in the introduction, there are already many delooping results in the context of embedding calculus in the literature. \cref{thm:delooping} is related to several of them:
\begin{enumerate}[leftmargin=*]
	\item For the tangential structure $\theta=(\ast\ra B)$, \cref{thm:delooping} recovers the codimension zero case of \cite[Equation (13)]{DucoulombierTurchin}, and for $\theta=(\BO(d)\ra \BAut(E_d))$ it recovers \cite[Equation (12)]{DucoulombierTurchinWillwacher} and the codimension zero case of \cite[Equation (1.3)]{BoavidaWeissConfiguration}. 
	\item For the identity tangential structure $\theta=(\BAut(E_{d,\le k})\ra \BAut(E_{d,\le k}))$, \cref{thm:delooping} becomes $T_k\Emb^p_\partial(D^d,D^d)\simeq \ast$, i.e.\,the space of endomorphisms of $D^d$ in particle embedding calculus is trivial. This should be compared to Boavida de Brito and Weiss' Alexander trick for configuration categories \cite[Theorem 1.4]{BoavidaWeissConfiguration}.
\end{enumerate}
\end{rem}

\begin{rem}\, 
\begin{enumerate}[leftmargin=*]
\item Variants of the method of proof of \cref{thm:delooping} can be used to obtain many other delooping results for embedding calculus, e.g.\,deloopings of $\smash{T_k\Emb^\theta_\partial((M \times D^d,\ell),(M \times D^d,\ell))^\times}$ in terms of $E_d$-algebra automorphisms of $\smash{E_{(M \times \bfR^d,\ell)}}$, or delooping results in positive codimension. 
\item The left side of \cref{thm:delooping} can be shown to carry an $E_{d+1}$-algebra structure of geometric origin corresponding to composition and stacking $D^d\cong [-1,1]^d$ in the $d$ directions (this follows e.g.\,by implementing \cite[Remark 3.1 (i)]{KKDisc} and extending it to the level of towers). With respect to this $E_{d+1}$-structure, the equivalence in \eqref{thm:delooping} ought to be one of $E_{d+1}$-algebras.
\end{enumerate}
\end{rem}

\noindent We prove \cref{thm:delooping} below, after discussing two of its key ingredients.

\subsubsection{Delooping spaces of $E_d$-module maps} \label{sec:delooping-module-maps} One of the ingredients in the proof of \cref{thm:delooping} is a delooping result due to Lurie for the space of maps of $E_d$-modules over $A$ from $A$ to $\varphi^*B$ where $\varphi\colon A\ra B$ is any map of $E_d$-algebras in $\cC$. To state this result in a form suitable for our purposes, we write $\smash{M^\times\in\Alg_{E_d}(\cS)}$ for the \emph{units} of an $E_d$-algebra $\smash{M\in\Alg_{E_d}(\cS)}$ in spaces, i.e.\,the value at $M$ of the right-adjoint of the inclusion $\smash{\Alg^{\mathrm{gp}}_{E_d}(\cS)\subset \Alg_{E_d}(\cS)}$ of the full subcategory of group-like $E_d$-algebras \cite[p.\,898]{LurieHA}. Combined with the functor $\Map_{\cC}(1,-)\colon \Alg_{E_d}(\cC)\ra \Alg_{E_d}(\cS)$ resulting from the lax monoidal lift of the functor $\Map_{\cC}(1,-)\colon \cC\ra\cS$ (see \cref{ex:mapping-unit-in-is-lax}), this allows one to define a group-like $E_d$-space $\Map_\cC(1,B)^\times\in\Alg_{E_d}(\cS)$, known as the \emph{unit space}, for any $E_d$-algebra $B$ in any symmetric monoidal category $\cC$. In these terms the delooping result reads as follows:

\begin{thm}\label{thm:delooping-technical}For a presentable symmetric monoidal category $\cC$ and a morphism $\varphi\colon A\ra B$ of $E_d$-algebras in $\cC$, there is an $E_d$-algebra structure on $\Map_{\smash{\oMod_{A}^{E_d}(\cC)}}(A,\varphi^*B)\in\cS$ and a fibre sequence
\[
	\smash{\Omega^{d}\Map_{\Alg_{E_d}(\cC)}(A,B)\lra \Map_{\smash{\oMod_{A}^{E_d}(\cC)}}(A,\varphi^*B)^\times \lra  \Map_{\cC}({1},B)^\times}
\]
of $E_d$-algebras where the loop space is taken at $\smash{\varphi}$. This is natural in pre- and post-composition with morphisms of $E_d$-algebras in $\cC$. If $\cC$ is unital, it in particular yields a natural equivalence of $E_d$-algebras
\[
	\smash{\Omega^{d}\Map_{\Alg_{E_d}(\cC)}(A,B)\simeq \Map_{\smash{\oMod_{A}^{E_d}(\cC)}}(A,\varphi^*B)^\times}.
\]
\end{thm}

\begin{proof}We write $Z_{E_d}(\varphi)\in\Alg_{E_d}(\cC)$ for the \emph{centraliser} of $\varphi$ in the sense of \cite[5.3.1.2, 5.3.1.12, 5.3.1.13]{LurieHA}, which exists if $\cC$ is presentable monoidal by 5.3.1.15 loc.cit.. This centraliser corepresents the functor $\smash{\Map_{\smash{\oMod_{A}^{E_d}(\cC)}}(A \otimes (-),\varphi^*B)\colon \Alg_{E_d}(\cC)^\op \ra \cS}$ as a result of 5.3.1.30 loc.cit., so for $C=1$ one has an equivalence of spaces $\smash{\Map_{\cC}(1,Z_{E_d}(\varphi))\simeq \Map_{\smash{\oMod_{A}^{E_d}(\cC)}}(A,\varphi^*B)}$ (c.f.\, (3) on p.\,881 loc.cit.), which induces an $E_d$-algebra structure on the right-hand side as claimed. The fibre sequence and its naturality then follows by setting $A'=1$ in 5.3.2.5 loc.cit., applying $\Omega^d(-)$ to the resulting pullback to get a pullback in $\Alg_{E_d}(\cS)$, using the second part in this result and 5.3.2.7 as well as 5.3.2.14 to identify one of the corners naturally with $\Map_{\cC}(1,Z_{E_d}(\varphi))^\times$ and another one as being trivial. Finally, if $\cC$ is unital, then $\Map(1,B)\simeq\ast$, so the fibre sequence reduces to an equivalence. \end{proof}

\subsubsection{The enveloping algebra for $E_d$-modules}Another ingredient in the proof of \cref{thm:delooping} is the equivalence of categories from \cref{thm:ed-modules-as-leftmodules} in the case $\cO=E_d$. For this operad, the associative algebra $U_{E_d}$ in $\RMod(E_d)=\PSh(\Env(E_d))$ featuring in the statement, as well as the module $M_{E_d}$ over it admits an explicit description: as mentioned above, the standard framing $\ell_\st$ of $D^d$ lifts the nullbordism $D^d\colon S^{d-1}\leadsto \varnothing$ from $\BordInf^t(d)$ to $\BordInf^\fr(d)$, so by applying $E$ this lifts the pair $(E_{S^{d-1}\times I},E_{D^d})\in\LLMod(\RMod(E_d^t))$ to $\smash{(E_{(S^{d-1}\times I,\ell_\st)},E_{(D^d,\ell_\st)})\in \LLMod(\RMod(E_d))}$.

\begin{prop}\label{prop:universal-algebra-for-ed}
We have $(U_{E_d},M_{E_d})\simeq (E_{(S^{d-1}\times I,\ell_\st)},E_{(D^d,\ell_\st)})$ in $\LLMod(\RMod(E_d))$.
\end{prop}

\noindent The proof of this proposition will be an application of the following general lemma:

\begin{lem}\label{lem:criterion-module}Fix a monoidal category $\cC$ and two pairs $(A_0,M_0), (A_1,M_1)\in \LLMod(\cC)$ of an associative algebra and a left-module over it. Assume that the following two properties are satisfied:

\begin{enumerate}
	\item $M_0\simeq M_1$ as objects in $\cC$ and 
	\item for every object $C\in \cC$ and $i\in\{0,1\}$, the following composition is an equivalence.
	\[
		\smash{\Map_{\cC}(C,A_i)\xra{(-)\otimes M_i} \Map_{\cC}(C\otimes M_i,A_i\otimes M_i)\xra{\mathrm{act}\circ(-)} \Map_{\cC}(C \otimes M_i,M_i).}
	\]
\end{enumerate}
Then $(A_0,M_0)$ and $(A_1,M_1)$ are equivalent in $\LLMod(\cC)$.
\end{lem}

\begin{proof}The proof involves the \emph{endomorphism category} $\cC[M]$ of an object $M\in \cC$ from \cite[4.7.1]{LurieHA}. This is a monoidal category with objects given as pairs $(X,\eta)$ of $X\in \cC$ and a morphism $\eta\colon X\otimes M \ra M$, morphisms spaces $\smash{\Map_{\cC[M]}((X,\eta),(Y,\varphi))}$ given by the pullback of the forgetful map $\smash{\Map_{\cC_{/M}}(\eta,\varphi)\ra  \Map_{\cC}(X\otimes M,Y\otimes M)}$ along $\smash{(-)\otimes M\colon \Map_{\cC}(X,Y)\ra  \Map_{\cC}(X\otimes M,Y\otimes M)}$, and monoidal structure on objects given by $\smash{(X,\eta)\otimes(Y,\varphi)\simeq (X\otimes Y,\eta\circ (\id_X\otimes \varphi))}$ (see 4.7.1.1. loc.cit. and the preceding discussion, as well as 4.7.1.30). This category can be used to prove the claim:

\medskip
\noindent Since $M_0\simeq M_1$, we may assume $M_0=M_1\eqcolon M$. The pairs $(A_0,M), (A_1,M)$ define objects in the fibre $\fib_{M}(\LLMod(\cC)\ra \cC)$. The latter features  by 4.7.1.34 and 4.7.1.35 loc.cit.\,in an equivalence $\fib_{M}(\LLMod(\cC)\ra \cC)\simeq \Ass(\cC[M])$ whose composition with the forgetful map $\Ass(\cC[M])\ra \cC[M]$ sends an algebra $A\in \Ass(A)$ with an action on $M$ to the pair $(A,\mathrm{act}\colon A\otimes M\ra M)$. It thus suffices to show that $(A_i,\mathrm{act}\colon A_i\otimes M\ra M)$ for $i=0,1$ are equivalent in $\Ass(\cC[M])$. This would follow by showing that both are final objects in $\Ass(\cC[M])$, which we may test in $\cC[M]$ since the forgetful functor detects limits (see 3.2.2.5 loc.cit.). Using the above description the mapping spaces, one sees that the second condition in the statement implies that  $(A_i,M)$ is terminal in $\cC[M]$ for $i\in\{0,1\}$, so the claim follows.
\end{proof}

\begin{proof}[Proof of \cref{prop:universal-algebra-for-ed}]
Recall from the proof of \cref{thm:ed-modules-as-leftmodules} that $(U_{E_d},M_{E_d})$ is induced via the coherent nerve from an object $(\cat{U_{E_d}},\cat{M_{E_d}})$ in the simplicial category $\cat{LMod(PSh(Env(E_d))}$ of strict left-modules: we have $\cat{U_{E_d}(-)}=\Emb^{\rec}((-)\sqcup (-1,1)^d,(-1,1)^d)$ and  $\cat{M_{E_d}(-)}= \Emb^{\rec}(-,(-1,1)^d)$ as simplicial presheaves on $\cat{Env(E_d)}$, with the algebra and module structure induced by composition. 

\medskip
\noindent With this in mind, we apply \cref{lem:criterion-module} in the case $\cC=\RMod(E_d)$ to the two pairs in the claim. The first condition is satisfied since the inclusion $\smash{\Emb^{\rec}(-,(-1,1)^d)\subset \Emb^{\fr}(-,(-1,1)^d)}$ is an equivalence, so $M_{E_d}\simeq E_{D^d}$ in $\RMod(E_d)$. Since any presheaf is a colimit of representables, it suffices to verify the second condition for $C=E_{T\times D^d}\simeq E_{T\times (-1,1)^d}$ for finite sets $T$. Spelling out the condition and using the Yoneda lemma, it follows immediately for $(U_{E_d},M_{E_d})$ and for $\smash{(E_{(S^{d-1}\times I,\ell_\st)},E_{(D^d,\ell_\st)})}$ by using that $\smash{\Emb^\fr(T\times D^d,S^{d-1}\times I)}$ is equivalent to the fibre of the restriction $\smash{\Emb^\fr(T\times D^d\sqcup D^d,D^d)\ra \Emb^\fr(D^d,D^d)\simeq\ast}$, by isotopy extension.
\end{proof}

\subsubsection{The proof of \cref{thm:delooping}}We now turn to the proof of \cref{thm:delooping}, beginning with a preparatory lemma on the sequence of operads $\cO\ra\Env(\cO)\ra\RMod(\cO)$ given by the unit of the adjunction from \cref{sec:env}, followed by the symmetric monoidal Yoneda embedding.
\begin{lem}\label{lem:operad-yoneda}Let $\cO$ be an operad.
\begin{enumerate}[leftmargin=*]
	\item The map of operads $\cO\ra \RMod(\cO)$ is faithful, i.e.\,the map $\Map_{\Opd}(\cU,\cO)\ra \Map_{\Opd}(\cU,\RMod(\cO))$ induced by postcomposition is an inclusion of path-components for all $\cU\in\Opd$.
	\item If $\cO$ is $k$-truncated, then the map of operads $\cO\ra \RMod(\cO)$ factors uniquely over the faithful map of operads $\RMod_k(\cO)\ra \RMod(\cO)$ given by right Kan extension.
\end{enumerate}
\end{lem}

\begin{proof}
The map $\smash{\Map_{\Opd}(\cU,\cO)\ra \Map_{\Opd}(\cU,\RMod(\cO))}$ the first item is about is obtained from the functor $\smash{\Fun(\cU^{\otimes},\cO^\otimes )\ra \Fun(\cU^{\otimes},\RMod(\cO)^\otimes)}$ by passing to certain full subcategories in source and target---namely those over $\Fin_*$ that preserve cocartesian lifts of inert morphisms---followed by applying cores, so the first claim follows from the fact that it suffices to prove that the functor $\cO^\otimes\ra \RMod(\cO)^\otimes$ in $\Cat$ is fully faithful. The latter follows from the Yoneda lemma and the description of the mapping spaces in source and target resulting from \eqref{equ:maps-total-space-operads}, \eqref{equ:mapping-space-env}.

\medskip 

\noindent From the construction of the symmetric monoidal structure of $\RMod_k^\un(\cO)$ in \cref{prop:tower} \ref{enum:tower-monoidal-tower} resulting from \cref{lem:localise-day} we see that in order to show the second part it suffices to show that for $c\in \cO^\col$ the image $\Map_{\Env(\cO)}(-,c)\in \RMod(\cO)$ under $\cO \to \RMod(\cO)$ is right Kan extended from $\RMod_k(\cO)$, i.e.\,that the unit $\Map_{\Env(\cO)}(-,c)\ra \iota_*\iota^*\Map_{\Env(\cO)}(-,c)$ is an equivalence in $\RMod(\cO)$. This follows from \eqref{equ:counit-multioperations} and \cref{lem:simplified-right-kan}.
\end{proof}

\noindent As a next step in the proof of \cref{thm:delooping}, we consider the following sequence of spaces
\begin{equation}\label{equ:operad-fibre-diagram-delooping}
	\begin{tikzcd}[row sep=0.5cm,column sep=0cm,ar symbol/.style = {draw=none,"\textstyle#1" description,sloped},equivalent/.style = {ar symbol={\simeq}}] 
	\fib_*(B\ra\BAut(E_{d,\le k})) \rar[equivalent,pos=1.18] \hspace{-.3cm} & \hspace{.3cm} \Map_{\Opd^{\gc,\simeq_{\red}}}(E_{d,\le k},(E_{d,\le k})^\theta) \ar[draw=none]{ld}[name=X, anchor=center]{} \ar[hook', rounded corners, to path={ -- ([xshift=2ex]\tikztostart.east)	|- (X.center) \tikztonodes	-| ([xshift=-2ex]\tikztotarget.west) -- (\tikztotarget)}]{dl} \\
	\Map_{\Opd^{\un}}(E_{d,\le k},(E_{d,\le k})^\theta) \rar[equivalent] &\Map_{\Opd^\un}(E_{d},(E_{d,\le k})^\theta)\ar[draw=none]{ld}[name=X, anchor=center]{} \ar[hook', rounded corners, to path={ -- ([xshift=2ex]\tikztostart.east)	|- (X.center) \tikztonodes	-| ([xshift=-2ex]\tikztotarget.west) -- (\tikztotarget)}]{dl}\\
	\Map_{\Opd}(E_{d},\RMod_k((E_{d,\le k})^\theta))\rar[equivalent]&\Alg_{E_d}(\RMod_k((E_{d,\le k})^\theta))^\simeq.
	\end{tikzcd}
\end{equation}
The top equivalence is the equivalence of categories \eqref{equ:gc-redequ-identification} on mapping spaces from $E_{d,\le k}$ to $(E_{d,\le k})^\theta$ using that $\smash{\Map_{\cS_{/\BAut(E_{d,\le k})}}(*,B)\simeq \fib(B\ra \BAut(E_{d,\le k}))}$, the middle equivalence holds because truncation of unital operads is a localisation (see \cref{sec:truncation-operads}) and $(E_{d,\le k})^\theta$ is $k$-truncated by \cref{lem:truncated-tangential}, and the bottom equivalence holds by definition (see \cref{sec:algebra-over-operads}). The first arrow is induced by the subcategory inclusion $\Opd^{\gc,\simeq_\red}\subset\Opd^{\un}$, so it is an inclusion of components. The second arrow is induced by postcomposition with the map of operads considered in \cref{lem:operad-yoneda} for $\cO=(E_{d,\le k})^\theta$, and this lemma also shows that it is an inclusion of components. Note that the composition \eqref{equ:operad-fibre-diagram-delooping} sends a basepoint $\ast\in \fib(B\ra\BAut(E_{d,\le k}))$, viewed as a commutative diagram 
\begin{equation}\label{equ:point-is-change-of-tang-structure}
	\begin{tikzcd}[column sep=-0.2cm,row sep=0.3cm]
	\ast\arrow[dr]\arrow[rr,"\ast"]&& B\arrow[dl]\\
	&\BAut(E_{d,\le k})&,
	\end{tikzcd}
\end{equation}
to the composition of operads
\begin{equation}\label{equ:where-basepoint-goes}\smash{E_d\lra E_{d,\le k}\xlra{\ast}E_{d,\le k}^\theta\subset \Env(E_{d,\le k}^\theta)\xlhra{y} \RMod(E_{d,\le k}^\theta)}\end{equation}
whose first map is given by truncation, second map by the change of tangential structure \eqref{equ:point-is-change-of-tang-structure}, and final map by the symmetric monoidal Yoneda embedding. Since the target of this composition is a presentable symmetric monoidal category, this composition extends first to $\Env(E_d)$ and then along the Yoneda embedding to $\RMod(E_d)$ by taking colimit preserving extension. By the uniqueness of colimit preserving extensions, the resulting functor $\RMod(E_d)\ra \RMod(E_{d,\le k}^\theta)$ agrees with the one given by truncation followed by the change of tangential structures \eqref{equ:point-is-change-of-tang-structure} via left Kan extension.

\medskip

\noindent Next we consider a sequence of equivalences of $E_1$-algebras in spaces
\begin{align}
	\begin{split}\label{equ:delooping-chain-of-equ}
	&\Omega^{d+1}\fib\big(B\ra\BAut(E_{d,\le k})\big)\ {\simeq}\  \Omega^{d}\Aut_{\Alg_{E_d}(\RMod_k((E_{d,\le k})^\theta))}\big(A_\ast\big)\\{\simeq}\ \ &  \Aut_{\oMod^{E_d}_{A_\ast}(\RMod_k((E_{d,\le k})^\theta))}\big(A_\ast\big)\ {\simeq} \ \Aut_{\LLMod_{\lvert E_{(S^{d-1}\times I,\ell_\st)}\rvert_{A_\ast}}(\RMod_k((E_{d,\le k})^\theta))}\big(\lvert E_{(D^d,\ell_\st)}\rvert_{A_\ast}\big).\end{split}
\vspace{-0.3cm}\end{align} 
The first equivalence is one of $E_{d+1}$-algebras and results from applying $\Omega^{d+1}(-)$ to \eqref{equ:operad-fibre-diagram-delooping}, based at $\ast\in \fib(B\ra\BAut(E_{d,\le k}))$; the $E_d$-algebra $A_\ast$ is defined as the image of this basepoint, i.e.\,as the composition \eqref{equ:where-basepoint-goes}. The second one follows from the fibre sequence in \cref{thm:delooping-technical} for $\varphi=\id_{A_\ast}$, using that $\RMod_k((E_{d,\le k})^\theta)$ is presentable monoidal by an instance of \cref{lem:localise-day} \ref{enum:localise-day-iii} and the unitality of the operad $E_{d,\le k}^\theta$. It is an equivalence of $E_1$-algebras with respect to the $E_1$-structure by composition on the target and the one induced from the loop space structure on the source; this follows from the naturality part of \cref{thm:delooping-technical}. The final equivalence follows from specialising \cref{thm:ed-modules-as-leftmodules} to $\cO=E_d$, combined with \cref{prop:universal-algebra-for-ed}. Since we argued above that the colimit preserving extension $\lvert -\rvert_{A_\ast}\colon \RMod(E_d)\ra \RMod(E_{d,\le k}^\theta)$ is induced by truncation and the change of tangential structure \eqref{equ:point-is-change-of-tang-structure}, the algebra and left-module $\smash{(\lvert E_{(S^{d-1}\times I,\ell_\st)}\rvert_{A_\ast},\lvert E_{(D^d,\ell_\st)}\rvert_{A_\ast})}$ in $\smash{\RMod_k((E_{d,\le k})^\theta)}$ are obtained from the $k$-truncation of the algebra and left-module ${(E_{(S^{d-1}\times I,\ell_\st)}, E_{(D^d,\ell_\st)})}$ in $\RMod(E_d)$ by the change of tangential structure \eqref{equ:point-is-change-of-tang-structure} so the final space in \eqref{equ:delooping-chain-of-equ} agrees with $T_k\Emb^\theta_\partial((D^d,\ell_\ast),(D^d,\ell_\ast))^\times$. The latter is group-like by \cref{thm:tk-equivalences}, so we may leave out the $(-)^\times$-superscript and thus conclude the claim.

\subsubsection{Delooping the $\DiscInf$-structure space of a disc}\cref{thm:delooping} implies the following delooping result for the $\DiscInf$-structure space (see \cref{sec:s-disc} for a recollection). 
\begin{cor}\label{cor:delooping-cor-sdisc}For $d\ge0$, is an equivalence $S^{\DiscInf,t}_\partial(D^d){\simeq}\Omega^{d+1}\Aut(E_d)/\Top(d)$.
\end{cor}

\begin{proof}
This follows from the chain of equivalences 
\begin{align*}
	S^{\DiscInf,\ot}_\partial(D^d) 
&= \fib_{E_{D^d}}(\BordInf^{t,\simeq}(d)_{S^{d-1}}\ra \ModInf(E_d^t)_{E_{S^{d-1}}})\\
	&\simeq \textstyle{\fib(\bigsqcup_{[M]}\BHomeo_\partial(M) \to \oB T_\infty \Emb^\ot_\partial(D^d,D^d)^\times )} \\
	&\simeq T_\infty \Emb^\ot_\partial(D^d,D^d)^\times \simeq \Omega^{d+1}\Aut(E_d)/\Top(d)
\end{align*}
where $[M]$ runs over compact topological $d$-manifolds $M$ with an identification $\partial M\cong S^{d-1}$ such that $\smash{E_M\simeq E_{D^d}}$ in $\smash{\ModInf(E_d^t)_{E_{S^{d-1}}}^\simeq}$, up to homeomorphism fixed on the boundary. As $E_M\simeq E_{D^d}$ implies that $M$ is contractible (apply the functor $\smash{\ModInf(E_d^t)_{E_{S^{d-1}}}\ra (\cS_{/\BTop(d)})_{/S^{d-1}}}$ to see this), there is in fact a unique such class by the topological Poincaré conjecture, namely that of $D^d$. Since $\BHomeo_\partial(D^d)$ is contractible by the Alexander trick, this implies the penultimate equivalence in the above chain. The final equivalence is the case $\theta=(\BTop(d)\ra \BAut(E_d))$ and $k=\infty$ of \cref{thm:delooping}, together with the fact that $T_\infty \Emb^\ot_\partial(D^d,D^d)$ is grouplike by \cref{thm:tk-equivalences}.
\end{proof}

\begin{rem}Recall from \cref{thm:sdisc-smooth-vs-top} that $\smash{S^{\DiscInf,o}_\partial(D^d)\simeq S^{\DiscInf,t}_\partial(D^d)}$ as long as $d\neq4$, so \cref{cor:delooping-cor-sdisc} also implies that there is an equivalence \begin{equation}\label{equ:delooping-s-disc-smooth}
	S^{\DiscInf,o}_\partial(D^d)\simeq \Omega^{d+1}\Aut(E_d)/\Top(d)\quad{\text{for }}d\neq4.
\end{equation} 
Up to passing to certain components in the source and target, the equivalence \eqref{equ:delooping-s-disc-smooth} was already known from Boavida de Brito and Weiss' work \cite{BoavidaWeissConfiguration} (see \cite[Theorem 8.1]{KKDisc} for an explanation of how to deduce it from their work). Our proof of it here is independent of their work, and comes with the further advantages that it does not require any restriction on components, and that it admits a version that is also valid in dimension $d=4$, namely \cref{cor:delooping-cor-sdisc}.
\end{rem}

\subsection{Embedding calculus in positive codimension}\label{sec:positive-codim}So far we only considered embedding calculus for manifolds of a fixed dimension $d$, for reasons we explained in the introduction. However, for the sake of completeness, we briefly explain how embedding calculus for spaces of embeddings between manifolds of potentially different dimensions can be incorporated into our set-up. This not meant as a comprehensive discussion: we leave a full development of the theory in positive codimension to future work.

\medskip

\noindent 

\noindent Fixing $c\in\{\oo,\ot\}$, the first observation is that there is an operad $E^c\in\Opd^\un$ that contains the operads $E_d^c$ for all $d$ as full suboperads (in the sense of \cite[2.7.1]{KKDisc}), but also involves embeddings between Euclidean spaces of different dimension: it is the operadic nerve of the ordinary coloured simplicial operad with colours the nonnegative integers and  spaces of multi-operations from $(d_s)_{s \in S}$ for $S\in\Fin$ and $d_s\ge0$ to $d\ge0$ given by the space $\Emb^c(\sqcup_{s \in S} \bfR^{d_s},\bfR^d)$ of smooth (if $c = \oo$) or locally flat topological (if $c = \ot)$ embeddings. Note that the operad $E^c$ is unital, but not groupoid-coloured.

\medskip 

\noindent The second observation is that all constructions from \cite[3]{KKDisc} and their extensions to topological manifolds from \cref{sec:e-for-topological-mfds} go through almost verbatim for manifolds that do not have a fixed dimension (disjoint unions of $d$-manifolds for potentially different values of $d$). We denote the resulting objects by dropping the reference to the dimension $d$ in the notation. In particular, there is 
\begin{itemize}[leftmargin=0.5cm]
	\item a symmetric monoidal category $\ManInf^c$ of manifolds without boundary and without fixed dimension (smooth or topological, depending on $c$), and spaces of locally flat embeddings between them. Its underlying operad contains $E^c$ as the full suboperad,
	\item the full subcategory $\DiscInf^c\subset \ManInf^c$ on manifolds of the form $\sqcup_{s\in S}\bfR^{d_s}$ for $S\in\Fin$ and $d_s\ge0$, which is equivalent to the symmetric monoidal envelope $\Env(E^c)$ of $E^c$,
	\item a bordism category $\ncBordInf^c\in\CMon(\CCat(\Cat))$ of manifolds without fixed dimension, containing $\ncBordInf^c(d)$ for all $d\ge0$ as levelwise full subcategories,
	\item a functor $E\colon \ncBordInf^c\ra \ModInf^\un(E^c)$ of symmetric monoidal  double $\infty$-categories.
\end{itemize}

\medskip

\noindent Applying the theory from \cref{sec:operadic-framework} to the unital operad $E^c$, in particular upgrades $\ModInf^\un(E^c)=\ModInf^\un_\infty(E^c)$ to a tower $E\colon \ncBordInf^c\ra\ModInf^\un_\bullet(E^c)$ of symmetric monoidal double categories under $\ncBordInf^c$. This extends the tower $E\colon \ncBordInf^c(d)\ra\ModInf^\un_\bullet(E_d^c)$ for fixed $d\ge0$ we considered in \cref{sec:relation-embcalc} in that it factors as the following composition in $\CMon(\CCat(\Cat))$
\[
	\smash{\ncBordInf^c(d)\subset \ncBordInf^c\lra\ModInf^\un_\bullet(E^c)\xlra{\iota^*} \ModInf^\un_\bullet(E_d^c)}
\]
where the final map is induced by restriction along the inclusion of operads $\iota\colon E_d^c\subset E^c$ which induces a lax symmetric monoidal functor $\RMod(E^c)\ra \RMod(E_d^c)$ by \eqref{equ:restriction-is-lax} that turns out to be strong monoidal (e.g.\,by the formula for Day convolution from \cref{lem:day-con-env}), and extends to the level of towers by an argument similar to the proof of \cref{thm:naturality}.

\medskip

\noindent Recall from \cref{sec:sm-top-emb-calc} that the embedding calculus tower $\Emb^c_\partial(M,N)\ra T_\bullet\Emb^c_\partial(M,N)$ to the space of embeddings between two $d$-dimensional bordisms $M,N\colon P\leadsto Q$ was obtained from the tower $E\colon \ncBordInf^c(d)\ra\ModInf^\un_\bullet(E_d^c)$ by taking mapping spaces in mapping categories. For embedding calculus for embeddings of arbitrary codimension, we can use relative mapping spaces in the sense of \cref{sec:relative-mapping-spaces}: given an $m$-dimensional bordism $M\colon P\leadsto Q$ and an $n$-dimensional bordism $N\colon R\leadsto S$, both viewed as objects in the category of morphisms of $\ncBordInf^c$, and embeddings $e_P\colon P \hookrightarrow R$ and $e_Q\colon Q\hookrightarrow S$, seen as maps in the category of objects of $\ncBordInf^c$, the mapping space between $M$ and $N$ relative to $e_P$ and $e_Q$ in the sense of  \cref{sec:relative-mapping-spaces} recovers the space $\Emb_\partial(M,N)$ of locally flat embeddings $M\hookrightarrow N$ that extend $e_P\sqcup e_Q$:
\[
\Map_{\ncBordInf^c}(M,N;e_P,e_Q)\simeq\Emb^c_\partial(M,N).
\] 
The embedding calculus tower 
\begin{equation}\label{equ:emb-calc-pos-codim-tower-def}
	\Emb_\partial^c(M,N)\lra T_\bullet\Emb_\partial^c(M,N)
\end{equation} 
can then be obtained as the tower on relative mapping spaces induced by the composition 
\begin{equation}\label{equ:emb-calc-pos-codim-fixed-target-dim}
	\smash{\ncBordInf^c\lra  \ModInf^\un_\bullet(E^c)\xlra{\iota^*} \ModInf^\un_\bullet(E_m^c).}
\end{equation}

\begin{rem}\label{rem:rems-on-pos-codim-calc}Some comments on this definition:
\begin{enumerate}[leftmargin=*]
	\item The reader might wonder why relative mapping spaces made no appearance in our discussion of embedding calculus in codimension $0$. The conceptual reason is that $\ncBordInf^c(d)$ in contrast to $\ncBordInf^c$ satisfies the condition of \cref{rem:st-cartesian}: given $d$-dimensional bordisms $N\colon R\leadsto S$, viewed as an object in the category of morphisms $\ncBordInf^c(d)_{[1]}$, and embeddings $e_P\colon P\hookrightarrow R$ and $e_Q\colon Q\hookrightarrow S$, viewed as morphisms in the category of objects $\ncBordInf^c(d)_{[0]}$, a cartesian lift of $(e_P,e_Q)$ with target $N$ along $(s,t)\colon \ncBordInf^c(d)_{[1]}\ra \ncBordInf^c(d)_{[0]}\times\ncBordInf^c(d)_{[0]}$ is given by the inclusion $N'\subset N$ of $N'\coloneqq \interior(N)\cup e(P)\cup e(Q)$. This also gives a conceptual reason for the ``trick'' in \cref{rem:triads} to deal with embedding spaces of triads: in the notation of this remark, $N'$ results from $N$ as the cartesian lift of $(e_{\partial_0},\id_\varnothing)$.
	\item\label{enum:more-symmetric-pos-codim-calculus} In the definition of \eqref{equ:emb-calc-pos-codim-tower-def} as being induced by \eqref{equ:emb-calc-pos-codim-fixed-target-dim}, we could have omitted the final functor restriction functor to $\ModInf^\un_\bullet(E_m^c)$ and obtained the same result. Using \cref{rem:st-cartesian} and \cref{rem:morita-st-cartesian} to rewrite relative mapping spaces as absolute ones, this follows by showing that for $1\le k\le \infty$ the map
	\[
		\Map_{\ModInf_k(E^c)_{E_P,E_Q}}(E_M,X) \lra \Map_{\ModInf_k(E^c_m)_{E_P,E_Q}}(E_M,X)
	\] 
	induced by the second functor in \eqref{equ:emb-calc-pos-codim-fixed-target-dim} is an equivalence for all $m$-dimensional bordisms $M\colon P\leadsto Q$, and $X\in \RMod_k(E^c)_{E_P,E_Q}$. To do so, firstly, we use both sides satisfy descent with respect to complete Weiss $k$-covers of $M$ (by an argument as in \cref{lem:tk-descent}) and, secondly, that the map is an equivalence for $M=[0,1) \times P \sqcup S \times \bfR^m \sqcup (-1,0] \times Q$ with $|S| \leq k$ by the Yoneda lemma, using that $E_M$ is in this case the free $(E_{P\times I},E_{Q\times I})$ on the representable presheaf $E_{S\times \bfR^m}$, both in $\RMod_k(E^c)_{E_P,E_Q}$ and in $\RMod_k(E^c)_{E_P,E_Q}$ (by an argument as in the proof of \cref{lem:bimodule-tower-is-presheaf-tower}). Restricting to $\ModInf^\un_\bullet(E_m^c)$ to define \eqref{equ:emb-calc-pos-codim-tower-def} has the advantage that we can make use of the results on $\ModInf^\un_\bullet(E_m^c)$ from the previous subsections, e.g.\,the geometric identification of the layers from \cref{sec:higher-layers-bdy}. Not restricting to $\ModInf^\un_\bullet(E^c)$ on the other hand equips \eqref{equ:emb-calc-pos-codim-tower-def} with better naturality properties, e.g.\,with respect to precomposition with positive codimension embeddings.
	\item \label{enum:pos-codim-as-manifold-calc}A more direct way to construct \eqref{equ:emb-calc-pos-codim-tower-def} that does not involve the operad $E^c$ is to apply manifold calculus as in \cref{sec:manifold-calc} to the presheaf $\Emb_{\partial}(-,N)\in\PSh(\ncBordInf(m)_{P,Q})$. To see that the resulting tower agrees with \eqref{equ:emb-calc-pos-codim-tower-def}, one notes that the $k$th stage of \eqref{equ:emb-calc-pos-codim-tower-def} is an equivalence for $M=[0,1)\times P\sqcup S\times \bfR^m\sqcup (-1,0]\times Q$ with $|S|\le k$ and that the target satisfies descent with respect to complete Weiss $k$-covers of $M$ (see Item \ref{enum:more-symmetric-pos-codim-calculus} of this remark). This in particular shows that \eqref{equ:emb-calc-pos-codim-tower-def} agrees with the previously considered models for embedding calculus, e.g.\,with the original one from \cite{WeissEmbeddings}. The advantage of \eqref{equ:emb-calc-pos-codim-tower-def} is that it extends (by construction) to the level of bordism categories, endowing it with more naturality properties (e.g.\,with respect to pre- and postcompositions or gluings along parts of the boundary). 
\end{enumerate}
\end{rem}
 
\begin{rem}[Embeddings of positive codimension triads]\label{rem:triad-pos-codim}Given manifolds $M$ and $N$ (possibly noncompact and of potentially different dimension), a codimension $0$-submanifold $\partial_0M\subset \partial M$ and an embedding $e_{\partial_0}\colon \partial_0M\hookrightarrow \partial N$, we can proceed as in \cref{rem:triads} by replacing $M$ up to isotopy equivalence relative to $\interior(\partial_0M)$ with $M'\coloneqq (\interior(M)\cup\interior(\partial_0M))\colon \interior(\partial_0M)\leadsto \varnothing$ and consider 
\[
	\Emb_{\partial_0}^c(M,N)\simeq\Emb^c_{\partial}(M',N)\lra T_\bullet\Emb^c_{\partial}(M',N)\coloneqq T_\bullet\Emb^c_{\partial_0}(M,N)
\]
to obtain an embedding calculus tower for the space of embeddings $M\hookrightarrow N$ extending $e_{\partial_0}$.
\end{rem}

\subsubsection{The bottom layer}
Recall from \cref{sec:sm-top-emb-calc} the equivalence $\ModInf^\un_1(E_m^c) \simeq \Cosp(\cS_{/\oB K(m)})$ for $K(m)\in\{\oO(m),\Top(m)\}$ depending on $c\in\{\oo,\ot\}$ and the description of $\ncBordInf^c(m)\ra \Cosp(\cS_{/\oB K(m)})$ from \cref{prop:bottom-manifold-layer}. To generalise this to describe the effect of the bottom layer 
\begin{equation}\label{equ:bottom-layer-pos-codim}
	\ncBordInf^c\lra \ModInf^\un_1(E_m^c)\simeq \Cosp(\cS_{/\oB K(m)})
\end{equation} 
of \eqref{equ:emb-calc-pos-codim-fixed-target-dim}, we denote by $\oO(n;m)\le \oO(n)$ and $\Top(n;m)\le \Top(n)$ for $m\le n$ the topological subgroups of those orthogonal transformations or homeomorphisms of $\bfR^n$ that preserve the subspace $\bfR^m=\bfR^m\times\{0\}^{n-m}\subset \bfR^n$ setwise, and by $\oO(n,m)\le \oO(n;m)$ and $\Top(n,m)\le \Top(n;m)$ the smaller subgroups where $\bfR^m$ is pointwise fixed. These subgroups are related by fibre sequences 
\begin{equation}\label{equ:fix-or-move}
	\oO(n,m)\ra \oO(n;m)\ra \oO(m)\quad\text{and}\quad \Top(n,m)\ra \Top(n;m)\ra \Top(m),
\end{equation}
induced by restriction, and by maps induced by taking products and composition
\begin{equation}\label{equ:product-maps-top}
	\begin{tikzcd}[row sep=0cm, column sep=0.5cm]
	\oO(m)\times \oO(n-m)\rar{\simeq}& \oO(m)\times \oO(n,m)\rar{\simeq}&\oO(n;m)\\\Top(m)\times \Top(n-m)\rar& \Top(m)\times \Top(n,m)\rar{\simeq}&\Top(n;m).
	\end{tikzcd}
\end{equation}
All maps labelled with a $\simeq$-sign are easily seen to be equivalences (isomorphisms of topological groups, in fact), but the unlabelled arrow is not an equivalence in general (e.g.\,for $m=20$ and $n=29$ it is not; see \cite[Example 2]{Millett}).

\begin{lem}\label{lem:positive-codim-t1-fact}Fix $n,m\ge0$.
\begin{enumerate}[leftmargin=*]
	\item On categories of objects, the functor \eqref{equ:bottom-layer-pos-codim} sends an $(n-1)$-manifold $P$ to 
	\[
		\big(P\times_{\BTop(n)}\BTop(n;m)\ra\BTop(m)\big)\in\cS_{/\BTop(m)}
	\] 
	if $m\le n$ (involving the once stabilised tangent bundle), and to $\varnothing\ra\BTop(m)$ otherwise. \item On categories of morphisms, \eqref{equ:bottom-layer-pos-codim} sends an $m$-dimensional bordism $N\colon R\leadsto S$ to the cospan
	\[
		\big(P\times_{\BTop(n)}\BTop(n;m)\big)\ra (N\times_{\BTop(n)}\BTop(n;m)\big)\la(Q\times_{\BTop(n)}\BTop(n;m)\big)
	\]
	in $\cS_{/\BTop(m)}$ if $m\le n$, and to the cospan $\varnothing\ra\varnothing\la\varnothing$ otherwise.
	\item On relative mapping spaces, the functor \eqref{equ:bottom-layer-pos-codim} is induced by taking topological derivatives.
\end{enumerate}
Replacing $\Top$ by $\oO$ and $\ot$ by $\oo$, the same statement holds in the smooth case.
\end{lem}

\begin{proof}We prove the second part in the case where $c=\ot$ and $P=Q=\varnothing$; all other statements follow analogously, similarly to the proof of \cref{prop:bottom-manifold-layer}. By construction the value of $\eqref{equ:bottom-layer-pos-codim}$ at $N$ is given by $\Emb^t(\bfR^m,N)/\Top(m)\ra \BTop(m)$. If $\dim(N)=n>m$ then this is $\varnothing\ra \BTop(m)$ since $\Emb^t(\bfR^m,N)=\varnothing$. In the other case $n \leq m$, we consider the commutative diagram 
\[
	\begin{tikzcd}[row sep=0.3cm] 
	\Emb^\ot(\bfR^{m},N)/ \Top(m) \dar& \lar\Emb^\ot(\bfR^{n},N) /  \Top(n;m) \rar\dar& \Emb^\ot(\bfR^{n},N)/ \Top(n)\overset{}{\simeq} M\dar\\
	\BTop(m)&\BTop(n;m)\lar\rar&\BTop(n)
	\end{tikzcd}
\]
where the horizontal arrows are given restriction, the vertical arrows are the evident ones, and the top-right equivalence holds by the argument in the proof of \cref{prop:bottom-manifold-layer}. The leftmost vertical map is the one we like to identify. Since the right square is a pullback, it suffices to show that the top-left horizontal map is an equivalence. Using isotopy extension, it suffices to show this for $N=\bfR^n$. Since the inclusion $\Emb^t(\bfR^n,\bfR^n)\subset \Top(n)$ is an equivalence by the Kister--Mazur theorem (see the proof of \cref{prop:bottom-manifold-layer}) and restriction induces an equivalence $\Top(n)/\Top(n,m)\ra \Emb^t(\bfR^m,\bfR^n)$ (see below), we conclude the claimed equivalence
\[
	\hspace{-0.4cm}\Emb^\ot(\bfR^{m},\bfR^{n})/ \Top(m)\simeq \big(\Top(n)/\Top(n,m)\big)/\Top(m)\simeq \Top(n)/(\Top(m)\times \Top(n,m))\simeq \Top(n)/\Top(n;m).
\]
To justify that restriction induces an equivalence $\Top(n)/\Top(n,m)\simeq \Emb^t(\bfR^m,\bfR^n)$, we add subscripts $(-)_0$ to denote the subspaces of homeomorphisms or embeddings that fix the origin and superscripts $(-)^{(0)}$ the spaces of germs of homeomorphisms or embeddings near the origin, so that the inclusion and taking germs yields vertical maps in the commutative diagram
\[
	\begin{tikzcd}[row sep=0.2cm]  \Top(n,m) \rar & \Top(n) \rar & \Emb^t(\bfR^m,\bfR^n) \\
	\Top_0(n,m) \rar \dar \uar[swap] & \Top_0(n) \rar \dar \uar[swap] & \Emb^t_0(\bfR^m,\bfR^n) \dar \uar[swap]\\
	\Top_0(n,m)^{(0)} \rar & \Top_0(n)^{(0)} \rar & \Emb^t_0(\bfR^m,\bfR^n)^{(0)}.\end{tikzcd}
\]
The claim follows by arguing that all vertical maps are equivalences, the bottom-right horizontal arrow is a surjective Kan fibration, and the bottom-left horizontal map the inclusion of a fibre. The top vertical maps admit homotopy inverses given by composition with the appropriate translations. The left and middle bottom vertical maps are equivalences by the argument in the proof of \cite[Theorem 1]{Kister}, and the right bottom vertical map is an equivalence by restricting the domain to a small neighbourhood of the origin. The bottom-right horizontal map is a Kan fibration by isotopy extension \cite{EdwardsKirby} whose fibre over the inclusion is the bottom-left horizontal map. For the surjectivity, choose a chart around the origin that witnesses a given embedding is locally flat.
\end{proof}

\subsubsection{The higher layers} By construction, the tower $\smash{T_\bullet\Emb^c_\partial(M^m,N^n)}$ is induced by the tower $\smash{\ModInf_\bullet(E_m^c)_{E_{P\times I},E_{Q\times I}}}$ on mapping spaces from $E_M$ to a bimodule $\iota^*E_N\in \ModInf_\bullet(E_m^c)_{E_{P\times I},E_{Q\times I}}$ whose underling right-module is given by $\Emb(-,N)\in\RMod(E_m^c)\simeq\PSh(\DiscInf^c_m)$. The discussion of the higher layers in $\ModInf_\bullet(E_m^c)_{E_{P\times I},E_{Q\times I}}$ from \cref{sec:layers-embcalc} thus applies to this tower. In particular, an application of \cref{cor:layer-description-explicit} \ref{enum:layer-corollary-general} to $X=\iota^*E_N$ shows that the description of the higher layers in the codimension zero case from part \ref{enum:layer-corollary-geometric} of this corollary is also valid in positive codimension.

\section{Convergence results for embedding calculus}\label{sec:convergence}
When considering embedding calculus for topological embeddings---as done in \cref{sec:emb-calc}---one is led to ask whether there is a convergence result analogous to that by Goodwillie, Klein, and Weiss for smooth embeddings (see \cref{sec:convergence-statements} below for a recollection). As part of this section, we show that this is the case, as long as the manifold one embeds into is smoothable and of dimension $\ge5$. The strategy is to deduce topological convergence from smooth convergence via classical smoothing theory and our smoothing theory for embedding calculus from \cref{sec:smoothing-embcalc}. For this argument, it is convenient to first prove a strengthening of the existing smooth convergence result,; this should be of independent interest. As an additional application of the latter, we extend work of Galatius--Randal-Williams \cite{GRWContractible} on homeomorphisms of contractible manifolds and embeddings of one-sided h-cobordisms from dimensions $d\ge6$ to dimension $d=5$ (see \cref{sec:grw-alexander-extension}). 

\medskip

\noindent The plan for this section is to state the convergence results in \cref{sec:convergence-statements}, prove them in \cref{sec:proof-convergence}, and conclude with the extension of Galatius--Randal-Williams' work in \cref{sec:grw-alexander-extension}.

\subsection{Statements of convergence results}\label{sec:convergence-statements}The convergence results are stated in the following setting: we fix a quadruple $(M,\partial_0M,N,e_{\partial_0})$ consisting of 
\begin{itemize}
	\item a compact topological $m$-manifold $M$,
	\item a compact codimension $0$ submanifold $\partial_0M\subset \partial M$ 
	\item a potentially noncompact topological $n$-manifold $N$ with $n\ge m$, and
	\item a locally flat embedding $e_{\partial_0}\colon \partial_0M\hookrightarrow \partial N$,
\end{itemize}
and consider the embedding calculus approximation (see \cref{rem:triad-pos-codim}).
\[\Emb^t_{\partial_0}(M,N)\lra T_\infty\Emb^t_{\partial_0}(M,N)\]
to the space of locally flat topological embeddings $e\colon M\hookrightarrow N$ with $e|_{\partial_0M}=e_{\partial_0}$ and $e^{-1}(\partial N)=\partial_0M$ (the latter condition can be dropped without effecting the homotopy type of the space of embeddings). If the quadruple  $(M,\partial_0M,N,e_{\partial_0})$ is \emph{smooth}; i.e.\,if $M$ and $N$ are smooth, $\partial_0M\subset M$ is a smooth submanifold, and $e_{\partial_0}$ a smooth embedding, we also consider the embedding calculus approximation 
\begin{equation}\label{equ:smooth-embcalc-con}
	\Emb^o_{\partial_0}(M,N)\lra T_\infty\Emb^o_{\partial_0}(M,N)
\end{equation}
to the space of \emph{smooth} embeddings with the same properties as before. By Goodwillie, Klein, and Weiss' convergence result, the smooth version \eqref{equ:smooth-embcalc-con} is known to be an equivalence when the relative handle dimension of $\partial_0 M\subset M$ is at most $n-3$, i.e.\,when $M$ can be built from a collar on $\partial_0M$ by attaching handles of index $\le n-3$ (for a reference, combine \cite[5.1]{GoodwillieWeiss} with \cite[Theorem B]{GoodwillieKlein} and use Remarks \ref{rem:triad-pos-codim} and \ref{rem:rems-on-pos-codim-calc} \ref{enum:pos-codim-as-manifold-calc}). Our strengthening and extensions of this result are easiest stated in codimension zero, i.e.\,when $m=n$, so we focus on this case first. 

\subsubsection{Convergence in codimension $0$}
In the codimension $0$ case $m=n$, the condition in Goodwillie, Klein, and Weiss' convergence result (that the relative handle dimension of $\partial_0 M\subset M$ is at most $n-3$) says equivalently that $M$ can be built from a collar on $\partial_1 M\coloneqq \partial M\backslash\interior(\partial_0M)$ by attaching $(\ge3)$-handles. We will weaken this condition to only assuming that $\partial_1M\subset M$ is \emph{a tangential $2$-type equivalence}, meaning that the map $\fib(\partial_1M\ra\BO)\ra\fib(M\ra\BO)$ between the fibres of the classifiers of the stable tangent bundles is an equivalence on fundamental groupoids (this is equivalent to the definition of tangential $2$-type equivalence used in \cite[5.1.1]{KKDisc}). For instance, if $M$ is spin, this is equivalent to $\partial_1M\subset M$ being an equivalence on fundamental groupoids (see \cref{rem:rem-smooth-conv} \ref{enum:rem-smooth-conv-2type} below). In these terms, our improved convergence result in codimension $0$ reads as:

\begin{thm}[Smooth convergence in codimension $0$]\label{thm:conv-smooth-codim0}Fix a smooth quadruple $(M,\partial_0M,N,e_{\partial_0})$ with $m=n\ge5$. If $\partial_1M\subset M$ is a tangential $2$-type equivalence, the map 
\[
	\Emb^o_{\partial_0}(M,N)\lra T_\infty\Emb^o_{\partial_0}(M,N)
\]
is an equivalence.
\end{thm}

\begin{rem}\label{rem:rem-smooth-conv}Some remarks on the statement of \cref{thm:conv-smooth-codim0}:
\begin{enumerate}[leftmargin=*]
	\item \label{enum:rem-smooth-conv-2type}Since $(w_1, w_2)\colon \BO\ra K(\bfZ/2,1)\times K(\bfZ/2,2)$ is an equivalence on $2$-truncations, the condition that $\partial_1M\subset M$ is an equivalence on tangential $2$-types is equivalent to:
	\begin{enumerate}[label=(\alph*)]
		\item $\partial_1M\subset M$ is an equivalence on fundamental groupoids, and
		\item at all basepoints, $w_2\colon \pi_2(M)\ra\bfZ/2$ is trivial if and only if $w_2\colon \pi_2(\partial_1M)\ra\bfZ/2$ is.
	\end{enumerate}
	Note that the second condition can be rephrased in terms of the components of universal covers of $\partial_1M$ and $M$ being spin or not.
	\item For $n\le 2$, the map \eqref{equ:smooth-embcalc-con} is an equivalence without assumptions on $\partial_1 M\subset M$, by \cite{KKsurfaces}.
	\item If the inclusion $\partial_0 M\subset M$ has relative handle dimension $\le d-3$, the convergence result of Goodwillie, Klein, and Weiss not only shows that \eqref{equ:smooth-embcalc-con} an equivalence, but also that the connectivity of the maps $T_k\Emb^o_{\partial_0}(M,N)\ra T_{k-1}\Emb^o_{\partial_0}(M,N)$ in the tower whose limit is $\Emb^o_{\partial_0}(M,N)\simeq T_{\infty}\Emb^o_{\partial_0}(M,N)$ increases in $k$, so in particular for $i\ge1$ the tower of homotopy groups $\{\pi_i(T_k\Emb^o_{\partial_0}(M,N),\iota)\}_{k\ge 1}$ based at a fixed embedding $\iota \in \Emb_{\partial_0}(M,N)$ is eventually constant and thus Mittag--Leffler (equivalently, the Bousfield--Kan spectral sequence based at $\iota$ Mittag--Leffler converges; see \cite[IX.5.6]{BousfieldKan}). In the more general situation of \cref{thm:conv-smooth-codim0} and the generalisations to positive codimension and topological embeddings below, the connectivity of the maps between the finite stages need not increase, but the tower $\{\pi_i(T_k\Emb^o_{\partial_0}(M,N),\iota)\}_{k\ge 1}$ is still Mittag--Leffler for $i \geq 1$. In fact, this is the case whenever the map \eqref{equ:smooth-embcalc-con} is an equivalence, by the following argument: $\pi_{i-1}(\Emb^o_{\partial_0}(M,N),\iota)$ is countable for $i\ge1$ as (e.g.\,as a result of \cite[Lemma 8.7]{KKDisc}, using that embedding spaces are second countable locally weakly-contractible when equipped with the smooth Whitney topology, because they are open subsets in the space of all smooth maps and the latter has these properties) so by \cite[9.3.1]{BousfieldKan} the $\lim^1$-term $\lim^1_k\pi_i(T_k\Emb^o_{\partial_0}(M,N),\iota)$ is countable. Since $\pi_i(T_k\Emb^o_{\partial_0}(M,N),\iota)$ is countable for all $k \geq 1$ (do an induction over $k$ and use the description of the layers of the tower), this implies that the tower $\{\pi_i(T_k\Emb^o_{\partial_0}(M,N),\iota)\}_{k\ge 1}$ is Mittag--Leffler by \cite[Theorem 2]{McGibbonMoller}.
	\end{enumerate}
\end{rem}

\noindent Combining \cref{thm:conv-smooth-codim0} with smoothing theory for embedding calculus, we will prove a similar result in the topological setting. Note that in view of \cref{rem:rem-smooth-conv} \ref{enum:rem-smooth-conv-2type}, the condition that $\partial_1M\subset M$ is a tangential $2$-type equivalence makes sense without smooth structures.

\begin{thm}[Topological convergence in codimension $0$]\label{thm:conv-top-codim0}For a quadruple $(M,\partial_0M,N,e_{\partial_0})$ with $m=n\ge5$, the target $N$ being smoothable and $\partial_1M\subset M$ a tangential $2$-type equivalence, the map 
\[
	\Emb^t_{\partial_0}(M,N)\lra T_\infty\Emb^t_{\partial_0}(M,N)
\]
is an equivalence. Moreover, for $n=5$, the smoothability condition on $N$ can be weakened to only assuming that the topological tangent bundle $N\ra \BTop(5)$ lifts along $\BO(5)\ra\BTop(5)$.
\end{thm}
	
\subsubsection{Smooth convergence in positive codimension}
We continue by stating the generalisation of \cref{thm:conv-smooth-codim0} to positive codimension. Note that for a smooth quadruple $(M,\partial_0M,N,e_{\partial_0})$ the normal bundle of an embedding $e\in\Emb^o_{\partial_0}(M,N)$ depends only on the image of $e$ in $T_\infty \Emb^o_{\partial_0}(M,N)$, since it can be recovered as the image under the composition 
\[
	\smash{T_\infty \Emb^o_{\partial_0}(M,N)\ra T_1 \Emb^o_{\partial_0}(M,N)\simeq \Map_{\partial_0}^{{/\BO(m)}}(M,N\times_{\BO(n)}\BO(n;m))\ra \Map(M,\BO(n-m))}.
\]
Here we used \cref{lem:positive-codim-t1-fact} and the equivalence $\BO(n;m) \simeq \BO(m) \times \BO(n-m)$. As a result, the map $\Emb^o_{\partial_0}(M,N)\ra T_\infty\Emb^o_{\partial_0}(M,N)$ decomposes as a disjoint union
\[
	\textstyle{\bigsqcup_{[\xi]}\Emb^o_{\partial_0}(M,N)_{[\xi]}\lra \bigsqcup_{[\xi]}T_\infty\Emb^o_{\partial_0}(M,N)_{[\xi]}}
\] 
indexed over isomorphism classes $[\xi]\in[M,\BO(n-m)]$ of $(n-m)$-dimensional vector bundles over $M$; here the $[\xi]$-subscripts indicate the collection of components with normal bundle isomorphic to $\xi$. We write $D(\xi)$ and $S(\xi)$ for the closed disc- and sphere bundle of a vector bundle $\xi$.

\begin{thm}[Smooth convergence]\label{thm:conv-smooth-general}Fix a smooth quadruple $(M,\partial_0M,N,e_{\partial_0})$ and $[\xi] \in[M,\BO(n-m)]$. If $n\ge5$ and $D(\xi)|_{\partial_1M}\cup S(\xi)\subset D(\xi)$ is a tangential $2$-type equivalence, then the map
\[
	\Emb^o_{\partial_0}(M,N)_{[\xi]}\ra T_\infty\Emb^o_{\partial_0}(M,N)_{[\xi]}
\] 
is an equivalence.
\end{thm}

\begin{rem}\label{rem:codim-three-smooth}Some remarks on the statement of \cref{thm:conv-smooth-general}:
\begin{enumerate}[leftmargin=*]
	\item\label{enum:pos-codim-condition-alternative} The tangential $2$-type condition is equivalent to the following two conditions:
	\begin{enumerate}[label=(\alph*)]
		\item\label{enum:admissible-a} the map $\partial_1M\cup_{S(\xi)|_{\partial_1M}}S(\xi)\ra M$ is an equivalence on fundamental groupoids and
		\item\label{enum:admissible-b} at all basepoints, the map $w_2(M)+w_2(\xi)\colon \pi_2(M)\ra\bfZ/2$ is trivial if and only if its precomposition with the map $\pi_2(\partial_1M\cup_{S(\xi)|_{\partial_1M}}S(\xi))\ra \pi_2(M)$ is trivial.
	\end{enumerate}
	Indeed, note that the projection induces an equivalence of pairs $(D(\xi),D(\xi)|_{\partial_1M}\cup S(\xi))\simeq (M,\partial_1M\cup_{S(\xi)|_{\partial_1M}}S(\xi))$, and  that $T(D(\xi))$ is the pullback of $TM\oplus \xi$ along the projection, so $w_2(TM\oplus \xi)=w_2(M)+w_2(\xi)+w_1(M)w_1(\xi)$. On the Hurewicz image the cup product $w_1(M)w_1(\xi)$ vanishes and thus $w_2(D(\xi))\colon \pi_2(D(\xi))\cong\pi_2(M)\ra \bfZ/2$ agrees with $w_2(M)+w_2(\xi)\colon \pi_2(M)\ra \bfZ/2$, so the claimed characterisation follows from \cref{rem:rem-smooth-conv} \ref{enum:rem-smooth-conv-2type}. This in particular implies that the condition only depends on the underlying spherical fibration of $\xi$.
	\item \label{enum:pos-codim-gkw} The statement of \cref{thm:conv-smooth-general} includes Goodwillie, Klein, and Weiss' convergence result in codimension $n-m\ge 3$, since in this case the condition in  \cref{thm:conv-smooth-general} is always satisfied: this follows from the characterisation in \ref{enum:pos-codim-condition-alternative} and the fact that the map  $\partial_1M\cup_{S(\xi)|_{\partial_1M}}S(\xi)\ra \partial_1M\cup_{\partial_1M}M\simeq M$ is $2$-connected, because it is part of a map of pushout squares whose other maps are  $\id\colon \partial_1M\ra \partial_1M$, $S(\xi)|_{\partial_1M}\ra \partial_1M$ and $S(\xi)\ra M$ which are all $2$-connected (the latter two because the fibres are spheres of dimension $\ge2$). In general the bottom right map in a map between pushout squares is $n$-connected if this holds for the other three maps (one way to see this is to use the groupoid version of the Seifert--van Kampen theorem together with the Mayer--Vietoris sequence with local coefficients).
\end{enumerate}
\end{rem}

\noindent In some cases, the tangential $2$-type condition on in \cref{thm:conv-smooth-general} can be further simplified:

\begin{lem}\label{lem:simpler-cond}In the situation of \cref{thm:conv-smooth-general}, if the codimension is $n-m=2$, $M$ is $1$-connected, and $\partial_0M=\partial M$ then $D(\xi)|_{\partial_1M}\cup S(\xi)=S(\xi)\subset D(\xi)$ is a tangential $2$-type equivalence if and only if 
\begin{enumerate}
	\item the Euler class $e(\xi)\in \oH^2(M;\bfZ)$ is indivisible and 
	\item the second Stiefel--Whitney class $w_2(M)\in  \oH^2(M;\bfZ/2)$ is nontrivial.
\end{enumerate}
\end{lem}
\begin{proof}We use the characterisation of the tangential $2$-type condition from \cref{rem:codim-three-smooth} \ref{enum:pos-codim-condition-alternative}. First we show that condition \ref{enum:admissible-a} is under the stated assumptions equivalent to the indivisibility of $e(\xi)\in \oH^2(M;\bfZ)$. The long exact sequence of $S(\xi)\ra M$ shows that $S(\xi)$ is $1$-connected if and only if the boundary morphism $\pi_2(M)\ra \pi_1(S^1)\cong\bfZ$ is surjective. The latter agrees with the Hurewicz isomorphism $\pi_2(M)\cong \oH_2(M)$ followed by evaluation of $e(\xi)$, so is surjective if and only if $e(\xi)$ is indivisible. This leaves us with showing that if \ref{enum:admissible-a} holds, then \ref{enum:admissible-b} is equivalent to $w_2(M)\neq 0$. Since $M$ and $S(\xi)$ are both $1$-connected, \ref{enum:admissible-b} is equivalent to the statement \begin{equation}\label{equ:sw-2-equivalence}
		\smash{\Big(w_2(M)=w_2(\xi)\in \oH^2(M;\bfZ/2)\Big) \Leftrightarrow\Big( \pi^*(w_2(M))=\pi^*(w_2(\xi))\in\oH^2(S(\xi);\bfZ/2)\Big)}
\end{equation} 
where $\pi\colon S(\xi)\ra M$ is the projection. Using that $e(\xi)$ and $w_2(\xi)$ agree modulo $2$, the $\bfZ/2$-Gysin sequence shows that the right hand side holds if and only if $w_2(M)\in\oH^2(M;\bfZ/2)$ is a multiple of $w_2(\xi)$. As $w_2(\xi)$ is nontrivial because of \ref{enum:admissible-a} and $e(\xi)=w_2(\xi)$ modulo $2$, this implies that \eqref{equ:sw-2-equivalence} is equivalent to the nontriviality of $w_2(M)$, as claimed.
\end{proof}

\begin{ex}\label{ex:convergence-cp2n}For an explicit example of convergence not captured by previous results, note that \cref{lem:simpler-cond} applies to the normal bundle of the inclusion $\bfC P^{2n}\subset \bfC P^{2n+1}$, so \cref{thm:conv-smooth-general} shows that embedding calculus converges for the components of the space of embeddings $\Emb^o(\bfC P^{2n},\bfC P^{2n+1})$ whose underlying immersion are regularly homotopic to the inclusion $\bfC P^{2n}\subset \bfC P^{2n+1}$.
\end{ex}

\subsubsection{Topological convergence in positive codimension}
We now turn to the analogue of \cref{thm:conv-smooth-general} in the topological category. Since not all topological embeddings admit normal bundles, there is no analogue of the map $\BO(n;m)\ra\BO(n-m)$ in the topological case, but we can use instead that there is a map $\BTop(n;m)\ra \BG(n-m) \coloneqq \BhAut(S^{n-m-1})$ resulting from $\bfR^{n}\backslash \bfR^{m}\simeq S^{n-m-1}$, and that the condition in \cref{thm:conv-smooth-general} only depends on the underlying spherical fibration of $\xi$ by \cref{rem:codim-three-smooth} \ref{enum:pos-codim-condition-alternative}. More precisely, for a quadruple $(M,\partial_0M,N,e_{\partial_0})$ we consider the composition 
\vspace{-0.1cm}
\begin{equation}\label{equ:comp-to-spherical}
	\hspace{-0.2cm}\begin{tikzcd}[row sep=0.3cm,column sep=0cm] T_\infty \Emb^t_{\partial_0}(M,N) \rar &[-20pt] T_1 \Emb^t_{\partial_0}(M,N)\simeq \Map_{\partial_0}^{/\BTop(m)}(M,N\times_{\BTop(n)}\BTop(n;m))  \ar[draw=none]{ld}[name=X, anchor=center]{} \ar[rounded corners, to path={ -- ([xshift=2ex]\tikztostart.east)	|- (X.center) \tikztonodes	-| ([xshift=-2ex]\tikztotarget.west) -- (\tikztotarget)}]{dl} \\
	 \Map^{/\BTop(m)}(M,\BTop(n;m)) \rar & \Map(M,\BG(n-m)); \end{tikzcd}
 \end{equation}
here we used \cref{lem:positive-codim-t1-fact} for the equivalence. Pulling back the decomposition of $\Map(M,\BG(n-m))$ into path components we arrive at a decomposition of $\Emb^t_{\partial_0}(M,N)\ra T_\infty\Emb^t_{\partial_0}(M,N)$ as \begin{equation}\label{equ:decomp-spherical-fibration}
	\smash{\textstyle{\bigsqcup_{[\xi]}\Emb^t_{\partial_0}(M,N)_{[\xi]}\lra \bigsqcup_{[\xi]}T_\infty\Emb^t_{\partial_0}(M,N)_{[\xi]}}}
\end{equation}
where $[\xi]$ runs over equivalence classes of $(n-m-1)$-dimensional spherical fibrations over $M$. The topological analogue of \cref{thm:conv-smooth-general} is phrased in terms of this decomposition:

\begin{thm}[Topological convergence]\label{thm:conv-top-general}Fix a quadruple $(M,\partial_0M,N,e_{\partial_0})$ and $[\xi]\in [M,\BG(n-m)]$. If $n\ge5$, $N$ is smoothable, and $S(\xi)\cup_{S(\xi)|_{\partial_1}}\partial_1M\ra M$ satisfies the condition in \cref{rem:codim-three-smooth} \ref{enum:pos-codim-condition-alternative}, then 
\[
	\Emb^t_{\partial_0}(M,N)_{[\xi]}\lra T_\infty\Emb^t_{\partial_0}(M,N)_{[\xi]}
\] 
is an equivalence. Moreover, for $n=5$, the smoothability condition on $N$ can be weakened to only assuming that the topological tangent bundle $N\ra\BTop(5)$ lifts along $\BO(5)\ra\BTop(5)$.\end{thm}

\begin{rem}\label{rem:simpler-cond-top} Some remarks on the statement of \cref{thm:conv-top-general}:
\begin{enumerate}[leftmargin=*]
	\item By the argument in \cref{rem:codim-three-smooth} \ref{enum:pos-codim-gkw}, the condition on $\xi$ in \cref{thm:conv-top-general} is redundant if the codimension satisfies $n-m\ge3$, so \cref{thm:conv-top-general} in particular gives $\smash{\Emb^t_{\partial_0}(M,N) \simeq T_\infty\Emb^t_{\partial_0}(M,N)}$ in codimension $\ge3$ as long as $N$ is smoothable and of dimension $\ge5$. 
	\item The simplification of the tangential $2$-type condition in \cref{lem:simpler-cond} only used the underlying spherical fibration of $\xi$, so it also applies to the situation of \cref{thm:conv-top-general}.
\end{enumerate}
\end{rem}

\subsection{Proofs of convergence}\label{sec:proof-convergence}It is time to prove the promised convergence results.

\subsubsection{Proof of \cref{thm:conv-smooth-codim0}}The proof is by reduction to the convergence result by Goodwillie, Klein, and Weiss. The key ingredient for this reduction is the following proposition. It involves the notion of a \emph{tangential $k$-type equivalence} for the inclusion $P\subset M$ of a submanifold with trivial normal bundle, which says that the map $\fib(P\ra\BO)\ra\fib(M\ra\BO)$ involving the stable tangent bundles is an equivalence on Postnikov $(k-1)$-truncations, or equivalently that there exists a factorisation $M\ra B\ra\BO$ such that the maps $M\ra B$ and $P\subset M\ra B$ are both $k$-connected.

\begin{lem}\label{lem:trivial-handles}Let $M$ be a compact smooth $m$-manifold and $\partial_1M\subset \partial M$ a compact smooth codimension $0$ submanifold, such that the inclusion $\partial_1M\subset M$ is a tangential $k$-type equivalence for some $0\le k<\min(m-2,m/2)$. Then $M$ can be obtained from $\partial_1M\times [0,1]$ by attaching trivial $k$-handles and arbitrary $>k$-higher handles.
\end{lem}

\begin{proof}
Fix a closed collar $c(\partial_1M)\subset M$ of $\partial_1M$. The hypothesis implies that $\partial_1 M\subset M$ is an equivalence on $(k-1)$-truncations, so in particular it is $(k-1)$-connected. As $k\le m-3$ so $k-1\le m-4$, it follows from handle trading  \cite[Theorem 3]{WallGeometrical} that $M$ can be obtained from $c(\partial_1M)$ by attaching handles of index $\ge k$, so the claim follows from showing that we can choose the $k$ handles to be trivially attached to $c(\partial_1M)$, that is, the attaching embedding $e\colon \sqcup^g S^{k-1}\times D^{m-k}\hookrightarrow \partial_1M$ is null-isotopic, i.e.\,isotopic to one that extends to an embedding $\sqcup^g D^k \times D^{m-k-1} \hookrightarrow \partial_1 M$ (by isotopy extension, $e$ then extends itself as well). It suffices to show this in the case $g=1$; if $g>1$ then isotope the attaching embedding of each handle separately, using that any null-isotopy is generically disjoint from the other handles, since $2k-1<m-1=\dim(\partial_1M)$. We thus assume $g=1$.

\medskip

\noindent Next we observe that $e\colon S^{k-1}\times D^{m-k}\hookrightarrow \partial_1M$ is null-isotopic if and only if it is null-concordant. The non-obvious direction is implied by either an elementary general position argument (given an concordance $E \colon [0,1] \times S^{k-1} \hookrightarrow [0,1] \times \partial_1 M$, the composition $\pr_2 \circ E$ generically only has intersections between slices $E|_{\{t\} \times S^{k-1}}$ for distinct $t$, using $2k \leq m-1$), or by concordance-implies-isotopy \cite[\S 2]{HudsonIsotopy} using that $k-1 \leq (m-1)-3$ holds by assumption.

\medskip

\noindent To construct a null-concordance, note that since $\partial_1 M\subset M$ is an equivalence on $(k-1)$-truncations and the attaching sphere $e_0\coloneqq e|_{S^{k-1}\times\{0\}}$ is nullhomotopic in $M$ since it bounds a handle, it is also nullhomotopic in $\partial_1 M$, so we may pick an arbitrary nullhomotopy of $e_0$ in $\partial_1 M$. This gives a homotopy $[0,1] \times S^{k-1} \times \{0\} \to [0,1] \times \partial_1 M$ between $\smash{e_0=e|_{S^{k-1} \times \{0\}}}$ and $\smash{i|_{S^{k-1} \times \{0\}}}$ for some embedding $i \colon D^k \times D^{m-k-1}\hookrightarrow \partial_1M$. By general position using that $k<2d$, we can assume that this homotopy is given by an embedding $E \colon [0,1] \times S^{k-1} \times \{0\} \hookrightarrow [0,1] \times \partial_1 M$ (that is, a concordance), so it remains to extend $E$ to an embedding on $[0,1] \times S^{k-1} \times D^{m-k}$, or equivalently that we can extend the given trivialisations of the normal bundle on $\{0,1\} \times S^{k-1}\times\{0\}$ to $[0,1] \times S^{k-1}\times\{0\}$. By potentially reversing the orientation of the second term in the domain of $i$, we can do so over an interval connecting both boundary components and it then remains to extend it over a $k$-disc. The obstruction to doing so is represented by the homotopy class of a map $u\colon S^k \to \BO(m-k)$ given by the normal bundle of an embedded sphere $e_u\colon S^k\hookrightarrow M$ obtained from gluing together $E$, $i$, and the bounding handle in $M$. By stability, since $2k<m$, the map $u$ is trivial if its stabilisation $u\colon S^k\ra \BO$ is. The latter is the composition of $e_u$ with the stable normal bundle of $M$, so since $\partial_1M\subset M$ is an equivalence on tangential $k$-types, we find an element in $ \pi_k(\partial_1 M)$ which maps under $\partial_1M\subset M\ra \BO$ to $-[u]$. By immersion theory and general position, this element can be represented by the normal bundle of an embedded sphere $e' \colon S^d \hookrightarrow [0,1] \times \partial_1 M$ disjoint from the image of $E$, so we can take the embedded connected sum of it with $E$ to remove the obstruction.\end{proof}

Using this proposition, we now prove \cref{thm:conv-smooth-codim0}. Given a smooth quadruple $(M,\partial_0M,N,e_{\partial_0})$ as in the statement, we view $M$ as a bordism $M\colon \partial_1M\leadsto \partial_0M$ of manifolds with boundary and apply \cref{lem:trivial-handles} for $k=2$ (note that $2<\min(m-2,m/2)$ holds exactly for $d\ge5$) to factor $M$ as a composition of bordisms $M_{\le 2}\colon \partial_1M\leadsto K$ and $M_{> 2}\colon K\leadsto \partial_0M$ where $M_{\le 2}$ is obtained from $\partial_1M$ by attaching trivial $2$-handles and $M_{>2}$ from $K$ by attaching handles of index $>2$. Fixing an embedding $e\in \Emb^o_{\partial_0M}(M_{>2},N)$ and abbreviating $N\backslash e\coloneqq N\backslash e(M_{>2}\backslash K)$, we consider the commutative diagram
\[
	\begin{tikzcd}
	\Emb^o_{K}(M_{\le 2},N\backslash e)\rar{\ext}\dar{\circled{3}}& \Emb^o_{\partial_0M}(M,N)\rar{\res}\dar{\circled{2}}&\Emb^o_{\partial_0M}(M_{>2},N)\dar{\circled{1}}\\
	T_\infty\Emb^o_{K}(M_{\le 2},N\backslash e)\rar{\ext}& T_\infty\Emb^o_{\partial_0M}(M,N)\rar{\res}&T_\infty\Emb^o_{\partial_0M}(M_{>2},N)
	\end{tikzcd}
\]
induced by restriction and extension by the identity (see \cref{rem:comp-and-gluing}). The map $\circled{1}$ is an equivalence by the convergence of embedding calculus in handle codimension $\ge3$ since $M_{>2}$ is obtained from $\partial_0M$ by attaching handles of index $<m-2$ (read the handle structure backwards). The upper row is a fibre sequence by ordinary isotopy extension, and the bottom row is a fibre sequence by isotopy extension for embedding calculus (see \cite[6.1, 6.5]{KnudsenKupers} and \cite[4.4]{KKDisc}). To show that $\circled{2}$ is an equivalence, it thus suffices to show that $\circled{3}$ is an equivalence for all choices of $e$. For this, note that since $M_{\le 2}\colon K\leadsto \partial_1 M$ is obtained by attaching \emph{trivial} $2$-handles to $\partial_1 M$, we can attach cancelling $3$-handles to $K$ to build a bordism $M^-_{\le 2}\colon \partial_1M\leadsto K$ such that the composed bordism $\smash{\overline{M}\coloneqq M^-_{\le 2}\cup_{K}M_{\le 2}\colon \partial_1M\leadsto \partial_1M}$ is trivial, i.e.\,diffeomorphic to $\partial_1M \times[0,1]$. Writing $\smash{\overline{N\backslash e}\coloneqq M^-_{\le 2}\cup_{e(K)}(N\backslash e)}$, we consider the diagram 
\[
	\begin{tikzcd}
	\Emb^o_{K}(M_{\le 2},N\backslash e)\rar{\ext}\dar{\circled{3}}& \Emb^o_{\partial_1 	M}(\overline{M},\overline{N\backslash e})\rar{\res}\dar{\circled{4}}&\Emb^o_{\partial_1 M}(M^-_{\le 2},\overline{N\backslash e})\dar{\circled{5}}\\
	T_\infty\Emb^o_{K}(M_{\le 2},N\backslash e)\rar{\ext}& \Emb^o_{\partial_1 M}(\overline{M},\overline{N\backslash e})\rar{\res}&T_\infty\Emb^o_{\partial_1 M}(M^-_{\le 2},\overline{N\backslash e}).\end{tikzcd}
\]
The maps $\circled{4}$ and $\circled{5}$ are an equivalence by the convergence of embedding calculus in handle codimension $\ge3$ since $M^-_{\le 2}$ is obtained from $\partial_1 M$ by attaching handles of index $m-3$ and $\overline{M}$ is diffeomorphic to a collar on $\partial_1M$. For the same reason as above, this is a map of fibre sequences, so $\circled{3}$ is an equivalence as well and the proof is finished.

\subsubsection{Proof of \cref{thm:conv-top-codim0}}
To deduce \cref{thm:conv-top-codim0} from \cref{thm:conv-smooth-codim0}, note that a smooth structure of $N$ induces a smooth structure on $\partial N$ and thus one on $\interior(\partial_0M)$ via $e_{\partial_0}$ such that $e_{\partial_0}$ becomes a smooth embedding. We fix such a choice. Up to removing $M$ and $N$ up to isotopy equivalence with $M'\coloneqq \interior(M)\cup\interior(\partial_0M)$ and $N'\coloneqq \interior(N)\cup\interior(e_{\partial_0}(N))$ we can consider $M$ and $N$ as nullbordisms of $\interior(\partial_0 M)$ in $\ncBordInf^{\ot}(m)$ (see \cref{rem:triads}). Now consider the diagram of categories (we omit $\interior(-)$, $\varnothing$, and $E_\varnothing$ from the notation for brevity)
\[
	\begin{tikzcd}[row sep=0.3cm]
	\ncBordInf^{\oo}(m)_{\partial_0M}\dar\rar &\ModInf^\un(E^{\oo}_m)_{E_{\partial_0M\times I}}\dar\rar&(\cS_{/\BO(m)})_{\partial_0M\times I/}\dar\\
	\ncBordInf^t(m)_{\partial_0M}\rar&\ModInf^\un(E^{\ot}_m)_{E_{\partial_0M\times I}}\rar&(\cS_{/\BTop(m)})_{\partial_0M\times I/}
	\end{tikzcd}
\]
obtained from \eqref{equ:smoothing-theory-diagram} by taking mapping categories from $\partial_0M$ to $\varnothing$. By  \cref{thm:smoothing-theory-emb-calc} both squares are pullbacks for $m \neq 4$. Since the rightmost vertical functor is a right-fibration and right-fibrations are stable under pullbacks, the middle vertical map is a right-fibration. The map in question is the map on mapping spaces of the bottom-left horizontal functor between $M$ and $N$, viewed as nullbordisms of $\partial_0M$. The chosen smoothing of $N$ and $e_{\partial_0}\colon \partial_0 M\hookrightarrow \partial N$ lifts $N\in \ncBordInf^t(m)_{\partial_0M}$ to $\ncBordInf^{\oo}(m)_{\partial_0M}$, so by an application of \cref{lem:criterion-mapping-space-equivalence} it suffices to show that the map $\Emb^{\oo}_{\partial_0}(M,N) \to T_\infty\Emb^{\oo}_{\partial_0}(M,N)$ induced by the top-left horizontal functor in the diagram on mapping spaces is an equivalence for all choices of smoothings of $M$ extending the given one on $\partial_0 M$. This holds by \cref{thm:conv-smooth-codim0}.

\medskip

\noindent To prove the addendum regarding the case $m=n=5$, recall that any chosen lift $N\ra\BO(5)$ of the topological tangent bundle is induced by a smooth structure on $N$ by smoothing theory \emph{as long as} one already knows that the restriction of this lift to $\partial N$ is induced by a smooth structure on $\partial N$. Smoothing theory fails in dimension $4$, so the latter is not automatic. However, by the sum-stable smoothing theorem \cite[p.~125]{FreedmanQuinn} there \emph{is} a smooth structure on $\partial N \sharp (S^2 \times S^2)^{\sharp g}$ for some $g\ge0$ inducing the given lift. Let us assume for the moment that $\partial N \backslash \partial_0 M \subset \partial N$ is $0$-connected, so all connected sums with $S^2 \times S^2$'s can be taken at $\partial N \backslash \partial_0 M$. Then we have a commutative diagram
\begin{equation}\label{equ:retract-diagram-conv}
	\begin{tikzcd}[row sep=0.3cm] \Emb^t_{\partial_0}(M,N) \rar{\inc} \dar & \Emb^t_{\partial_0}(M,N \natural (S^2 \times D^3)^{\natural g}) \rar{\inc} \dar{\simeq} &  \Emb^t_{\partial_0}(M,N) \dar \\
	T_\infty \Emb^t_{\partial_0}(M,N) \rar{\inc} & T_\infty \Emb^t_{\partial_0}(M,N \natural (S^2 \times D^3)^{\natural g}) \rar{\inc} &  T_\infty \Emb^t_{\partial_0}(M,N) \end{tikzcd}
\end{equation}
induced by the inclusion of $N$ into $N \natural (S^2 \times D^3)^{\natural g}$ and the inclusion of $N \natural (S^2 \times D^3)^{\natural g}$ into $N \natural (D^5)^{\natural g} \cong N$. The middle arrow is an equivalence by the previous case, since $N \natural (S^2 \times D^3)^{\natural g}$ is smoothable by the above discussion. Both rows are compatibly homotopic to the identity, so the outer vertical map is an equivalence because it is a retract of one. To deal with the case where $\partial N \backslash \partial_0 M \subset \partial N$ is not $0$-connected, note that $\partial_1 M \subset M$ is an equivalence of tangential 2-types, so in particular $0$-connected. We thus find embedded tubes $\sqcup^k D^4 \times [0,1] \hookrightarrow M$ connecting each component of $\partial_0 M$ to a component of $\partial_1 M$. Setting $M'$ to be the closure of the complements of these tubes in $M$ and $\partial_0M'\coloneqq M'\cap\partial_0M$, we consider the commutative diagram
\[
	\begin{tikzcd}[row sep=0.3cm]
	\Emb^t_{\sqcup^k D^4\times\{0\}\cup \partial D^4\times [0,1]}(\sqcup^k D^4 \times [0,1],N\backslash e)\rar{\ext}\dar& \Emb^t_{\partial_0}(M,N)\rar{\res}\dar& \Emb^t_{\partial_0}(M',N)\dar\\
	T_\infty\Emb^t_{\sqcup^k D^4\times\{0\}\cup \partial D^4\times [0,1]}(\sqcup^k D^4 \times [0,1],N\backslash e)\rar{\ext}&  T_\infty\Emb^t_{\partial_0}(M,N)\rar{\res}& T_\infty\Emb^t_{\partial_0}(M',N)\end{tikzcd}
\]
for embeddings $e\in\Emb^t_{\partial_0}(M',N)$. Since $\partial N \backslash \partial_0 M' \to \partial N$ is $0$-connected and $\partial_1M'\subset M'$ an equivalence on tangential $2$-types by general position the right-hand vertical map is an equivalence by the previous case and the same also holds after replacing $M'$ by $M'\sqcup (D^d)^{\sqcup r}$ for all $r\ge0$. As $D^4 \times [0,1]$ is a collar on $D^4\times\{0\}\cup \partial D^4\times [0,1]$, the left-hand vertical map is an equivalence as well. By isotopy extension for topological embedding calculus (see \cref{sec:topological-isotopy-extension}), this is a map of fibre sequences, so the middle vertical arrow is an equivalence as well and we conclude the claim.

\begin{rem}[Connectivity of the maps to the finite stages]
\label{rem:top-convergence-rate}\,
\begin{enumerate}[leftmargin=*]
	\item\label{enum:connectivity-t-k} Replacing the Morita category $\ModInf(E_{m}^{c})$ for $c\in\{\oo,\ot\}$ by its truncated analogue $\ModInf_k(E_{m}^{c})$ for some fixed $k\ge1$, arguing as in the first part of the previous proof shows that under the hypotheses of \cref{thm:conv-top-codim0}, the map $\smash{ \Emb^\ot_{\partial_0}(M,N) \to T_k \Emb^\ot_{\partial_0}(M,N)}$ is $r$-connected for some fixed $r\ge0$ if this is the case for its smooth analogue $\smash{\Emb^\oo_{\partial_0}(M,N) \to T_k \Emb^\ot_{\partial_0}(M,N)}$ with respect to all smoothings of $M$ relative to a fixed smoothing of $\partial_0M$. If we assume that $\partial_1 M \subset M$ is $2$-connected and $m \geq 6$, then $M$ admits with respect to any smoothing relative to $\partial_0M$ a smooth handle decomposition relative to $\partial_0 M$ with only handles of index $\le m-3$ by handle trading \cite[Theorem 3]{WallGeometrical}. Then \cite[Corollary 2.5, 5.1]{GoodwillieWeiss} implies that  $\smash{\Emb^\oo_{\partial_0}(M,N)\simeq T_\infty\Emb^\oo_{\partial_0}(M,N) \to T_k \Emb^\oo_{\partial_0}(M,N)}$ is at least $(k-m+4)$-connected, so the same is true for $\smash{\Emb^\ot_{\partial_0}(M,N)\simeq T_\infty\Emb^\ot_{\partial_0}(M,N) \to T_k \Emb^\ot_{\partial_0}(M,N)}$. 
	\item \label{enum:connectivity-t-k-dim-5} The conclusions of \ref{enum:connectivity-t-k} also hold for $m=5$, but this requires an additional argument, since even if $\partial_1M\subset M$ is $2$-connected, $M$ need not have a handle decomposition relative to $\partial_0M$ with only $\le m-3=2$-handles for $m=5$. However by \cite[Theorem 1.2]{QuinnStable}, there \emph{is} such a handle decomposition after replacing $M$ by $M\natural (S^2\times D^3)^{\natural k}$ for some $k$ (the boundary connected sums are taken at discs in $\partial_1M$). Then one can argue as in the final part of the proof of \cref{thm:conv-top-codim0}, using the retract diagram \eqref{equ:retract-diagram-conv} with $T_\infty\Emb^t_{\partial_0}(-,-)$ replaced by $T_k\Emb^c_{\partial_0}(-,-)$ for $c\in\{o,t\}$.
	\item  One should also be able to obtain the connectivity estimates in \ref{enum:connectivity-t-k} and \ref{enum:connectivity-t-k-dim-5} for the maps $T_\infty\Emb^c_{\partial_0}(M,N)\ra T_k\Emb^c_{\partial_0}(M,N)$ directly from the layer identification in \cref{cor:layer-description-explicit} which would have the advantage that it gives a connectivity estimate for $T_\infty\Emb^\ot_{\partial_0}(M,N) \to T_k \Emb^\ot_{\partial_0}(M,N)$ \emph{without} the assumption that $N$ is smoothable, in which case it is not known (but likely) that the map $\smash{\Emb^\ot_{\partial_0}(M,N)\ra T_\infty\Emb^\ot_{\partial_0}(M,N)}$ is an equivalence. 
\end{enumerate}
\end{rem} 

\subsubsection{Proof of \cref{thm:conv-smooth-general}} We reduce smooth convergence in positive codimension to codimension zero as in the proof of \cite[Lemma 2.9]{KKsurfaces}. In the notation used here, Lemma 2.7 loc.cit.~provides a commutative diagram consisting of two pullback squares
\begin{equation}\label{equ:smooth-thickening}
	\begin{tikzcd}[row sep=0.3cm] \Emb_{\partial_0}^\oo(D(\xi),N) \rar \dar &[-5pt] T_\infty \Emb_{\partial_0}^\oo(D(\xi),N) \rar \dar &[-5pt] \Map^{/\BO(n)}_{\partial_0}(D(\xi),N) \dar \\
	\Emb_{\partial_0}^\oo(M,N)_{[\xi]} \rar & T_\infty \Emb_{\partial_0}^\oo(M,N)_{[\xi]} \rar & \Map^{/\BO(m)}_{\partial_0}(M,N \times_{\BO(n)} \BO(n;m))_{[\xi]}.\end{tikzcd}
\end{equation}
Here $\partial_0 D(\xi)\coloneqq D(\xi)|_{\partial_0M}$, and the additional subscript $(-)_{[\xi]}$ indicates we restrict to those components that map to $[\xi] \in \Map_{\partial_0}(M,\BO(n-m))$. This makes the rightmost vertical map surjective on components, so the same holds for the other two vertical maps. Hence the bottom-left horizontal map is an equivalence if and only if the top-left horizontal map is. The latter holds by \cref{thm:conv-smooth-codim0} because $D(\xi)$ and $N$ are of the same dimension and $\partial(D(\xi))\backslash \interior(D(\xi)|_{\partial_0M})=D(\xi)|_{\partial_1M}\cup_{S(\xi)|_{\partial_1M}} S(\xi)\subset D(\xi)$ is a tangential $2$-type equivalence by assumption.

\subsubsection{Proof of \cref{thm:conv-top-general}} We try to proceed analogously to the smooth case, but will encounter additional difficulties due to a lack of topological normal bundles. Let us be more precise about this: for a locally flat $m$-dimensional submanifold $M\subset W$ in a $n$-dimensional manifold $W$, the map $T^{\ot}M\times T^{\ot}W|_{M}\colon M\ra \BTop(n)\times \BTop(m)$ has a preferred lift $\nu(W;M)$ along $\BTop(n;m)\ra \BTop(n)\times \BTop(m)$.
An \emph{open normal bundle} of $M$ in $W$ is a lift of $\nu(W;M)$ along $\BTop(n-m)\times\BTop(m)\ra \BTop(n;m)$, a \emph{closed normal bundle} is a lift along the map $\BHomeo(D^{n-m})\times\BTop(m)\ra \BTop(n;m)$ induced by the map $\BHomeo(D^{n-m})\ra\BTop(n-m)$ that takes interiors, and a \emph{normal vector bundle} is a lift along  $\BO(n-m)\times\BTop(m)\ra \BTop(n;m)$. There are examples of locally flat submanifolds $M\subset W$ which do not admit an open normal bundle \cite{RourkeSanderson}, and hence also no closed or vector bundle one. Conversely to the existence of normal bundles, given a lift
\begin{equation}\label{equ:lift-xi}
\begin{tikzcd}
&\BTop(n;m)\dar\\
M\arrow[ur,dashed,"\xi"]\arrow["T^{\ot}M",swap,r]&\BTop(m),
\end{tikzcd}
\end{equation}
one can ask whether there is a manifold $W$ containing $M$ that ``realises'' $\xi$, in the following senses:

\begin{dfn}\label{def:thickenings}Fix a lift $\xi$ as in \eqref{equ:lift-xi}.
\begin{enumerate}[leftmargin=*]
	\item A \emph{$\xi$-thickening} of $M$ is a $n$-dimensional manifold $W$ with a codimension $0$ submanifold $\partial_0W\subset \partial W$ such that $W$ and $\partial_0W$ contain $M$ and $\partial_0M$ as a locally flat submanifold respectively such that $\partial W\cap M=\partial_0M$, together with a homotopy between $\nu(W;M)$ and $\xi$ over $\BTop(m)$.
	\item A \emph{compact $\xi$-thickening} is a $\xi$-thickening such that $W$ and $\partial_0W$ are compact.
	\item\label{enum:mcg-thickening} A \emph{mapping cylinder $\xi$-thickening} is a  compact $\xi$-thickening for which there exists a map $\pi\colon \partial_1W\coloneqq \partial W\backslash \interior(\partial_0W)\ra M$ with $\pi^{-1}(\partial_0M)=\partial_0W\cap\partial_1W$ and a homeomorphism $h\colon \cyl(\pi)\cong W$ extending $\id_M$ such that $h^{-1}(\partial_0W)=\cyl(\pi_{\partial_0W\cap\partial_1W})$ where $\pi_{\partial_0W\cap\partial_1W}\colon \partial_0W\cap\partial_1W\ra \partial_0M$ is the restriction (see \cref{fig:mapping-cylinder} for an example).
\end{enumerate}
\end{dfn}

\begin{figure}
	\begin{tikzpicture}[scale=1.5]
		\foreach \n in {0,...,10} 
		{
			\draw [thin,black!20!white] (0.2*\n,1) -- (0.2*\n,0.05);
		}
		\foreach \n in {1,...,5} 
		{
			\draw [thin,black!20!white] (2+0.2*\n,1) -- (2+0.01*\n,0.05-0.005*\n);
		}
		\foreach \n in {1,...,9} 
		{
			\draw [thin,black!20!white] (3,1-.2*\n) -- (2+0.05,0.025-0.005*\n);
		}
		\foreach \n in {1,...,5} 
		{
			\draw [thin,black!20!white] (2+0.2*\n,-1) -- (2+0.01*\n,-0.05+0.005*\n);
		}
		\foreach \n in {0,...,10} 
		{
			\draw [thin,black!20!white] (0.2*\n,-1) -- (0.2*\n,-0.05);
		}
		\draw [ultra thick,green!50!black] (0,0) -- (2,0);
		\draw [ultra thick] (0,1) -- (3,1) -- (3,-1) -- (0,-1);
		\draw (0,1) -- (0,-1);
		\node at (1.5,-1) [below] {$\partial_1 W$};
		\node at (0,.5) [left] {$\partial_0 W$};
		\node at (0,0) [left] {$\partial_0 M$};
		\node at (1,0) [below] {$M$};
		\node at (2,0) [below] {$\partial_1 M$};
		\node at (1.5,.5) [fill=white] {$W$};
	\end{tikzpicture}
	\caption{A mapping cylinder thickening $W$ of $M$ as in \cref{def:thickenings} \ref{enum:mcg-thickening}. The grey lines indicate the map $\pi \colon \partial_1 W \to M$. Note that $\pi$ need not be injective.}
	\label{fig:mapping-cylinder}
\end{figure}
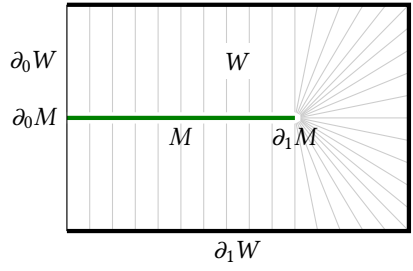

\begin{rem}\label{rem:codim-2-thickening}
Note that if $\xi$ lifts to $\BTop(n-m)\times \BTop(m)\ra \BTop(n;d)$, then the total space of the associated $\bfR^{n-m}$-bundle gives a $\xi$-thickening (strictly speaking only after removing a collar of $\partial_1M$ from $M$), and if $\xi$ lifts along $\BHomeo(D^{n-m})\times \BTop(m)\ra \BTop(n;m)$ then the total space of the associated $D^{n-m}$-bundle gives a mapping cylinder $\xi$-thickening. In particular, if $n-m\le 2$, then any $\xi$ has a mapping cylinder thickening, since the map $\BO(n-m)\times \BTop(m)\ra \BTop(n;m)$ is $(m+1)$-connected in this case (see \cite[Lemma 8.16]{KKDisc}, which relies on \cite[Theorem B]{KirbySiebenmannCodim} and \cite[9.3A, 9.3D]{FreedmanQuinn}). For $n \leq 4$ this is true in any codimension $n-m$ using the existence and uniqueness of smooth structures for $n \leq 3$ \cite{Moise} and \cite[9.3A, 9.3D]{FreedmanQuinn} for $n=4$.
\end{rem}

As a preparation to proving \cref{thm:conv-top-general}, we record two facts related to $\xi$-thickenings:

\begin{lem}\label{lem:shrink-to-mc}Fix a lift $\xi$ as in \eqref{equ:lift-xi}. If $n\ge5$ and $n-m\ge 3$, then any $\xi$-thickening of $M$ contains a mapping cylinder $\xi$-thickening of $M$.
\end{lem}

\begin{proof}Given a $\xi$-thickening $\partial_0W\subset W$ of $M$, we like to apply \cite[Theorem 14, 15]{Pedersen} to find compact codimension $0$-submanifolds $W'\subset W$ and $\partial_0W'\subset \partial_0W$ such that $\partial_0W\subset \partial W$, and $M\subset W'$, $\partial_0M\subset \partial_0W'$ are locally flat submanifolds with $\partial W\cap M=\partial_0M$, and such that $\partial_0W\subset W$ is homeomorphic to a mapping cylinder as in the definition of mapping cylinder $\xi$-thickening. To extract this from loc.cit.\,we use, firstly, that the ``handlebody'' and ``stable hypothesis'' in Theorem 14 loc.cit.\,are satisfied as a result of \cite[III.2.1]{KirbySiebenmann} and \cite[2.3.1]{QuinnIII} for the former and the latter by \cite{KirbyStable} and \cite[2.2.2]{QuinnIII}  (the $4$-dimensional results are needed to deal with boundaries of $5$-manifolds), secondly, that the assumption $n\ge6$ in \cite[Theorem 15]{Pedersen} can be weakened to $n\ge5$ in the case of manifolds, and finally that the proofs of \cite[Theorem 14, 15]{Pedersen} in the case $\partial_0M=\partial M$ can be extended to the more general case by slightly adapting the handle inductions involved. Since $W'\subset W$ is codimension $0$, we have $\nu(W';M)\simeq\nu(W;M)$ and as $W$ is a $\xi$-thickening we have $\nu(W';M)\simeq\xi$, so $(W',\partial_0W')$ is mapping cylinder $\xi$-thickening as required. 

\medskip

\noindent We point out for the interested reader that \cite[Theorem 3.6]{Johnson}, \cite[3.1.1]{QuinnI}, \cite[4.1]{QuinnControl}, and \cite[Theorem 2.1]{Edwards} contain results similar to the ones by Pedersen used above. In particular, the statement of \cite[3.1.1]{QuinnI} includes the relative case needed above.
\end{proof}

\begin{lem}\label{lem:relative-connectivity}
If $n\ge m\ge 0$, the map $(-)\times\bfR \colon \BTop(n-1,m-1)\ra\BTop(n,m)$ is $(m+1)$-connected.

\end{lem}
\begin{proof}For $n-m\le 2$ or $n \leq 4$, this follows from \eqref{equ:fix-or-move} and the $(m+1)$-connectivity of the map $\BO(n-m)\times \BTop(m)\ra \BTop(n;m)$ mentioned in \cref{rem:codim-2-thickening}. For $n \geq 5$ and $n-m\ge 3$, we use that the forgetful map $\Top(n)/\Top(n,m)\ra G(n)/G(n-m)$ is $(2n-m-3)$-connected by work of Haefliger and Millett \cite[p.\,147]{LashofEmbedding}, so as $\BG(n-1)\ra \BG(n)$ is $(n-1)$-connected by Freudenthal's suspension theorem, and $2(n-1)-(m-1)-3\ge n-2$, it follows that $\Top(n-1)/\Top(n-1,m-1)\ra \Top(n)/\Top(n,m)$ is $(n-2)$-connected. As $\BTop(n-1)\ra\BTop(n)$ is $(n-1)$-connected \cite[V.§5.(4)]{KirbySiebenmann} and $n-m\ge 3$, it follows that $\BTop(n-1,m-1)\ra \BTop(n,m)$ is $(m+1)$-connected.\end{proof}

To prove \cref{thm:conv-top-general}, we refine the decomposition \eqref{equ:decomp-spherical-fibration} to be indexed over the components of the space $\smash{\Map^{/\BTop(m)}(M,\BTop(n;m))}$ in \eqref{equ:comp-to-spherical} as opposed to those of $\smash{\Map(M,\BG(n-m))}$. For the rest of this section, we fix a class $[\xi]\in \pi_0\Map^{/\BTop(m)}(M,\BTop(n;m))$ such that the target in
\begin{equation}\label{equ:xi-decomposition}
	\Emb^{\ot}_{\partial_0}(M,N)_{[\xi]}\lra T_\infty\Emb^{\ot}_{\partial_0}(M,N)_{[\xi]}
\end{equation}
is nonempty. \cref{thm:conv-top-general} follows once we show that this map is an equivalence if $\xi$ satisfies the assumption in the statement. If $m=n$ this is \cref{thm:conv-top-codim0}, so we assume $m<n$. First we prove:

\begin{lem}\label{lem:proof-with-thickening}If $n\ge5$, $N$ is smoothable, and there exists a mapping cylinder $\xi$-thickening $(W,\partial_0W,\pi,h)$ of $M$ such that $\partial_1W\subset W$ is a tangential $2$-type equivalence, then \eqref{equ:xi-decomposition} is an equivalence.
\end{lem}

\begin{proof}Assuming for a moment that there exists an embedding $e_{\partial_0W}\colon \partial_0W\hookrightarrow \partial N$ extending $e_{\partial_0M}\colon \partial_0M\hookrightarrow \partial M$, we consider the following topological analogue of the diagram \eqref{equ:smooth-thickening} 
\[
	\hspace{-0.6cm}\begin{tikzcd}[column sep=0.2cm,row sep=0.3cm] \Emb_{\partial_0}^t(W,N) \rar{\simeq} \dar & T_\infty \Emb_{\partial_0}^t(W,N) \rar \dar & \Map^{/\BTop(n)}_{\partial_0}(W,N)\simeq \Map^{/\BTop(n;m)}_{\partial_0}(M,N\times_{\BTop(n)} \BTop(n;m))\dar \\
	\Emb_{\partial_0}^t(M,N)_{[\xi]} \rar & T_\infty \Emb_{\partial_0}^t(M,N)_{[\xi]} \rar & \Map^{/\BTop(m)}_{\partial_0}(M,N \times_{\BTop(n)} \BTop(n;m))_{[\xi]},\end{tikzcd}
\]
where the top-right equivalence is induced by the fact that $M\subset W$ is an equivalence since $W=h(\cyl(\pi))$, and that $M\subset W\ra \BTop(n)$ factors as $\nu(W;M)$ followed by $\BTop(n;m)\ra \BTop(n)$. The top left map is an equivalence by \cref{thm:conv-top-codim0}, by the assumption on $W$. The rightmost vertical map is induced by $\BTop(n;m)\ra \BTop(m)$ and surjects onto the components indicated with the $[\xi]$ subscript since $\nu(W;M)\simeq \xi$ because $W$ is a $\xi$-thickening. If we knew that
\begin{enumerate}
	\item \label{enum:top-conv-i}the right square is a pullback and
	\item\label{enum:top-conv-ii} the outer square is a pullback,
\end{enumerate}
then it would follow that the middle vertical map is also surjective on components and that the left square is a pullback, so the bottom left horizontal map would be an equivalence as claimed. To finish the proof, we are thus left to justify \ref{enum:top-conv-i} and \ref{enum:top-conv-ii}, as well as
\begin{enumerate}[resume]
	\item \label{enum:top-conv-iii} an extension  $e_{\partial_0W}$ of $e_{\partial_0M}$ as assumed above exists.
\end{enumerate}
The argument for \ref{enum:top-conv-i} is the essentially same as that in the proof of \cite[Lemma 2.7]{KKsurfaces}: applying topological manifold calculus (see \cref{sec:manifold-calc}) to the presheaf $\smash{M\supset U\mapsto\Emb^{\ot}_{\partial_0 W}(h(\cyl(\pi_{U})),N)}$ on the poset $\cO(M)$ of open subsets of $M$ containing a collar on $\partial_0M$ where $\pi_{U}\colon \pi^{-1}(U)\ra U$ is the restriction, arguing as in the cited proof reduces the claim to finding a complete Weiss $\infty$-cover $\cU \subset \cO(M)$ (see \cref{sec:descent}) such that for all $U \in \cU$ the maps $\smash{\Emb^{\ot}_{\partial_0}(U,N) \to T_\infty\Emb^{\ot}_{\partial_0}(U,N)}$ and $\smash{\Emb^{\ot}_{\partial_0W}(h(\cyl(\pi_{U})),N) \to T_\infty\Emb^{\ot}_{\partial_0 W}(h(\cyl(\pi_{U})),N)}$ are equivalences.

\medskip

\noindent We now produce such a cover. By gluing on external collars, we may assume that there is a collar $\partial_0 W \times [0,\infty)$ of $\partial_0 W$ which is respected by the mapping cylinder $\xi$-thickening. We then take $\cU$ to be the complete Weiss $\infty$-cover $\cU \subset \cO$ given by those open subsets $U \subset M$ that are the disjoint union of $k$ open $m$-discs in $M$ and a collar on $\partial_0 M$, such that (a) the collar is given as $\partial_0 M \times [0,\epsilon)\subset \partial_0 W \times [0,\infty)$ for some $\epsilon>0$, and (b) each of the open $m$-discs is contained in a chart exhibiting $M$ as locally flat. Since $U$ is a disjoint union of collar and some discs, the map $\smash{\Emb^{\ot}_{\partial_0}(U,N) \to T_\infty\Emb^{\ot}_{\partial_0}(U,N)}$ is an equivalences, so we are left to argue that $\smash{\Emb^{\ot}_{\partial_0W}(h(\cyl(\pi_{U})),N) \to T_\infty\Emb^{\ot}_{\partial_0 W}(h(\cyl(\pi_{U})),N)}$ is one as well. To do so, we show that $h(\cyl(\pi_{U}))$ is isotopy equivalent relative to $\partial_0 W$ to a subset $U'\subset W$ that clearly has this property. Namely, we choose $U' \subset W$ so that $U' \cap \partial W = \partial_0 W$, $U' \cap M = U$ and there is a homeomorphism of pairs $(U',U) \cong (\partial_0 W \times [0,\infty),\partial_0 M \times [0,\infty)) \sqcup \ul{k} \times (\bfR^n,\bfR^m)$. To prove that $U'$ and $h(\cyl(\pi_U))$ are isotopy equivalent, note that by construction, there is an equality $U' \cap M = h(\cyl(\pi_U)) \cap M$. This induces a bijection between their components so it suffices to construct the isotopy equivalence between each component of $U'$ and the corresponding component of $h(\cyl(\pi_U))$ separately. For the collar components there is nothing to do, as they are equal as a consequence of the assumption that the mapping cylinder $\xi$-thickening respects the fixed collar. To give the construction for the disc components we may assume $k=1$ and $\partial_0 W = \varnothing$, and identify $U$ with $\bfR^m$ and $U'$ with $\bfR^{n} = \bfR^m \times \bfR^{n-m}$. Then there exist for all integers $N \geq 0$ a real number $\delta_N>0$ such that $\{(x,y)\in \bfR^m \times \bfR^{n-m}\mid |x| \leq N+1,\,|y|\leq \delta_N\} \subset h(\cyl(\pi_U))$, and picking a continuous function $\rho \colon \bfR_{\geq 0} \to \bfR_{\geq 1}$ so that $1/\rho(x) < \delta_{\lfloor x \rfloor}$ the self-embedding of $\bfR^{n}$ given by $(x,y) \mapsto (x,y/\rho(|x|))$ has image in $h(\cyl(\pi_U))$ and is isotopic to the identity. Similarly, there is for all $N \geq 0$ an $\epsilon_N$ such that $h(\{(z,t) \mid |\pi(z)| \leq N+1,\,|z| \geq 1-\epsilon_N\})$ is contained in $U'$, and picking a continuous function $\sigma \colon \bfR_{\geq 0} \to \bfR_{\geq 1}$ so that $1/\sigma(x) < \epsilon_{\lfloor x \rfloor}$ the self-embedding of $h(\cyl(\pi_U))$ given by $h(z,t) \mapsto h(z,(1-1/\sigma(|x|))t)$ has image in $U'$ and is isotopic to the identity (see \cref{fig:isotopy-piu}).
	
\medskip
	
\noindent To show \ref{enum:top-conv-ii}, we identify the rightmost vertical map (without the $[\xi]$-subscript) via topological immersion theory \cite{Lees, Gauld,LashofImmersion, Kurata} with the restriction map $\smash{\Imm^\ot_{\partial_0}(W,N)\ra \Imm^\ot_{\partial_0}(M,N)}$. Here the domain is the space of topological immersions $f\colon M\looparrowright N$ with $f|_{\partial_0M}=e_{\partial_0 M}$ and $f^{-1}(\partial N)=\partial_0M$. The target is defined similarly, using $W$ and $e_{\partial_0W}$. By isotopy extension, the map on vertical fibres over an embedding $e\in\Emb^{\ot}_{\partial_0}(W,M)$ of the outer square thus agrees with the forgetful map $\Emb_{\partial_0}^t(W,N\,\mathrm{rel}\,M) \to \Imm_{\partial_0}^t(W,N\,\mathrm{rel}\,M)$ between the spaces of embeddings respectively immersions that extend the given embeddings $e_{\partial_0 W}, e$ on $\partial_0 W$ and $M$. This is an equivalence because any compact family of locally flat immersions that are embeddings on $M \cup \partial_0 W$ consists of embeddings on an open neighbourhood $V$ of $M \cup \partial_0 W$, and we can isotope the mapping cylinder neighbourhood $W$ into $V$ fixing a neighbourhood of $M \cup \partial_0 W$. 

\medskip
	
\noindent We are left to show \ref{enum:top-conv-iii}. It suffices to construct an extension of $e_{\partial_0M}\colon \partial_0M\hookrightarrow \partial N$ to an immersion $f_{\partial_0M}\colon \partial_0W\looparrowright \partial N$, since the latter will be an embedding in a neighbourhood of $\partial_0M$ in $ \partial_0W$ and we can isotope $\partial_0W=h(\cyl(\pi_{\partial_0W\cap\partial_1W}))$ into any neighbourhood of ${\partial_0M}$. Translated via immersion theory as recalled above, the underlying immersion of $e_{\partial_0M}$ corresponds to a map $\partial_0M\ra \partial N \times_{\BTop(n-1)}\BTop(n-1;m-1)$ over $\BTop(m-1)$ and the task is to extend it to a map over $\BTop(n-1;m-1)$ with respect to the map $\nu(\partial_0W;\partial_0M)\colon \partial_0M\ra \BTop(n-1;m-1)$, in other words to solve the dashed lifting problem in
\[
	\begin{tikzcd}
	\{0,1\}\times\partial_0M\dar{\subset}\rar&\BTop(n-1;m-1)\dar\rar &\BTop(n;m)\dar\\
	\left[0,1\right]\times\partial_0M\arrow[dashed,ur]\arrow[dotted,rru]\rar&\BTop(m-1)\rar&\BTop(m)
	\end{tikzcd}
\]
The class $\xi$ induces a dotted lift as indicated, and this can be extended to dashed lift as required by obstruction theory, since $\partial_0M$ is $(m-1)$-dimensional and the map between vertical fibres of the right hand square is $(m+1)$-connected by \cref{lem:relative-connectivity}.
\end{proof}

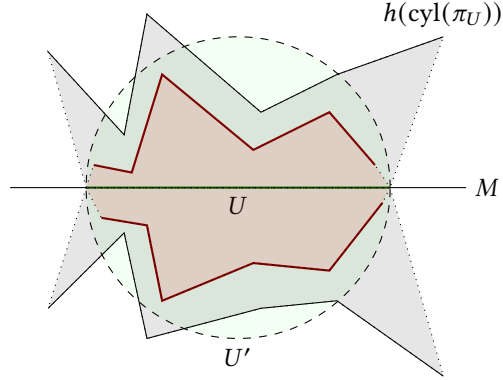
\begin{figure}
	\begin{tikzpicture}
		\draw [dashed,fill=green!5!white] (0,0) circle (2cm);
		\draw (-3,0) -- (3,0);
		\draw [thick,green!50!black] (-2,0) -- (2,0);
		\node at (3,0) [right] {$M$};
		\node at (0,-2) [below] {$U'$};
		\node at (0,0) [below] {$U$};
		
		\draw [dotted,fill=black,fill opacity=.1](-2,0) -- (-2.5,1.8) -- (-1.5,.7) -- (-1.2,2.3) -- (0.3,1) -- (1.3,1.5) -- (2.7,2) -- (2,0) -- cycle;
		\node at (2.7,2) [above] {$h(\cyl(\pi_U))$};
		\draw (-2.5,1.8) -- (-1.5,.7) -- (-1.2,2.3) -- (0.3,1) -- (1.3,1.5) -- (2.7,2);
		\draw [dotted,fill=red,fill opacity=.1](-2,0) -- (-1.9,.3) -- (-1.4,.2) -- (-1,1.5) -- (0.2,.5) -- (1.2,1) -- (1.8,.3) -- (2,0) -- cycle;
		\draw [thick,red!50!black](-1.9,.3) -- (-1.4,.2) -- (-1,1.5) -- (0.2,.5) -- (1.2,1) -- (1.8,.3);
		
		\draw [dotted,fill=black,fill opacity=.1](-2,0) -- (-2.5,-1.6) -- (-1.5,-.6) -- (-1.2,-2) -- (0.3,-1.6) -- (1.3,-1.5) -- (2.7,-2.5) -- (2,0) -- cycle;
		\draw (-2.5,-1.6) -- (-1.5,-.6) -- (-1.2,-2) -- (0.3,-1.6) -- (1.3,-1.5) -- (2.7,-2.5);
		\draw [thick,red!50!black] (-1.8,-.4) -- (-1.2,-.5) -- (-1,-1.5) -- (0.2,-1) -- (1.2,-1.1) -- (1.9,-0.2);
		\draw [dotted,fill=red,fill opacity=.1](-2,0) -- (-1.8,-.4) -- (-1.2,-.5) -- (-1,-1.5) -- (0.2,-1) -- (1.2,-1.1) -- (1.9,-0.2) -- (2,0) -- cycle;
	\end{tikzpicture}
	\caption{By pushing along the interval direction of the mapping cylinder increasingly much as we approach the boundary of $U$, we can isotope $h(\cyl(\pi_U))$ (the dark region) into $U'$ (the green region), with image given by the red region.}
	\label{fig:isotopy-piu}
\end{figure}

To conclude the proof of \cref{thm:conv-top-general}, we are left to show:
\begin{lem}
A mapping cylinder $\xi$-thickening as in \cref{lem:proof-with-thickening} always exists.
\end{lem}

\begin{proof}For $n-m\le 2$ or $n \leq 4$, we can use the thickening from \cref{rem:codim-2-thickening} whose inclusion $\partial_1W\subset W$ is given by $D(\xi')|_{\partial_1M}\cup_{S(\xi')_{\partial_1M}}S(\xi')\subset D(\xi')$ where $(\xi',T^{\ot} M)\colon M\ra\BO(n-m)\times\BTop(m)$ is a lift of $\xi$. Since the composition $\BO(n-m)\ra \BTop(n;m)\ra \BG(n-m)$ is the forgetful map, $\smash{D(\xi')|_{\partial_1M}\cup_{S(\xi')_{\partial_1M}}S(\xi')\subset D(\xi')}$ is a tangential $2$-type equivalence if the conditions in \cref{rem:rem-smooth-conv} \ref{enum:rem-smooth-conv-2type} are satisfied for the $S^{n-m-1}$-fibration underlying $\xi$. This holds by assumption.

\medskip

\noindent For $n-m\ge 3$, we proceed differently: since $\smash{\Map_{\partial_0}^{\BTop(m)/}(M,N\times_{\BTop(n)}\BTop(n;m))_{[\xi]}}$ was assumed to be nonempty, immersion theory as recalled in the proof \cref{lem:proof-with-thickening} ensures that there exists an immersion  $f\colon M\looparrowright N$ with $f|_{\partial_0M}=e_{\partial_0 M}$ and $f^{-1}(\partial N)=\partial_0M$ whose induced lift $M\ra \BTop(n;m)$ of $T^{\ot}M\colon M\ra\BTop(m)$ is equivalent to $\xi$. Every such immersion admits an \emph{induced neighbourhood}, meaning that there is a $n$-dimensional manifold $U$ together with a codimension $0$ submanifold $\partial_0U\subset \partial U$ such that $U$ and $\partial_0U$ contain $M$ and $\partial_0M$ as locally flat submanifold with $\partial U\cap M=\partial_0M$ and such that $f\colon M\looparrowright N$ extends to an immersion $F\colon U\looparrowright N$ with $F^{-1}(\partial N)=\partial_0U$ (see \cite[Section 3]{Kurata}, following \cite[\S 7]{HaefligerPoenaru}). Since $F$ extends $f$, it follows that this induced neighbourhood is a $\xi$-thickening, and applying \cref{lem:shrink-to-mc}, we can shrink it to a mapping cylinder $\xi$-thickening $(W,\partial_0W)$. Since $W\cong \cyl(\pi \colon \partial_1W\ra M)$, the inclusion $\partial_1W\subset M$ is equivalent to $W\backslash M\subset W$, so it is a tangential $2$-type equivalence because it is even $2$-connected since it is the inclusion of the complement of a locally flat submanifold of codimension $\ge3$ (see e.g.\,\cite[Theorem 2]{EilenbergWilder} for a reference for this fact that applies to topological manifolds).
\end{proof}
	
\subsection{The Alexander trick for homology $4$-spheres} \label{sec:grw-alexander-extension} As a sample application of our improved convergence results, we extend the results in \cite{GRWContractible} on one-sided $h$-cobordisms and homeomorphisms of contractible manifolds from dimensions $d\ge6$ to $d=5$:

\begin{thm}\label{thm:grw-extended}
Theorems A--C in \cite{GRWContractible} as well as the results in Section 4 loc.cit.\,hold for $d=5$.
\end{thm}

\begin{proof}Fix a compact $d$-dimensional one-sided  $h$-cobordism $C \colon B \leadsto B'$ in the sense of Definition 1.1 loc.cit., a $d$-manifold $M$, and an embedding $e\colon B\hookrightarrow \partial M$. The ``first proof'' of Theorem C in loc.cit.\,in dimensions $d\ge6$ proceeds in two steps: it first reduces the statement to $\Emb^\oo_B(C,M) \to T_\infty \Emb^\oo_B(C,M)$ being an equivalence and then uses that this is true, by Goodwillie, Klein, and Weiss' convergence results. The first step goes through for $d\ge 3$, the second step needs $d\ge6$. For $d=5$, we can use \cref{thm:conv-smooth-codim0} instead, since the inclusion $B'\cong \partial M\backslash \mathrm{int}(B)\subset M$ is an equivalence on tangential $2$-types because it is an equivalence by assumption. Using this extension of Theorem C loc.cit.\,to dimension $d=5$, the proof of Theorem B loc.cit.\,applies verbatim. The same holds for the proof of Theorem A on the contractibility of $\Homeo_\partial(\Delta)$ for compact contractible topological $d$-manifolds $\Delta$ \emph{as long as $\Delta$ admits a smooth structure}. In loc.cit.\,it is used that this is automatic for $d\ge6$ by smoothing theory but this is not available for $d=5$ if we do not already have a smooth structure on the boundary, so we need to argue differently: 

\medskip

\noindent The reduction of Theorem A to Theorem C in loc.cit.\,goes through in the topological category and reduces the claim to the space of topological embeddings $\Emb^t_{\partial \Delta}(\Delta\backslash\interior(D^d),\Delta)$ being contractible. A mild adaptation of the argument in the proof of Theorem C loc.cit.\,further reduces the latter statement to showing that topological embedding calculus approximation $\Emb^t_{\partial \Delta}(\Delta\backslash\interior(D^d),\Delta)\ra T_\infty\Emb^t_{\partial \Delta}(\Delta\backslash\interior(D^d),\Delta)$ is an equivalence. By \cref{thm:conv-top-codim0}, for this it suffices that $\Delta$ is \emph{formally} smoothable, i.e.\,that $\Delta\ra \BTop(5)$ lifts to $\BO(5)$. This holds trivially since $\Delta$ is contractible.

\medskip

\noindent The deduction of the results in Section 4 loc.cit.\,from Theorem C works verbatim for $d=5$.
\end{proof}

\begin{rem}Galatius and Randal-Williams observe in \cite[p.~2]{GRWContractible} that their Theorem C combined with the $s$-cobordism theorem recovers the classification of one-sided $h$-cobordisms in dimension $d\ge 6$. The $s$-cobordism theorem fails smoothly for $d=5$, so \cref{thm:grw-extended} does not yield the analogous result in that dimension. It does, however, prove that a $5$-dimensional one-sided $h$-cobordism has a one-sided inverse (for $h$-cobordisms this is originally due to Stallings \cite[Theorem 4]{Stallings}), since  any one-sided $h$-cobordism $C \colon B \leadsto B'$ of dimension $d\ge5$ embeds by \cref{thm:grw-extended} into $B\times [0,1]$ and the complement gives a one-sided inverse.\end{rem}

\appendix

\section{Truncation of unital operads}\label{sec:truncation} In this appendix, we establish the results announced in \cref{sec:truncation-operads}: we introduce the categories $\Opd_{\le k}^{\un}$ of $k$-truncated unital operads, show that they fit into a converging tower 
\begin{equation}\label{equ:truncation-tower-app}
	\Opd^{\le \bullet,\un}=\Big(\Opd^\un=\Opd^{\infty,\un} \lra \cdots \lra  \Opd^{\le 2,\un}\lra \Opd^{\le 1,\un}\simeq\Cat\Big)\in\Tow(\Cat),
\end{equation}
and prove that all functors in this tower admit explicit right adjoints which are fully faithful. 

\begin{rem}The results in this appendix were independently proven as part of a preprint by Dubey--Liu \cite[Theorem 5.17]{DubeyLiu} which appeared while this work was finalised. \end{rem}

\subsection{Dendroidal Segal spaces}In this appendix, we use a model of the category of unital operads $\Opd^\un$ different to the one used in the body of this work: \emph{complete dendroidal closed Segal spaces}. We summarise the relevant definitions and results.
\subsubsection{Operads and dendroidal spaces}\label{sec:op-dend}
Recall from \cite[Section 3.2]{HeutsMoerdijk} the category $\Omega$ of (finite and rooted) trees. The inclusion of linear trees gives a fully faithful functor $i\colon \Delta\rightarrow \Omega$ from the simplicial category (see p.\,96 loc.cit.). A \emph{dendroidal space} $X$ is a space-valued presheaf on $\Omega$. A \emph{complete dendroidal Segal space} is a dendroidal space $X$ which 
\begin{enumerate}[leftmargin=*]
	\item satisfies the \emph{Segal condition}: for any tree $\smash{T=T_1\cup_eT_2}$ obtained by grafting a tree $T_2$ along its root onto a leaf $e$ of a tree $T_1$, the map $\smash{X(T)\ra X(T_1)\times_{X(e)}X(T_2)}$ induced by the inclusions $e\subset T_i\subset T$ for $i=1,2$ is an equivalence (see Lemma 12.7 loc.cit.), and
	\item is \emph{complete}: $i^*X\in\PSh(\Delta)$ is complete Segal space in the sense of Rezk (see p.\,491 loc.cit.).
\end{enumerate}

\noindent By combining results of Barwick \cite[Section 10]{Barwick} and Chu--Haugseng--Heuts \cite[Theorem 1.1]{ChuHaugsengHeuts}, there is an equivalence of categories
\begin{equation}\label{equ:dendroidal-lurie-comp}
	\PSh(\Omega)^{\seg,\com}\simeq \Opd
\end{equation}
where $\PSh(\Omega)^{\seg,\com}\subset\PSh(\Omega)$ is the full subcategory on the complete dendroidal Segal spaces. Recall that the nerve $N\colon \Cat\ra\PSh(\Delta)$ (i.e.\,the Yoneda embedding followed by restriction along the inclusion $\Delta\subset\Cat$) induces an equivalence 
\begin{equation}\label{equ:rezk-equ}
	\Cat\simeq \PSh(\Delta)^{\seg,\com}
\end{equation} 
onto the full subcategory $\PSh(\Delta)^{\seg,\com}\subset \PSh(\Delta)$ of complete Segal spaces \cite[Theorem 4.11]{JoyalTierney}. Via this equivalence and \eqref{equ:dendroidal-lurie-comp}, the functor $(-)^{\col}\colon \Opd\ra \Cat$ from \cref{sec:operads-intro} agrees with the restriction $i^\ast\colon \PSh(\Omega)^{\seg,\com}\ra \PSh(\Delta)^{\seg,\com}$. In particular, the core $(\cO^{\col})^{\simeq}\in\cS$ of the category of colours of an operad $\cO\in\Opd$ agrees with the value $X_{\cO}(\eta)\in \cS$ of the corresponding complete dendroidal Segal space $X_{\cO}\in\PSh(\Omega)^{\seg,\com}$ at the unique tree $i([0])\coloneqq \eta\in\Omega$ without vertices \cite[p.\,92]{HeutsMoerdijk}. The space of multi-operations $\Mul_{\cO}((c_i)_{i\in S};c)$ can be recovered as follows: writing $C_{S}$ for the \emph{$S$-corolla}---the tree with one vertex and leaves indexed by $S$ (see p.\,93 loc.cit.)---we have
\begin{equation}\label{equ:dendroidal-multiops}
	\textstyle{\Mul_\cO((c_i)_{i\in S}; c)\simeq  \fib_{((c_i)_{i\in S},c)}\Big(X_{\cO}(C_S)\ra (\bigsqcap_{s\in S}X_{\cO}(\eta))\times X_{\cO}(\eta)\Big)}
\end{equation}
where the map we take fibres of is induced by the maps $\eta\ra C_S$ in $\Omega$ corresponding to the inclusion of the leaves and the root. To obtain the operadic composition, given finite sets $S\in \Fin$ and $S_s\in \Fin$ for $s\in S$, one considers the unique leaf-preserving morphism $l\colon C_{\sqcup_{s\in S} S_s}\ra T_S$ where $T_S$ is the result of of grafting $C_{T_{s}}$ for $s\in S$ to $C_S$ at the $s$th leaf. Using the Segal property, this gives a map $X_{\cO}(C_S)\times_{\bigsqcap_{s\in S}X_\cO(\eta)}\bigsqcap_{s\in S}X_{\cO}(C_{S_s})\ra X_{\cO}(C_{\sqcup_{s\in S} S_s})$ which corresponds to the operadic composition. From \eqref{equ:dendroidal-multiops} one sees in particular that $\cO$ is unital if and only if the map $X(C_0)\ra X(\eta)$ induced by the inclusion $\eta\ra C_0$ of the root is an equivalence. Writing $\PSh(\Omega)^{\seg,\com,\wc}\subset \PSh(\Omega)^{\seg,\com}$ for the full subcategory on those $X$ for which this is the case, \eqref{equ:dendroidal-lurie-comp} thus restricts to an equivalence \begin{equation}\label{equ:dendroidal-lurie-comp-unital}
	\PSh(\Omega)^{\seg,\com,\wc}\simeq \Opd^{\un}.
\end{equation}

\subsubsection{Unital operads and closed dendroidal spaces}\label{sec:op-dend-un}
Following \cite[p.\,92, 97]{HeutsMoerdijk}, we write $\cOmega\subset \Omega$ for the full subcategory of \emph{closed trees}, i.e.\,trees with no leaves. A \emph{closed dendroidal space} is a space-valued presheaf on $\cOmega$. The inclusion $\iota\colon \cOmega\hookrightarrow \Omega$ has a left adjoint $\cl\colon  \Omega\ra \cOmega$ induced by sending a tree $T$ to its \emph{closure} $\smash{\overline{T}}$, obtained from $T$ by adding a new vertex at the end of each leaf (see p.\,98 loc.cit.). Left Kan extension along $\iota$ induces an equivalence
\begin{equation}\label{equ:weakly-closed-equ}
	\iota_!\colon \PSh(\cOmega)\xlra{\simeq } \PSh(\Omega)^\wc
\end{equation}
from the category of closed dendroidal spaces to the category to the full subcategory $\PSh(\Omega)^\cl\subset \PSh(\Omega)$ of \emph{weakly closed} dendroidal spaces, which are dendroidal spaces $X$ for which the map $\smash{X(\overline{T})\ra X(T)}$ induced by the inclusion $T\ra\overline{T}$ is an equivalence for all trees. Moreover, this left Kan extension $\iota_!$ agrees with the restriction $\cl^*$ along the closure functor (see p.\,510 loc.cit.). For dendroidal \emph{Segal} spaces $X$, the condition of being closed is equivalent to the condition considered above, namely that $X$ maps $\eta\ra\overline{\eta}=C_0$ to an equivalence (see Proposition 12.49 loc.cit.). Combining \eqref{equ:dendroidal-lurie-comp-unital} and \eqref{equ:weakly-closed-equ}, we thus obtain an equivalence
\begin{equation}\label{equ:untial-op-dend}
	\PSh(\cOmega)^{\seg,\com}\simeq  \Opd^{\un}
\end{equation} 
where $\PSh(\cOmega)^{\seg,\com}\subset \PSh(\cOmega)$ is the full subcategory of closed dendroidal spaces whose associated dendroidal space is complete and satisfies the Segal condition.

\begin{rem}\label{rem:poset-model} One advantage of $\cOmega$ over $\Omega$ is that there is the following equivalent way to think about $\cOmega$ (see \cite[p.\,97]{HeutsMoerdijk}): $\cOmega$ is equivalent to the category of finite posets $(E,\le)$ with a unique maximal element such that for every $e\in E$ the subposet $E_{\le e}\coloneqq \{x\in E\mid e\le x\}$ is totally ordered. Morphisms are given by map of posets $\varphi\colon (E,\le) \ra (E',\le)$ that preserve the incompatibility relation, i.e.\,for $e,e'\in E$ with $e'\not\ge e\not\le e'$, we have $\varphi(e')\not\ge \varphi(e)\not\le \varphi(e')$. For a closed tree $T$, the associated poset $(E(T),\le )$ is given by its set of edges with the relation that $e\le e'$ if $e'$ is ``closer to the root'', i.e.\,if the unique shortest path from the root to $e$ contains $e'$ (see page 92 loc.cit.). Writing $\oP(E(T))$ for the poset of subsets of $E(T)$ ordered by inclusion, we have a functor $\oP(E(T))\ra \cOmega_{/T}$ sending $A\subset E(T)$ to its \emph{downward-completion} $T_A\coloneqq \{e\in E(T)|\exists a\in A\colon e\ge a\}\cup\{\text{root}\}\subset E(T)$ which we can think of as a tree with a map to $T$.\end{rem}

\subsection{Truncated operads} Writing $\cOmega_{\leq k}\subset \cOmega$ for the full subcategory on those closed trees all of whose vertices have at most $k$ incoming edges, we have inclusions of full subcategories $\smash{\cOmega_{\leq 1} \subset \cOmega_{\leq 2} \subset \cdots \subset \cOmega_{\le \infty}\eqcolon\cOmega}$. The first subcategory has a simple description: it is isomorphic to $\Delta$ induced by corestricting the composition $(\cl\circ\iota)\colon \smash{ \Delta\ra \Omega\ra\cOmega}$. The tower of full subcategories induces by restriction a tower of categories that factors $(-)^\col\colon \Opd^{\un}\ra\Cat$
\begin{equation}\label{equ:truncated-operads}
	{\Opd^{\un}\overset{\eqref{equ:untial-op-dend}}{\simeq} \PSh(\cOmega_{\le \infty})^{\seg,\com} \lra \cdots \lra \PSh(\cOmega_{\leq 2})^{\seg,\com} \lra \PSh(\cOmega_{\leq 1})^{\seg,\com}\overset{\eqref{equ:rezk-equ}}{\simeq}}\Cat
\end{equation}
where $\smash{\PSh(\cOmega_{\leq k})^{\seg,\com}\subset \PSh(\cOmega_{\leq k})}$ is the full subcategory on those presheaves that satisfy the analogue of the Segal and completeness condition for trees in $\cOmega_{\leq k}$ (note that the completeness condition for a dendroidal space only depends on its restriction to $\cOmega_{\leq 1}$, but the Segal condition does not). The tower \eqref{equ:truncated-operads} is our preferred model for truncation of unital operads, meaning that we \emph{define} \eqref{equ:truncation-tower-app} as \eqref{equ:truncated-operads}. In particular, the category of \emph{$k$-truncated unital operads} is
\[
	\Opd^\un_{\leq k} \coloneqq \PSh(\cOmega_{\leq k})^{\seg,\com}\quad\text{for }1\le k\le \infty.
\]
By mimicking the dendroidal description of spaces of multi-operations and operadic composition from \cref{sec:op-dend-un}, a $k$-truncated operad $\cO\in \Opd^\un_{\leq k}$ has a category of colours $\cO^{\col}$, spaces of multi-operations $\Mul_{\cO}((c_i)_{i\in S};c)$ for all sets $S\in \Fin_{\le k}$ of cardinality at most $k$ and composition maps between them that satisfy the classical axioms of an operad up to higher coherent homotopy. 

\begin{lem}The tower \eqref{equ:truncated-operads} in $\Cat$ converges.\end{lem}

\begin{proof}By the same argument as for \cref{prop:tower} \ref{enum:tower-convergence}, the map $\smash{\PSh(\cOmega_{\leq \infty})\ra \lim_k\PSh(\cOmega_{\leq k})}$ is an equivalence, so it suffices to show that this equivalence restricts to an equivalence when imposing the completeness and Segal conditions. For the former, this follows from the observation that for $X\in\PSh(\cOmega_{\leq \infty})$ being complete only depends on $X|_{\cOmega_{\leq 1}}$. For the latter, it follows by observing that the Segal condition resulting from a grafting of a fixed tree $T$ depends only on $X|_{\cOmega_{\leq k}}$ for a choice of finite $k$ such that all vertices in $T$ have $\le k$ incoming edges.
\end{proof}
The strategy of proof for the following result was kindly suggested to us by Gijs Heuts:

\begin{thm}\label{thm:truncation}For $1\le k\le j\le \infty$, truncation $\tau^\ast\colon \Opd^\un_{\le j} \ra \Opd^\un_{\leq k}$ has a fully faithful right adjoint 
\[
	\tau_\ast\colon \Opd^\un_{\le k} \lra \Opd^\un_{\leq j}.
\] 
For $\cO\in \Opd^\un_{\le j}$, a finite set $S\in \Fin_{\le k}$, and colours $c,c_i\in \cO^{\col}$ for $i\in S$, there is an equivalence
\[
	\Mul_{(\tau_*\cO)}((c_i)_{i\in S}; c) \simeq \lim_{T \subseteq S,\,\lvert T\rvert \leq k} \Mul_{\cO}((c_i)_{i\in T}; c)
\]
where the limit is induced by restriction along the inclusion of corollas $C_T\ra C_S$ for $T\subset S$.
\end{thm}

\noindent As preparation, we show the following lemma in which we use the notation of \cref{rem:poset-model}:

\begin{lem}\label{lem:resolve-by-leaves}Fix trees $S\in\cOmega_{\le k}$ and $T\in \cOmega\backslash \cOmega_{\le k}$ and consider the subset $E(T)_{\min}\subset E(T)$ of minimal elements in $(E(T),\le)$ (i.e.~the leaves). Then the map
\[
	\colim_{A\subsetneq E(T)_{\min}}\Map_{\cOmega}(S,T_A)\lra  \Map_{\cOmega}(S,T)
\]
induced by the functor $T_{(-)}\colon \oP(E(T))\ra \cOmega_{/T}$ is an equivalence. 
\end{lem}

\begin{proof}From the poset-perspective on $\cOmega$ (see \cref{rem:poset-model}) it is clear that for $A\subset E(T)$, the map $\Map_{\cOmega}(S,T_A)\ra \Map_{\cOmega}(S,T)$ is an inclusion of components and for $A,B\subset T$ we have $\Map_{\cOmega}(S,T_A)\cap \Map_{\cOmega}(S,T_B)=\Map_{\cOmega}(S,T_{A\cap B})$, so by the same argument as in the middle part of the proof of \cref{thm:layers-tower}, the colimit in the statement is given by the union, so it suffices to show that any map $\varphi\colon S\ra T$ factors through $T_A$ for some $A\subsetneq E(T)_{\min}$. 

\medskip
	
\noindent To show this, we choose each $s\in E(S)_{\min}$ a minimal edge $\overline{s}\in E(T)_{\min}$ with $\overline{s}\le \varphi(s)$ and set $A_\varphi$ to be the union of all the $\overline{s}$. Since for any $s'\in S$ there is a minimal edge $s$ with $s\le s'$, we have $\overline{s}\le \varphi(s)\le \varphi(s')$ so $\varphi(s')$ is contained in the downward completion $T_{A_\varphi}$ of $A_\varphi$ and thus $\varphi$ factors through $T_{A_\varphi}$, so we are left with arguing that $A_\varphi$ is a \emph{proper} subset of $E(T)_{\min}$. We assume the contrary and argue by contradiction: using that $T\not\in\cOmega_{\le k}$, choose pairwise different edges $t_1,\ldots t_{k+1}\in E(T)$ and an edge $t\in E(T)$ such for each $i$ we have $t_i< t$ for all $i$ and the $t_i$ is minimal with respect to this property (in other words: a sub-$(k+1)$-corolla in $T$). Since we assumed $A_\varphi = E(T)_{\min}$ we find $s_i\in E(S)_{\min}$ with $\overline{s_i}\le t_i$. Note that the $\overline{s_i}$ are pairwise different (and hence also the $s_i$) since the $t_i$ are incomparable and $E(T)_{\ge \overline{s_i}}$ is totally ordered. Now choose $s\in E(S)$ with $s_i<s$ for all $i$ and $s$ is maximal with this property, and choose for each $i$ an edge $s_i'\in E(S)$ with $s_i\le s_i'<s$ that is maximal with respect to this property. We now claim that the $s_i'$ are pairwise different so span a sub-$(k+1)$-corolla in $S$ which would contradict $S\in\cOmega_{\le k}$. In a moment, we will show that  $t_i\ge \varphi(s_i')$. Assuming this for now, we can finish the proof: if $s_i'=s_j'$ for $i\neq j$ then  $t_i\ge \varphi(s_i')=\varphi(s_i')\le t_j$, so since $E(T)_{\ge \varphi(s_i')}$ is totally ordered $t_i$ and $t_j$ are comparable which is not the case. To show the remaining claim that $t_i\ge \varphi(s_i')$, we use $s_i'\ge s_i$, so $\varphi(s_i')\ge \varphi(s_i)\ge \overline{s_i}$. Because $t_i\ge \overline{s_i}$ and $E(T)_{\ge \overline{s_i}}$ is totally ordered, we have  $t_i\ge \varphi(s_i')$ or $t_i<\varphi(s_i')$ but in the latter case we conclude $t\le \varphi(s_i')$ so as $E_{\ge t}(T)$ is totally ordered, it follows that all the $\varphi(s_i')$s are comparable which is a contradiction since the $s_i'$ are incomparable and $\varphi$ preserves incomparable elements.
\end{proof}

\begin{proof}[Proof of \cref{thm:truncation}] From the discussion in Sections \ref{sec:op-dend} and \ref{sec:op-dend-un}, we see that the category $\Opd^\un_{\leq k} = \PSh(\cOmega_{\leq k})^{\seg,\com}$ can be written as a reflexive localisation \[\PSh(\cOmega_{\leq k})^{\seg,\com} \simeq \PSh(\cOmega_{\le k})[W_{\seg,k}^{-1},W_{\com,k}^{-1}]\] 
that inverts two classes of morphisms: interpreting trees as representable presheaves, the first class $W_{\seg,k}$ consists of the maps $T_1 \cup_{e} T_2 \to T$ as in the Segal condition with $\smash{T,T_1,T_2\in \cOmega_{\le k}}$ and the second class $W_{\com,k}$ consists of the two maps given by applying $(\cl^*\circ i_!) \colon \PSh(\Delta) \to \PSh(\cOmega)$ to the two endpoint inclusions $\ast \to J$ into the nerve of the groupoid $J$ with two objects and a unique isomorphism between them, followed by restriction along $\cOmega_{\le k}\subset \cOmega$ . Note that before localisation the restriction $\tau^*\colon \PSh(\cOmega_{\le j}) \lra \PSh(\cOmega_{\leq k})$ has a right adjoint $\tau_\ast$ given by right Kan extension, and that the localisation $L_k\colon \PSh(\cOmega_{\leq k})\ra \PSh(\cOmega_{\leq k})^{\seg,\com}$ has a right adjoint given by the subcategory inclusion $j\colon \PSh(\cOmega_{\leq k})^{\seg,\com}\hookrightarrow \PSh(\cOmega_{\leq k})$, so the composition $(\tau_\ast\circ j)$ is a right adjoint to $(L_k\circ\tau^*)$. To prove that the asserted right adjoint exists, it thus suffices (see e.g.\,\cite[Proposition 7.1.14]{Cisinki}) to show that $(L_k \circ \tau^*)$ sends the morphisms in $W_{\seg,j}$ and $W_{\com,j}$ to equivalences in $\PSh(\cOmega_{\leq k})^{\seg,\com}$; then $(L_j\circ\tau_*\circ j)$ is the desired right-adjoint. Since $\tau^*(W_{\com,j})=W_{\com,k}$ by definition, we only need to show that for $T_1 \cup_{e} T_2 \to T$ in $W_{\seg,j}$, the map $L_k(\tau^*(T_1 \cup_{e} T_2))\ra  L_k(\tau^*(T))$ is an equivalence in $\PSh(\cOmega_{\leq k})^{\seg,\com}$. We will prove this by an induction on $\lvert E(T)_{\min}\rvert$. If $\lvert E(T)_{\min}\rvert\le k$, then we have $T,T_1,T_2\in \cOmega_{\le k}$, so $\tau^*(T_1 \cup_{e} T_2))\ra \tau^*(T)$ is in $W_{\seg,k}$, so becomes an equivalence after applying $L_k$. If $\lvert E(T)_{\min}\rvert>k$, then if $T\in \cOmega_{\le k}$ the previous argument applies, so we may assume $T\not\in \cOmega_{\le k}$ in which case \cref{lem:resolve-by-leaves} together with the fact that colimits of presheaves are computed objectwise shows that $\tau^*(T)$ is the colimit of the diagram $E(T)_{\min}\supsetneq A\mapsto \tau^*(T_A)$. The same argument also shows that $\tau^*(T_1 \cup_{e} T_2))$ is the colimit of $E(T)_{\min}\supsetneq A\mapsto \tau^*((T_1\cap T_A)\cup_{e\cap T_A}(T_2\cap T_A))$, so the map $\tau^*(T_1 \cup_{e} T_2)\ra  \tau^*(T)$ is a colimit of maps of the form $\tau^*(T_A)\ra \tau^*((T_1\cap T_A)\cup_{e\cap T_A}(T_2\cap T_A))$ whose value under $L_k$ is an equivalence by induction since $\lvert E(T_A)_{\min}\rvert<\lvert E(T)_{\min}\rvert$. As $L_k$ is a left adjoint, so preserves colimits, this implies the claim. 

\medskip
	
\noindent By construction, the unit of the constructed adjunction is obtained from the unit of the adjunction $\tau^*\dashv\tau_*$ before localisation by restricting along $j$ and applying $L_k$. Since the unit before localisation is an equivalence as $\cOmega_{\le k}\subset \cOmega_{\le j}$ is fully faithful, the unit of the constructed adjunction is an equivalence too, so the right adjoint is fully faithful as claimed. To justify the final assertion about spaces of operations, note that for $X\in \PSh(\cOmega_{\leq k})^{\seg,\com}$ and $S\in \Fin$ with $\lvert S\rvert\le j$, we have $\smash{L_j(\tau_*(j(X)))(\overline{C}_S) \simeq\Map_{\PSh(\cOmega_{\leq k})^{\seg,\com}}(L_j(\tau^*(\overline{C}_S)),X)}$ by the Yoneda lemma and adjunction. An iterated application of  \cref{lem:resolve-by-leaves} shows that $\smash{\tau^*(\overline{C}_S)\simeq \colim_{T\subset S, \lvert T\rvert\le k}(\tau^*(\overline{C}_T))}$. As $\smash{\overline{C}_T\in\cOmega_{\le k}}$ so $\smash{\overline{C}_T}$ is $L_j$-local, we conclude $L_j(\tau_*(j(X)))(\smash{\overline{C}_S})\simeq \holim_{T\subset S, \lvert T\rvert\le k}X(\smash{\overline{C}_T})$ which gives the claimed identification of the spaces of multi-operations by taking fibres of the corresponding map to the appropriate product of $X(\smash{\overline{\eta}})$s.\end{proof}

\bibliographystyle{amsalpha}
\bibliography{literature}

\bigskip

 \end{document}